\title{Spherical twists and Lagrangian spherical manifolds}

\documentclass[11pt]{article}

\usepackage{amsfonts}
\usepackage{amsthm}
\usepackage{latexsym, amssymb, bbm}
\usepackage{xcolor}
\usepackage{graphicx}
\usepackage{amsmath}
\usepackage{enumerate}
\usepackage[all]{xy}
\usepackage[margin=1.1in]{geometry}
\usepackage[mathscr]{eucal}
\usepackage{latexsym, amssymb, bbm, mathrsfs,mathabx}
 \usepackage{caption} \usepackage{subcaption}
\usepackage{amscd}

\usepackage{amsmath,amsthm,amsfonts,amssymb,graphicx,psfrag}
\usepackage{color}

\usepackage{hyperref}

\usepackage{tikz}
\usepackage{comment}
\usepackage{tikz-cd}
\newtheorem{thm}{Theorem}[section]

\newtheorem{prop}[thm]{Proposition}
\newtheorem{lemma}[thm]{Lemma}
\newtheorem{rmk}[thm]{Remark}
\newtheorem{question}[thm]{Question}
\newtheorem{defn}[thm]{Definition}
\newtheorem{corr}[thm]{Corollary}

\newtheorem{convention}[thm]{Convention}

\theoremstyle{definition}
\newtheorem{defnlemma}[thm]{Definition-Lemma}
\newtheorem{example}[thm]{Example}

\newtheorem{assumption}[thm]{Assumption}

\DeclareMathOperator{\perf}{perf}
\DeclareMathOperator{\cchar}{char}
\DeclareMathOperator{\rank}{rank}

\DeclareMathOperator{\sign}{sign}
\DeclareMathOperator{\ind}{ind}

\DeclareMathOperator{\virdim}{virdim}
\DeclareMathOperator{\mbb}{mb}
\DeclareMathOperator{\critp}{critp}

\newcommand{\wt}[1]{\widetilde{#1}}

\newcommand{\ov}[1]{\overline{#1}}

\newcommand{\bdf}{\begin{defn}}
\newcommand{\edf}{\end{defn}}

\newcommand{\bthm}{\begin{thm}}
\newcommand{\ethm}{\end{thm}}

\newcommand{\blem}{\begin{lemma}}
\newcommand{\elem}{\end{lemma}}

\newcommand{\bcor}{\begin{corr}}
\newcommand{\ecor}{\end{corr}}

\newcommand{\bprop}{\begin{prop}}
\newcommand{\eprop}{\end{prop}}

\newcommand{\brmk}{\begin{rmk}}
\newcommand{\ermk}{\end{rmk}}

\newcommand{\bpf}{\begin{proof}}
\newcommand{\epf}{\end{proof}}

\newcommand{\bex}{\begin{example}}
\newcommand{\eex}{\end{example}}

\newcommand{\etalchar}[1]{$^{#1}$}
\numberwithin{equation}{section}

\def\eD{\EuScript{D}}
\def\eE{\EuScript{E}}
\def\eF{\EuScript{F}}

\def\eK{\EuScript{K}}
\def\eL{\EuScript{L}}
\def\eM{\EuScript{M}}

\def\eP{\EuScript{P}}

\def\eT{\EuScript{T}}

\def\C{\mathbb{C}}
\def\bH{\mathbb{H}}
\def\K{\mathbb{K}}
\def\N{\mathbb{N}}

\def\bP{\mathbb{P}}

\def\R{\mathbb{R}}
\def\Z{\mathbb{Z}}

\def\RP{\mathbb{RP}}
\def\bP{\mathbb{P}}

\def\fE{\mathbf{E}}

\def\fJ{\mathbf{J}}

\def\fP{\mathbf{P}}

\def\fU{\mathbf{U}}

\def\fL{\mathbf{L}}

\def\fq{\mathbf{q}}
\def\fp{\mathbf{p}}
\def\fc{\mathbf{c}}
\def\fu{\mathbf{u}}
\def\fe{\mathbf{e}}
\def\ff{\mathbf{f}}
\def\fx{\mathbf{x}}
\def\fY{\mathbf{Y}}

\def\cA{\mathcal{A}}
\def\cB{\mathcal{B}}
\def\cC{\mathcal{C}}
\def\cD{\mathcal{D}}
\def\cE{\mathcal{E}}
\def\cF{\mathcal{F}}
\def\cG{\mathcal{G}}
\def\cH{\mathcal{H}}

\def\cJ{\mathcal{J}}

\def\cL{\mathcal{L}}
\def\cM{\mathcal{M}}

\def\cQ{\mathcal{Q}}
\def\cR{\mathcal{R}}
\def\cS{\mathcal{S}}
\def\cT{\mathcal{T}}

\def\cX{\mathcal{X}}
\def\cY{\mathcal{Y}}

\def\ov{\overline}

\def\w{\omega}

\def\xkm2{\overline{X}_{k-2}}

\begin{document}

\author{Cheuk Yu Mak and Weiwei Wu}

\AtEndDocument{\bigskip{\footnotesize
  \textsc{Cheuk Yu Mak, DPMMS, University of Cambridge, Cambridge, UK} \par
  \textit{E-mail address}: \texttt{cym22@dpmms.cam.ac.uk} \par

}\bigskip{\footnotesize
  \textsc{Weiwei Wu, Department of Mathematics,
      University of Georgia,
      Athens, Georgia 30602} \par
  \textit{E-mail address}: \texttt{weiwei.wu@math.uga.edu} \par

}}

\date{\today}
\maketitle

\begin{abstract}

We study Dehn twists along Lagrangian submanifolds that are finite quotients of spheres.
We decribe the induced auto-equivalences to the derived Fukaya category and explain its relation to twists along spherical functors.

\end{abstract}

\tableofcontents

\section{Introduction}

One of the most groundbreaking work on symplectomorphisms goes back to Seidel's thesis on Dehn twist along Lagrangian spheres 
and his study on the induced auto-equivalence to the derived Fukaya category
\cite{Se03}, \cite{Seidelbook}.
As predicted via homological mirror symmetry, this auto-equivalence induces non-geometric symmetries in the derived category of coherent sheaves on the mirror \cite{ST01}.
It turns out that this auto-equivalence, called a spherical twist, can be described purely categorically 
and 
%The ingredient needed to define a spherical twist is a spherical object that satisfies a Calabi-Yau condition.
there are a lot of generalizations of spherical twists and spherical objects, including $\mathbb{P}$-twist \cite{HT06}, family twist \cite{Horja05}, etc.

Many of these examples are motivated by the corresponding symplectomorphisms associated to Lagrangian objects--more precisely, the speculation on the effects of the symplectimorphisms 
in Floer theory provides a way to construct new auto-equivalences in derived categories on the mirrors.  
The nature of Lagrangian Dehn twists allows various generalizations of this kind, for example, to submanifolds whose geodesics are all closed with a same period.  However, most of such 
auto-equivalences are only conjecturally 
related to their symplectomorphism counterparts.  In \cite{MW15}, the authors verified that
 Dehn twists along Lagrangian submanifolds that are diffeomorphic to projective spaces match the mapping cone operation
 predicted by the $\mathbb{P}$-twist on the Fukaya category, as well as its family version, as conjectured in \cite{HT06}.  

In this paper, we investigate a new type of Dehn twist and their twist auto-equivalences.
%, of which there was no known constructions before.

\begin{question}
 What is the induced auto-equivalence of Dehn twist along a a spherical Lagrangian, i.e. a Lagrangian submanifold $P$ whose universal cover is $S^n$?
\end{question}

%Such a twist auto-equivalence is not known previously on either the symplectic nor the algebraic side.  
A particularly interesting feature of these twist auto-equivalences, 
which distinguishes this question from all previous twist auto-equivalences, is its sensitivity to the characteristic of the ground field.

Consider the basic example of $P=\RP^n$.  In characteristic zero, $P$ is a spherical object instead of a $\mathbb{P}$-object in the Fukaya category.  
In Corollary \ref{c:PtwistSStwist} we showed that the induced auto-equivalence is a composition of two spherical twists.  
However, when $char=2$, 
the auto-equivalence becomes a $\bP$-twist as defined in \cite{HT06}.  
Indeed, given any spherical Lagrangian, its twist auto-equivalence decomposes into a composition of spherical twists 
in characteristic zero,
%, which is still of interests but can be understood by well-established techniques; 
but when one considers ground field of non-zero characteristics, such twists yield an entire family of new auto-equivalences.
The main novelty here is that,
even though it is rare in the algebraic geometry literature to study how twist auto-equivalences change under a change of characteristic,
it is a natural thing to do in symplectic geometry because all these auto-equivalences
are connected to geometry: they come from symplectimorphisms.
In particular, via mirror symmetry, it should shed light on new discovery of  twist auto-equivalences in algebraic/arithmetic geometry.

To explain our result, let $\mathbb{K}$ be a field of any characteristic and $\Gamma \subset SO(n+1)$ be a finite subgroup such that there exists 
$\widetilde{\Gamma} \subset Spin(n+1)$ so that the covering homomorphism $Spin(n+1) \to SO(n+1)$ restricts to an isomorphism $\widetilde{\Gamma} \simeq \Gamma$.
Let $P$ be a Lagrangian submanifold that is diffeomorphic to $S^n/\Gamma$ in a Liouville manifold $(M,\w)$ with $2c_1(M,\omega)=0$.
Pick a Weinstein neighborhood $U$ of $P$ and take the universal cover $\fU$ of $U$.
The preimage of $P$ is a Lagrangian sphere $\fP$ in $\fU$.
We can pick a parametrization to identify $\fP$ with the unit sphere in $\R^{n+1}$, and the deck transformation with $\Gamma \subset SO(n+1)$.
Then we can define the Dehn twist $\tau_{\fP}$ along $\fP$ in $\fU$.
Since $\tau_{\fP}$ is defined by geodesic flow with respect to the round metric on $\fP$
and the antipodal map lies in the center of $SO(n+1)$, $\tau_{\fP}$ is $\Gamma$-equivariant and descends to a symplectomorphism $\tau_P$ in $U$. 
We call $\tau_P$ the \textit{Dehn twist along} $P$.

We equip $P$ with the induced spin structure from $S^n$ and with the universal local system $E$ corresponding to the canonical representation of $\Gamma:=\pi_1(P)$ to $\K[\Gamma]$.
The pairs $(P,E)$ defines an object $\eP$ in the compact Fukaya category $\mathcal{F}$.
For any Lagrangian brane (ie. an exact Lagrangian submanifold with a choice of grading, spin structure and local system)
$\eE$ in $(M,\w)$, we have
a left $\Gamma$-module structure on $hom_{\mathcal{F}}(\eE,\eP)$ and a right $\Gamma$-module structure on $hom_{\mathcal{F}}(\eP,\eE)$. Our main result is

\begin{thm}\label{t:Twist formula}
Let $(M^{2n},\w)$ be a Liouville manifold with $2c_1(M)=0$ and $n \ge 3$.
For any Lagrangian brane $\eE \in \mathcal{F}$, there is a quasi-isomorphism of the obejcts
 \begin{equation}
 \tau_P(\eE) \simeq Cone(hom_{\mathcal{F}}(\eP,\eE) \otimes_{\Gamma} \eP \xrightarrow{ev_{\Gamma}} \eE) \label{eq:EqTwist}
 \end{equation}
 in $\mathcal{F}^{\perf}$, where $ev_{\Gamma}$ is the evaluation map (see Section \ref{ss:EquivariantEv}) and $\mathcal{F}^{\perf}$ is the category of 
 perfect $A_{\infty}$ right $\cF$ modules 
 (see Definition \ref{d:PerfMod}).
\end{thm}

\begin{corr}\label{c:PtwistSStwist}
 If $P$ is diffeomorphic to $\RP^{n}$ for $n=4k-1$ and $\cchar(\K) \neq 2$, then there are two orthogonal spherical objects in $\mathcal{F}$
 coming from equipping $P$ with different rank one local systems.  $\tau_P(\eE)$ is quasi-isomorphic to the composition of the spherical twists to $\eE$ along these two spherical objects.
 
 If $P=\RP^{n}$ for $n$ odd and $\cchar(\K) =2$, then $P$ is a $\bP$-object and $\tau_P(\eE)$ is quasi-isomorphic to applying $\bP$-twist to $\eE$ along $P$.
\end{corr}

Similar to the spherical twist, there is a purely categorical reformulation of this auto-equivalence in terms of twist along a spherical functor \cite{AL17}.
The terminology will be reviewed in Section \ref{s:Algebraic perspective}.
In particular, we will define an $R$-spherical object (Definition \ref{d:R-twist}) for any $\K$-algebra $R$.
Roughly speaking, an $R$-spherical object is an object with cohomological endomorphism algebra being $R \otimes_\K H^*(S^n)$ that satisfies a Calabi-Yau condition.
We will explain how an $R$-spherical object defines a spherical functor and hence an auto-equivalence (Proposition \ref{p:R-twist})
when an additional mild assumption is satisfied (Assumption \ref{a:perfect}).
Finally, we will explain that Theorem \ref{t:Twist formula} can be reformulated as follows:

\begin{thm}\label{t:Twist formula(alg)}
 Under the assumption of Theorem \ref{t:Twist formula}, $\eP$ is a $\K[\Gamma]$-spherical object and it defines a spherical functor
 \begin{align*}
  \cS: \K[\Gamma]^{\perf} \to \mathcal{F}^{\perf}
 \end{align*}
given by $V \mapsto V \otimes_\Gamma \eP$.
Moreover, for any $\eE \in \mathcal{F}$, we have
\begin{align*}
 \tau_P(\eE) \simeq  \cT_{\cS}(\eE)
\end{align*}
in $\mathcal{F}^{\perf}$, where $\cT_{\cS}$ is the twist auto-equivalence of $\mathcal{F}^{\perf}$ associated to $\cS$.
\end{thm}

%relate Fukaya categories of plumbings of 3-spheres along a circle with derived categories of sheaves on local CY 3-folds containing two floppable curves, 
%and that both Lens space twists and fat spherical twists naturally arise in characteristic p in that setting

\subsection*{Examples and outlooks} % (fold)
\label{sub:examples_and_outlooks}

The current paper is focused on the foundations of the theory of twist auto-equivalences associated to $\tau_P$ and is the starting point of a series of 
subsequent works investigating examples involving Lagrangian spherical space forms.  Although we will not discuss in-depth these examples, we include several ongoing works that will appear in the near future to give 
the readers a peek on potential applications of the twist formula and relations to works in the literature.

\begin{itemize}
  \item In an upcoming paper \cite{MR18}, the first author and Ruddat construct 
 Lagrangian embeddings of graph manifolds (e.g. spherical space forms) systematically in some Calabi-Yau 3-folds using toric degenerations and tropical curves. 
Construction in smooth toric varieties and open Calabi-Yau manifolds using tropical curves can be found in \cite{Mik} and \cite{Mat}, respectively.

% If we condsider a deformation of a Calabi-Yau 3-fold, which contains a Lagrangian submanifold $P$, such that $P$ degenerates to a point
% in the special fiber, then one can realize the Dehn twist along $P$ as the monodromy 
% around the special fiber.

Lagrangian spherical space forms $P$ have been studied in some physics literature (see e.g. \cite{HaghighatKlemm}) and 
the Dehn twist along $P$ can be realized as the monodromy 
around a special point in the complex moduli.
Our study in this paper can be viewed as the mirror-dual of the intensive study of monodromy actions on 
the derived category of coherent sheaves in the stringy K\"ahler moduli (\cite{AHK05}, \cite{Horja05}, \cite{DonovanSegal}, \cite{DonovanWemyss}, \cite{HS16}, etc).

%   Through mirror symmetry, one expects to get
% a family of auto-equivalences in the derived category of the mirror as specified by \eqref{eq:EqTwist}, which in particular, 
% are sensitive to characteristics in a particular way.  
% Motivated by this expectation, 
 \item  Hong, Lau and the first author study the local mirror symmetry in all characteristics in a subsequent paper \cite{HLM} when two Lens spaces $P$, $P'$ are plumbed together. In this case, the lens spaces can be identified 
% , where $\K[\Gamma]$-spherical objects are supported on (and $\Gamma$ is abelian in this case), 
 with fat spherical objects in the sense of Toda  \cite{Toda07} in specific characteristics.  In particular, 
 this justifies that Dehn twists along Lens spaces are mirror to fat spherical twists in this case.
 %and both objects becomes union of spherical objects in generic characteristics.

 Independently, in the upcoming work \cite{ESW}, Evans, Smith and Wemyss relate Fukaya categories of plumbings of $3$-spheres along a circle with derived categories of sheaves on local Calabi-Yau 3-folds containing two floppable curves.  Both Lens space twists and fat spherical twists naturally arise in specific characteristics in that setting.

\item In principle, Theorem \ref{t:Twist formula} can be deduced from the Lagrangian cobordism formalism \cite{BC13}, \cite{BC2}, \cite{BCIII}.  There are 
several ingredients that needs to be incorporated, though.  The cobordism for a result of this form will require one to use an 
immersed Lagrangian cobordism that does not have clean intersections, which would not even have compactness results on holomorphic disks.  
A fix could be to generalize the \textit{bottleneck immersed cobordism} \cite{MW15} to the categorical level, which should yield the desired mapping cone relation.  

Note that this bottleneck immersed formalism is different from an ongoing work of the immersed Lagrangian cobordism due to Biran and Cornea, 
where the latter should also enter the picture.  We have not adopted this approach due to the relevant tools are still under constructions, 
but such an alternative approach should be of independent interests.

\item Another possible approach to Theorem \ref{t:Twist formula}, explained to us by Ivan Smith, is to realize the
Dehn twists as the monodromy in certain symplectic fibrations and apply the Ma'u-Wehrheim-Woodward quilt formalism \cite{MWWfunctor}. 
It is particularly suitable to study the case that $P=\mathbb{RP}^n$ because, in this case, 
$\tau_P$ can be realized as the monodromy of a Morse-Bott Lefschetz fibration, and the techniques developed by Wehrheim and Woodward 
in \cite{WeWo16} can possibly be adopted.
When $P$ is a general spherical space form, the symplectic fibration is no longer Morse-Bott and more technicality will be involved.
Carrying out this approach would be of independent interests because it would possibly give a functor level statement.

\end{itemize}

Examples mentioned above mostly focused on lens spaces where the group $\Gamma$ is a cyclic group.   The algebro geometric counterparts of Dehn twists along more general spherical space forms will be investigated in future works.

%Such auto-equivalences seems to be governed by a universal object defined over $\Z$, where the spherical twists and our $P$-twists reflects part of its geometry.  
%This will be studied in future works.

%sketch of proof

% One distinguished feature that $S^n/\Gamma$ has but sphere and projective spaces do not have is
% the presence of exceptional geodescis (i.e. some geodesics on $S^n/\Gamma$ has a shorter shortest period than the generic one)
% when $|\Gamma|>2$.
% This is the essential reason why the approaches to the auto-equivalences of Dehn twists taken in \cite{Se03} and \cite{MW15} fail in the current setting.
% In this paper, we choose a new path and use symplectic field theory  together with some algebraic argument to obtain Theorem \ref{t:Twist formula}.

The paper is organized as follows.
In Section \ref{sec:floer_cohomology_with_local_systems}, we review the Fukaya category for Lagrangian with local systems 
and discuss the object $\eP$ and the evaluation map in \eqref{eq:EqTwist}.
In Section \ref{sec:review_of_symplectic_field_theory_and_dimension_formulae}, we explain how to use symplectic field theory
to compute some moduli which define the $A_{\infty}$ structure of the Fukaya category.
In Section \ref{sec:quasi_isomorphism}, we apply symplectic field theory on Floer differential and product to prove 
a cohomological version of Theorem \ref{t:Twist formula}.
In Section \ref{s:categorical}, we  construct an appropriate degree zero cocycle that is found based on Section \ref{sec:quasi_isomorphism}
to induce the quasi-isomorphism in \eqref{eq:EqTwist} (and hence finish the proof of Theorem \ref{t:Twist formula}).
The discussion on the algebraic aspect will be given in Section \ref{s:Algebraic perspective}, where $R$-spherical objects and $R$-spherical twists will be introduced.

\section*{Acknowledgment}
The authors thank Richard Thomas for his interest in real projective space twist in our previous work \cite{MW15}, which motivates
our investigation on Dehn twist along general spherical space forms in this project.
The first author is deeply indepted to Paul Seidel for many insightful conversation and encouragement.  This work would be impossible without his support.
%during the 2016-2017 special year on homological mirror symmetry at the Institute for Advanced Study.
The authors are very grateful to Ivan Smith for many helpful discussion, which in particular, greatly simplified the technical difficulties involved in this project.
Discussion with Mohammed Abouzaid, Matthew Ballard, Yu-Wei Fan, Sheel Ganatra, Mauricio Romo, Ed Segal, Zachary Sylvan and Michael Wemyss have influenced our understanding on spherical functors and we thank all of them.
We also thank Ailsa Keating, Nick Sheridan, Michael Usher and Jingyu Zhao for helpful communication.

C.Y.M was supported by the National Science Foundation under agreement No. DMS-1128155 
%during his membership in the 2016-2017 special year on homological mirror symmetry at the Institute for Advanced Study.
and by EPSRC (Establish Career Fellowship EP/N01815X/1), and W.W. is partially supported by Simons Collaboration Grant 524427.
Any opinions, findings and conclusions or recommendations expressed in this material are those of the author(s) and do not necessarily reflect the views of the National Science Foundation nor Simons Foundation.

%\begin{comment}

\vskip 1cm

\noindent{\bf Some standing notations.}

\begin{itemize}
  \item $\Gamma$ is a finite group.
  \item $P$ is a Lagrangian submanifold diffeomorphic to $S^n/\Gamma$ for some $\Gamma \subset SO(n+1)$ and $P$ is spin (see Remark \ref{r:Spin})
  \item $\mathbf{L}$ is the universal cover of $L$ and $\pi:\fL \to L $ (or $\pi:T^*\fL \to T^*L$) is the covering map.
  %\item If $p\in P$, then $\mathbf{p}\in \mathbf{P}$ is a (non-canonically chosen) lift of $p$.  The rest of lifts can be denoted as $\mathbf{p}\cdot g$.
  \item $\fp \in \fL$ is a lift of $p \in L$.
  \item $c_{\fp,\fq}$ is the geometric intersection of $\pi( T^*_{\fp} \fP \cap \tau_{\mathbf{P}}(T^*_{\fq} \fP)) \in T^*P$ (see \eqref{e:genCorr}).
  %\item $c_{p,q,g}=\pi(F^0_{\mathbf{q}}\cap \tau_{\mathbf{P}}F^1_{\mathbf{p}g})$ picks out a geometric intersections of $F^0_{q}\cap\tau_{P}F^1_{pg}$.  Note that $\pi(F^0_{\mathbf{q}}\cap \tau_{\mathbf{P}}F^1_{\mathbf{p}g})=\pi(F^0_{\mathbf{q}h}\cap \tau_{\mathbf{P}}F^1_{\mathbf{p}gh})$
  %\item $c_{p,q}=c_{p,q,id}$
  \item $\eP$ is $P$ being equipped with the universal local system.
\end{itemize}

%\end{comment}

\section{Floer theory with local systems}  % (fold)
\label{sec:floer_cohomology_with_local_systems}

In this section, we review the Floer theory for Lagrangian with local systems (see \cite{Abouzaid12}). 
We always assume that $(M,\omega)$ is a Louville manifold with $2c_1(M,\omega)=0$, and a choice of a trivialization of $(\Lambda_\C^{top}T^*M)^{\otimes 2}$ is chosen.  All 
Lagrangians are equipped with a $\mathbb{Z}$-grading and a spin structure.
More discussion about grading can be found in \cite{SeGraded}, \cite[Section 11,12]{Seidelbook}.

\subsection{Fukaya categories with local systems}\label{ss:Fuk}

Let $L$ be a closed exact Lagrangian submanifold in $(M,\omega)$ with a base point $o_L \in L$.
Let $E$ be a finite rank local system on $L$ with a flat connection $\nabla$.
For a path $c:[0,1] \to L$, we use $I_c$ to denote the parallel transport  from $E_{c(0)}$ to $E_{c(1)}$ along $c$ with respect to the connection $\nabla$.
We use the monodromy action from $\Gamma:=\pi_1(L)$ to $E_{o_L}$ to identify $(E,\nabla)$ as a {\it right} $\Gamma$-module.
More explicitly, the right action is given by
\begin{align}
 &\rho: \Gamma \to End(E_{o_L}) \label{eq:RightAction}\\
 &g \mapsto (y \mapsto I_{g}y)
\end{align}
for $y \in E_{o_L}$ and $g \in \Gamma$.
In particular, $(yg)h=I_{h}(I_{g}y)=I_{g*h}y=y(g*h)$, where $*$ stands for concatenation of paths (i.e. $g$ goes first).
We use $\eE$ to denote the triple $(L,E,\nabla)$.
For a Hamiltonian diffeomorphism $\phi \in Ham(M,\omega)$, we define $\phi( \eE):=(\phi(L),\phi_* E, \phi_* \nabla)$.

Let $\eE^i:=(L_i,E^i, \nabla^i)$ for $i=0,1$.
A family of {\it compactly supported} Hamiltonian functions $H=(H_t)_{t \in [0,1]}$ is called \textbf{$(L_0,L_1)$-admissible} if
\begin{align}\label{eq:admissibleHam}
 \phi^H(L_0) \pitchfork L_1
\end{align}
where $\phi^H$ is the time one flow of the Hamiltonian vector field $X_H=(X_{H_t})_{t \in [0,1]}$.
Let $\cX(L_0,L_1)$ be the set of $H$-Hamiltonian chord from $L_0$ to $L_1$ (i.e. $x:[0,1] \to M$ such that $\dot{x}(t)=X_H(x(t))$, $x(0) \in L_0$ and $x(1) \in L_1$).
The Floer cochain complex between $\eE^0$ and $\eE^1$ is defined by
\begin{align}\label{eq:CochainComplex}
CF(\eE^0,\eE^1):=\oplus_{x \in \cX(L_0,L_1)} Hom_{\mathbb{K}}(E^0_{x(0)},E^1_{x(1)}) 
\end{align}

Now, we want to introduce some notations to define the differential for $CF(\eE^0,\eE^1)$
as well as the $A_\infty$-structure for a collection of Lagrangians with local systems.

Let $\cR^{d+1}$ be the space of holomorphic disks with $d+1$ boundary punctures.
For each $S \in \cR^{d+1}$, one of the boundary punctures is distinguished and it is denoted by $\xi_0$.
The other boundary punctures are ordered counterclockwisely along the boundary and are denoted by $\xi_1, \dots, \xi_d$, respectively.
We denote the boundary component of $S$ from $\xi_j$ to $\xi_{j+1}$ by $\partial_j S$ for $j=0,\dots,d-1$.
The boundary component from $\xi_d$ to $\xi_0$ is denoted by $\partial_d S$.
For $j=1,\dots,d$, we pick an outgoing/positive strip-like end for $\xi_j$, which is a holomorphic embedding $\epsilon_j: \R_{\ge 0} \times [0,1] \to S$ such that
\begin{align}\label{eq:outgoingStripEnds}
 \left\{
 \begin{array}{ll}
  \epsilon_j(s,0) \in \partial_{j-1} S \\
  \epsilon_j(s,1) \in \partial_{j} S \\
  \lim_{s \to \infty} \epsilon_j(s,t)=\xi_j
 \end{array}
\right.
\end{align}
We also pick an incoming/negative strip-like end for $\xi_0$, which is a holomorphic embedding $\epsilon_0: \R_{\le 0} \times [0,1] \to S$ such that
\begin{align}\label{eq:incomingStripEnds}
 \left\{
 \begin{array}{ll}
  \epsilon_0(s,0) \in \partial_{0} S \\
  \epsilon_0(s,1) \in \partial_{d} S \\
  \lim_{s \to -\infty} \epsilon_0(s,t)=\xi_0
 \end{array}
\right.
\end{align}
%The embeddings $\{\epsilon_j\}_{j=1}^d$ are called outgoing/positive strip-like ends and the embedding $\epsilon_0$ is called an incoming/negative strip-like end.
The strip-like ends are assumed to have pairwise disjoint image and they vary smoothly with respect to  $S$ in $\cR^{d+1}$.

Let $\{\eE^j\}_{j=0}^d$ be a finite collection of Lagrangians with local systems.
For $j=1,\dots,d$, let $H_j$ be a $(L_{j-1},L_j)$-admissible Hamiltonian (see \eqref{eq:admissibleHam}).
We also pick a $(L_0,L_d)$-admissible Hamiltonian $H_0$.
For each $S \in \cR^{d+1}$ and each collection $\{H_j\}_{j=0}^d$, we pick a $C_{cpt}^\infty(M)$-valued one-form $K \in \Omega^1(S,C_{cpt}^\infty(M))$.
Let $X_K \in \Omega^1(S,C^{\infty}(M,TM))$ be the corresponding Hamiltonian-vector-field-valued one-form.
We require that
\begin{align} \label{eq:FloerK}
 \left\{
 \begin{array}{ll}
  \epsilon_j^* X_K=X_{H_j}dt \\
  X_K|_{\partial_j S}=0
 \end{array}
\right.
\end{align}
When $d=1$, we assume that $K(s,t)=H_{0,t}=H_{1,t}$ for all $(s,t) \in \R \times [0,1]$.
We also assume that $K$ varies smoothly with respect to $S$ and is consistent with respect to gluing near boundary strata of the Deligne-Mumford-Stasheff compactification of $\cR^{d+1}$.

Let $J^M$ be an $\omega$-compatible almost complex structure that is cylindrical over the infinite end of $M$ (see Definition \ref{d:cylindricalJ}).
Let $\cJ(M,\omega)$ be the space of $\omega$-compactible almost complex structures $J$ such that $J=J^M$ outside a compact set.
For $j=0,\dots,d$, let $J_j=(J_{j,t})_{t \in [0,1]}$ be a family such that $J_{j,t} \in \cJ(M,\omega)$ for all $t$.
For each $S \in \cR^{d+1}$ and each collection $\{J_j\}_{j=0}^d$, we pick a domain-dependent $\omega$-compatible almost complex structure $J=(J_z)_{z \in S}$
such that 
\begin{align}
\left\{\label{eq:FloerJ}
\begin{array}{ll}
 J_z \in \cJ(M,\omega) \text{ for all }z\\
 J \circ \epsilon_j(s,t)=J_{j,t} \text{ for all }j,s,t
\end{array}
\right.
\end{align}
When $d=1$, we require that $J=(J_{s,t})_{(s,t) \in \R \times [0,1]}=(J_t)_{t \in [0,1]}$ is independent of the $s$-direction.
We assume that $J$ varies smoothly with respect to $S$ in $\cR^{d+1}$
and is consistent with respect to gluing near boundary strata of the Deligne-Mumford-Stasheff compactification of $\cR^{d+1}$.

Let $x_j \in \cX(L_{j-1},L_{j})$ for $j=1,\dots,d$ and $x_0 \in \cX(L_0,L_d)$.
For $d>1$, we define $\eM^{K,J}(x_0;x_d,\dots,x_1)$ to be the space of smooth maps $u:S \to M$ such that
\begin{align} \label{eq:FloerModuli}
 \left\{
 \begin{array}{ll}
 S \in \cR^{d+1} \\
 (du-X_K)^{0,1}=0 \text{ with respect to }(J_z)_{u(z)} \\
 u(\partial_j S) \subset L_j \text{ for all }j  \\
 \lim_{s \to \pm \infty} u(\epsilon_j(s,t))=x_j(t) \text{ for all }j
 \end{array}
\right.
\end{align}
When $d=1$, we define $\eM^{K,J}(x_0;x_1)$ to be the corresponding space of maps after modulo the $\R$ action by translation in the $s$-coordinate.
For simplicity, we may use $\eM(x_0;x_d,\dots,x_1)$ to denote $\eM^{K,J}(x_0;x_d,\dots,x_1)$ for an appropriate choice of $(K,J)$.

\begin{rmk}\label{r:ModuliNotationJ}
 In Section \ref{sec:review_of_symplectic_field_theory_and_dimension_formulae}, we will encounter situations where $K \equiv 0$ and $J$ is a domain
 independent almost complex structure.
 In these cases, $J$ has to be chosen carefully to achieve regularity so we will emphasize $J$ and denote the moduli by $\eM^J(x_0;x_d,\dots,x_1)$. 
\end{rmk}

When every element in $\eM(x_0;x_d,\dots,x_1)$ is transversally cut out, $\eM(x_0;x_d,\dots,x_1)$ is a smooth manifold 
of dimension $|x_0|-\sum_{j=1}^d |x_j|+(d-2)$, where $|\cdot |$ denotes the Maslov grading (see Section \ref{ss:Grading}).

For each transversally cut out rigid element $u \in \eM(x_0;x_d,\dots,x_1)$, we define
 \begin{align}
 &\mu^u:Hom(E^{d-1}_{x_d(0)},E^d_{x_d(1)}) \times \dots \times Hom(E^0_{x_1(0)},E^1_{x_1(1)}) \to Hom(E^0_{x_0(0)},E^d_{x_0(1)}) \nonumber \\
 &\mu^u(\psi^d,\dots,\psi^1)(a)= \sign(u) I_{\partial_d u} \circ \psi^d \circ \dots \circ \psi^1 \circ I_{\partial_0 u} (a) \label{eq:AinfLocalSystem}
 \end{align}
where $\partial_d u=u|_{\partial_d S}$ for $\partial_d S$ being equipped with the counterclockwise orientation, and $\sign(u) \in \{\pm 1\}$ is the sign determined by $u$ (see Appendix \ref{sec:orientations}).
Finally, we define the $A_{\infty}$-operation by
\begin{align}
&\mu^d:CF(\eE^{d-1},\eE^d) \times \dots \times CF(\eE^0,\eE^1) \to CF(\eE^0,\eE^d) \nonumber \\
&\mu^d(\psi^d,\dots,\psi^1)=\sum_{u \in \eM(x_0;x_d,\dots,x_1), u \text{ rigid}}  \mu^u(\psi^d,\dots,\psi^1) \label{eq:AinfLocalSystem2}
\end{align}

To summarize, we have a Fukaya category $\mathcal{F}$ of closed exact Lagrangians with local systems.
An object of $\mathcal{F}$ is a closed embedded exact Lagrangian equipped with a finite rank local system, a $\mathbb{Z}$-grading and a spin structure.
For $\eE^0,\eE^1 \in \cF$, we pick a $(L_0,L_1)$-admissible Hamiltonian $H$ and a family of almost complex structure $J$ 
to define the Floer cochain \eqref{eq:CochainComplex}, which is the morphism space $hom_\cF(\eE^0,\eE^1)$.
For each collection of objects $\{\eE^j\}_{j=0}^d$, we pick $K$ and $J$ subject to the conditions \eqref{eq:FloerK} \eqref{eq:FloerJ}.
After choosing $K,J$, we can define 
$\eM^{K,J}(x_0;x_d,\dots,x_1)$ \eqref{eq:FloerModuli} for any $x_0 \in \cX(L_0,L_d)$ and $x_j \in \cX(L_{j-1},L_{j})$ for $j=1,\dots,d$.
When $K$ and $J$ are chosen generically, $\eM^{K,J}(x_0;x_d,\dots,x_1)$ is a smooth manifold and we use \eqref{eq:AinfLocalSystem} and
\eqref{eq:AinfLocalSystem2} to define $\{\mu^d\}_{d=1}^\infty$.
The fact that $K$ and $J$ are chosen consistently with respect to the Deligne-Mumford-Stasheff compactification implies that $\{\mu^d\}_{d=1}^\infty$ satisfies the $A_{\infty}$-relations.

\begin{comment}
For an element $\psi \in  Hom_{\mathbb{K}}(E^0_{x_1(0)},E^1_{x_1(1)})$, the Floer differential of $CF(\eE^0,\eE^1)$ is given by
$$\mu^1(\psi)=\sum_{u \in \mathcal{M}(x_0,x_1), |x_0|-|x_1|=1} \mu^u(\psi) $$
where
$$\mu^u(\psi)(a):=(-1)^* I^1_{\partial_1 u} \circ \psi \circ I^0_{\partial_0 u}(a)$$
for $u \in \mathcal{M}(x_0,x_1)$ defined using a regular $J$, where $\partial_i u$ are the image of the boundary
of the domain of $u$ oriented counterclockwise.

Equivalently, we can define $CF(\eE^0,\eE^1)=\oplus_{\gamma \in \cX(L_0,L_1)} Hom_{\mathbb{K}}(E^0_{\gamma(0)},E^1_{\gamma(1)})$
where $\cX(L_0,L_1)$ is the set of $H:=(H_t)_{t \in [0,1]}$-Hamiltonian chord $\gamma:[0,1] \to M$ from $L_0$ to $L_1$
and $H$ is a choice of Hamiltonian such that the time $1$ Hamiltonian flow generated by $H$ is $\psi$.
The differential is given by co
\end{comment}

\subsection{Unwinding local systems}\label{ss:unwinding}

The goal of this subsection is to give a computable presentation of $CF(\eE^0,\eE^1)$, where $\eE^i$ are local systems of the same underlying Lagrangian.  
In particular, the identification \eqref{eq:TrivializedIdentification} and \eqref{eq:groupCoh} will be used frequently later.

Let $L$ be a closed exact Lagrangian and $\mathbf{L}$ be its universal cover with covering map $\pi:\mathbf{L} \to L$.  Let $o_L \in L$ be a base point of $L$ and we pick a lift $o_\mathbf{L} \in \mathbf{L}$ such that $\pi(o_\mathbf{L})=o_L$.
We assume throughout that $\Gamma:=\pi_1(L,o_L)$ is a finite group so that $\fL$ is compact.
For each $\fq \in \mathbf{L}$, there is a unique path $c_\fq$ (up to homotopy) from  $o_\mathbf{L}$ to $\fq$
and we identify $\fq$ with the homotopy class $[\pi \circ c_\fq]$.
We have a {\it left} $\Gamma$-action on $\mathbf{L}$ given by
\begin{align}\label{eq:GammaActionGen}
 g \fq :=g*[(\pi \circ c_\fq)]
\end{align}
for $g \in \pi_1(L,o_L)$, 
where $g*[(\pi \circ c_\fq)]$ is a homotopy class of path from $o_L$ to $\pi(\fq)$ and we identify it as a point in $\mathbf{L}$.
It is clear that $h(g\fq)=(h*g)\fq$.
If we pick a Morse function and a Riemannian metric on $L$ to define a Morse cochain complex $C^*(L)$,
we can lift the function and metric to $\mathbf{L}$ to define a Morse cochain complex $C^*(\mathbf{L})$.
The $\Gamma$-action on $\fL$ induces a left $\Gamma$-action on $C^*(\mathbf{L})$.
The $\Gamma$-invariant part of $C^*(\mathbf{L})$ can be identified with $C^*(L)$, in other words,
\begin{align}
 C^*(L)=Rhom_{\mathbb{K}[\Gamma]-mod}(\mathbb{K},C^*(\mathbf{L}))=(C^*(\fL))^\Gamma
\end{align}
We want to discuss the analogue when $L$ is equipped with local systems.

Given a local system $E$ on $L$, we use $\mathbf{E}=\pi^*E$ to denote the pull-back local system.
For a path $c:[0,1] \to \mathbf{L}$, we use $I_c$ to denote the parallel transport with respect to the pull-back flat connection on $\fE$.

Let $E^i$ be local systems on $L$ for $i=0,1$.
We have {\it right} actions (see \eqref{eq:RightAction})
\begin{align}
 \rho^i:\Gamma \to End(E^i_{o_L})
\end{align}
for $i=0,1$. It induces a {\it left} $\Gamma$-module structure on $Hom_\K(E^0_{o_L},E^1_{o_L})$ by
\begin{align}\label{eq:InducedLeftAction}
 \psi \mapsto g \cdot \psi :=\rho^1(g^{-1}) \circ \psi \circ \rho^0(g)
\end{align}
%To check that it give a left action, we consider 
%$((\psi \cdot g) \cdot h)(y)=\psi(y \rho^0(h^{-1}) \rho^0(g^{-1}))\rho^1(g) \rho^1(h)=(\psi \cdot (g * h))(y)$.

\begin{lemma}\label{l:gammaModIsom}
 Let $E^i$ be local systems on $L$ for $i=0,1$.
 Then there is a DG left $\Gamma$-module isomorphism
 \begin{align}
  \Phi: CF((\fL,\fE^0),(\fL,\fE^1)) \simeq C^*(\mathbf{L}) \otimes_{\mathbb{K}} Hom_\K(E^0_{o_L},E^1_{o_L}) \label{eq:TrivializedIdentification}
 \end{align}
 where the differential on $C^*(\mathbf{L}) \otimes_{\mathbb{K}} Hom_\K(E^0_{o_L},E^1_{o_L})$ is only the differential on the first factor, and the $\Gamma$-action on it is given by
 $g \cdot (x\otimes \psi):=gx \otimes g \cdot \psi $ (see \eqref{eq:GammaActionGen} and \eqref{eq:InducedLeftAction}).
\end{lemma}

\begin{proof}
We use the Morse model to compute the Floer cochain complex.
Let $C^*(L)$ be a Morse cochain complex and $C^*(\mathbf{L})$ be its lift.
We use $\partial_L$ and $\partial_{\fL}$ to denote the differential of $C^*(L)$ and $C^*(\mathbf{L})$, respectively.
 
For each $\fq \in \fL$ and both $i=0,1$, there is a canonical identification 
\begin{align}
 I_{c_{\fq}^{-1}}: \fE^i_{\fq} \to \fE^i_{o_{\fL}}
\end{align}
where $c_{\fq}$ is the unique (up to homotopy) path from  $o_\fL$ to $\fq$.
Therefore, it induces a trivialization of $\fE^i$.
We can also trivialize $Hom_\K(\fE^0,\fE^1)$ using the canonical isomorphism
\begin{align}
& Hom_\K(\fE^0_\fq,\fE^1_\fq) \to Hom_\K(\fE^0_{o_\fL},\fE^1_{o_\fL})=Hom_\K(E^0_{o_L},E^1_{o_L}) \label{eq:trivialization} \\
& \psi \mapsto I^1_{c_\fq^{-1}} \circ \psi \circ I^1_{c_\fq} \label{eq:trivialization2}
\end{align}

 Using the trivialization \eqref{eq:trivialization}, \eqref{eq:trivialization2}, we have a graded vector space isomorphism \eqref{eq:TrivializedIdentification}.
 To compare the differential on both sides of \eqref{eq:TrivializedIdentification}, let $\fu$ be a Morse trajectory from $\fq_0$ to $\fq_1$ contributing to $\partial_\fL$ and hence the differential of $CF((\fL,\fE^0),(\fL,\fE^1))$.
 %on $CF((\fL,\fE^0),(\fL,\fE^1))$ becomes (with respect to the identification \eqref{eq:TrivializedIdentification})
 For $\fq_1 \otimes \psi \in C^*(\mathbf{L}) \otimes_{\mathbb{K}} Hom_\K(E^0_{o_L},E^1_{o_L})$,
 \begin{align}
  &\Phi(\mu^\fu(\Phi^{-1}(\fq_1 \otimes \psi))) \\
  =&\sign(\fu)\fq_0 \otimes I_{c_{\fq_0}^{-1}}I_{\partial_1 \fu}I_{c_{\fq_1}} \psi I_{c_{\fq_1}^{-1}}I_{\partial_0 \fu}I_{c_{\fq_0}} \\
  =&\sign(\fu)\fq_0 \otimes \psi
 \end{align}
 where the second equality uses the fact that $\pi_1(\fL)=1$.
 Therefore, $\Phi$ is an isomorphism of differential graded vector spaces if we define
 the differential on $C^*(\mathbf{L}) \otimes_{\mathbb{K}} Hom_\K(E^0_{o_L},E^1_{o_L})$ to be $\partial_{\fL}$ acting on the first factor.
 
 Finally, we want to compare the left $\Gamma$-module structures.
 In $CF((\fL,\fE^0),(\fL,\fE^1))$, the action on $\psi \in Hom_\K(\fE^0_{\fq},\fE^1_{\fq})=Hom_\K(E^0_{q},E^1_q)$ is given by
 \begin{align}
  \psi \mapsto g\psi=\psi
 \end{align}
 where the last $\psi$ lies in $Hom_\K(\fE^0_{g\fq},\fE^1_{g\fq})=Hom_\K(E^0_{q},E^1_q)$.
 For $\fq \otimes \psi \in C^*(\mathbf{L}) \otimes_{\mathbb{K}} Hom_\K(E^0_{o_L},E^1_{o_L})$,
  \begin{align}
  &\Phi(g(\Phi^{-1}(\fq \otimes \psi))) \\
  =&g\fq \otimes I_{c_{g\fq}^{-1}}I_{c_{\fq}} \psi I_{c_{\fq}^{-1}}I_{c_{g\fq}} \\
  =&g\fq \otimes I_{g^{-1}}\psi I_g
 \end{align}
 which is exactly the one given in \eqref{eq:GammaActionGen} and \eqref{eq:InducedLeftAction}.
 It finishes the proof.
\end{proof}

We have the following consequence of Lemma \ref{l:gammaModIsom}:

\begin{lemma}\label{l:cochainModel}
  Let $E^i$ be local systems on $L$ for $i=0,1$.
 Then
 \begin{align}\label{eq:groupCoh}
  CF(\eE^0,\eE^1)=Rhom_{\mathbb{K}[\Gamma]-mod}(\mathbb{K},C^*(\mathbf{L}) \otimes_{\mathbb{K}} Hom_\K(E^0_{o_L},E^1_{o_L}))
  %=(C^*(\mathbf{L}) \otimes_{\mathbb{K}} Hom_\K(E^0_{o_L},E^1_{o_L}))^{\Gamma}
 \end{align}
 
\end{lemma}

\begin{proof}
We use the notation in the proof of Lemma \ref{l:gammaModIsom}.
Let $u$ be a Morse trajectory from $q_0$ to $q_1$ contributing to $\partial_L(q_1)$.
Let $\fq_0 \in \fL$ be a lift of $q_0$ and let $\fq_1 \in \fL$ be the corresponding lift of $q_1$
such that $u$ lifts to a Morse trajectory $\fu$ from $\fq_0$ to $\fq_1$.
Let $\psi \in Hom_\K(E^0_{q_1},E^1_{q_1})$. By \eqref{eq:AinfLocalSystem}, we have
\begin{align}
 \mu^u(\psi)=\sign(u)I_{\partial_1u} \psi I_{\partial_0u}
\end{align}
By definition, $Hom_\K(\fE^0_{g\fq_j},\fE^1_{g\fq_j}) \cong Hom_\K(E^0_{q_j},E^1_{q_j})$ for $j=0,1$ and for all $g \in \Gamma$.
Therefore, for $\psi \in Hom_\K(\fE^0_{g\fq_1},\fE^1_{g\fq_1}) \subset CF((\fL,\fE^0),(\fL,\fE^1))$,
\begin{align}
 \mu^{g\fu}(\psi)&=\sign(g\fu)I_{\partial_1 g\fu} \psi I_{\partial_0 g\fu} \\
 &=\sign(u)I_{\partial_1u} \psi I_{\partial_0u}
\end{align}
where $\mu^{g\fu}$ is the term in the differential of $CF((\fL,\fE^0),(\fL,\fE^1))$ contributed by $g\fu$, and the second equality
uses $Hom_\K(\fE^0_{g\fq_1},\fE^1_{g\fq_1}) \cong Hom_\K(E^0_{q_1},E^1_{q_1})$.
The $\Gamma$ action on the generators ($\fq \otimes \psi \mapsto g\fq \otimes \psi$) and differentials ($\mu^u \mapsto \mu^{g\fu}$) of $CF((\fL,\fE^0),(\fL,\fE^1))$
are free and the invariant part can be identified with $CF(\eE^0,\eE^1)$ so
 \begin{align}
  CF(\eE^0,\eE^1)=Rhom_{\mathbb{K}[\Gamma]-mod}(\mathbb{K},CF((\fL,\fE^0),(\fL,\fE^1)))=(CF((\fL,\fE^0),(\fL,\fE^1)))^{\Gamma} \label{eq:GammaInvariantModel}
 \end{align}
and the result follows from Lemma \ref{l:gammaModIsom}.
\end{proof}

\subsection{The universal local system}\label{ss:UniSystem}

In this subsection, we introduce the universal local system and hence, in particular, the object $\eP$ in Theorem \ref{t:Twist formula}.
Some elemenary properties of the universal local system will also be given. 

%\begin{defn}[Universal local system]
% The {\it universal local system} on $L$ is 
%\end{defn}
\begin{defn}[Universal local system]\label{d:uniSystem}
The {\bf universal local system} $E$ on $L$ is a local system that is uniquely determined by the following conditions:
%We can define a local system $E$ on $L$ as follows:
As a vector space, $E_q=\K \langle \pi^{-1}(q) \rangle$ for $q \in L$.\footnote{$E$ has finite rank because we assumed $|\Gamma|=|\pi_1(L,o_L)| < \infty$}
For any $y \in \pi^{-1}(q)$ and $c:[0,1] \to L$ such that $c(0)=q$,
the parallel transport of $E$ satisfies $I_c(y)=\fc(1)$, where $\fc:[0,1] \to \fL$ is the unique path such that $\pi \circ \fc=c$ and $\fc(0)=y$.
%These conditions uniquely determine a local system on $L$, which we call 
\end{defn}

As usual, we have the monodromy right $\Gamma$-action $\rho$ on $E_{o_L}$ \eqref{eq:RightAction}.
On top of that, we can use the left $\Gamma$ action on $\fL$ \eqref{eq:GammaActionGen} to induce (by extending it linearly) a left $\Gamma$ action on $E_q$ for all $q \in L$.
%and the left $\Gamma$-action on $E_{o_L}$ is given by \eqref{eq:GammaActionGen} for $\fq \in \pi^{-1}(o_L)=E_{o_L}$.
These two actions on $E_{o_L}$ commute and in general, we have

\begin{lemma}\label{l:ActionCommute}
Let $E$ be the universal local system on $L$. 
 For $q \in L$, $y \in E_q$, $g \in \Gamma$ and $c:[0,1] \to L$ such that $c(0)=q$, we have
 \begin{equation}
  g(I_cy)=I_c(gy)
 \end{equation}
\end{lemma}

\begin{proof}
 Without loss of generality, let $y \in \pi^{-1}(q)$.
 We can identify $y$ with the homotopy class $[\pi \circ c_y]$ from $o_L$ to $q$. Then we have (see \eqref{eq:GammaActionGen})
  \begin{equation}
  g(I_cy)=g*[\pi \circ c_y]*[c]=I_c(gy)
 \end{equation}
 where $[c]$ is the homotopy class of path from $c(0)$ to $c(1)$ that $c$ represents.
\end{proof}

Since we have a left action on $E_q$ for all $q \in L$, it induces a left $\Gamma$ action on $CF(\eE',\eE)$
\begin{align}\label{eq:Left_on_cE}
 \psi \mapsto g\psi 
\end{align}
for any $\eE' \in Ob(\cF)$.
Similarly, for any $\eE' \in Ob(\cF)$, we have the induced right $\Gamma$ action on $CF(\eE,\eE')$
\begin{align}\label{eq:Right_on_cE}
 \psi(\cdot ) \mapsto \psi(g \cdot) 
\end{align}

As an immediate consequence of Lemma \ref{l:ActionCommute} and the definition of $\mu^u$ (see \eqref{eq:AinfLocalSystem}), we have

\begin{corr}\label{c:ActionCommute}
Let $E$ be the universal local system on $L$.
 Let $\eL_1 \dots, \eL_r, \eK_1,\dots,\eK_s \in Ob(\cF)$.
 Let $y_j \in CF(\eK_j,\eK_{j+1})$ for $j=1,\dots,s-1$,
 $x_j \in CF(\eL_j,\eL_{j+1})$ for $j=1,\dots,r-1$, $\psi_2 \in CF(\eE,\eK_1)$ and $\psi_1 \in CF(\eL_r,\eE)$,
 we have
 \begin{align}
 \mu^u(y_{s-1},\dots,y_1,\psi_2g,\psi_1,x_{r-1},\dots,x_1)= \mu^u(y_{s-1},\dots,y_1,\psi_2,g\psi_1,x_{r-1},\dots,x_1) \label{eq:movingG}
 \end{align}
for all $g \in \Gamma$, where $u$ is an element in the appropriate moduli contributing to the $A_{\infty}$-structural maps of $\cF$.
% $CF(K_{s-1},K_s) \dots CF(K_1,K_2) \times CF(\eE,K_1) \times CF(L_r,\eE) \dots CF(L_1,L_2)$.
When $s=0$ (resp. $r=0$), we have
 \begin{align}
 &g\mu^u(\psi_1,x_{r-1},\dots,x_1)= \mu^u(g\psi_1,x_{r-1},\dots,x_1) \text{, respectively} \\
 &\mu^u(y_{s-1},\dots,y_1,\psi_2g)= \mu^u(y_{s-1},\dots,y_1,\psi_2)g
 \end{align}

\end{corr}

\begin{rmk}
 We offer an alternative way to understand \eqref{eq:movingG} using $\fL$ instead of $\eE$.
 For each $q \in L$ and a lift $\fq$ of $q$, we can view $\fq$ as a point in $\fL$ or as an element in $E_q$.
 Therefore, we can identify the generators of $CF(T^*_qL,\eE)$ and the generators of $CF( \cup_{g \in \Gamma} T^*_{g\fq}\fL,\fL)$ by
 \begin{align}\label{e:LStoLift}
  E_q \ni \fq  \mapsto \fq \in T^*_\fq\fL \cap \fL 
 \end{align}
Dually, $CF(\eE, T^*_qL)$ can be identified with $CF(\fL, \cup_{g \in \Gamma} T^*_{g\fq}\fL)$ by
 \begin{align}\label{e:LStoLiftvee}
  Hom_\K(E_q,\K) \ni \fq^\vee  \mapsto \fq \in \fL \cap T^*_\fq\fL 
 \end{align}
 The right action \eqref{eq:Right_on_cE} on $Hom_\K(E_q,\K)$ is given by $\fq^\vee g=(g^{-1}\fq)^\vee$, which corresponds to the right action on $\fL$ by $\fq g=g^{-1}\fq$.
 
 Now, we want to make connection with Corollary \ref{c:ActionCommute}.
 %Let $U$ be a Weinstein neighborhood of $L$ and $\fU=\pi_{TL}^{-1}(U)$ where $\pi_{TL}:T^*\fL \to T^*L$ is the covering map.
 We adopt the notations from Corollary \ref{c:ActionCommute}.
 For simplicity, we assume that $K_1$ and $L_1$ are Lagrangians without local systems and $\psi_1 =\fq_1 \in E_{q_1}$, $\psi_2=\fq_2^\vee \in Hom_\K(E_{q_2},\K)$.
 Let $\gamma$ be $\partial_{r+1}S$, which is the component of $\partial S$ with label $L$.
 
 Since the parallel transport of $E$ can be identified with moving the points in $\fL$, for $\mu^u$ to be non-zero and contribute to the RHS of \eqref{eq:movingG}, there is exactly
 one $g \in \Gamma$ and one lift of $u|_\gamma$, which is denoted by $\fu:\gamma \to \fL$, such that $\fu$ goes from $g\fq_1$ to $\fq_2$. 
 For each $h \in \Gamma$, the maps $h \fu:\gamma \to \fL$ are the other lifts of $u|_\gamma$ and $h\fu$ goes from $hg\fq_1$ to $h\fq_2$.
 
 Roughly speaking, one can define a Floer theory by counting $(u, \fu)$, where $u$ is as in Corollary \ref{c:ActionCommute}
 and $\fu$ is a lift of $u|_\gamma$. This definition is explained in details in \cite{Da12} and the outcome is the same as Lagrangian Floer theory with local systems. 
 In this setting, the pair $(u,h\fu)$ contributes to
 \begin{align}
  \mu^{(u,h\fu)}(y_{s-1},\dots,y_1,h\fq_2,hg\fq_1,x_{r-1},\dots,x_1)
 \end{align}
 and it equals to $\mu^{(u,\fu)}(y_{s-1},\dots,y_1,\fq_2,g\fq_1,x_{r-1},\dots,x_1)$.
Under the identification \eqref{e:LStoLift}, \eqref{e:LStoLiftvee}, it means that (when $h=g^{-1}$)
 \begin{align*}
  \mu^{(u,\fu)}(y_{s-1},\dots,y_1,\fq_2^\vee,g\fq_1,x_{r-1},\dots,x_1)=\mu^{(u,g^{-1}\fu)}(y_{s-1},\dots,y_1,(g^{-1}\fq_2)^\vee,\fq_1,x_{r-1},\dots,x_1)
 \end{align*}
 which is exactly the same as \eqref{eq:movingG}
\end{rmk}

The rest of this subsection is devoted to the discussion of $CF(\eE,\eE)$ when $E$ is the universal local system of $L$.
Let $R:=\K[\Gamma]$ and $1_\Gamma$ be the unit of $\Gamma$.
For $h \in \Gamma$, we define $\tau_h \in Hom_\K(R,R)$ by 
\begin{align}
 \tau_h(g)=\left\{
 \begin{array}{ll}
  1_\Gamma &\text{ if }g=h^{-1} \\
  0 &\text{ if } g \in \Gamma \backslash \{h^{-1}\}
 \end{array}
 \right.
\end{align}
Note that $R \cong E_{o_L}$ so, by Lemma \ref{l:cochainModel}, we have
\begin{align}
CF(\eE,\eE)=(C^*(\fL) \otimes Hom_\K(R,R))^{\Gamma}
\end{align}
In particular, we have $\mu^1,\mu^2$ on $(C^*(\fL) \otimes Hom_\K(R,R))^{\Gamma}$ inherited from $CF(\eE,\eE)$.
In Lemma \ref{l:gammaModIsom}, we proved that $\mu^1$ coincides with the Morse differential $\partial_\fL$ on the first factor.
The same line of argument can prove that $\mu^2$ coincides with the 
Floer multipliciation on $C^*(\fL)$ tensored with the composition in $Hom_\K(R,R)$ (i.e. $\mu^2_{\fL}(-,-) \otimes - \circ - $).

Let $\Phi_2:C^*(\fL) \otimes R \to (C^*(\fL) \otimes Hom_\K(R,R))^{\Gamma}$ be the graded vector space isomoprhism given by
\begin{align}
 \Phi_2:x \otimes h \mapsto \sum_{g \in \Gamma} gx \otimes g \cdot \tau_h=\sum_{g \in \Gamma} gx \otimes I_{g^{-1}} \tau_h I_g \label{eq:Phi2}
\end{align}

\begin{lemma}\label{l:Identifyingm1m2}
We have the following equalities
\begin{align}
 \Phi_2^{-1} \circ \mu^1 \circ \Phi_2(x \otimes h)&=\partial_{\fL}(x) \otimes h \label{eq:Simplifiedm1}\\
 \Phi_2^{-1} \circ \mu^2 \circ (\Phi_2(x_2 \otimes h_2),\Phi_2(x_1 \otimes h_1))&=\mu^2_{\fL}(x_2,h_2x_1) \otimes h_2h_1 \label{eq:Simplifiedm2}
\end{align}
As a consequence of \eqref{eq:Simplifiedm1}, we have $H^*(CF(\eE,\eE))=H^*(\fL) \otimes R$ as a vector space.
\end{lemma}

\begin{proof}
 For $x \otimes h \in C^*(\fL) \otimes R$,
 \begin{align}
  &\Phi_2^{-1} \circ \mu^1 \circ \Phi_2(x \otimes h) \\
  =&\Phi_2^{-1} ( \sum_{g \in \Gamma} \partial_\fL(gx) \otimes I_{g^{-1}} \tau_h I_g )\\
  =&\Phi_2^{-1} ( \sum_{g \in \Gamma} g\partial_\fL(x) \otimes I_{g^{-1}} \tau_h I_g )\\
  =&\partial_{\fL}(x) \otimes h 
 \end{align}
where the second equality uses Corollary \ref{c:ActionCommute}.

For $x_i \otimes h_i \in C^*(\fL) \otimes R$, $i=1,2$, we have
 \begin{align}
  &\Phi_2^{-1} \circ \mu^2 \circ (\Phi_2(x_2 \otimes h_2),\Phi_2(x_1 \otimes h_1))\\
  =&\Phi_2^{-1} ( \sum_{g_1,g_2 \in \Gamma} \mu^2_\fL(g_2x_2,g_1x_1) \otimes I_{g_2^{-1}} \tau_{h_2} I_{g_2} I_{g_1^{-1}} \tau_{h_1} I_{g_1})
 \end{align}
For $\tau_{h_2} I_{g_2} I_{g_1^{-1}} \tau_{h_1}$ and hence $I_{g_2^{-1}} \tau_{h_2} I_{g_2} I_{g_1^{-1}} \tau_{h_1} I_{g_1}$ to be non-zero, we must have
\begin{align}
 1_{\Gamma} * g_1^{-1} * g_2=h_2^{-1}
\end{align}
and for any $g_2$, there is a unique $g_1$  ($=g_2h_2$) such that $g_1^{-1}g_2=h_2^{-1}$.
Therefore, the sum becomes
\begin{align}
 &\Phi_2^{-1} (\sum_{g_2 \in \Gamma} \mu^2_\fL(g_2x_2,g_2h_2x_1) \otimes I_{g_2^{-1}} \tau_{h_2} I_{h_2^{-1}} \tau_{h_1} I_{g_1} I_{g_2^{-1}} I_{g_2}) \\
=&\Phi_2^{-1} ( \sum_{g_2 \in \Gamma} g_2\mu^2_\fL(x_2,h_2x_1) \otimes I_{g_2^{-1}} \tau_{h_2} I_{h_2^{-1}} \tau_{h_1} I_{h_2} I_{g_2}) \\
=&\Phi_2^{-1} (\sum_{g_2 \in \Gamma} g_2\mu^2_\fL(x_2,h_2x_1) \otimes I_{g_2^{-1}} \tau_{h_2h_1} I_{g_2} )\\
=&\mu^2_{\fL}(x_2,h_2x_1) \otimes h_2h_1
 \end{align}
where the second equality uses that $\mu^2_\fL$ is $\Gamma$-equivariant, and the third equality uses $\tau_{h_2} I_{h_2^{-1}} \tau_{h_1} I_{h_2}= \tau_{h_1} I_{h_2}=\tau_{h_2h_1}$.

\end{proof}
\begin{comment}
\begin{rmk}
 In general, if one chooses certain auxiliary data to define the $A_{\infty}$ operations on $C^*(L)$ and lifts them to $C^*(\fL)$, then the same line of argument would give
 \begin{align}
  \Phi_2^{-1} \circ \mu^d \circ (\Phi_2(x_d \otimes h_d),\dots,\Phi_2(x_1 \otimes h_1))&=\mu^d_{\fL}(x_d,h_dx_{d-1}, \dots h_2x_1) \otimes h_d \dots h_1
 \end{align}
\end{rmk}
\end{comment}

\subsection{Spherical Lagrangians}
In this subsection, we apply the results from the previous subsections to the case that $L=P$ and
\begin{align}\label{eq:GammaCondition}
 \text{$P$ is diffeomorphic to $S^n/\Gamma$ for some $\Gamma \subset SO(n+1)$ and $P$ is spin}
\end{align}

\begin{rmk}\label{r:Spin}
A finite free quotient of a sphere $S^n/\Gamma$ is spin if and only if 
there exists $\widetilde{\Gamma} \subset Spin(n+1)$ such that the covering homomorphism
$Spin(n+1) \to SO(n+1)$ restricts to an isomoprhism $\widetilde{\Gamma} \simeq \Gamma$. 
\end{rmk}

First, we apply the discussion from Section \ref{ss:unwinding}:

\begin{lemma}\label{l:semisimpleCal}
% Suppose $L$ satisfies \eqref{eq:GammaCondition}.
 Let $E^i$ be local systems on $P$ for $i=0,1$.
 If $\cchar(\K)$ does not divide $|\Gamma|$, then
 $HF(\eE^0,\eE^1)=H^*(S^n)\otimes Hom_{\mathbb{K}[\Gamma]}(\eE^0_{o_L},\eE^1_{o_L})$ as a $\mathbb{K}$-vector space.
\end{lemma}

\begin{proof}
 We apply the Leray spectral sequence to Lemma \ref{l:cochainModel}.
 The $E_2$-page is given by
 $$E_2^{p,q}=H^p(\Gamma, H^q(S^n) \otimes_{\mathbb{K}} Hom_\K(\eE^0_{o_L}, \eE^1_{o_L}))$$
 where the $\Gamma$-action is given by $x \otimes \psi \mapsto x \otimes g \cdot \psi $
 and $g \cdot \psi =\rho^1(g^{-1}) \circ \psi \circ \rho^0(g)$.
 As a result, we have
 $$E_2^{p,q}=H^q(S^n) \otimes Ext^p_{\Gamma}(\Gamma, Hom_\K(\eE^0_{o_L}, \eE^1_{o_L}))$$
 When $\cchar(\K)$ does not divide $|\Gamma|$, $\mathbb{K}[\Gamma]$ is semi-simple by Maschke's theorem.
 Therefore, $Ext^p_{\Gamma}(\Gamma, Hom_\K(\eE^0_{o_L}, \eE^1_{o_L})) \neq 0$ only if $p=0$.
 It implies that the spectral sequence degenerate at $E_2$-page and the result follows from the fact that
 $Ext^0_{\Gamma}(\Gamma, Hom_\K(\eE^0_{o_L}, \eE^1_{o_L}))$ consists of $\psi \in Hom_\K(\eE^0_{o_L}, \eE^1_{o_L})$
 such that $g \cdot \psi=\psi$, which is clearly $Hom_{\mathbb{K}[\Gamma]}(\eE^0_{o_L},\eE^1_{o_L})$.
\end{proof}

\begin{corr}\label{c:HSphere}
% Suppose $L$ satisfies \eqref{eq:GammaCondition}.
  Let $\eE^0$ be any local system on $P$ corresponding to an irreducible representation of $\Gamma$.
 If $\cchar(\K)$ does not divide $|\Gamma|$, then $HF(\eE^0,\eE^0)=H^*(S^n)$.
\end{corr}

\begin{proof}
 It follow from Lemma \ref{l:semisimpleCal} and Schur's lemma $Hom_{\mathbb{K}[\Gamma]}(\eE^0_{o_L},\eE^0_{o_L})=\mathbb{K}$. 
 %(The semi-simplicity is not needed to apply Schur's lemma).
 Notice that, the ring structure is also determined uniquely by dimension and degree reason.
\end{proof}

Now, we want to compute the cohomological endomorphism algebra structure of the universal local system on $P$ using Lemma \ref{l:Identifyingm1m2}.
Since the universal local system on $P$ plays a distinguished role in the paper, we denote it by $\eP$.
We define $\mu^1,\mu^2$ on $C^*(\fP) \otimes R$ by \eqref{eq:Simplifiedm1} and \eqref{eq:Simplifiedm2}, respectively.
By \eqref{eq:Simplifiedm1}, we know that $H^*(C^*(\fP) \otimes R)$ is given by $H^*(\fP) \otimes R$.
We are going to determine the algebra structure in the next lemma.
Before that, we recall a convention 

\begin{convention}\label{conv:DG2Ainf}
 If $C$ is a differential graded algebra (eg. a $\K$-algebra with no differential), then $C$ is viewed as an $A_{\infty}$ algebra by
 \begin{align}
  \mu^1(a)&=(-1)^{|a|}\partial(a) \\
  \mu^2(a_1,a_0)&=(-1)^{|a_0|}a_1a_0
 \end{align}
and $\mu^k=0$ for $k \ge 3$, where $a,a_0,a_1 \in C$ and $\partial$ is the differential of $C$.
\end{convention}

\begin{lemma}\label{l:universal}
% Suppose $L$ satisfies \eqref{eq:GammaCondition}.
%diffeomorphic to $S^n/\Gamma$ for $\Gamma \subset SO(n+1)$ that can be lifted to $Spin(n+1)$. 
Let $\eP$ be the universal local system on $P$ and $R:=\K[\Gamma]$.
 Then the Floer cohomology $HF(\eP,\eP)=H^*(S^n) \otimes_{\K} R$ as a $\K$-algebra, where the 
 ring structure on the right is the product of the standard ring structure.
\end{lemma}

\begin{proof}
Pick a Morse model such that $C^*(P)$ has only one degree $0$ generator $e$ and one degree $n$ generator $f$.
 The corresponding Morse complex $C^*(\fP)$ has 
 $|\Gamma|$ degree $0$ generator $\{g\fe\}_{g \in \Gamma}$ and $|\Gamma|$ degree $n$ generator $\{g\ff\}_{g \in \Gamma}$.
 It is clear that $\sum_g g\fe$ represents the unit of $H^0(\fP)$.
Therefore, $\{[\sum_g g\fe] \otimes h\}_{h \in \Gamma}$ are the degree $0$ generators of $H(C^*(\fP) \otimes R)$ (see \eqref{eq:Simplifiedm1}).
 
 Similarly, if $x$ represents a generator of $H^n(\fP)$, then $\{[x] \otimes h\}_{h \in \Gamma}$
 are the degree $n$ generators of $H(C^*(\fP) \otimes R)$.
 It follows from \eqref{eq:Simplifiedm2} that
 \begin{align}
  &\mu^2([\sum_g g\fe] \otimes h_2,[\sum_g g\fe] \otimes h_1)=[\sum_g g\fe] \otimes h_2h_1 \\
  &\mu^2([x] \otimes h_2,[\sum_g g\fe] \otimes h_1])=[x] \otimes h_2h_1 \\
  &\mu^2([\sum_g g\fe] \otimes h_2,[x] \otimes h_1)=(-1)^{|x|}[h_2x] \otimes h_2h_1=(-1)^{|x|}[x] \otimes h_2h_1
 \end{align}
 Therefore, $H(C^*(\fP) \otimes R)=H^*(S^n) \otimes_{\K} R$  as a $\K$-algebra (see Convention \ref{conv:DG2Ainf}).
The result now follows from Lemma \ref{l:cochainModel}, \ref{l:Identifyingm1m2} (see \eqref{eq:GammaInvariantModel}, \eqref{eq:Phi2}).
\end{proof}

Let $\theta_g=1_{H^0(S^n)} \otimes g \in  H^0(S^n) \otimes R$. 
By Lemma \ref{l:universal}, we have a left $\Gamma$-action on $HF(\eE,\eP)$ given by
\begin{align}\label{eq:AnotherLeftAction}
 x \mapsto [\mu^2(\theta_g,x)]
\end{align}
for any $\eE \in \cF$.
On the other hand, we have another left $\Gamma$-action on $CF(\eE,\eP)$ given by \eqref{eq:Left_on_cE}, which descends to a  left $\Gamma$-action
on the cohomology $HF(\eE,\eP)$.

\begin{lemma}\label{l:ActionCoincide}
 When $\eE=\eP$, the two left $\Gamma$-actions \eqref{eq:AnotherLeftAction} and \eqref{eq:Left_on_cE} on $\theta_{1_\Gamma} \in HF(\eP,\eP)$ coincide.
\end{lemma}

\begin{proof}
 We use the notations in the proof of Lemma \ref{l:universal}.
 The element $\theta_h \in H^0(S^n) \otimes R$ is represented by $\sum_g g\fe \otimes h \in C^0(\fP) \otimes R$.
 We have (see \eqref{eq:Phi2})
 \begin{align}
  \Phi_2(\sum_g g\fe \otimes h)&=\sum_{g_2,g_1}  g_2g_1\fe \otimes I_{g_2^{-1}} \tau_h I_{g_2} \\
  &=\sum_{g}  g\fe \otimes I_{g^{-1}}( \sum_{g'} I_{g'} \tau_h I_{(g')^{-1}})I_{g}
 \end{align}
Undoing the trivialization \eqref{eq:TrivializedIdentification}, we have
\begin{align}
 \Phi^{-1}(\Phi_2(\sum_g g\fe \otimes h))=\sum_{g}  g\fe \otimes I_{c_\fe}( \sum_{g'} I_{g'} \tau_h I_{(g')^{-1}})I_{c_\fe^{-1}}
\end{align}
where $c_\fe:[0,1] \to \fP$ is a path from $o_\fP$ to $\fe$.
With respect to the identification $(CF((\fP,\fE),(\fP,\fE)))^{\Gamma} =CF(\eP,\eP)$ (see \eqref{eq:GammaInvariantModel}),
\begin{align}
 \Phi^{-1}(\Phi_2(\sum_g g\fe \otimes h))=I_{\pi \circ c_\fe}( \sum_{g'} I_{g'} \tau_h I_{(g')^{-1}})I_{(\pi \circ c_\fe)^{-1}} \in Hom(E_e,E_e) \subset CF(\eP,\eP)
\end{align}
Without loss of generality, we can assume $e=o_L$ so
\begin{align}
\sum_{g'} I_{g'} \tau_h I_{(g')^{-1}} \in Hom(E_{o_L},E_{o_L}) \subset CF(\eP,\eP)
\end{align}
represents $\theta_h$ under the isomorphism $HF^0(\eP,\eP)=H^0(S^n) \otimes R$.

For each $y \in \Gamma \subset E_{o_L}$, there is a unique $g'$($=h*y$)
such that $\tau_h I_{(g')^{-1}}(y) \neq 0$.
Therefore, $\sum_{g'} I_{g'} \tau_h I_{(g')^{-1}}(y) =hy$ for all $y \in E_{o_L}$.
In particular, it means that 
\begin{align}
 \sum_{g'} I_{g'} \tau_h I_{(g')^{-1}}= h  (\sum_{g'} I_{g'} \tau_{1_{\Gamma}} I_{(g')^{-1}})
\end{align}
so $\theta_{h}=h \theta_{1_{\Gamma}}$ and hence $\mu^2(\theta_h,\theta_{1_{\Gamma}})=(-1)^{|\theta_{1_{\Gamma}}|}\theta_h=h\theta_{1_{\Gamma}}$ as desired.

\begin{comment}
Now, we can compute the left action $\mu^2(\theta_h, \cdot)$.
If $\psi \in Hom(E_{q},E_{q}) \subset CF(\eP,\eP)$ for some $q \in C^*(L)$, then for $u$ contributing to $\mu^2_L(o_L,q)$ in $C^*(L)$, we have
\begin{align}
 &\mu^u(\sum_{g'} I_{g'} \tau_h I_{(g')^{-1}}, \psi) \\
 =& \sign(u) I_{\partial_2 u} (\sum_{g'} I_{g'} \tau_h I_{(g')^{-1}}) I_{\partial_1 u}  \psi I_{\partial_0 u}
\end{align}
Moreover, we can arrange the auxiliary data such that there is a unique such $u$, $\sign(u)=1$, $I_{\partial_0 u}=Id$ and $I_{\partial_1 u}=I_{\partial_2 u}^{-1}$.
For any $a \in \pi^{-1}(q) \subset E_q$, we can identify $\psi(a)$ with a homotopy class of path from $o_L$ to $q$.
Therefore, for $\tau_h I_{(g')^{-1}} I_{\partial_1 u}  \psi(a)$ to be non-zero, we need to have
\begin{align}
 \psi(a) * [\partial_1 u] * (g')^{-1}=h^{-1}
\end{align}
There is a unique such $g'$ ($=h*\psi(a) * [\partial_1 u]$) so 
\begin{align}
 &\mu^2(\sum_{g'} I_{g'} \tau_h I_{(g')^{-1}}, \psi)(a)\\
 =&I_{\partial_2 u} (\sum_{g'} I_{g'} \tau_h I_{(g')^{-1}}) I_{\partial_1 u}  \psi I_{\partial_0 u}(a) \\
 =&I_{\partial_1 u}^{-1} I_{h*\psi(a) * [\partial_1 u]} (1_{\Gamma}) \\
 =&I_h\psi(a)
\end{align}
which coincide with the left $\Gamma$-action \eqref{eq:Left_on_cE} on $\psi$ by $h$.

Since the auxiliary choices we made (such that there is a unique $u$) does not affect $\mu^2$ on the cohomological level, the result follows. 
\end{comment}
\end{proof}

\begin{rmk}\label{r:CohomologicalUnit}
 From the proof of Lemma \ref{l:ActionCoincide}, we see that the identity morphism at $E_{o_L}$ represents the cohomological unit.
 It is in general true that if one picks a Morse cochain complex for a Lagrangian submanifold $L$ such that there is a unique degree $0$
 generator $e_L$ representing the cohomological unit of $C^*(L)$, then the identity morphism of $E_{o_{L}}$ is a cohomological unit of $CF(\eE,\eE)$, where $\eE$
 is a local system on $L$.
\end{rmk}

\begin{corr}\label{c:ActionCoincide}
 The two left $\Gamma$-actions \eqref{eq:AnotherLeftAction} and \eqref{eq:Left_on_cE} on $HF^k(\eE,\eP)$ coincide, up to $(-1)^k$, for all $\eE \in \cF$.
\end{corr}

\begin{proof}
 Let $x \in HF(\eE,\eP)$.
 We have
 \begin{align}
  [\mu^2(\theta_g,x)]&=[\mu^2(\mu^2(\theta_g,\theta_{1_{\Gamma}}),x)]\\
  &=[\mu^2(g\theta_{1_{\Gamma}},x)] \\
  &=[g\mu^2(\theta_{1_{\Gamma}},x)] \\
  &=(-1)^{|x|}gx
 \end{align}
where the first equality uses Lemma \ref{l:universal}, the second equality uses Lemma \ref{l:ActionCoincide}, the third equality uses Corollary \ref{c:ActionCommute}
and the last equality uses that $\theta_{1_{\Gamma}}$ is a cohomological unit.
\end{proof}

%\begin{rmk}
% For a DG left $\Gamma$-module $V$, when we view $V$ as  
%\end{rmk}

Similarly, for any $\eE \in \cF$, we have a right $\Gamma$-action on $HF(\eP,\eE)$ given by
\begin{align}\label{eq:AnotherRightAction}
 x \mapsto [\mu^2(x,\theta_g)]
\end{align}
and another right action on $HF(\eP,\eE)$ given by \eqref{eq:Right_on_cE}.
The analogue of Corollary \ref{c:ActionCoincide} holds, (i.e. $\mu^2(\theta_{1_{\Gamma}},\theta_{h})=\theta_{h}= \theta_{1_{\Gamma}}h$)
and we leave the details to readers.

\begin{corr}\label{c:ActionCoincideRight}
 The two right $\Gamma$-actions \eqref{eq:AnotherRightAction} and \eqref{eq:Right_on_cE} on $HF(\eP,\eE)$ coincide (without additional factor of $-1$) for all $\eE \in \cF$.
\end{corr}

\begin{comment}

\begin{rmk}
 A quick verification of $Ext^p_\Gamma(\mathbb{K}, (\eE^0)^*_{o_L} \otimes_{\mathbb{K}} E^1_{o_L})=Ext^p_\Gamma(\eE^0_{o_L}, E^1_{o_L})$.
 Pick a projective resolution $P$ of $\mathbb{K}$.
 We want to identify $$Hom_\Gamma(P^k,(\eE^0)^*_{o_L} \otimes_{\mathbb{K}} E^1_{o_L})=Hom_\Gamma(P^k \otimes E^0_{o_L},E^1_{o_L})$$
 as well as the differentials $d:Hom_\Gamma(P^k,(\eE^0)^*_{o_L} \otimes_{\mathbb{K}} E^1_{o_L}) \to Hom_\Gamma(P^{k+1},(\eE^0)^*_{o_L} \otimes_{\mathbb{K}} E^1_{o_L})$
 and $d:Hom_\Gamma(P^k \otimes E^0_{o_L},E^1_{o_L}) \to Hom_\Gamma(P^k \otimes E^0_{o_L},E^1_{o_L})$ for all $k$.
 We define $\Phi: Hom_\Gamma(P^k,(\eE^0)^*_{o_L} \otimes_{\mathbb{K}} E^1_{o_L}) \to Hom_\Gamma(P^k \otimes E^0_{o_L},E^1_{o_L})$
 by $\Phi(\phi)(p \otimes e)=\phi(p)(e)$.
 If $\phi$ is $\Gamma$-equivariant, then  $\Phi(\phi)(pg \otimes eg)=\phi(pg)(eg)=(\phi(p)g)(eg)=\phi(p)(e)$.
 Similarly, if $\Phi(\phi)$ is $\Gamma$-equivariant, then $\phi(pg)(eg)=(\phi(p)(e))g$ for all $p,g$.
 Therefore, $\phi$ is $\Gamma$-equivariant.
 The identification of differential is left as an exercise.

\end{rmk}

\begin{rmk}
 In Lemma \ref{c:universal}, if $char\mathbb{K}$ does not divide $|\Gamma|$, then as a vector space
 $$HF(\eE,\eE)=H^*(S^n) \otimes_{\K} \K[\Gamma]=\oplus_{i} H^*(S^n) \otimes Ext(\eE^i,E^i)^{\oplus dim(\eE^i)^2}=\oplus_{i} H^*(S^n) \otimes \mathbb{K}^{\oplus dim(\eE^i)^2}$$
 where the sum is over all irreducible representation of $\Gamma$.
\end{rmk}
\end{comment}

\subsection{Equivariant evaluation}\label{ss:EquivariantEv}

In this subsection, we want to give the definition of 
\begin{align}
 T_{\eP}(\eE):=Cone(hom_{\mathcal{F}}(\eP,\eE) \otimes_{\Gamma} \eP \xrightarrow{ev} \eE)
\end{align}
that arises in \eqref{eq:EqTwist} in the context of Fukaya category.
%For simplicity, we 
We will keep the exposition minimal and self-contained here, and a thorough treatment of 
the general algebraic mechanism is postponed to Section \ref{s:Algebraic perspective}.

Let $\cF^{\perf}$ be the DG category of perfect $A_{\infty}$ right $\cF$-modules.
We have a cohomologically full and faithful Yoneda embedding \cite[Section (2g)]{Seidelbook}
\begin{align}
 \cY: \cF \to \cF^{\perf}
\end{align}
By abuse of notation, we use $\eE$ to denote $\cY(\eE)$ for $\eE \in Ob(\cF)$.

Let $P$ be a Lagrangian brane such that $\pi_1(P)=\Gamma$, and $\eP$ be the object with underlying Lagrangian $P$ equipped with the universal local system $E$.
Let $\eE \in Ob(\cF)$.
By Corollary \ref{c:ActionCommute}, we know that (see  \eqref{eq:Right_on_cE})
\begin{align}
 \mu^1_\cF(\psi)g=\mu^1_\cF(\psi g)
\end{align}
for $\psi \in hom_{\mathcal{F}}(\eP,\eE)$ so $hom_{\mathcal{F}}(\eP,\eE)$ is a DG right $\Gamma$-module.
%with $\partial (\psi)=(-1)^{|\psi|}\mu^1_\cF(\psi)$ and $\psi \cdot g=\psi g$.
%We prefer to view every DG category/module as an $A_{\infty}$ category/module with no higher order terms
%so the $A_{\infty}$ structural maps are given by  
%$\mu^{1|0}_{hom_{\mathcal{F}}(\eE,L)}(\psi) =\mu^1_\cF(\psi)$,
%$\mu^{1|1}_{hom_{\mathcal{F}}(\eE,L)}(\psi,g)=\psi g$.

Given a DG right $\Gamma$-module $V$, we define an object $V\otimes_\Gamma \eP\in Ob(\cF^{perf})$ as follows:
For every $X \in Ob(\cF)$, we have a cochain complex
\begin{align}
 (V \otimes_{\Gamma} \eP)(X):=V \otimes_{\Gamma} hom_\cF(X,\eP)
\end{align}
where the left  $\Gamma$-actions on $hom_\cF(X,\eP)$ is given by \eqref{eq:Left_on_cE}.
By Corollary \ref{c:ActionCommute}, we have
\begin{align}
\left\{
 \begin{array}{ll}
  \mu^1_V(vg) \otimes \psi=\mu^1_V(v)g \otimes \psi \\
  v \otimes \mu^1_\cF(g\psi)=v \otimes g\mu^1_\cF(\psi)
 \end{array}
\right.
\end{align}
for $v \otimes \psi \in V \otimes_{\Gamma} hom_\cF(X,\eP)$ so 
\begin{align}
\mu^{1|0}: v \otimes \psi \mapsto (-1)^{|\psi|-1}\mu^1_V(v) \otimes \psi + v \otimes \mu_\cF^1(\psi)
\end{align}
is a well-defined differential on $V \otimes_{\Gamma} hom_\cF(X,\eP)$.

The $A_{\infty}$ right $\cF$ module structure on $ V \otimes_{\Gamma} \eP$ is given by
\begin{align}
 \mu^{1|d-1}: (v \otimes \psi, x_{d-1}, \dots,x_1) \mapsto v \otimes \mu_\cF^d(\psi,x_{d-1},\dots,x_1)
\end{align}
for $v \otimes \psi \in V \otimes_{\Gamma} hom_\cF(X_d,\eP)$
and $x_j \in hom_\cF(X_{j},X_{j+1})$.
The morphism $\mu^{1|d-1}$ is well-defined by Corollary \ref{c:ActionCommute} and we leave it to readers to check that $\{\mu^{1|j}\}_{j=0}^\infty$
satisfies $A_{\infty}$ module relations \cite[Equation (1.19)]{Seidelbook}.
In particular, we have an $A_{\infty}$ right $\cF$ module
$hom_{\mathcal{F}}(\eP,\eE) \otimes_{\Gamma} \eP$.

Now we want to define an $A_{\infty}$ morphism
\begin{align}
 ev_\Gamma: hom_{\mathcal{F}}(\eP,\eE) \otimes_{\Gamma} \eP \to \eE
\end{align}
as follows. For $\psi^2 \otimes \psi^1 \in hom_{\mathcal{F}}(\eP,\eE) \otimes_{\Gamma} hom_\cF(X_d,\eP)$
and $x_j \in hom_\cF(X_{j},X_{j+1})$, we define
\begin{align}
 ev_\Gamma^d: (\psi^2 \otimes \psi^1, x_{d-1}, \dots,x_1) \mapsto \mu^{d+1}_\cF(\psi^2,\psi^1, x_{d-1}, \dots,x_1)。 \label{eq:evaluationMap}
\end{align}
The well-definedness follows from Corollary \ref{c:ActionCommute} again.
The fact that $ev_\Gamma=\{ev_\Gamma^d\}_{d=1}$ defines an $A_{\infty}$ morphism follows from the $A_{\infty}$ relations of $\cF$.
As a consequence, we can define 
\begin{align}
 T_{\eP}(\eE):=Cone(hom_{\mathcal{F}}(\eP,\eE) \otimes_{\Gamma} \eP \xrightarrow{ev_\Gamma} \eE) \label{eq:TcE}
\end{align}
as the $A_{\infty}$ mapping cone for the $A_{\infty}$ morphism $ev_\Gamma$ (see \cite[Section (3e)]{Seidelbook}).
In particular, for $X \in Ob(\cF)$, we have a cochain complex
\begin{align}
 T_{\eP}(\eE)(X)=(hom_{\mathcal{F}}(\eP,\eE) \otimes_{\Gamma} hom_{\cF}(X,\eP))[1] \oplus hom_{\cF}(X,\eE)
\end{align}
with differential and multiplication given by
\begin{align}
 \mu^1_{T_{\eP}(\eE)}(\psi^2 \otimes \psi^1,x)&=((-1)^{|\psi^1|-1}\mu^1_\cF(\psi^2) \otimes \psi^1 + \psi^2 \otimes \mu_\cF^1(\psi^1),\mu^1_\cF(x)+\mu^2_\cF(\psi^2, \psi^1)) \label{eq:mappingCone}\\
 \mu^2_{T_{\eP}(\eE)}((\psi^2 \otimes \psi^1,x),a)&=(\psi^2 \otimes \mu^2_\cF(\psi^1,a), \mu^2_\cF(x,a)+\mu^3_\cF(\psi^2,\psi^1,a)) \label{eq:mappingMu2}
 \end{align}

Finally, we want to state a functorial property of $T_{\eP}(\eE)$.

\begin{corr}[Corollary \ref{c:functoralityAuto}]\label{c:ConeIndepAuxData}
 Let  $\cF_0, \cF_1$ be the Fukaya categories with respect to two different sets of choices of auxiliary data. 
 %Let $\eE$ be the universal local system on $P$ and 
 The Lagrangian branes $\eP, \eE$ above will be denoted by $\eP_j, \eE_j$, respectively, when
 we regard them as objects in $\cF_j$, for $j=0,1$.
 Let $\cG:\cF_0 \to \cF_1$ be a quasi-equivalence
 sending $\eP_0$ to $\eP_1$ and $\eE_0$ to $\eE_1$.
 Then 
 \begin{align*}
  \cG(T_{\eP_0}( \eE_0)) \simeq T_{\eP_1}(\eE_1)
 \end{align*}

\end{corr}

There is a direct proof along the same line as \cite[Lemma $5.6$]{Seidelbook} and is left to interested readers.  We will show that it is also an immediate consequence
from the general disucssion in Section \ref{s:Algebraic perspective}.

\section{Symplectic field theory package} % (fold)
\label{sec:review_of_symplectic_field_theory_and_dimension_formulae}

In this section, we review the necessary background for symplectic field theory.
More details can be found in \cite{SFTcompact}, \cite{EES05}, \cite{EES07}, \cite{CEL10}, \cite{CDGG} and etc.

\subsection{The set up} % (fold)
\label{sub:the_set_up_and_dimension_formulae}

% subsection the_set_up_and_dimension_formulae (end)

Let $(Y,\alpha)$ be a contact manifold with a contact form $\alpha$.

\begin{defn}\label{d:cylindricalJ}
A cylindrical almost complex structure on $SY:=(\R \times Y,d(e^r\alpha))$ is an almost complex structure such that
\begin{itemize}
 \item $J$ is invariant under $\R$ action
 \item $J(\partial_r)=R_{\alpha}$, where $R_{\alpha}$ is the Reeb vector field of $\alpha$
 \item $J(\ker(\alpha))=\ker(\alpha)$
 \item $d\alpha(\cdot, J\cdot)|_{\ker(\alpha)}$ is a metric on $\ker(\alpha)$
\end{itemize}
\end{defn}

The set of cylindrical almost complex structures is denoted by $\mathcal{J}^{cyl}(Y,\alpha)$.
If $I \subset \mathbb{R}$ is an interval, we call $J$ a cylindrical almost complex structure of $(I \times Y,d(e^r\alpha))$ if $J=J'|_{I \times Y}$
for some $J' \in \mathcal{J}^{cyl}(Y,\alpha)$.
Let $(M,\omega,\theta)$ be a Liouville domain with a separating contact hypersurface $(Y, \alpha=\theta|_Y)$ such that $Y \cap \partial M=\emptyset$.
By the neighborhood theorem, there is a neighborhood $N(Y) \subset M$ of $Y$ such that we have a symplectomorphism
\begin{align}
\Phi_{N(Y)}:(N(Y),\omega|_{N(Y)}) \simeq ((-\epsilon,\epsilon) \times Y,d(e^t \alpha)) \label{eq:Yneighborhood}
\end{align}
for some $\epsilon >0$.

Let $J^0$ be a compatible almost complex structure on $M$ such that $(\Phi_{N(Y)})_*(J^0|_{N(Y)})$ is cylindrical.
We say that a smooth family of compactible almost complex structure $(J^{\tau})_{\tau \in [0,\infty)}$ on $M$ is {\it adjusted to $N(Y)$} if
\begin{align}
 \left\{ \label{eq:SFTJ}
 \begin{array}{ll}
 J^\tau|_{M \setminus N(Y)}=J^0|_{M \backslash N(Y)} \text{ for all } \tau \\
 \text{for each $\tau$, we have }\Phi^{\tau}_{N(Y)}:(N(Y),J^{\tau}|_{N(Y)}) \simeq ((-(\tau+\epsilon),\tau+\epsilon) \times Y, (J^\tau)') 
 \end{array}
\right.
\end{align}
where $\Phi^{\tau}_{N(Y)}$ is an isomorphism of almost complex manifolds, the diffeomorphism $\Phi^{\tau}_{N(Y)} \circ (\Phi_{N(Y)})^{-1}$ is the identity on the $Y$ factor, and 
$(J^\tau)'$ is the unique cylindrical almost complex structure such that $(J^\tau)'|_{(-\epsilon,\epsilon) \times Y}=(\Phi_{N(Y)})_*(J^0|_{N(Y)})$.

Let $M^-$ be the Liouville domain in $M$ bounded by $Y$ and $M^+=M \backslash (M^- \backslash \partial M^-)$.
Let $SM^-$ and $SM^+$ be the positive and negative symplectic completion of $M^-$ and $M^+$, respectively.
Given $(J^{\tau})_{\tau \in [0,\infty)}$, there is a unique almost complex structure $J^-$, $J^Y$ and $J^+$ on $SM^-$, $SY$ and $SM^+$, respectively, such that
$(M^-,J^\tau|_{M^-})$, $(N(Y),J^{\tau}|_{N(Y)})$ and $(M^+,J^\tau|_{M^+})$ converges to $(SM^-,J^-)$, $(SY,J^Y)$ and $(SM^+,J^+)$, respectively, as $\tau$ goes to infinity.
More details about this splitting procedure can be found in \cite[Section $3$]{SFTcompact}.

\begin{rmk}\label{r:FlexibleJ}
 There is a variant for being adjusted to $N(Y)$.
 For a fixed number $R \ge 0$, we call a smooth family of compactible almost complex structure $(J^{\tau})_{\tau \in [3R,\infty)}$ on $M$ is {\it $R$-adjusted to $N(Y)$} if
 \eqref{eq:SFTJ} is satisfied but the property of $(J^\tau)'$ is replaced by the following conditions.
 \begin{align*}
 \left\{ 
 \begin{array}{ll}
 (J^\tau)'|_{[-(\tau+\epsilon-2R),\tau+\epsilon-2R] \times Y} \text{ is cylindrical for all }\tau \\
 (J^\tau)'|_{(-\epsilon,\epsilon) \times Y}=(\Phi_{N(Y)})_*(J^0|_{N(Y)}) \text{ for all }\tau \\
 (J^{\tau_1})'|_{(-(\tau_1+\epsilon),-(\tau_1+\epsilon-2R)] \times Y}= (\phi_{\tau_1,\tau_2}^-)_* (J^{\tau_2})'|_{(-(\tau_2+\epsilon),-(\tau_2+\epsilon-2R)] \times Y} 
 \text{ for all } \tau_1,\tau_2\\
 (J^{\tau_1})'|_{[\tau_1+\epsilon-2R,\tau_1+\epsilon) \times Y}= (\phi_{\tau_1,\tau_2}^+)_* (J^{\tau_2})'|_{[\tau_2+\epsilon-2R,\tau_2+\epsilon) \times Y} \text{ for all } \tau_1,\tau_2
 \end{array}
\right.
\end{align*}
 where $\phi_{\tau_1,\tau_2}^-:(-(\tau_2+\epsilon),-(\tau_2+\epsilon-2R)] \times Y \to (-(\tau_1+\epsilon),-(\tau_1+\epsilon-2R)] \times Y $
 and $\phi_{\tau_1,\tau_2}^+:[\tau_2+\epsilon-2R,\tau_2+\epsilon) \times Y \to [\tau_1+\epsilon-2R,\tau_1+\epsilon) \times Y$
 are the $r$-translation.
% where $J'|_{(-(\tau+\epsilon-2T),\tau+\epsilon-2T) \times Y}$ is the unique cylindrical almost complex structure such that $J'|_{(-\epsilon,\epsilon) \times Y}=(\Phi_{N(Y)})_*(J^0|_{N(Y)})$.
% we can require that for $\tau >3T$

When $R=0$, being $R$-adjusted to $N(Y)$ is the same as being adjusted to $N(Y)$.
For $R > 0$, we can also define $J^\pm, J^Y$ accordingly.
 %Instead of having $J^+$ (resp $J^-$) to be cylindrical over the end $(-\infty,0] \times \partial M^+$ (resp $[0,\infty) \times \partial M^-$), we allow
 In this case, $J^+$ (resp. $J^-$) are cylindrical over the end $(-\infty,-2R] \times \partial M^+ \subset SM^+$ (resp. $[2R,\infty) \times \partial M^- \subset SM^-$).
\end{rmk}

Let $L$ be a Lagrangian submanifold in $M$ such that 
$L \cap N(Y)=(-\epsilon,\epsilon)\times \Lambda$ for some (possibly empty) Legendrian submanifold $\Lambda$.
Let $L^\pm:=L \cap M^\pm$. We define $SL^-=L^- \cup (\R_{\ge 0} \times \Lambda) \subset SM^-$
and $SL^+=L^+ \cup (\R_{\le 0} \times \Lambda) \subset SM^+$ which are the cylindrical extensions of $L^-$ and $L^+$ with respect to the symplectic completion.
We denote $\R \times \Lambda \subset SY$ by $S\Lambda$.

The main ingredient we needed from \cite{SFTcompact} is the following compactness result in symplectic field theory.

\begin{thm}[\cite{SFTcompact} Theorem 10.3 and Section 11.3; see also \cite{CEL10}]\label{t:SFTcompactness}
    Let $L_j$, $j=0,\dots,d$ be a collection of embedded exact Lagrangian submanifolds in $M$ such that $L_i \pitchfork L_j$ for all $i \neq j$.  
    Let $(Y,\alpha)\subset M$ be a contact type hypersurface and 
    $(N(Y),\w|_{N(Y)})\cong ((-\epsilon,\epsilon)\times Y, d(e^r\alpha))$ be a neighborhood of $Y$ such that 
    $L_i\cap N(Y)=(-\epsilon,\epsilon)\times \Lambda_i$ for some (possibly empty) Legendrian submanifold $\Lambda_i$ of $Y$.

    Let $J^\tau$ be a smooth family of almost complex structures $R$-adjusted to $N(Y)$.
    Let $x_0 \in CF(L_0,L_d)$ and $x_j \in CF(L_{j-1},L_j)$ for $j=1,\dots,d$.
    If there exists a sequence $\{\tau_k\}_{k=1}^\infty$ such that $\lim_{k \to \infty} \tau_k =\infty$, and a sequence
    $u_k \in \eM^{J^{\tau_k}}(x_0;x_d,\dots,x_1)$,  then $u_k$ converges to a holomorphic building $u_\infty=\{u_v\}_{v \in V(\eT)}$ in the sense of \cite{SFTcompact}.
\end{thm}

We remark that each $J^\tau$ above is a domain independent almost complex structure (see Remark \ref{r:ModuliNotationJ})
and we do not need to assume $u_k$ to be transversally cut out to apply Theorem \ref{t:SFTcompactness}.

The rest of this subsection is devoted to the description/definition of $u_\infty=\{u_v\}_{v \in V(\eT)}$ in Theorem \ref{t:SFTcompactness}.
The definition is quite well-known so we only give a quick review and introduce necessary notations along the way.

First, $\eT$ is a tree with $d+1$ semi-infinite edges and one of them is distinguished which is called the root.
The other semi-infinite edges are ordered from $1$ to $d$ and called the leaves.
Let $V(\eT)$ be the set of vertices of $T$.
For each $v \in V(\eT)$, we have a punctured Riemannian surface $\Sigma_v$.
If $\partial \Sigma_v \neq \emptyset$, there is a distinguished boundary puncture which is denoted by $\xi^v_0$.
After filling the punctures of $\Sigma_v$, it is a topological disk so we can label the other boundary punctures of $\Sigma_v$
by $\xi^v_1, \dots \xi^v_{d_v}$ counterclockwise along the boundary, where $d_v+1$ is the number of boundary punctures of $\Sigma_v$.
Let $\partial_{j} \Sigma_v$ be the component of $\partial \Sigma_v$ that goes from $\xi^v_j$ to $\xi^v_{j+1}$ for $j=0,\dots,d_v-1$, and 
$\partial_{d_v} \Sigma_v$ be the component of $\partial \Sigma_v$ that goes from $\xi^v_{d_v}$ to $\xi^v_{0}$.
If $\partial \Sigma_v = \emptyset$, then $\Sigma_v$ is a sphere after filling the punctures.
%In both cases, the interior punctures of $\Sigma_v$ are denoted  by $\eta^v_{1}, \dots, \eta^v_{n_v}$ for $n_v \in \mathbb{N}$.

There is a bijection $f_v$ from the punctures of $\Sigma_v$ to the edges in $\eT$ adjacent to $v$.
Moreover, $f_v(\xi^v_0)$ is the edge closest to the root of $\eT$ among edges adjacent to $v$.
If $v,v'$ are two distinct vertices adjacent to $e$, then $f_v^{-1}(e)$ and $f_{v'}^{-1}(e)$
are either both boundary punctures or both interior punctures.
We call $e$ a {\it boundary edge} (resp. an {\it interior edge}) if $f_v^{-1}(e)$ is a boundary (resp. an interior) puncture.
We can glue $\{\Sigma_v\}_{v \in V(\eT)}$ along the punctures according to the edges and $\{f_v\}_{v \in V(\eT)}$ 
(i.e. $\Sigma_v$ is glued with $\Sigma_{v'}$ by identifying $f_v^{-1}(e)$ with $f_{v'}^{-1}(e)$ if $v,v'$ are two distinct vertices adjacent to $e$).
After gluing, we will get back $S$, the domain of $u_k$, topologically.
Therefore, there is a unique way to assign Lagrangian labels to $\partial \Sigma_v$ such that it is compatible with gluing and coincides
with that on $\partial S$ after gluing all $\Sigma_v$ together.
We denote the resulting Lagrangian label on $\partial_j \Sigma_v$ by $L_{v,j}$.

There is a level function 
$l_\eT:V(\eT) \to \{0,\dots, n_\eT\}$ for some positive integer $n_\eT$.
If $l_\eT(v)=0$, then  $u_v : \Sigma_v \to SM^-$ is a $J^-$-holomorphic curve such that $u_v({\partial_j \Sigma_v}) \subset SL_{v,j}^-$.
If $l_\eT(v)=1,\dots, n_\eT-1$, then $u_v : \Sigma_v \to SY$ is a $J^Y$-holomorphic curve such that $u_v({\partial_j \Sigma_v}) \subset S\Lambda_{v,j}$.
If $l_\eT(v)=n_\eT$, then  $u_v : \Sigma_v \to SM^+$ is a $J^+$-holomorphic curve such that $u_v({\partial_j \Sigma_v}) \subset SL_{v,j}^+$.

If $v \neq v'$ are adjacent to the same edge $e$ in $\eT$, then $|l_\eT(v)-l_\eT(v')| \le 1$.
If $l_\eT(v)+1=l_\eT(v')$ and $e$ is a boundary (resp. interior) edge, then there is a Reeb chord (resp. orbit)
which is the positive asymptote of $u_v$ at $f^{-1}_v(e)$, and the negative asymptote of $u_{v'}$ at $f^{-1}_{v'}(e)$ (see Convention \ref{con:Ends}).
If $l_\eT(v)=l_\eT(v')$, then $e$ is necessarily a boundary edge, $l_\eT(v)=l_\eT(v') \in \{0,n_\eT\}$ and  
$u_v,u_{v'}$ converges to the same Lagrangian intersection point at $f^{-1}_v(e),f^{-1}_{v'}(e)$, respectively.
If $e$ is the $j^{th}$ semi-infinite edge adjacent to $v$, then $u_v$ is asymptotic to $x_j$ at $f^{-1}_v(e)$.

Finally, for each $j=1,\dots,n_\eT-1$, there is at least one $v \in V(\eT)$ such that $l_\eT(v)=j$ and $u_v$ is not a trivial cylinder
(i.e. $u_v$ is not a map $\R \times [0,1] \to SY$ (or $\R \times S^1 \to SY$) such that 
\begin{align}
 u_v(s,t)=(f_r(s),f_Y(t)) \in \R \times Y \label{eq:TrivialCyl}
\end{align}
for some $f_r$, $f_Y$).
We use $\eM^{J^{\infty}}(x_0;x_d,\dots,x_1)$ to denote the set of such holomorphic buildings.

 %Let $u_\infty=\{u_v\}_{v \in V(\eT)}$ be a holomorphic building obtained as a limit in Theorem \ref{t:SFTcompactness}.
 Let $V^{core}$ be the set of vertices $v \in V(\eT)$ such that more than one Lagrangian appears in the Lagrangian labels of $\partial \Sigma_v$.
 Let $V^{\partial}$ be the set of vertices $v \in V(\eT)$ such that there is only one Lagrangian appears in the Lagrangian labels of $\partial \Sigma_v$.
Let $V^{int}$ be the set of vertices $v \in V(\eT)$ such that $\partial \Sigma_v= \emptyset$.
In particular, we have $V(\eT)=V^{core} \sqcup V^{\partial} \sqcup V^{int}$.
 Let $\eT^{core}$, $\eT^{\partial}$ and $\eT^{int}$ be the subgraphs of $\eT$ which consists of vertices $V^{core}$, $V^{\partial}$ and $V^{int}$, respectively,
 and edges adjacent to its vertices (see Figure \ref{fig:eg_tree-2} for an example).
 
\begin{lemma}\label{l:TreeConfig}
 The graphs $\eT^{(1)}:=\eT^{core} \backslash \eT^{int} $ and $\eT^{(2)}:=(\eT^{core} \cup \eT^{\partial}) \backslash \eT^{int}$  are planar trees (so, in particular, they are connected).
\end{lemma}

\begin{proof}
 Let $G$ be the smallest subtree of $\eT$ containing $\eT^{(1)}$. If there is a vertice $v$ in $G$ such that $v \in V^{int}$, then it would imply that $S$, the domain of $u_k$, is not a disk.
 If there is a vertice $v$ in $G$ such that $v \in V^{\partial}$, then  it would imply that there is a Lagrangian 
 that appears more than once in the Lagrangian label of $\partial S$. Both of these situations are not possible.
 
 Similarly, let $G'$ be the smallest subtree of $\eT$ containing $\eT^{(2)}$. If there is a vertice $v$ in $G'$ such that $v \in V^{int}$, then 
 it would imply that $S$ is not a disk and we get a contradiction.
 
 As a result, $G=\eT^{(1)}$ and $G'=\eT^{(2)}$ so
 both $\eT^{(1)}$ and $\eT^{(2)}$  are trees.
 
 The fact that $\eT^{(1)}$ and $\eT^{(2)}$   are planar follows from the fact that we can order the boundary punctures of $\Sigma_v$,
 for $v \in V^{core} \cup V^{\partial}$, in a way that is compatible with the boundary orientation.
\end{proof}

\begin{figure}
  \centering
  \includegraphics[]{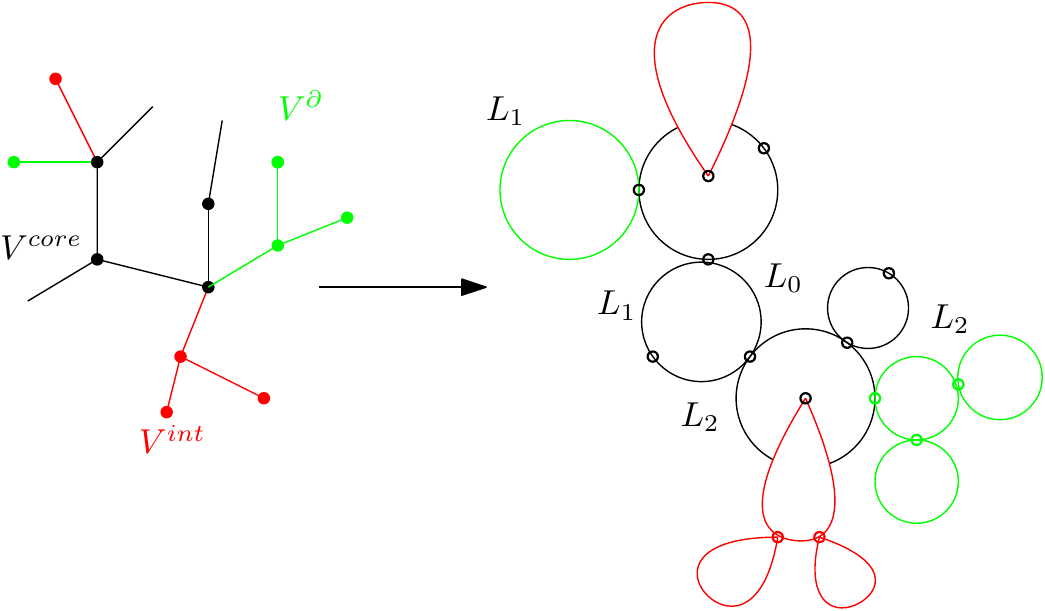}
  \caption{A tree $\eT$ with $2$ leaves. Black dots: elements in $V^{core}$; Green dots: elements in $V^{\partial}$; Red dots: elements in $V^{int}$;
  Black tree: $\eT^{core} \backslash (\eT^{\partial} \cup \eT^{int})$; Green sugbraph: $\eT^{\partial} \backslash \eT^{int}$; Red subgraph: $\eT^{int}$}
  \label{fig:eg_tree-2}
\end{figure}

\begin{convention}\label{con:Ends}
 We need to explain the convention of strip-like ends and cylindrical ends we use for punctures of $\Sigma_v$.
 Let $e$ be an edge in $\eT$ and $v \neq v'$ are the vertices adjacent to $e$.

 First assume that $l_\eT(v)+1=l_\eT(v')$.
 If $e$ is a boundary (resp. interior) edge, we use an outgoing/positive strip-like end \eqref{eq:outgoingStripEnds} (resp. cylindrical end) for $f_v^{-1}(e)$,
 where an outgoing/positive cylindrical end for $f_v^{-1}(e)$ is a holomorphic embedding of $\epsilon_{v,e}:\{z=s\exp(\sqrt{-1}t) \in \C| s \ge 1\} \to \Sigma_v$ such that
 $\lim_{|z| \to \infty} \epsilon_{v,e}(z)=f_v^{-1}(e)$.
 With respect to coordinates given by the strip-like (resp. cylindrical) end $\epsilon_{v,e}$, we have
 \begin{align}
 \left\{
 \begin{array}{ll}
  \lim_{s \to \infty} \pi_Y(u_v(\epsilon_{v,e}(s,t)))=x(Tt) \text{   (resp. } \gamma(Tt) \text{)} \\
  \lim_{s \to \infty} \pi_\R(u_v(\epsilon_{v,e}(s,t)))=\infty
 \end{array}
\right.
 \end{align}
for some Reeb chord $x$ (resp. orbit $\gamma$) and some $T>0$, where $\pi_Y,\pi_\R$ are the projection from $SY$ to the two factors.
In this case, we call $x$ (resp. $\gamma$) the positive asymptote of $u_v$ at $f_v^{-1}(e)$.

 %In this case, $u_v$ converges to a Reeb chord (or orbit) at $r=\infty$ near $f_v^{-1}(e)$
 %With respect to coordinates near $f_v^{-1}(e)$ provided by an outgoing/positive strip-like end \eqref{eq:outgoingStripEnds} (resp. cylindrical end),
 %where an outgoing/positive cylindrical end for $f_v^{-1}(e)$ is a holomorphic embedding of $\epsilon_{v,e}:\{z \in \C| |z| \ge 1\} \to \Sigma_v$ such that
 %$\lim_{|z| \to \infty} \epsilon_{v,e}(z)=f_v^{-1}(e)$.
 On the other hand, we use an incoming/negative strip-like end \eqref{eq:incomingStripEnds} (resp. cylindrical end) for $f_{v'}^{-1}(e)$,
 where an  incoming/negative cylindrical end for $f_{v'}^{-1}(e)$ is a 
 holomorphic embedding of $\epsilon_{v',e}:\{z=s\exp(\sqrt{-1}t) \in \C| 0 < s \le 1\} \to \Sigma_{v'}$ such that
 $\lim_{|z| \to 0} \epsilon_{v',e}(z)=f_{v'}^{-1}(e)$.
 With respect to coordinates given by the strip-like (resp. cylindrical) end $\epsilon_{v',e}$, we have
 \begin{align}
 \left\{
 \begin{array}{ll}
  \lim_{s \to 0} \pi_Y(u_v(\epsilon_{v',e}(s,t)))=x(Tt) \text{   (resp. } \gamma(Tt) \text{)} \\
  \lim_{s \to 0} \pi_\R(u_v(\epsilon_{v',e}(s,t)))=-\infty
 \end{array}
\right.
 \end{align}
 or some Reeb chord $x$ (resp. orbit $\gamma$) and some $T>0$.
 In this case, we call $x$ (resp. $\gamma$) the negative asymptote of $u_{v'}$ at $f_{v'}^{-1}(e)$.

 %$u_{v'}$ converges to a Reeb chord (or orbit) at $r=-\infty$ near $f_{v'}^{-1}(e)$
 %with respect to coordinates near $f_{v'}^{-1}(e)$ 
%provided by an incoming/negative strip-like end \eqref{eq:incomingStripEnds} (resp. cylindrical end), where an  incoming/negative cylindrical end for $f_{v'}^{-1}(e)$ is a 
 %holomorphic embedding of $\epsilon_{v',e}:\{z \in \C| 0 < |z| \le 1\} \to \Sigma_{v'}$ such that
 %$\lim_{|z| \to 0} \epsilon_{v',e}(z)=f_{v'}^{-1}(e)$.
 
 If $l_\eT(v)=l_\eT(v')$ and, say $v$ is closer to the root of $\eT$ than $v'$, then
 we use an  outgoing/positive strip-like end for $f_v^{-1}(e)$ and
  an incoming/negative strip-like end for  $f_{v'}^{-1}(e)$.
  Similarly, the intersection point that they are asymptotic to is the positive asymptote of $u_v$ at $f_v^{-1}(e)$ and the negative 
  asymptote of $u_{v'}$ at $f_{v'}^{-1}(e).$
\end{convention}

\subsection{Gradings}\label{ss:Grading}

Let $P \subset (M,\omega,\theta)$ be a Lagrangian submanifold which satisfies \eqref{eq:GammaCondition}.
In particular, $H^1(P,\R)=0$ and $P$ is an exact Lagrangian.
The round metric on $S^n$ descends to a Riemannian metric on $P$.
Let $U$ be a Weinstein neighborhood of $P$ and we identify $\partial U$ with the set of covectors of $P$ having a common small fixed norm.
Without loss of generality, we can assume that $\theta|_U=\theta_{T^*P}$, where $\theta_{T^*P}$ is the standard Liouville one-form on $T^*P$.
Let $\alpha_0:=\theta|_{\partial U}$ be the standard contact form on $\partial U$.
Eventually, we will apply Theorem \ref{t:SFTcompactness} along a perturbation  $(\partial U)'$ of $\partial U$.
Since $((\partial U)',\theta|_{(\partial U)'}) \simeq (\partial U, \alpha')$ for a perturbation $\alpha'$ of $\alpha_0$, we will need to understand
the Reeb dynamic of $\alpha'$.
Therefore, it is helpful to explain the Reeb dynamic of $(\partial U, \alpha_0)$ first.
We assume $\Lambda_i:=L_i \cap \partial U$ are (possibly empty) unions of cospheres at points of $P$.
There are four types of asymptotes that can appear for $u_v$ near the punctures.

\begin{enumerate}
 \item Lagrangian intersection points between $SL_i^\pm$ and $SL_j^\pm$ in $SM^\pm$,
 \item Reeb chords from $\Lambda_i$ to  $\Lambda_j$ in $Y$ for $i \neq j$, 
 \item Reeb chords from $\Lambda_i$ to  itself in $Y$, and
 \item  Reeb orbits from $\Lambda_i$ to  itself in $Y$
\end{enumerate}

We want to discuss the grading for each of these types.

\subsubsection{Type one}

Let $\Omega^2$ be the nowhere-vanishing section of $(\Lambda_\C^{top}T^*M)^{\otimes 2}$ which equals to $1$ with respect to the chosen trivialization (see the beginning of Section \ref{sec:floer_cohomology_with_local_systems}).
For a Lagrangian subspace $V \subset T_pM$ and a choice of basis $\{X_1, \dots, X_n\}$ of $V$, we define
\begin{align}
 Det^2_{\Omega}(V):=\frac{\Omega^2(X_1,\dots,X_n)}{\|\Omega^2(X_1,\dots,X_n)\|} \in S^1
\end{align}
which is independent of the choice of basis.
We assumed that the Lagrangians are $\mathbb{Z}$-graded.
It means that $L_i$ is equipped with a continuous function $\theta_{L_i}:L_i \to \R$, which is called the grading function,
such that $e^{2\pi\sqrt{-1} \theta_{L_i}(p)}=Det^2_{\Omega}(T_pL_i)$ for all $p \in L_i$.

At each transversal intersection point $x \in L_i \cap L_j$, we have two graded Lagrangian planes $T_xL_i$, $T_xL_j$ inside $T_xM$.
The grading of $x$ as a generator of $CF(L_i,L_j)$ is given by the Maslov grading from $T_xL_i$ to $T_xL_j$ which is
\begin{align}
 |x|=\iota(T_xL_i,T_xL_j)=n+\theta_{L_j}(x)-\theta_{L_i}(x)-2Angle(T_xL_i,T_xL_j) \label{eq:MaslovGrading}
\end{align}
where $Angle(T_xL_i,T_xL_j)= \sum_{j=1}^n \beta_j$ and $\beta_j \in (0,\frac{1}{2})$ are such that there is a unitary basis $u_1, \dots,u_n$
of $T_xL_i$ satisfying $T_xL_j=Span_\R\{e^{2\pi\sqrt{-1}\beta_j}u_j\}_{j=1}^n$.
If we regard $x$ as an element in $CF(L_j,L_i)$, then we have  $\iota(T_xL_j,T_xL_i)=n-\iota(T_xL_i,T_xL_j)$.

\begin{convention}\label{c:DualGenNotation}
 For a generator $x \in CF(L_i,L_j)$, we use $x^\vee$ to denote the generator of $CF(L_j,L_i)$ which represents the same intersection point as $x$.
 Therefore, we have $|x|=n-|x^\vee|$.
\end{convention}

Since $SM^-=T^*P$, there is a preferred choice of trivialization of $(\Lambda_\C^{top}T^*SM^-)^{\otimes 2}$ such that the grading 
functions on cotangent fibers and the zero section are constant functions (see \cite{SeGraded}).
Without loss of generality, we can assume that the restriction to $M^-$ of the choice of trivialization of $(\Lambda_\C^{top}T^*M)^{\otimes 2}$ we picked 
coincides with that of $(\Lambda_\C^{top}T^*SM^-)^{\otimes 2}$.
We call that the cotangent fibers and the zero section are in {\it canonical relative grading} if the following holds:
\begin{align}\label{eq:CanonicalGrading}
 CF(P,T^*_qP) \text{ is concentrated at degree $0$}
\end{align}
for all $q \in P$.

We refer readers to \cite{SeGraded}, \cite[Section $11,12$]{Seidelbook} for more about Maslov gradings.

\subsubsection{Type two}\label{sss:TypeTwoGrading}

In general, if we have a Reeb chord $x=(x(t))_{t \in [0,1]}$ from  $\Lambda_0$ to $\Lambda_1$ in a contact manifold $(Y,\alpha)$ such that $S\Lambda_i$ are graded Lagrangians in $SY$ for both $i$, 
we will assign a grading to $x$ by regarding $x$ as a Hamiltonian chord 
between graded Lagrangians  $S\Lambda_0$ and $S\Lambda_1$ in the symplectic manifold $SY$ as follows:
There is an appropriate Hamiltonian $H$ in $SY$ that depends only on the radial coordinate $r$ 
such that the Reeb vector field $R_{\alpha}$ in $Y$ coincides with the restriction of the Hamiltonian vector field $X_H$ to $\{0\} \times Y$.
Let $\phi^H$ be the time-one flow of $H$.
We identify $x(t) \in Y$ with $(0,x(t)) \in SY$ so $x$ is a $H$-Hamiltonian chord.
We have graded Lagrangian subspaces $(\phi^H)_*T_{x(0)}S\Lambda_0$ and $T_{x(1)}S\Lambda_1$ in $T_{x(1)}SY$.
Let 
\begin{align}
 K_x:=(\phi^H)_*T_{x(0)}S\Lambda_0 \cap T_{x(1)}S\Lambda_1 \label{eq:Kx}
\end{align}
The grading $|x|$ of $x$ is defined to be
\begin{align}
 |x|=\iota((\phi^H)_*T_{x(0)}S\Lambda_0/K_x, T_{x(1)}S\Lambda_1/K_x) \label{eq:ChordIndex}
\end{align}
where the Maslov grading (see \eqref{eq:MaslovGrading}) is computed in the symplectic vector space $T_{x(1)}M/(K_x+ J(K_x))$.
More details about the Maslov grading assigned to non-transversally intersecting graded Lagrangian subspaces can be found in, for example, \cite[Section $4.1$]{MW15}.

\begin{comment}
Let $x=(x(t))_{t \in [0,1]}$ be a Reeb chord from $\Lambda_i$ to $\Lambda_j$ and identify it with a Hamiltonian chord from $L_i$ to $L_j$  inside $\partial U$
for some appropriate Hamiltonian $H$ that depends only on the radial coordinate $r$ along the Liouville flow.
In this case, the Reeb vector field $R_{\alpha_0}$ in $\partial U$ coincides with the restriction of the Hamiltonian vector field $X_H$ to $\partial U$.
Let $\phi^H$ be the time-one flow of $H$.
We have graded Lagrangian subspaces $(\phi^H)_*T_{x(0)}L_i$ and $T_{x(1)}L_j$ in $T_{x(1)}M$.
Let 
\begin{align}
 K_x:=(\phi^H)_*T_{x(0)}L_i \cap T_{x(1)}L_j \label{eq:Kx}
\end{align}
The grading $|x|$ of $x$ is defined to be
\begin{align}
 |x|=\iota((\phi^H)_*T_{x(0)}L_i/K_x, T_{x(1)}L_j/K_x) \label{eq:ChordIndex}
\end{align}
where the Maslov grading is computed in the symplectic vector space $T_{x(1)}M/(K_x+ J(K_x))$.
%\begin{align}
%T_{x(1)}M/(K_x+ J(K_x)) =T_{x(1)}M/\R \langle \partial_r, R_{\alpha_0} \rangle 
% (\phi^H)_*T_{x(0)}L_i \cap T_{x(1)}L_j
%\end{align}

In the generic case (see \eqref{eq:GenericPointCondition}), we have
\begin{align}\label{eq:IntersectingTangent}
 K_x=(\phi^H)_*T_{x(0)}L_i \cap T_{x(1)}L_j= \R \langle \partial_r \rangle
\end{align}
so we have
\begin{align}
|x|=\iota((\phi^H)_*T_{x(0)}L_i/\R \langle \partial_r \rangle, T_{x(1)}L_j/\R \langle \partial_r \rangle) 
\end{align}
\end{comment}

Now, we go back to our situation and assume $x$ is a Reeb chord from $\Lambda_i$ to $\Lambda_j$ in $(\partial U,\alpha_0)$.
Since $L_i$ is graded, $S\Lambda_i$ has a grading function in $S(\partial U)$ inherited from $L_i$. 
The computation of $|x|$  is done in the literature (e.g. \cite{AbbS12} \cite{Ab12}, where they indeed proved $HW(T_{\mathbf{q}_i})\cong k[u]$ for $|u|=-(n-1)$) and we recall it here.

Without loss of generality, we assume $\Lambda_i$ and $\Lambda_j$ are connected.
Let $q_i,q_j \in P$ be such that $T^*_{q_i}P \cap \partial U=\Lambda_i$ and $T^*_{q_j}P \cap \partial U=\Lambda_j$.
We equip the cotangent fibers and $P$ with the canonical relative grading (see \eqref{eq:CanonicalGrading}).
The grading functions of $L_i$ and $L_j$ differs from the grading functions of  $T^*_{q_i}P$ and $T^*_{q_j}P$
near $\Lambda_i$ and $\Lambda_j$, respectively, by an integer.
In the following, we will assume the grading functions coincide and the actual $|x|$ can be recovered by adding back the integral differences of the grading functions. 

Let $\fq_i \in \fP$ be a lift of $q_i$.
Each Reeb chord $x$ from $\Lambda_i$ to $\Lambda_j$ corresponds to a geodesic from $q_i$ to $q_j$, which can be lifted to
a geodesic $\fx$ from $\fq_i$ to a point $\fq_j \in \fP$ such that $\pi(\fq_j)=q_j$.
If 
\begin{align}\label{eq:GenericPointCondition}
 \text{$\fq_j$ is not the antipodal point of $\fq_i$ and $\fq_j \neq \fq_i$}
\end{align}
then there is a unique closed geodesic (assumed to have length $2\pi$) passing through $\fq_i$ and $\fq_j$.
Therefore, for each interval $I_k=(k\pi,(k+1 \pi))$, $k\in \N$, there is a unique geodesic from $\fq_i$ to $\fq_j$
with length lying inside $I_k$.
If the length of $\fx$ lies in $I_k$, then
\begin{align}
 |x|=-k(n-1) \label{l:ChordNondegenGrading}
\end{align}
For generic $q_i,q_j$, every lifts $\fq_i,\fq_j$ of $q_i,q_j$ satisfies \eqref{eq:GenericPointCondition}.

\subsubsection{Type three}\label{sss:TypeThreeGrading}

There are four kinds of Reeb chords from $\Lambda_i$ to itself.
First, if $x$ is a Reeb chord from one connected component of $\Lambda_i$ to a different one, then the computation of $|x|$ reduces to the previous case (Section \ref{sss:TypeTwoGrading}).
For the remaining three kinds, we assume $\Lambda_i=T^*_{q_i}P \cap \partial U$, i.e. it has exactly one connected component.
Let $\fq_i$ be a lift of $q_i$ and $\fx:[0,1] \to \fP$ be the lift of the geodesic such that $\fx(0)=\fq_i$, $\fq_i':=\fx(1)$ and $\pi(\fq_i')=q_i$.
The three possibilities are
\begin{enumerate}
 \item $\fq_i,\fq_i'$ satisfy \eqref{eq:GenericPointCondition} (with $\fq_i'$ replacing $\fq_j$), or
 \item $\fq_i'$ is the antipodal point of $\fq_i$, or
 \item $\fq_i=\fq_i'$
\end{enumerate}
For the first case, the computation of $|x|$ reduces to the previous one again (Section \ref{sss:TypeTwoGrading}).
For the second and the third cases, we have (see \eqref{eq:Kx} for the meaning of $K_x$)
\begin{align}
 &K_x= T_{x(1)}L_i \\
 &|x|=\iota((\phi^H)_*T_{x(0)}L_i/K_x, T_{x(1)}L_j/K_x)=-k(n-1)
\end{align}
where $k\pi$ is the length of $\fx$, and the term $-k(n-1)$ is exactly the (integral)
difference of the values of the grading functions at $(\phi^H)_*T_{x(0)}L_i$ and $T_{x(1)}L_i$ as graded Lagrangian planes.

We want to point out that in the second and third cases $x$ lies in a Morse-Bott family $S_{x}$ of Reeb chords from $\Lambda_i$ to itself
and $\dim(S_{x})=n-1$.

\subsubsection{Type four}\label{sss:TypeFourGrading}

Reeb orbits of $\partial U$ are graded by the  Robbin-Salamon index \cite{RobbinSalamon} (see also \cite[Section $5$]{Bo02}), which is a generalization of the
Conley-Zehnder index to the degenerated case.
To define the  Robbin-Salamon index of a Reeb orbit $\gamma$, we need to pick a 
symplectic trivialization $\Phi_\gamma$ of $\xi$ along  $\gamma$ subject to the following compatiblity condition:
Together with the obvious trivialization of $\R \langle \partial_r,R_{\alpha_0} \rangle$, $\Phi_\gamma$ gives a 
symplectic trivialization of $TM$ along $\gamma$, and hence a trivialization of $(\Lambda_\C^{top}T^*M)^{\otimes 2}$ along $\gamma$.
The compatiblity condition is that the induced trivialization of $(\Lambda_\C^{top}T^*M)^{\otimes 2}$ along $\gamma$ 
coincides with the trivialization of $(\Lambda_\C^{top}T^*M)^{\otimes 2}$ we picked in the beginning of Section \ref{sec:floer_cohomology_with_local_systems}.
One may show that there is $\Phi_\gamma$ satisfying the compatiblity condition.

We can now define a path of symplectic matrices $(\Phi_t)_{t \in [0,1]}$ given by $\Phi_t:=(\phi^{R}_t)_*:\xi_{\gamma(0)} \to \xi_{\gamma(t)} \simeq \xi_{\gamma(0)}$,
where $\phi^{R}_t$ is the time-$t$ flow generated by $R_{\alpha_0}$ and the last isomorphism is given by $\Phi_\gamma$.
We can assign the Robbin-Salamon index for $\Phi_t$ as follows:
First, we isotope (relative to end points) $\Phi_t$ to a path of symplectic matrices $\Phi_t'$
such that $\ker(\Phi_t'-Id) \neq 0$ happens at finitely many times $t=t_1,\dots,t_k $ and for each $t_j$,
the crossing form $J\frac{d}{dt}|_{t=t_j}(\Phi_t')$ is non-degnerate on $\ker(\Phi_{t_j}'-Id)$.
The signature of $J\frac{d}{dt}|_{t=t_j}(\Phi_t')$ is denoted by $\sigma(t_j)$ and the Robbin-Salamon index is defined by
\begin{align}
 \mu_{RS}(\Phi_t):=\frac{1}{2}\sigma(0) +\sum_{j=1}^k \sigma(t_j) +\frac{1}{2}\sigma(1) \label{eq:RSindexFormula}
\end{align}
where $\sigma(1)$ is defined to be zero if $\Phi_1$ is invertible.
The index is independent of the choice of $\Phi_{t_j}'$.
The  Robbin-Salamon index of $\gamma$ with respect to the trivialization $\Phi_\gamma$ is
\begin{align}
 \mu_{RS}(\gamma):=\mu_{RS}(\Phi_t)
\end{align}
Any two choices of $\Phi_\gamma$ satisfying the compatiblity condition would give the same index.

There are two kinds of Reeb orbits $\gamma$ in $\partial U$, namely, contractible in $U$ or not.
We are only interested in the case that $\gamma$ is contractible in $U$, which means that it can be lifted to a Reeb orbit in $\partial \fU$
that is contractible in $\fU$.
The lifted Reeb orbit corresponds to a geodesic loop $l_{\gamma}$ in $\fP$.
The Robbin-Salamon index $\mu_{RS}(\gamma)$ is computed in \cite[Lemma $7$]{Hind04} (the proof there can be directly generalized to all $n$)
\begin{align}
 \mu_{RS}(\gamma)=2k(n-1)
\end{align}
where $k$ is the covering multiplicity of $l_{\gamma}$ with respect to the simple geodesic loop, or equivalently, $2k\pi$ is the length of $l_{\gamma}$.

We want to point out that $\gamma$ lies in a Morse-Bott family $S_{\gamma}$ of (unparametrized) Reeb orbits and $\dim(S_{\gamma})=2n-2$.

\subsection{Dimension formulae}

In this section, we first review the virtual dimension formula from \cite{Bo02}, where the domain of the pseudo-holomorphic map only has interior punctures.
Then we consider the case that the domain only has boundary punctures, and finally obtain the general formula by gluing.

Let $(Y^{\pm},\alpha^{\pm})$ be contact manifolds with contact forms $\alpha^\pm$.
We assume that every Reeb orbit $\gamma$ of $Y^\pm$ lies in a Morse-Bott family $S_{\gamma}$ of (unparametrized) Reeb orbits.
Let $(X,\omega_X)$ be a symplectic manifold such that there exists a compact set $K_X \subset X$ and $T_X \in \R_{>0}$ so that
$(X \backslash K_X, \omega_X|_{X \backslash K_X})$ is the disjoint union of the ends $([T_X,\infty) \times Y^+, d(e^r\alpha^+))$
and $((-\infty,-T_X] \times Y^-, d(e^r\alpha^-))$.
In this case, we have

\begin{lemma}[\cite{Bo02}, Corollary $5.4$]\label{l:VirdimFormula1}
 Let $\Sigma$ be a punctured Riemannian surface of genus $g$ and $\partial \Sigma =\emptyset$.
 Let $J$ be a compactible almost complex structure on $X$ that is cylindrical over the ends. 
 Let $u:\Sigma \to X$ be a $J$-holomorphic map with positive asymptotes $\{\gamma_j^+ \}_{j=1}^{s_+}$
 and negative asymptotes $\{\gamma_j^- \}_{j=1}^{s_-}$ (see Convention \ref{con:Ends}). 
 %for some $J \in \cJ^{cyl}(Y,\alpha)$.
 Then the virtual dimension of $u$ is given by
 \begin{align}
  \virdim(u)=&(n-3)(2-2g-s^+-s^-)+\sum_{j=1}^{s^+} \mu_{RS}(\gamma^+_j)-\sum_{j=1}^{s^-} \mu_{RS}(\gamma^-_j) 
  +\frac{1}{2}\sum_{j=1}^{s^+} \dim(S_{\gamma^+_j}) \nonumber \\
  &+\frac{1}{2}\sum_{j=1}^{s^-} \dim(S_{\gamma^-_j})+2c_1^{rel}(TX)([u]) \label{eq:VirdimFormula1}
 \end{align}
where $2c_1^{rel}(TX)([u])$ is 
the relative first Chern class computed with respect to the fixed symplectic trivializations along the Reeb orbits that we chose to compute $\mu_{RS}$ (see Section \ref{sss:TypeFourGrading}). 
\end{lemma}

\begin{proof}[Sketch of proof]
 As explained in Section \ref{sss:TypeFourGrading}, the trivialization $\Phi_{\gamma^\pm_j}$ of $\xi$ along $\gamma^\pm_j$ determines a path of symplectic matrices 
 $\Phi_t^{\pm,j}$. We can trivialize $TX$ along $\gamma^\pm_j$ using $\Phi_{\gamma^\pm_j}$ by adding the invariant directions $\partial_r, R_{\alpha}^\pm$.
 The corresponding path of symplectic matrices become $\overline{\Phi}_t^{\pm,j}=\Phi_t^{\pm,j} \oplus I_{2 \times 2}$, where $I_{2 \times 2}$ is the $2$ by $2$ identity matrix.
 By additivity property of $\mu_{RS}$, we have $\mu_{RS}(\overline{\Phi}_t^{\pm,j})=\mu_{RS}(\Phi_t^{\pm,j})+\mu_{RS}(I_{2 \times 2})=\mu_{RS}(\Phi_t^{\pm,j})$.
 
 If $\ker(\overline{\Phi}_1^{\pm,j}-Id)=0$ (which is never the case) for all $\gamma^\pm_j$, then the index of $u$ is given by
  \begin{align}
  \ind(u)=&n(2-2g-s^+-s^-)+\sum_{j=1}^{s^+} \mu_{RS}(\gamma^+_j)-\sum_{j=1}^{s^-} \mu_{RS}(\gamma^-_j)+2c_1^{rel}(TX)([u]) \label{eq:indOrbit}
 \end{align}
 If $\ker(\overline{\Phi}_1^{\pm,j}-Id) \neq 0$, then it contributes $\dim(\ker(\overline{\Phi}_1^{\pm,j}-Id))=\dim(S_{\gamma^\pm_j})+2$ (resp. $0$) to $\ind(u)$
 when $\gamma^\pm_j$ is a positive (resp. negative) asymptote.
 However, the definition of $\mu_{RS}$ already takes into account $\frac{1}{2}\dim(\ker(\overline{\Phi}_1^{\pm,j}-Id))$ so we have
  \begin{align}
 \ind(u)=&n(2-2g-s^+-s^-)+\sum_{j=1}^{s^+} (\mu_{RS}(\gamma^+_j)+\frac{1}{2}(\dim(S_{\gamma^+_j})+2)) \\
 &-\sum_{j=1}^{s^-} (\mu_{RS}(\gamma^-_j)-\frac{1}{2}(\dim(S_{\gamma^-_j})+2))+2c_1^{rel}(TX)([u]) \\
  =&n(2-2g-s^+-s^-)+\sum_{j=1}^{s^+} \mu_{RS}(\gamma^+_j)-\sum_{j=1}^{s^-} \mu_{RS}(\gamma^-_j) 
  +\frac{1}{2}\sum_{j=1}^{s^+} \dim(S_{\gamma^+_j}) \nonumber \\
  &+\frac{1}{2}\sum_{j=1}^{s^-} \dim(S_{\gamma^-_j})+(s^++s^-)+2c_1^{rel}(TX)([u]) \label{eq:IndexU1}
 \end{align}
 Finally, to obtain the virtual dimension, we need to
 add the dimension of the Teichm\"uller  space that $\Sigma$ lies, which is $6g-6+2(s^++s^-)$.
 It gives the formula \eqref{eq:VirdimFormula1}.
\end{proof}

%We want to explain how \eqref{l:VirdimFormula1} is obtained.
\begin{comment}
We can extend $R_\alpha$ to a $r$-translational invariant vector field $\tilde{R}_\alpha$ on the symplectization $SY=\R_r \times Y$.
Let $\tilde{\gamma}(t)=(0,\gamma(t)) \in SY$.
Let $\phi^{\tilde{R}_\alpha}_t$ be the time $t$ flow generated by $\tilde{R}_\alpha$.
With respect to a choice of symplectic trivialization of $TSY$ along $\tilde{\gamma}$,
we have a 
path of symplectic matrices $\Phi_t$ given by $\Phi_t:=(\phi^{\tilde{R}_\alpha}_t)_*:T_{\tilde{\gamma}(0)}SY \to T_{\tilde{\gamma}(t)}SY \simeq T_{\tilde{\gamma}(0)}SY$,
where the last equality comes from the trivialization.
%such that $\Phi_0=Id$ and $\Phi_t=(\phi^{\tilde{R}_\alpha}_t)_*$.
The Robbin-Salamon index is computed using $\Phi_t$ (so it depends on the choice of symplectic trivialization of $TSY$ along $\tilde{\gamma}$).

By the Morse-Bott assumption,
\begin{align}\label{eq:LastTermRS}
 \ker(\Phi_0-\Phi_1)=T_{\gamma}S_\gamma + \langle \partial_r, R_\alpha \rangle
\end{align}
so $\dim(\ker (\Phi_0-\Phi_1))=\dim(S_\gamma)+2$.
The Robbin-Salamon index is computed using $\Phi_t$
and \eqref{eq:LastTermRS} contributes to the Robbin-Salamon index by $\frac{1}{2}(\dim(\ker (\Phi_0-\Phi_1)))=\frac{1}{2}\dim(S_\gamma)+2$.
\end{comment}

Lemma \ref{l:VirdimFormula1} applies when $Y^-=\emptyset$ (resp. $Y^+=\emptyset$).
%$SY$ is replaced by a symplectic manifold with positive (resp. negative) cylindrical end $[0,\infty) \times Y$ (resp. $(-\infty,0] \times Y$).
%In this case, $J$ is only required to be cylindrical on the positive (resp. negative) cylindrical end
In this case,
we only have positive (resp. negative) asymptotes and $s^-=0$ (resp. $s^+=0$).
However, when $\omega$ is exact, and if $u$ only has negative asymptotes, then $u$ is necessarily a constant map by action reason (see Corollary \ref{c:actionBoundoutside}) so it is not interesting.

\begin{example}
 The virtual dimension of $u:\C \to T^*S^n$ with the puncture asymptotic to a simple Reeb orbit is given by
 \begin{align}
  \virdim(u)=(n-3)(2-2(0)-1-0)+2(n-1)+\frac{1}{2}(2n-2)=4n-6
 \end{align}
 because $c_1^{rel}(TT^*S^n)=0$.
When $n=2$, then we have $\virdim(u)=2$ which is obtained in \cite[Lemma $7$]{Hind04}.
\end{example}

Now, we consider the relative setting.
A Lagrangian cobordism $L$ in $X$ 
is a Lagrangian
such that there exists $T>T_X$ so that $L \cap (-\infty,-T] \times Y^-=(-\infty,-T] \times \Lambda^-$
and $L \cap [T, \infty) \times Y^+=[T, \infty) \times \Lambda^+$
for some Legendrian submanifolds $\Lambda^\pm$ in $Y^\pm$.
Let $L_0,L_1$ be exact Lagrangian cobordisms such that $L_0 \pitchfork L_1$.
We assume that every Reeb chord $x$ from $\Lambda_0^{\pm}$ to $\Lambda_1^{\pm}$ lies in a Morse-Bott family $S_x$ of Reeb chords.
In this case, we define (see \eqref{eq:Kx})
\begin{align}
\mbb(x)=\dim(S_x)+1=\dim(K_x)
\end{align}
If $x \in L_0 \cap L_1$, we define $\mbb(x)=0$.
As always, we assume that $L_0,L_1$ are $\Z$-graded so all elements in $L_0 \cap L_1$, and all Reeb chords from
$\Lambda_0^{\pm}$ to $\Lambda_1^{\pm}$ are graded (see Section \ref{ss:Grading}).

\begin{lemma}\label{l:VirdimFormula2}
 Let $S \in \cR^{d+1}$ be equipped with Lagrangian labels $L_j$ on $\partial_j S$, where each $L_j$ is a Lagrangian cobordism.
 Let $J$ be a compatible almost complex structure on $X$ that is cylindrical on the ends.
 Let $u:S \to X$ be a $J$-holomorphic map with positive asymptotes $\{x_j^+ \}_{j=1}^{r^+}$ and
 negative asymptotes $\{x_j^- \}_{j=1}^{r^-}$ such that $u(\partial_j S) \subset L_j$.  Assume all asymptotes are Morse-Bott, then the virtual dimension of $u$ is given by
 \begin{align}
  \virdim(u)=n(1-r^-)+\sum_{j=1}^{r^-} (\iota(x^-_j)+\mbb(x^-_j))-\sum_{j=1}^{r^+} \iota(x^+_j)+(r^-+r^+-3) \label{eq:VirdimFormula2}
 \end{align} 
\end{lemma}

\begin{proof}[Sketch of proof]
 When all $x^\pm_j$ are Lagrangian intersection points, then the index of $u$ is given by (see \cite[Proposition $11.13$]{Seidelbook})
 \begin{align}
  \ind(u)=n(1-r^-)+\sum_{j=1}^{r^-} \iota(x^-_j)-\sum_{j=1}^{r^+} \iota(x^+_j) \label{eq:indChord}
 \end{align}
If $x^\pm_j$ is a Reeb chord, then the intersection of the graded Lagrangian subspaces $K_{x^\pm_j}$ is non-zero.
Similar to the proof of Lemma \ref{l:VirdimFormula1}, there are extra contributions to $\virdim(u)$ from the asymptotes.
This time, $x^\pm_j$ contributes $\dim(K_{x^\pm_j})=\mbb(x^\pm_j)$ (resp. $0$) to $\ind(u)$ when $x^\pm_j$ is a negative (resp. positive) asymptote.
The reversing of the roles of positive and negative asymptotes between here and the proof of Lemma \ref{l:VirdimFormula1}
can be understood from the fact that in \eqref{eq:indOrbit}, positive asymptotes contribute positively while in \eqref{eq:indChord}, positive asymptotes contribute negatively,
which in turn boils down to the reversing convention of the definition of indices between orbits and chords.

After all, we have
 \begin{align}
  \ind(u)=n(1-r^-)+\sum_{j=1}^{r^-} (\iota(x^-_j)+\mbb(x^-_j))-\sum_{j=1}^{r^+} \iota(x^+_j) \label{eq:IndexU2}
 \end{align}
 The last term of \eqref{eq:VirdimFormula2} comes from the dimension of $\cR^{d+1}$.
\end{proof}

Lemma \ref{l:VirdimFormula2} applies when $Y^-=\emptyset$ or $Y^+=\emptyset$.

\begin{example}
Let $q_0,q_1 \in S^n$ and $\Lambda_{i}$ be the unit cospheres at $q_i$.
Let $x$ be the shortest Reeb chord from $\Lambda_{0}$ to $\Lambda_1$ in the unit cotangent bundle of $S^n$.
The virtual dimension of $u:S \to T^*S^n$ such that $S \in \cR^3$, $u(\partial_0 S) \subset S^n$, 
$u(\partial_1 S) \subset T^*_{q_0}S^n$, $u(\partial_2 S) \subset T^*_{q_1}S^n$
with positive asymptotes $q_0$ and $x$ at $\xi_1$ and $\xi_2$, respectively, and a negative asymptote $q_1$ at $\xi_0$ is given by
\begin{align}
 \virdim(u)=n(1-1)+0-0-0=0 \label{eq:ExampleVirdim}
\end{align}
The computation of \eqref{eq:ExampleVirdim} is done by using the canonical relative grading (see Section \ref{ss:Grading}).
\end{example}

Now, we combine Lemma \ref{l:VirdimFormula1} and \ref{l:VirdimFormula2}.

\begin{lemma}\label{l:VirdimFormula3}
 Let $S$ be a disk with $r^++r^-$ boundary punctures and $s^++s^-$ interior punctures.
 Let $u:S \to X$ be a $J$-holomorphic map with 
 positive asymptotes $\{x_j^+ \}_{j=1}^{r^+}$ and
 negative asymptotes $\{x_j^- \}_{j=1}^{r^-}$  at boundary punctures, and 
  positive asymptotes $\{\gamma_j^+ \}_{j=1}^{s_+}$
 and negative asymptotes $\{\gamma_j^- \}_{j=1}^{s_-}$ at interior punctures such that $u(\partial S)$ lies in the corresponding Lagrangians determined by the boundary asymptotes.
 Then the virtual dimension of $u$ is given by
 \begin{align}
  \virdim(u)=
  &(n-3)(1-s^+-s^-)+\sum_{j=1}^{s^+} \mu_{RS}(\gamma^+_j)-\sum_{j=1}^{s^-} \mu_{RS}(\gamma^-_j) 
  +\frac{1}{2}\sum_{j=1}^{s^+} \dim(S_{\gamma^+_j}) \nonumber \\
  &+\frac{1}{2}\sum_{j=1}^{s^-} \dim(S_{\gamma^-_j})+2c_1^{rel}(TX)([u]) \nonumber \\
  &+\sum_{j=1}^{r^-} (\iota(x^-_j)+\mbb(x^-_j))-\sum_{j=1}^{r^+} \iota(x^+_j)-(n-1)r^-+r^+ \label{eq:VirdimFormula3}
 \end{align} 
\end{lemma}

\begin{proof}
  We follow the proof in \cite[Proposition $11.13$]{Seidelbook}.
  The domain $S$ is the connected sum of a disk $S_1$ with $r^++r^-$ boundary punctures and a sphere $S_2$ with  $s^++s^-$ interior punctures.
  Let $u_1:S_1 \to X$ be a $J$-holomorphic map with 
 positive asymptotes $\{x_j^+ \}_{j=1}^{r^+}$ and
 negative asymptotes $\{x_j^- \}_{j=1}^{r^-}$ 
 such that $u_1(\partial S_1)$ lies in the corresponding Lagrangians determined by the boundary asymptotes.
 Let $u_2:S_2 \to X$ be a $J$-holomorphic map with 
 positive asymptotes  $\{\gamma_j^+ \}_{j=1}^{s_+}$
 and negative asymptotes $\{\gamma_j^- \}_{j=1}^{s_-}$.
 Then, we have
 \begin{align}
  \ind(u)=\ind(u_1)+\ind(u_2)-2n
 \end{align}
which can be computed by \eqref{eq:IndexU1} and \eqref{eq:IndexU2}.
Finally, to get the virtual dimension, we need to
 add the dimension of the Teichm\"uller  space that $S$ lies, which is $(r^-+r^+-3)+2(s^++s^-)$.
 It gives the formula \eqref{eq:VirdimFormula3}.
\end{proof}

We want to use Lemma \ref{l:VirdimFormula1} and \ref{l:VirdimFormula3}
to derive some corollaries for the holomorphic buildings $u_{\infty}=(u_v)_{v \in V(\eT)}$ obtained in Theorem \ref{t:SFTcompactness}.
Let $v \neq v'$ be adjacent to the edge $e$. 
If $l_\eT(v)+1=l_\eT(v')$, then there is a Reeb chord $x$ (or orbit $\gamma$) which is the positive asymptote of $u_v$ at $f_v^{-1}(e)$ and
the negative asymptote of $u_{v'}$ at $f_{v'}^{-1}(e)$.
Let $u_{v} \#_x u_{v'}$ (resp $u_{v} \#_\gamma u_{v'}$) be a pseudo-holomorpic map with boundary and asymptotic conditions determined by gluing $u_v$ and $u_{v'}$ along $x$ (resp. $\gamma$).
%We want to compute $\virdim(u_{v} \#_x u_{v'})$ (resp $\virdim(u_{v} \#_\gamma u_{v'})$) in terms of $\virdim(u_v)$ and $\virdim(u_{v'})$.
By a direct application of Lemma \ref{l:VirdimFormula1} and \ref{l:VirdimFormula3}, we get
 \begin{align} \label{eq:virdimGlue1}
  \left\{
  \begin{array}{ll}
   \virdim(u_{v} \#_x u_{v'})&=\virdim(u_{v})+\virdim(u_{v'})-\dim(S_x) \\
   \virdim(u_{v} \#_\gamma u_{v'})&=\virdim(u_{v})+\virdim(u_{v'})-\dim(S_\gamma)
  \end{array}
\right.
 \end{align}
On the other hand, if $l_\eT(v)=l_\eT(v')$ so that there is a Lagrangian intersection point $x$ which is the positive asymptote of $u_v$ at $f_v^{-1}(e)$ and
the negative asymptote of $u_{v'}$ at $f_{v'}^{-1}(e)$, then we  have
\begin{align}
 \virdim(u_{v} \#_x u_{v'})&=\virdim(u_{v})+\virdim(u_{v'})+1 \label{eq:virdimGlue2}
\end{align}
where $u_{v} \#_x u_{v'}$ is defined analogously.

\subsection{Action}\label{ss:action}

This subsection discuss the action of the generators.
A similar discussion can be found in \cite{SFTcompact} and \cite[Section $3$]{CDGG}.

Let $L_0,L_1$ be exact Lagrangians in $(M,\omega,\theta)$.
It means that, for $j=0,1$, there exists a primitive function $f_j \in C^{\infty}(L_j,\R)$ such that $df_j=\theta|_{L_j}$.
For a Lagrangian intersection point $p \in CF(L_0,L_1)$, the action is 
\begin{align}
 A(p)=f_0(p)-f_1(p)
\end{align}
so $ A(p^\vee)=-A(p)$ (see Convention \ref{c:DualGenNotation}).
For a contact hypersurface $(Y,\alpha=\theta|_Y) \subset (M,\omega, \theta)$
and a Reeb chord $x:[0,l_x] \to Y$ from $\Lambda_0=L_0 \cap Y$ to $\Lambda_1=L_1 \cap Y$.
The length of $x$ is
\begin{align}
 L(x)= \int x^*\alpha =l_x
\end{align}
and the action is
\begin{align}
 A(x)=L(x)+f_0(x(0))-f_1(x(l_x)) \label{eq:ActionChordDefn}
\end{align}
Reeb orbits are special kinds of Reeb chords so the length and action of a Reeb orbit $\gamma$ is
\begin{align}
 L(\gamma)=A(\gamma)=\int \gamma^*\alpha
\end{align}

We have the following action control

\begin{lemma}[see \cite{CDGG}(Lemma $3.3$, Proposition $3.5$)]\label{l:actionBound}
 Let $u_{\infty}=\{u_v\}_{v \in V(\eT)}$ be a holomorphic building obtained in Theorem \ref{t:SFTcompactness}.
 If $u_v$ has positive asymptotes $\{x_j^+\}_{j=1}^{r^+}$, $\{\gamma_j^+\}_{j=1}^{s^+}$
 and negative asymptotes $\{x_j^-\}_{j=1}^{r^-}$, $\{\gamma_j^-\}_{j=1}^{s^-}$, then
 \begin{align}
  E_{\omega}(u_v):=\sum_{j=1}^{r^+} A(x_j^+) + \sum_{j=1}^{s^+} A(\gamma_j^+) -\sum_{j=1}^{r^-} A(x_j^-)-\sum_{j=1}^{s^-} A(\gamma_j^-) \ge 0 \label{eq:SumofA}
 \end{align}
The equality holds if and only if $u_v$ is a trivial cylinder (see \eqref{eq:TrivialCyl}).
 
\end{lemma}

Since $A(\gamma) >0$ for any Reeb orbit $\gamma$ and $A(x)>0$ if $x$ is a non-constant Reeb chord such that $x(0)=x(l_x)$,
a direct consequence of Lemma \ref{l:actionBound} is

\begin{corr}\label{c:actionBoundoutside}
 If $u_v:\Sigma_v \to SM^+$ has only negative asymptotes, then at least one of the asymptotes 
 is not a Reeb orbit nor a Reeb chord $x$ such that $x(0)=x(l_x)$.
\end{corr}

\begin{lemma}\label{l:actionBoundGlobal}
Let $u_{\infty}=\{u_v\}_{v \in V(\eT)}$ be a holomorphic building obtained in Theorem \ref{t:SFTcompactness}.
 If $\sum_{j=0}^d |A(x_j)| <T$, then for every $v \in V(\eT)$, the action of every asymptote of $u_v$ lies in $[-T,T]$.
\end{lemma}

\begin{proof}
 Let us assume the contrary.
 Then there is an asymptote of $u_v$ with action lying outside $[-T,T]$. 
 We assume that this is a boundary asymptote and denote it by $x$. The case for interior asymptote is identical.
 If $A(x)>T$ (resp. $A(x)<-T$), we pick $v' \in V(\eT)$ (which might be $v$ itself) such that $x$ is a negative (resp. positive) asymptote of $u_{v'}$. 
 Let $e$ be the edge in $\eT$ corresponds to this asymptote.
 Let $G$ be the subtree of $\eT \backslash \{e\}$ containing $v'$.

 Denote $x_{v,i}$ by the asymptotes corresponding to the vertex $v$.  
 Let $sgn(x)=0$ (resp. $sgn(x)=1$) if $x$ is a positive (resp. negative) asymptote.
 Then 
%To ease our notation, we will use the following convention in the rest of our paper for a boundary or interior puncture $x$ of a Riemann surface $S$
%\begin{align}
% sgn(x)=\left\{
% \begin{array}{ll}
%  0, &\text{ if } x\text{ is a positive end,}\\
%  1, &\text{ if } x\text{ is a negative end.}
% \end{array}
%\right.
% \end{align}
\begin{equation}
       \begin{aligned}
  0 &\le \sum_{v \in G} E_\omega(u_v) \\
    &= \sum_{v\in G}\sum_{i}(-1)^{sgn(x_{i,v})}A(x_{i,v})\\
    &\le (-1)^{sgn(x)}A(x)+ \sum_{j} |A(x_j)|<0
 \end{aligned}
\end{equation}

Here $x_{i,v}$ runs over all asymptotes of $u_v$, and $j$ over all semi-infinite edges.  The second inequality holds because all finite edges are cancelled between the components they connect.  This concludes the lemma.

%  Notice that $\sum_{v \in G} E_\omega(u_v)$ is given by the sum of (plus or minus) action of the asymptotes corresponding to the semi-infinite edges of
%  $G$ together with the minus action (resp. plus action) of $x$.
%  Therefore, we have
%  \begin{align}
%   0 \le \sum_{v \in G} E_\omega(u_v) < \mp A(x)+ \sum_{j=0}^d |A(x_j)|<0
%  \end{align}
% which is a contradiction.
\end{proof}

\begin{lemma}\label{l:PositiveActionAsymptote}
Let $u_{\infty}=\{u_v\}_{v \in V(\eT)}$ be a holomorphic building obtained in Theorem \ref{t:SFTcompactness}.
 If $v \in V^{\partial} \cup V^{int}$, then only the action of the asymptote of $u_v$
that corresponds to the edge $e_v$ closest to the root of $\eT$ contributes positively to $E_{\omega}(u_v)$.
\end{lemma}

\begin{proof}
Let $G_v$ be the subtree of $\eT \backslash \{e_v\}$ containing $v$.
%By Lemma \ref{l:TreeConfig}, we have $v' \in V^{\partial} \cup V^{int}$ for all $v' \in G_v$.
We apply induction on the number of vertices in $G_v$.

If $G_v$ has only one vertex, then $0<E_{\omega}(u_v)$ is only contributed
by the asymptote corresponds to $e_v$ so the base case is done.

Now we consider the general case.
Let $e$ be an edge in $G_v$ (so $e \neq e_v$).
Let $v' \neq v$ be the other vertex adjacent to $e$ so $v' \in V^{\partial} \cup V^{int}$ by Lemma \ref{l:TreeConfig}.
By induction on $G_{v'}$, we know that the asymptote corresponds to $e$ contributes positively to $E_{\omega}(u_{v'})$
and hence negatively to $E_{\omega}(u_{v})$.
Finally, for $E_{\omega}(u_{v})$ to be non-negative, we need to have at least one term which contributes positively to $E_{\omega}(u_{v})$.
This can only be contributed by the asymptote corresponding to $e_v$.
% It follows from inductively using Lemma \ref{l:actionBound}.
\end{proof}

\begin{lemma}[Distinguished asymptote]\label{l:DisAsymptote}
 Let $u_{\infty}=\{u_v\}_{v \in V(\eT)}$ be a holomorphic building obtained in Theorem \ref{t:SFTcompactness}.
 If $\partial \Sigma_v \neq \emptyset$ and $u_v$ is not a trivial cylinder (see \eqref{eq:TrivialCyl}),
 then there is a boundary asymptote $x$ of $u_v$ that appears only once among all the asymptotes $\{x_i^\pm\}$ of $u_v$.
\end{lemma}

\begin{proof}
By Lemma \ref{l:PositiveActionAsymptote}, when $v \in V^{\partial}$, the asymptote of $u_v$ at $\xi^v_0$ 
is the only asymptote that contributes positively to energy and hence appears only once among the asymptotes of $u_v$.

%We first consider $v \in V^{int}$.
%Let $v_1,\dots,v_k$ be the vertices that are further away from the root of $T$ than $v$ (i.e. the unique path from $v_j$ to the root passes through $v$).
%We have $E_{\omega}(v)+\sum_{j=1}^k E_{\omega}(v_j)= \pm A(\xi^v_0)$

 Now, we consider $v \in V^{core}$.
 If there are more than two Lagrangians appear in the Lagrangian labels of $\partial \Sigma_v$, say $\partial_j S$
 and $\partial_{j+1} S$ are labelled by $L_{k_1}$ and $L_{k_2}$, respectively, for $k_1 \neq k_2$, then the asymptote that $u_v$
 converges to at $\xi^v_{j+1}$ can only appears once among the asymptotes of $u_v$, by Lagrangian boundary condition reason.
 
 If there are exactly two Lagrangians appear in the Lagrangian labels of $\partial \Sigma_v$, then there are exactly two $j$ 
 such that the Lagrangian labels on $\partial_{j} S$ and $\partial_{j+1} S$ are different.
 Let the two $j$ be $j_1$ and $j_2$.
 It is clear that $f_v(\xi_{j_1+1}^v)$ and $f_v(\xi_{j_2+1}^v)$ are the only two edges in $\eT^{core} \backslash (\eT^{\partial} \cup \eT^{int})$ that are adjacent to $v$.
 Therefore, by our first observation, the action of the asymptotes corresponding the other edges of $v$ contributes negatively to  $E_{\omega}(u_v)$.
 
 If $u_v$ converges to the same Reeb chord at $\xi_{j_1+1}^v$ and $\xi_{j_2+1}^v$, then one of it must be a 
 positive asymptote and the other is a negative asymptote by Lagrangian boundary condition.
 Therefore, the contribution to $E_{\omega}(u_v)$ by this same asymptote cancels.
 %By Lemma \ref{l:TreeConfig}, we know that one of $f_v(\xi^{j_i+1})$ is the edge of $v$ that is closest to the root of $T$.
 Similarly, if $u_v$ converges to the same Lagrangian intersection point at $\xi_{j_1+1}^v$ and $\xi_{j_2+1}^v$, 
 then the contribution to $E_{\omega}(u_v)$ by this same asymptote cancels because of the order of the Lagrangian boundary condition.
 As a result, we have $E_{\omega}(u_v) \le 0$ which happens only when $u_v$ is a trivial cylinder (see \eqref{eq:TrivialCyl}), by Lemma \ref{l:actionBound}.
\end{proof}

\begin{rmk}
Notice that, when $u_v$ maps to $SY$, the sum \eqref{eq:SumofA} becomes
\begin{align}
\sum_{j=1}^{r^+} L(x_j^+) + \sum_{j=1}^{s^+} L(\gamma_j^+) -\sum_{j=1}^{r^-} L(x_j^-)-\sum_{j=1}^{s^-} L(\gamma_j^-) \label{eq:SimSumofA}
\end{align}
because the terms involving the primitive functions on the Lagrangians add up to zero. 
\end{rmk}

\subsection{Morsification}

We come back to our focus on $U=T^*P$, where $P$ satisfies \eqref{eq:GammaCondition}. We will need to use a perturbation of the standard contact form $\alpha_0$ on $\partial U$ to achieve transversality later. In this section, we explain how the action and index of the Reeb chord/orbit are changed under such a perturbation.

 As explained in Section \ref{ss:Grading}, $(\partial U,\alpha_0)$ is foliated by Reeb orbits.
 The quotient of $\partial U$ by the Reeb orbits is an orbifold, which is denoted by $Q_{\partial U}$.
 We can choose a Morse function $f_Q:Q_{\partial U} \to \R$ compatible with the strata of $Q_{\partial U}$
 and lifts $f_Q$ to a $R_{\alpha_0}$-invariant function $f_\partial:\partial U \to \R$ (see \cite[Section $2.2$]{Bo02}).
 Let $\critp(f_Q)$ be the set of critical points of $f_Q$.
Let $\alpha=(1+\delta f_{\partial})\alpha_0$, which is a contact form for $|\delta| \ll 1$.
 Let $L(\partial U)$ be the length of a generic simple Reeb orbit of $\partial U$.
 
 \begin{lemma}[\cite{Bo02} Lemma $2.3$]\label{l:Morsification}
  For all $T> L(\partial U)$, there exists $\delta>0$ such that every simple $\alpha$-Reeb orbit $\gamma$ with $L(\gamma)<T$ is non-degenerate and is a simple $\alpha_0$-Reeb orbit.
  Moreover, the set of simple $\alpha$-Reeb orbits $\gamma$ with $L(\gamma)<T$ is in bijection to $\critp(f_Q)$.
  
  Furthermore, if $\gamma$ is the $m$-fold cover of a simple $\alpha$-Reeb orbit $\gamma^s$ such that $L(\gamma)<T$, then
  \begin{align}
   \mu_{RS}^{\alpha}(\gamma) \ge \mu_{RS}^{\alpha_0}(\gamma) - \frac{1}{2}\dim(S_{\gamma}) \label{eq:IndexChange}
  \end{align}
where $\mu_{RS}^{\alpha}(\gamma)$, $\mu_{RS}^{\alpha_0}(\gamma)$ are the Robbin-Salamon index of $\gamma$ with respect to $\alpha$ and $\alpha_0$, respectively,
and $S_{\gamma}$ is the Morse-Bott family with respect to $\alpha_0$ that $\gamma$ lies.
 \end{lemma}

 \begin{proof}
  The first statement follows from \cite[Lemma $2.3$]{Bo02}.
  
  For the second statement, we need to compare the path of symplectic matrices $\Phi_t^{\alpha}$, $\Phi_t^{\alpha_0}$ corresponding to $\alpha$ and $\alpha_0$, respectively.
  We can isotope $\Phi_t^{\alpha_0}$ relative to end points, by changing the trivialization, to $\wt\Phi_t^{\alpha_0}$ such that  $\ker(\wt\Phi_t^{\alpha_0}-Id) \neq 0$
  only happens at finitely many $t \in [0,1]$.
  For a fixed $T$, we can choose $\delta$ sufficiently small such that $\Phi_t^{\alpha}$ and $\wt\Phi_t^{\alpha_0}$ are arbitrarily close but with $\ker(\wt\Phi_t^{\alpha}(1)-Id) \neq 0$.
  As a result, only the last contribution to $\mu_{RS}^{\alpha_0}(\gamma)$ at $t=1$ may not persist (see \eqref{eq:RSindexFormula}) and we obtain the result.
 \end{proof}
 
 \begin{corr}\label{c:IndexClosedCap}
  For all $T> L(\partial U)$, there exists $\delta>0$ such that every $\alpha$-Reeb orbit $\gamma$ with $L(\gamma)<T$ and being contractible in $U$ has $\mu_{RS}^{\alpha}(\gamma) \ge n-1$.
 As a result, the virtual dimension of $u:\C \to SM^-$ with positive asymptote $\gamma$ satisfies $\virdim(u) \ge 2n-4$
 \end{corr}

 \begin{proof}
  The underlying simple Reeb orbit $\gamma^s$ of $\gamma$ must have $L(\gamma^s)<T$ so it is also a $\alpha_0$-Reeb orbit, by Lemma \ref{l:Morsification}.
  Since $\gamma$ is contractible in $U$, by the explanation in Section \ref{sss:TypeFourGrading}, we have
  $\mu_{RS}^{\alpha_0}(\gamma)=2k(n-1)$ for some $k>0$ and $\dim(S_{\gamma})=2n-2$.
  Therefore, $\mu_{RS}^{\alpha}(\gamma) \ge n-1$ by Lemma \ref{l:Morsification} and $\virdim(u)=(n-3)+\mu_{RS}^{\alpha}(\gamma) \ge 2n-4$.

 \end{proof}

We have a similar index calculation for Reeb chord. Let $\Lambda_q \subset \partial U$ be the cosphere at $q$.
 
 \begin{lemma}\label{l:IndexChangeChord}
  There exists $f_Q$ such that for all $T> L(\partial U)$, there exists $\delta>0$ such that 
   every $\alpha$-Reeb chord $x$ from $\Lambda_{q_1}$ to $\Lambda_{q_2}$ with $L(x)<T$ has $|x| \le 0$ in the canonical relative grading. Here, we allow $q_1=q_2$.
   
   Moreover, if $q_i$ are in relatively generic position on $P$, for each lift $\fq_i$ of $q_i$, there is exactly one such chord $x_{\mathbf{q_1},\mathbf{q_2}}$ with $|x_{\mathbf{q_1},\mathbf{q_2}}|=0$ in canonical relative grading such that $x_{\mathbf{q_1},\mathbf{q_2}}$ can be lifted to a Reeb chord from 
   $\Lambda_{\fq_1}$ to $\Lambda_{\fq_2}$.
 \end{lemma}

 \begin{proof}
 % First note that, the grading function on $\Lambda_q$ is a constant so the terms contributed by it to $|x|$ cancels (see \eqref{eq:MaslovGrading}).
 % Therefore, we can use the canonical relative grading to do the computation.
 For the first statement, when $\delta>0$ is sufficiently small, $x$ is $C^1$-close to a $\alpha_0$-Reeb chord from $\Lambda_q$ to itself.
 Recall from Section \ref{sss:TypeThreeGrading} that, a non-degenerate $\alpha_0$-Reeb chord $x_0$ from $\Lambda_q$ to itself has $\iota(x_0) \le 0$.
 Therefore, if $x$ is $C^1$-close to $x_0$, then $\iota(x) \le 0$.
 
 On the other hand, a degenerated $\alpha_0$-Reeb chord $x_0$ from $\Lambda_q$ to itself has $\iota(x_0)=-k(n-1) \le -(n-1)$ for some $k>0$.
 We have $\dim(S_{x_0})=n-1$ so if $x$ is $C^1$-close to $x_0$, then $\iota(x) \le \iota(x_0)+\dim(S_{x_0}) \le -(n-1)+(n-1)=0$.

 For the second statement, we only need to notice that $|x_{\mathbf{q_1},\mathbf{q_2}}|=0$ if an only if the chord can be lifted to (a perturbation of)
 the unique geodesic between $\mathbf{q_1}$ and $\mathbf{q_2}$ with length less than $\pi$ from \eqref{l:ChordNondegenGrading}.
 \end{proof}

 Note that, we do not need to assume $x$ is non-degenerate in Lemma \ref{l:IndexChangeChord}.

 After choosing $\alpha$ in Lemma \ref{l:Morsification}, there are only finitely many simple Reeb orbits of length less than $T$.
 They correspond to finitely many geodesic loops in $P$.
 Therefore, for generic (on the complement of the geodesic loops) $q \in P$, $\Lambda_q$ does not intersect with simple Reeb orbits of length less than $T$.
 Moreover, for generic perturbation of $f_Q$, we can achieve the following:
 
 \begin{lemma}\label{l:genericity}
  We assume $n \ge 2$.
   For generic $C^2$-small perturbation of $f_Q$ away from $\critp(f_Q)$ (such that the set $\critp(f_Q)$ is unchanged), 
every $\alpha$-Reeb chord $x$ from $\Lambda_q$ to itself with $L(x) <T$ satisfies $x(t) \notin \Lambda_q$ for $t \in (0,L(x))$. 
Moreover, we can assume every such $x$ is non-degenerate.
 \end{lemma}

 \begin{proof}
We regard a chord as an intersection between the image $\phi(\R\times \Lambda_q)$ and $\Lambda_q\subset M$.  Here $\phi(t,x)=\phi_t(x)$.  From \cite[Lemma 2.3]{Bo02}, if we consider a small perturbation $\alpha_\epsilon:=(1+\epsilon \bar f_T)\alpha$ and denote $R_\epsilon$ as the perturbed Reeb vector field, then $R_\epsilon=R+X$, where

\begin{align*}\label{e:perturbedField}
     &i(X)d\alpha=\frac{\epsilon \bar f_T}{(1+\epsilon\bar f_T)^2},\\
     &     \alpha(X)=-\frac{\epsilon\bar f_T}{1+\epsilon \bar f_T}.
\end{align*}

Assume that there is a chord $x(t)$, $t\in[0,1]$ such that for some $x(t_i)$, for $i=1,\cdots,k$, $0<t_1<\cdots<t_k<1$, $x(t_i)\in \Lambda_q$.  Denote $p=x(1)$  For generic perturbation of $\alpha$, we may assume $x(t)\notin \Lambda_q$ unless $t=t_i$ for some $i$.  

Moreover, we may assume all intersections $x(t_i)$ are transversal in the following sense: let $\phi_s$ be the time-$s$ flow of the Reeb vector field, then $d\phi_{t_i-t_j}(T_{x(t_j)}\Lambda_q\oplus T_qx(t))$ intersects transversally to $T_{x(t_i)}\Lambda_q$ for all $i, j\in \{1,\dots, k\}$.  This transversality can be achieved similarly to the case of Reeb orbits: if one considers a perturbation $\alpha_\epsilon$ supported near a chord as above, while $df_p=0$ on the chord, then the linearized Reeb flow has an additional term $\frac{\epsilon}{1+\epsilon f_p}JHess(f_p)$ (see \cite{Bo02}), which suffices to perturb $\phi_*(\partial_t\oplus T_{x(0)}\Lambda_q)$ to be transversal to $\Lambda_q|_p$.

After an $\epsilon f_p$-perturbation as above, there is a unique chord that is $O(\epsilon)$-close to $x(t)$.  To see this, notice $x(t)$ corresponds to a Lagrangian intersection between $\Lambda_q\times\R$ and $\phi_1(\Lambda_q\times\R)$ mod out by $\R$-translation.  The transversality assumption ensures the perturbation does not creates new chords locally, and also finiteness of chords with action less than a fixed value $T$.

We may now choose a contactomorphism $\tau$ with small $C^2$-norm supported near $x(t_i)$, which pushes $x(t_i)$ off $\Lambda_q$ for all $i$, and consider the contact form $\tau_*\alpha$.  Since we did not change the contact structure, $\Lambda_q$ remains Legendrian and the perturbation on the contact form is by a function $f$ supported near $x(t_i)$.  $\tau(x(t))$ is then a Reeb chord with no interior intersection with $\Lambda_q$, and from the transversality assumption and argument above, there is no new chords created.  The induction on the number of chords concludes the lemma.

 \end{proof}

 \begin{corr}\label{c:ReebDynamicMorsified}
  We assume $n \ge 2$.
  For all $T> L(\partial U)$ and $k \in \N$, there exists $\delta>0$, $f_\partial:\partial U \to \R $ and pairwise distinct $q_1,\dots, q_k \in P$ 
  such that $\alpha=(1+\delta f_\partial)\alpha_0$ satisfies

  \begin{enumerate}[(1)]
   \item every simple $\alpha$-Reeb orbit $\gamma$ that is contractible in $U$ and $L(\gamma) < T$ is non-degenerate and $\mu_{RS}(\gamma) \ge n-1$, and
   \item every $\alpha$-Reeb chord $x$ from $\cup_{i=1}^k \Lambda_{q_i}$ to $\cup_{i=1}^k \Lambda_{q_i}$ with $L(x)<T$ is non-degenerate, satisfies $x(t) \notin \cup_{i=1}^k \Lambda_{q_i}$ for $t \in (0,L(x))$
   and $|x| \le 0$ with respect to canonical relative grading.
  \end{enumerate}
  As a consequence, the image of the
  $\alpha$-Reeb chords $x$ from $\cup_{i=1}^k \Lambda_{q_i}$ to $\cup_{i=1}^k \Lambda_{q_i}$ with $L(x)<T$ are pairwise disjoint, and they are disjoint from the image of 
  simple $\alpha$-Reeb orbits.
 \end{corr}

 \begin{proof}
  After choosing $\delta, f_Q$ such that $(1)$ is satisfied by Lemma \ref{l:Morsification} and Corollary \ref{c:IndexClosedCap}, and we can apply Lemma \ref{l:genericity} to $\cup_{i=1}^k \Lambda_{q_i}$.
  Since the perturbation is arbitrarily $C^2$-small, we have $|x| \le 0$ by Lemma \ref{l:IndexChangeChord} and \eqref{l:ChordNondegenGrading}.
 \end{proof}

 In the rest of the paper, we always choose a contact form $\alpha $ on $\partial U$ such that Corollary \ref{c:ReebDynamicMorsified} holds, we denote the set of simple $\alpha$-Reeb orbit $\gamma$ with  $L(\gamma) < T$ by $\cX^o_T$.
 Similarly, we denote the set of $\alpha$-Reeb chord $x$ from $\cup_{i=1}^k \Lambda_{q_i}$ to $\cup_{i=1}^k \Lambda_{q_i}$ with $L(x)<T$ by $\cX^c_T$.

\subsection{Regularity}\label{ss:Regularity}

In this section, we address the regularity of curves $u_v$ in the holomorphic buildings obtained in Theorem \ref{t:SFTcompactness} for $v \in V^{core} \cup V^{\partial}$.
We adapt the techniques developed in \cite{EES05}, \cite{EES07}, \cite{Dragnev} and \cite{CDGG}.
The key observation that is made in \cite{EES05} is that if there is an asymptote that only appears once among the boundary asymptotes of a pseudo-holomorphic curve,
then one can achieve regularity by perturbing $J$ near the asymptote.

The main difference of our situation is that
we do not work in a contact manifold that is a contactization of an exact symplectic manifold, hence we don't have a projection of holomorphic curve as in \cite{EES05}\cite{EES07}.  We remedy this by localizing our consideration to a neighborhood of the Reeb chord.

We first explain the space of almost complex structure we use.  In what follows, we always assume that a contact form $\alpha$ on $\partial U$ is chosen such that Corollary \ref{c:ReebDynamicMorsified} is satisfied.
%By deforming $\partial U$, we can assume that $\alpha=\theta|_{\partial U}$ where $\theta$ is the Liouville one-form on $M$.

\begin{lemma}[Neighborhood Theorem]\label{l:ReebNeighborhood}
 For any Reeb chord $x \in \cX_T^c$, there exists a neighborhood $N_x$ of $Im(x)$, an open ball $B_x \subset \R^{2n-2}$ containing the origin, an open interval $I_x \subset \R$ 
 and a diffeomorphism $\phi_{N_x}:N_x \to B_x \times I_x$ such that 
 \begin{align}
  \left\{
  \begin{array}{ll}
   \alpha=\phi_{N_x}^*(dz+\sum_{i=1}^{n-1} x_i dy_i) \\
   \pi_{B_x}(\phi_{N_x}(x(t)))=0
  \end{array}
  \right.
 \end{align}
where $(x_i,y_i) \in B_x$, $z \in I_x$ and $\pi_{B_x}:B_x \times I_x \to B_x$ is the projection to the first factor.
\end{lemma}

\begin{proof}
 It follows from Moser argument. We give a sketch following \cite[Theorem $2.5.1$]{GeigesBook}.
 Since $d \alpha$ is non-degnerate on $T_pY/T_pIm(x)$ for all $p \in Im(x)$, we can use exponential map with respect to an appropriate metric to find 
 coordinates $(x_1,y_1,\dots,x_{n-1},y_{n-1},z)$ near $Im(x)$ such that $Im(x)=\{x_i=y_i=0\}$ and on $TY|_{Im(x)}$,
 \begin{align}
  \left\{
  \begin{array}{ll}
   \alpha(\partial_z)=1, \iota_{\partial_z} d\alpha=0 \\
   \partial_{x_i}, \partial_{y_i} \in \ker(\alpha), d\alpha=\sum_{i=1}^{n-1} dx_i \wedge dy_i
  \end{array}
  \right.
 \end{align}
 Let $\alpha_{\R^{2n-1},std}=dz+\sum_{i=1}^{n-1} x_i dy_i$ and $\alpha_t=(1-t)\alpha_{\R^{2n-1},std}+t \alpha$. It follows that
on $TY|_{Im(x)}$,
\begin{align}
 \alpha_t=\alpha, d\alpha_t=d\alpha \text{ for all }t
 \end{align}
 In particular, $\alpha_t$ is a family of contact forms in a sufficiently small neighborhood of $Im(x)$.
 By Moser trick, there exists a vector field $X_t$ near $Im(x)$ such that the flow $\psi_t$ satisfies $\psi_t^* \alpha_t=\alpha_{\R^{2n-1},std}$ for all $t \in [0,1]$
 and $X_t(p)=0$ for all $p \in Im(x)$. We set $\phi_{N_x}=(\psi_1)^{-1}$.
\end{proof}

\begin{rmk}\label{r:ModifyContactForm}
 If we replace $dz+\sum_{i=1}^{n-1} x_i dy_i$ by $dz+\sum_{i=1}^{n-1} x_i dy_i+dy_1$ in Lemma \ref{l:ReebNeighborhood}, the lemma still holds. 
\end{rmk}

\begin{corr}\label{c:LiftingJ}
  Let $B_x$ be one chosen in Lemma \ref{l:ReebNeighborhood} or Remark \ref{r:ModifyContactForm}.   If $J'$ is a compactible almost complex structure on $B_x$, then there is a  cylindrical almost complex structure $J$ on $\R \times N_x$ such that $(\pi_{B_x} \circ \pi_Y)_* \circ J(v)=J' \circ (\pi_{B_x} \circ \pi_Y)_*(v)$ for all $v \in \xi$.
\end{corr}

\begin{proof}
 We can use the symplectic decomposition $T_{(r,z)}(\R \times N_x)=\R\langle \partial_r, R_\alpha \rangle \oplus \xi_z$
 and the isomorphism $(\pi_{B_x})_*: \xi_z \simeq T_{\pi_{B_z}(z)}B_x$
 to define $J$ such that $J(\partial_r)=R_\alpha$ and $J(v)=((\pi_{B_x}\circ \pi_Y)_*)^{-1} \circ J' \circ (\pi_{B_x} \circ \pi_Y)_*(v)$ for $v \in \xi_z$.
 One can check that $J$ is a cylindrical almost complex structure.
\end{proof}

For the $T$ chosen in Corollary \ref{c:ReebDynamicMorsified},
there are finitely many Reeb orbits or Reeb chord from $\cup_j \Lambda_{q_j}$ to $\cup_j \Lambda_{q_j}$ with length less than $T$.
Moreover, the simple Reeb orbits $\cX^{o}_T$ and the Reeb chords $\cX^{c}_T$ have pairwise disjoint images.

For each $x \in \cX^{c}_T $, we pick a neighborhood $N_x$ of $Im(x)$ using Remark \ref{r:ModifyContactForm}.
We assume that all these neighborhoods are pairwise disjoint and disjoint from the Reeb orbits of $\alpha$.
%and we denote the quotient map in Lemma \ref{l:ReebNeighborhood} by $\pi_{B_x}$ (resp. $\pi_{B_{\gamma}}$).

Let $x \in \cX^{c}_T$, $x(0) \in \Lambda_{q_0}$ and $x(L(x)) \in \Lambda_{q_1}$.
By Corollary \ref{c:ReebDynamicMorsified}, for sufficiently small $N_x$, we can assume that 
\begin{align}
D_{i,x}:=\Lambda_{q_i} \cap N_x   \label{eq:Dix}
\end{align}
is a disk for $i=0,1$.
Moreover, by the fact that $x$ is non-degenerate, we know that $\pi_{B_x}(D_{0,x})$ and $\pi_{B_x}(D_{1,x})$ are transversally intersecting Lagrangians.
There exists a compatible $J_{B_x}$ on $B_x$ such that $J_{B_x}$ is integrable near the origin.
%\footnote{see \cite[Section $2.3$]{EES07}}
By possibly perturbing $\Lambda_{q,i}$, or equivalently perturbing $\alpha$, we can assume that $\pi_{B_x}(D_{i,x})$ are real analytic submanifolds 
near origin for all $x$. 
We fix a choice of $J_{B_x}$ for each $x \in \cX^c_T$.

Let $\cJ^{cyl}(\partial U; \{N_x\}_{x \in \cX^{c}_T})$ be the space of $J \in \cJ^{cyl}(\partial U)$
such that $J$ is $R_{\alpha}$-invariant in $N_x$ and
there is a compactible almost complex structures $J'$ on $B_x$
so that $J'=J_{B_x}$ near the origin and $(\pi_{B_x} \circ \pi_Y)_* \circ J(v)=J' \circ (\pi_{B_x} \circ \pi_Y)_*(v)$ for all $v \in \xi$.
By Corollary \ref{c:LiftingJ}, we know that 
$\cJ^{cyl}(\partial U; \{N_x\}_{x \in \cX^{c}_T } )\neq \emptyset$.

Let $Y \subset (M,\omega,\theta)$ be a perturbation of $\partial U$ such that $(Y,\theta|_Y) \cong (\partial U, \alpha)$.
By abuse of notation, we denote $\theta|_Y$ by $\alpha$.
We define $N(Y)$ as in \eqref{eq:Yneighborhood}.
We can pick $J^0$ such that $(\Phi_{N(Y)})_*J^0|_{N(Y)} \in \cJ^{cyl}(\partial U; \{N_x\}_{x \in \cX^{c}_T})$.
Let $\{J^{\tau}\}_{\tau \in [3R,\infty)}$ be a smooth family $R$-adjusted to $(Y, \alpha)$
as explained in Section \ref{sec:review_of_symplectic_field_theory_and_dimension_formulae} (see Remark \ref{r:FlexibleJ}).

Let $\{L_j\}_{j=0}^d$ be a collection of Lagrangians satisfying the assumptions of Theorem \ref{t:SFTcompactness}.
Moreover, we assume that $\Lambda_j=\cup_{i=1}^{c_j} \Lambda_{q_{k_{j,i}}}$ for some $q_{k_{j,i}}$ in Corollary \ref{c:ReebDynamicMorsified}.
If $T$ was chosen sufficiently large, there exists $0<T^{adj}<T$ (depending only on the primitives of $\{L_j\}$, see Section \ref{ss:action})
such that 
\begin{align}
 \text{for all Reeb chords $x$ from $\Lambda_i$ to $\Lambda_j$, $|A(x)|<T^{adj}$ implies $|L(x)|<T$} \label{eq:AdjustedAction}
\end{align}
Without loss of generality, we can assume $T^{adj}$ exists
and $\sum_{j=0}^d|A(x_j)|<T^{adj}$.
Applying Theorem \ref{t:SFTcompactness} and Lemma \ref{l:actionBoundGlobal}, we get a holomorphic building $u_{\infty}=\{u_v\}_{v \in V(\eT)}$
such that all the asymptotes of $u_v$ are either Lagrangian intersection points, Reeb chords in $\cX^{c}_T$ or multiple cover of Reeb orbits in $\cX^{o}_T$.
For $u_{\infty}$, we have the following regularity result.

\begin{prop}\label{p:IntermediateRegular}
There is a residual set  $\cJ^{cyl, reg} \subset \cJ^{cyl}(\partial U; \{N_x\}_{x \in \cX^{c}_T})$ such that 
if (the cylindrical extension of) $(\Phi_{N(Y)})_*J^0|_{N(Y)}$ lies in $\cJ^{cyl, reg}$, then for $v \in V^{core} \cup V^{\partial}$
and $l_\eT(v) \in \{1,\dots, n_\eT-1\}$, the $J^Y$-holomorphic curve $u_v$ is transversally cut out.
\end{prop}

\begin{proof}
 By Lemma \ref{l:DisAsymptote}, $u_v$ has a boundary asymptote $x$ that appears only once among its asymptotes.
 We want to show that transversality can be achieved by considering variation of almost complex structures in $SN_x:=\R \times N_x$.
 
 Let $\Lambda^{tot}=\cup_i \Lambda_{q_i}$ and $S\Lambda=\cup_i S\Lambda_{q_i}$ where  $\Lambda_{q_i}$ are obtained in Corollary \ref{c:ReebDynamicMorsified}.
 There is a Banach manifold $\cB$ consisting of maps
 \begin{align}
  u:(\Sigma_v, \partial \Sigma_v) \to (SY, S\Lambda^{tot})
 \end{align}
in an appropriate Sobolev class with positive weight (see \cite{Abbas}, \cite{Dragnev}).
Let $U_{\Delta}$ be an appropriate Banach manifold that is dense inside $\cJ^{cyl}(\partial U; \{N_x\}_{x \in \cX^{c}_T})$.
The map
\begin{align}
 (u,J) \mapsto \overline{\partial}_J u \label{eq:section}
\end{align}
defines a section $\eF$ of a bundle $\eE^{0,1} \to \cB \times U_\Delta$ with differential 
\begin{align}
D\eF(u,J): T_u\cB \times T_JU_\Delta \to \eE^{0,1}_u  \\
(\eta,\fY) \mapsto D_u(\eta)+ \fY(u) \circ du \circ j_{\Sigma_v}
\end{align}
where $j_{\Sigma_v}$ is the complex structure on $\Sigma_v$.
By a choice a metric, we identify 
\begin{align}
T_u\cB \simeq \Gamma(u^*TSY, u|_{\partial \Sigma_v}^* S\Lambda^{tot}) 
\end{align}
where the right hand side is the completion of the space of smooth sections in $u^*TSY$, which takes value in $u|_{\partial \Sigma_v}^* S\Lambda^{tot}$ along the boundary,
with respect to an appropriate Sobolev norm.
On the other hand, we have $\eE^{0,1}_u=\Omega^{0,1}(u^*TSY)$, where the right hand side is the completion of the space of smooth $u^*TSY$-valued $(0,1)$-form 
with respect to an appropriate Sobolev norm.
We want to argue  $D\eF(u,J)$ is surjective at $(u,J)$ using that fact that there exists a
boundary asymptote $x \in \cX^c_T$ of $u$  that appears only once among its asymptotes and $\overline{\partial}_J u=0$.

Suppose not, then there exists $0 \neq l \in \eE^{0,1}_u$ such that 
\begin{align}
 \langle l, D\eF(u,J)(\eta,\fY) \rangle_{L^2, \Sigma_v} =0 \label{eq:CokerElement}
\end{align}
for all $\eta \in T_u \cB$ and $\fY \in T_JU_\Delta$.
 By unique continuation principle, it suffices to show that $l=0$ on some non-discrete set of $\Sigma_v$ to get a contradiction.
 
 Let $\cR=u^{-1}(N_x) \subset \Sigma_v$ and we will show that for $\eta$ supported in $\cR$ and $\fY$ supported in $SN_x$, it is sufficient to get $l|_\cR=0$.
 By Lemma \ref{l:ReebNeighborhood}, we can identify $SN_x$ with $\R_r \times (B_x)_{x_i,y_i} \times (I_x)_z$.
 Let $\overline{u}=\pi_{B_x} \circ \pi_{Y} \circ u_v|_\cR$.
 In the coordinates $((r,z),(\{x_i\},\{y_i\}))$, we can write $l|_\cR=(l_1,l_2)$.
 For $\eta=0$ and $\fY$ supported in $SN_x$ \footnote{$\fY$ vanishes along $\partial_r,\partial_z$ and takes values in $\partial_{x_i}, \partial_{y_i}$}, \eqref{eq:CokerElement} becomes 
 \begin{align}
 \langle l_2, \fY(\overline{u}) \circ d\overline{u} \circ j_{\Sigma_v} \rangle_{L^2, \cR} =0 \label{eq:CokerElement2}
\end{align}
 where $\fY$ is $r,z$-invariant in $SN_x$ by the definition of $\cJ^{cyl}(\partial U; \{N_x\}_{x \in \cX^{c}_T})$ so $\fY(\overline{u})$ is well-defined.
 
 \begin{lemma}\label{l:EESargument}
  It follows from \eqref{eq:CokerElement2} that $l_2=0$.
 \end{lemma}

 Assuming Lemma \ref{l:EESargument}, it suffices to show that $l_1=0$.
 Similarly, $l_1$ admits the unique continuation property (see \cite[page $754$]{Dragnev}) so we only need to show that $l_1=0$ on some non-discrete set of $\cR$.
 For $\fY=0$ and $\eta$ supported in $\cR$, \eqref{eq:CokerElement} becomes
 \begin{align}
 \langle l_1, D(\pi_{r,z}) \circ D_u \eta \rangle_{L^2, \cR} =0 \label{eq:CokerElement3}
\end{align}
 where $\pi_{r,z}: SN_x \to \R_r \times (I_x)_z$ is the projection.
 Notice that $J|_{T(\R_r \times (I_x)_z)}$ is the standard complex structure, and $\eta$ depends on the domain $\cR$ rather than the target $SN_x$.
 Therefore, we can find an interior point $p$ of $\cR$ and construct $\eta$ appropriately supported near $p$ to show that $l_1=0$.
 The details of the construction of $\eta$ can be found in \cite[page $754$]{Dragnev}.

 As a result, $l|_\cR=0$ and hence $l \equiv 0$.
 The existence of $\cJ^{cyl, reg}$ follows from applying Sard's-Smale theorem to the projection $\eF^{-1}(0) \to U_\Delta$.
\end{proof}

\begin{proof}[Proof of Lemma \ref{l:EESargument}]
 The proof is the same as \cite[Lemma $4.5(1)$]{EES07}. For readers' convenience, we will recall the proof using our notation.
 
 By the definition of $\cJ^{cyl}(\partial U; \{N_x\}_{x \in \cX^{c}_T})$, $\overline{u}$ is a $J'$-holomorphic curve for some compatible almost complex structure $J'$
 on $B_x$ such that $J'=J_{B_x}$ near origin.
 Moreover, exactly one boundary puncture, denoted by $\xi_{j_x}$, of $\cR$ is mapped to the origin by our choice of $x$.
 
 By the asymptotic behavior of holomorphic disks, we can assume that for sufficently small $\delta>0$, there exists a neighborhood $(E_0,\partial E_0) \subset (\cR, \partial \cR)$
 of $\xi_{j_x}$ such that
 \begin{enumerate}[(i)]
  \item $(\overline{u}(E_0),\overline{u}(\partial E_0)) \subset (B(0,2\delta), \pi_{B_x}(D_{0,x} \cup D_{1,x}) \cup \partial B(0,2\delta))$,
  \item $\pi_{B_x}(D_{0,x} \cup D_{1,x}) \cap \partial B(0,2\delta)$ are two real analytic disjoint  branches,
  \item $\overline{u}(\partial E_0)$ contains two regular oriented curves $\gamma\subset D_{0,x}$, $\tilde{\gamma}\subset D_{1,x}$ in $B(0,2\delta)$, respectively.
 \end{enumerate}
 
 Here $B(0,2\delta)$ is a $2\delta$-ball centered at the origin and $D_{i,x}$ are defined in \eqref{eq:Dix}.
 
 To prove $l_2$ is zero we consider the variation of $J'$ near a point on $\gamma$.
 To this end, we need to keep track of other parts of $\cR$ that map onto $\gamma$.

 Let $p_1 \dots, p_r \in \partial \cR$ be the preimages under $\overline{u}$ of $0$ with the property that one of the components
 of the punctured neighborhood of $p_j$ in $\partial \cR$ maps to $\gamma$.
 This set is finite and is identified with the set of boundary intersections between $u$ and $\R\times x$.
 
 Let $p_{r+1}, \dots p_s \in \cR \backslash \partial \cR$ be the preimages under $\overline{u}$ of $0$ with the property that the preimage of $\gamma$
 under $\overline{u}$ intersects some neighborhood of $p_j$ in a $1$-dimensional subset.
 By monotonicity lemma and maximum principle, this set is also finite, and is identified with the interior intersections between $u$ and $\R\times x$.
 
 For $1 \le j \le s$, let $E_j \subset \cR$ denote the connected coordinate neighborhood of $\overline{u}^{-1}(B(0,2\delta))$
 near $p_j$. Let $U_1=\overline{u}(E_0)$ and $U_2$ be the Schwartz reflection of $U_1$ through $\tilde{\gamma}$ (see Figure \ref{fig:EEStrick}).

 By monotonicity lemma and maximum principle, we can find $x_i \in U_i \backslash (B(0,\delta) \cup \pi_{B_x}(D_{0,x} \cup D_{1,x}))$
 and small neighborhoods $B(x_i,\epsilon)$, $\epsilon \ll r$, such that 
 \begin{align}
  \overline{u}^{-1}(B(x_i,\epsilon)) \subset \cup_{j=0}^s E_j
 \end{align}
% This is where we use essentially the assumption that $\xi_{j_x}$ is the only boundary puncture of $\cR$ that maps to the origin.
Note that for $j \ge 1$, $x_1 \in \overline{u}(E_j)$ if and only if $x_2 \in \overline{u}(E_j)$.
We exclude from our list any such $j \ge 1$ with $x_i \notin \overline{u}(E_j)$, $i=1,2$.
To simplify notation, we continue to index this possibly shortened list by $1 \le j \le s$.

For $1 \le j \le r$, we double the domain $E_j$ through its real analytic boundary $\partial E_j$.
We also double the local map $\overline{u}|_{E_j}$.
We continue to denote the open disk by $E_j$.
For $0 \le j \le s$, let $u_j=\overline{u}|_{E_j}$.
We can also double (for $1 \le j \le r$) the cokernel element $l_2$ (which is anti-holomorphic) locally and define (for $0 \le j \le s$) $(l_2)_j=l_2|_{E_j}$.

There exists a disk $E \subset \C$ and a map $f_E$ defined on $E$ such that for $1 \le j \le s$, there exists positive integers $k_j$
and bi-holomorphic identifications $\phi_j$ of $E$ with $E_j$ such that $(l_2)_j(\phi_j(z))=f_E(z^{k_j})$ for $z \in E$.

Via our choice of perturbation of the complex structure, we can choose $\fY$ to be supported in $B(x_2, \epsilon)$.
We get
\begin{align}
 \langle \sum_{j=1}^s (l_2)_j(\phi_j(z)), \fY(u_j \circ \phi_j) \circ d(u_j \circ \phi_j) \circ j_{E} \rangle_{L^2,E}=0 
\end{align}
where $j_E$ is the complex structure on $E$.
Varying $\fY$, this implies 
\begin{align}
\sum_{j=1}^s (l_2)_j(\phi_j(z)) =0 \label{eq:FirstVansihing}
\end{align}

We can also choose $\fY$ to be supported in $B(x_1, \epsilon)$. We get 
\begin{align*}
 \langle \sum_{j=1}^s (l_2)_j(\phi_j(z)), \fY(u_j \circ \phi_j) \circ d(u_j \circ \phi_j) \circ j_{E} \rangle_{L^2,E}
 +\langle (l_2)_0(z), \fY(u_0) \circ du_0 \circ j_{E_0} \rangle_{L^2,E_0}=0
\end{align*}
Since the first term is $0$ by \eqref{eq:FirstVansihing}, by varying $\fY$, it implies $l_2|_{E_0}=(l_2)_0=0$ and hence $l_2 \equiv 0$. 
 \end{proof}

\begin{figure}
  \centering
  \includegraphics[]{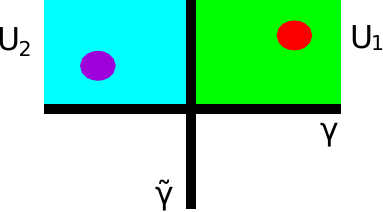}
  \caption{Green region: $U_1$; Blue region: $U_2$; Red dot: $B(x_1,\epsilon)$; Purple dot: $B(x_2,\epsilon)$. Only $u_0(E_0)$ hits $U_1$ but not $U_2$ among $u_j(E_j)$ because
  unlike $p_j$ (for $1 \le j \le s$), $\xi_{j_x}$ is a boundary puncture.}
  \label{fig:EEStrick}
\end{figure}

Proposition \ref{p:IntermediateRegular} only concerns $u_v$ in the intermediate levels in the building $u_\infty$.
We need another proposition to address the regularity when $u_v$ lies in the top/bottom level of $u_\infty$.
We will explain the case that $l_\eT(v)=n_\eT$ (i.e. top level) and the other case is similar.
%We start by explaining the space of almost complex structures we use.

Let $J_{M^+}$ be a compactible almost complex structure of $SM^+$
such that it is integrable near $SL_i^+ \pitchfork SL_j^+$, $i \neq j$.
We assume that $SL_i^+,SL_j^+$ are real analytic near $SL_i^+ \pitchfork SL_j^+$.

For $J^Y \in \cJ^{cyl}(Y,\alpha)$, we let 
$\cJ^{+}(SM^+)$ to be the set of compactible almost complex structure $J$ such that $J=J_{M^+}$ near $\cup_{i \neq j} SL_i^+ \cap SL_j^+$
and there exists $R>0$
so that $J^+|_{(-\infty,-R] \times \partial M^+}=J^Y|_{(-\infty,-R] \times Y}$.

\begin{prop}\label{p:TopRegular}
 There is a residual set  $\cJ^{+, reg} \subset \cJ^{+}(SM^+)$ such that 
if $J^+ \in \cJ^{+, reg}$, then for $v \in V^{core} \cup V^{\partial}$
and $l_\eT(v)=n_\eT$, the $J^+$-holomorphic curve $u_v$ is transversally cut out.
\end{prop}

\begin{proof}
 By Lemma \ref{l:DisAsymptote}, $u_v$ has a boundary asymptote $x$ that appears only once among its asymptotes.
If the distinguished asymptote of $u_v$ is a Lagrangian intersection point, then we can apply the argument in \cite[Lemma $4.5(1)$]{EES07} again to achieve the regularity of $u_v$.
If the distinguished asymptote of $u_v$ is a Reeb chord, we denote the corresponding puncture by $\xi_{j_x}^v$. 
By the asymptotic behavior of $u_v$, for a sufficiently large $R$, the preimage of a 
small neighborhood of $(-\infty,-R] \times Im(x)$ under $u_v$ is a neighborhood of $\xi_{j_x}^v$.
Therefore, we can find a somewhere injectivity point near $\xi_{j_x}^v$.
Unlike the situation in $SY$, where the almost complex structure is cylindrical, we can perturb $J$ in $SM^+$ as long as $J$ is cylindrical outside a compact set.
Therefore, we can use the somewhere injectivity point to achieve regularity (see \cite[Proposition $4.19$]{CDGG} for exactly the same argument).
\end{proof}

Similarly, one define $\cJ^{-}(SM^-)$ analogously and we have

\begin{prop}\label{p:BottomRegular}
 There is a residue set  $\cJ^{-, reg} \subset \cJ^{-}(SM^-)$ such that 
if $J^- \in \cJ^{-, reg}$, then for $v \in V^{core} \cup V^{\partial}$
and $l_\eT(v)=0$, the $J^-$-holomorphic curve $u_v$ is transversally cut out.
\end{prop}

\begin{rmk}\label{rem:Alex}
    There is a possible alternative approach to the above regularity results if one could generalize the work of Lazzarini \cite{La00} \cite{La11} and Perrier \cite{Pe18} to the SFT settings. This seems promising at least for $SM^\pm$, but the general regularity of $SY$ might be more difficult.
\end{rmk}

\subsection{No side bubbling}

We can now summarize the previous discussion on $u_{\infty}$.

Let $L_j$, $j=0,\dots,d$ be a collection of embedded exact Lagrangian submanifolds in $(M,\omega,\theta)$ such that $L_i \pitchfork L_j$ for all $i \neq j$.
    Let $P$ be a Lagrangian such that \eqref{eq:GammaCondition} is satisfied ($P$ can be one of the $L_j$).
    Let $U$ be a Weinstein neighborhood of $P$ and we assume that $\theta|_U$ coincides with the canonical Liouville one form on $T^*P$.
    For $T \gg 1$, we pick $\alpha$ satisfing Corollary \ref{c:ReebDynamicMorsified} and $T^{adj}$ satisfying \eqref{eq:AdjustedAction}.
    
Let $Y$ be a perturbation of $\partial U$ such that $(Y,\theta|_Y) \cong (\partial U,\alpha)$.
We denote $\theta|_Y$ by $\alpha$.
    We have a neighborhood $\Phi_{N(Y)}:(N(Y),\w|_{N(Y)})\cong ((-\epsilon,\epsilon)\times Y, d(e^r\alpha))$ of $Y$.
We assume that 
$L_j\cap N(Y)=(-\epsilon,\epsilon)\times \Lambda_{j}$ where $\Lambda_j=\Lambda_{q_{m_j}}=T^*_{q_{m_j}}P \cap Y$ for some $q_{m_j} \in P$ in Corollary \ref{c:ReebDynamicMorsified}. 

    Let $J^\tau$ be a smooth family of almost complex structures $R$-adjusted to $N(Y)$.
    such that $J^Y \in \cJ^{cyl,reg}$, where  $\cJ^{cyl,reg}$ is obtained in Proposition \ref{p:IntermediateRegular}.
    We also assume that $J^\pm \in \cJ^{\pm,reg}$, where $\cJ^{\pm,reg}$ is obtained in Proposition \ref{p:TopRegular}, \ref{p:BottomRegular}.
    
    Let $x_0 \in CF(L_0,L_d)$ and $x_j \in CF(L_{j-1},L_j)$ for $j=1,\dots,d$.
    We assume that 
    \begin{align}
     \sum_{j=0}^d |A(x_j)|<T^{adj}
    \end{align}

Suppose that there exists a sequence $\{\tau_k\}_{k=1}^\infty$ such that $\lim_{k \to \infty} \tau_k =\infty$, and a sequence
    $u_k \in \eM^{J^{\tau_k}}(x_0;x_d,\dots,x_1)$.
    We assume that $\virdim(u_k)=0$.
    Let $u_\infty=\{u_v\}_{v \in V(\eT)}$ be the holomorphic building obtained in Theorem \ref{t:SFTcompactness}.
Then we have
    
\begin{prop}[No side bubbling]\label{p:nosidebubble}
If $n\ge 3$, then $V^{int}=\emptyset$ and $n_\eT=1$.
Moreover, if $v \in V^{\partial}$, then $u_v$ is a rigid $J^+$-holomorphic map with exactly one boundary asymptote which is negative and goes to a Reeb chord.
Furthermore, the $d+1$ semi-infinite edges in $\eT$ are the only edges that correspond to Lagrangian intersection points.
\end{prop}

%The proof of Proposition \ref{p:nosidebubble} will be in the end of this section.
\begin{proof}

For a subtree $G \subset \eT$, 
we use $\virdim(G)$ to denote the virtual dimension of the map $\#_{v \in G} u_v$, where 
$\#_{v \in G} u_v$ refers to the map obtained by gluing all $u_v$ such that $v \in G$ along the asymptotes determined by the edges.
By \eqref{eq:virdimGlue1}, \eqref{eq:virdimGlue2} and the fact that all Reeb chords/orbits arising as asymptotes of $u_v$ are non-degnerate,
we have
\begin{align}
 \virdim(G)= \sum_{v \in G} \virdim(u_v)+k_G \label{eq:VirDimG}
\end{align}
where $k_G$ is the number of edges that correspond to Lagrangian intersections points and connect two distinct vertices in $G$.
By assumption, $\virdim(\eT)=0$.
Since $u_v$ are transversally cut out for $v \in V^{core} \cup V^{\partial}$ (Proposition \ref{p:IntermediateRegular}, \ref{p:TopRegular}, \ref{p:BottomRegular}), 
we have $\virdim(u_v) \ge 0$.
For $v \in V^{int}$, we cannot address the regularity but we have the following.

\begin{lemma}\label{l:noclosedstringbubble}
 For each connected component $G$ of $\eT^{int}$, we have $\virdim(G) >0$.
\end{lemma}

\begin{proof}
Let $v \in G$ be the vertex closest to the root.
By \ref{l:PositiveActionAsymptote}, we have a distinguished interior puncture $\eta^0 \in \Sigma_v$
which contributes positively to $E_\alpha(u_v)$.
Let $\gamma^0$ be the Reeb orbit that $u_v$ is asymptotic to at $\eta^0$.
Since $A(\gamma^0)=L(\gamma^0)>0$, $\gamma^0$ must be a positive asymptote of $u_v$.

Notice that, by Corollary \ref{c:actionBoundoutside}, there is no $v \in G$ such that $u_v$ maps to $SM^+$.
Therefore, $\#_{v \in G} u_v$ is a topological disk in $SM^-$ so $\gamma^0$ is contractible in $U$.
Moreover, $\virdim(G)$ is determined by $\gamma^0$ and it is given by $2n-4>0$ (see Corollary \ref{c:IndexClosedCap}).
\end{proof}

By combining \eqref{eq:VirDimG}, $\virdim(\eT)=0$, $\virdim(u_v) \ge 0$ for $v \in V^{core} \cup V^{\partial}$ and Lemma \ref{l:noclosedstringbubble}, we conclude
that $V^{int}= \emptyset$, $k_G=0$ and $\virdim(u_v)=0$ for all $v$.

Notice that if $u_v$ is not a trivial cylinder but $l_\eT(v) \notin \{0, n_\eT\}$, then $\virdim(u_v) \ge 1$ because one can translate $u_v$ along the $r$-direction.
Therefore, all intermediate level curves are trivial cylinders so $n_\eT=1$.
The last thing to show is that if $v \in V^{\partial}$, then $l_\eT(v)=1$ and $u_v$ has only one boundary asymptote.

We argue by contradiction. Suppose $l_\eT(v)=0$.
Due to the boundary condition, all asymptotes of $u_v$ are Reeb chords $y_0,\dots,y_{d_v}$.
Inside $SM^-$, we can compute the index of Reeb chords using the canonical relative grading.
By Corollary \ref{c:ReebDynamicMorsified}, we have $\iota(y_j) \le 0$ for all $j$.
It means that $\virdim(u_v)=n-\sum_{j=0}^{d_v} \iota(y_j)-(2-d_v) \ge n-2>0$.
This is a contradiction so $l_\eT(v)=1$ for all $v \in V^{\partial}$.

Finally, if there exists $v \in V^{\partial}$ such that $u_v$ has more than one boundary asymptote, then by the fact that $\eT$ is a tree, we must have $v \in V^{\partial}$
such that $l_\eT(v)=0$. 
This is a contradiction so we finish the proof of Proposition \ref{p:nosidebubble}.

\end{proof}

\subsection{Gluings in SFT} % (fold)
\label{sub:gluing_in_sft}

%Given a graded Lagrangian $L$ with cylindrical ends, possibly asymptotic to a Reeb chord.  The Reeb chord will 

To conclude our general discussion on generalities of neck-stretching, we recall the following gluing theorem for SFT, which will play an important role in our proof.

\begin{thm}\label{t:SFTgluing}
  Let $u_{\infty}=(u_v)_{v \in V(\eT)} \in \eM^{J^{\infty}}(x_0;x_d,\dots,x_1)$ be a holomorphic building such that $u_v$ is transversally cut out for all $v$ and $\virdim(u_{\infty})=0$.  Assume also that all asymptotic Reeb chords are non-degenerate.

  Then for any small neighborhood $N_{u_{\infty}}$ of $u_{\infty}$ in an appropriate topology,
  there exists $\Upsilon>0$ sufficienly large such that for each $\tau>\Upsilon$, there is a unique
  $u^{\tau} \in \eM^{J^{\tau}}(x_0;x_d,\dots,x_1)$ lying inside $N_{u_{\infty}}$.
  Moreover, $u^\tau$ is regular and $\{u^\tau\}_{\tau \in [\Upsilon,\infty)}$
  converges in SFT sense to $u_{\infty}$ as $\tau$ goes to infinity.
\end{thm}

A nice account for the above SFT gluing result can be found in Appendix A of \cite{Lip06}.  In the presence of conical Lagrangian boundary conditions, see also \cite[Proposition 4.6]{EES07} and \cite[Section $8$]{EES05}.   

The typical application of Proposition \ref{p:nosidebubble} and Theorem \ref{t:SFTgluing} goes as follows.
Given a collection of Lagrangians such that the assumption of Theorem \ref{t:SFTcompactness} is satisfied, 
we want to determine the signed count of rigid elements in $\eM^{J^{\tau}}(x_0;x_d,\dots,x_1)$ for some large $\tau$.
When $d=1$ (resp. $d=2$), the signed count is responsible to the Floer differential (resp. Floer multiplication).
If we pick $u_k \in \eM^{J^{\tau_k}}(x_0;x_d,\dots,x_1)$ such that $\lim_{k \to \infty} \tau_k=\infty$, we get a holomorphic building $u_\infty$ by Theorem \ref{t:SFTcompactness}.
By Proposition \ref{p:nosidebubble}, $u_\infty$ satisfies the assumption of Theorem \ref{t:SFTgluing}.
Therefore, for sufficently large $\tau$, $\eM^{J^{\tau}}(x_0;x_d,\dots,x_1)$ is in bijection to $\eM^{J^{\infty}}(x_0;x_d,\dots,x_1)$.
Moreover, all elements in $\eM^{J^{\tau}}(x_0;x_d,\dots,x_1)$ are transversally cut out.
It means that the Floer differential (resp. Floer multiplication) can be computed by determining $\eM^{J^{\infty}}(x_0;x_d,\dots,x_1)$, which is exactly what we will do in the
following section.
%\begin{corr}
% Lagrangian Floer cohomology can be computed at SFT limit.
%\end{corr}

\begin{comment}
\subsection{convention}

Equip gradings on $L_0,L_1,S$ such that $CF(S,L_i)$ are concentrated at degree $0$.
Then, $CF(L_i,S)$ are concentrated at degree $n$ and hence $CF(L_i,\tau(L_j))$ is concentrated at degree $n-1$ for $i \neq j$.
On the other hand, $CW(L_0,L_1)$ has one generator at degree $-k(n-1)$ for each $k \in \mathbb{N}$.
In particular, the one with degree $0$ is the shortest Hamiltonian chord from $L_0$ to $L_1$ with respect to wrapping, which in turn corresponds to the
shortest Reeb chord from $\Lambda_0$ to $\Lambda_1$.
We denote the Reeb chord corresponding to the generator of $CW(L_0,L_1)$ with index $-k(n-1)$ by $x_k$, which also corresponds to 
the generator of $CW(L_0,\tau(L_1))$ with index $-k(n-1)$.

% subsection gluing_in_sft (end)
\end{comment}

\section{Cohomological identification} % (fold)

\label{sec:quasi_isomorphism}

% section quasi_isomorphism (end)

% $G\cdot (p_0,p_1)$ then corresponds to $G\cdot (F_e^1\#S_e)\cap F_e^0$ (or equivalently, $\bigcup_{g\in G}\tau_{S_g}F_g^1 \cap F_e^0$).

% $G\cdot (p_0,p_1)\in hom(L_0,\wt P)\otimes_G hom(\wt P, L)$

% section dehn_twists_as_a_surgery (end)

Let $P$ be a Lagrangian such that \eqref{eq:GammaCondition} is satisfied and $\eP$ be the universal local system on $P$.
We pick a parametrization of $P$ so that $\tau_P$ can be defined.
In this section, we want to prove that 

\begin{prop}\label{p:CohLevelIso}
For $\eE^0,\eE^1 \in Ob(\cF)$, we have cohomological level isomoprhism
\begin{align}
 H(hom_{\cF^{\perf}}(\eE^0,T_{\eP}(\eE^1))) \simeq H(hom_{\cF}(\eE^0,\tau_P(\eE^1)))
\end{align}
\end{prop}

We will only consider the case that $\eE^i=L_i$ are Lagrangians without local system.
The proof of the general case is identical except that the notations become more involved.
In slightly more geometric terms, we would like to directly construct a chain map $\iota$ from 
\begin{equation}
C_0:=Cone(CF(\eP, L_1)  \otimes_\Gamma CF(L_0,\eP) \xrightarrow{ev_\Gamma} CF(L_0,L_1)) 
\end{equation}
 to 
\begin{equation}
 C_1:=CF(L_0, \tau_P L_1)
\end{equation}
which induces isomoprhism on cohomology.

By applying a Hamiltonian perturbation, we assume that $L_0 \pitchfork L_1$ and each connected component of $L_i\cap U$ is a cotangent fiber in $U$.  
The cotangent fiber $T^*_{q}P \cap U$ has $|\Gamma|$ diffenrent lifts $\{T^*_{g\mathbf{q}}\fP\}_{g\in \Gamma} \cap \fU$ in $\fU$, where $\fU \subset T^*\fP$ is the universal cover of $U$.  
We assume the Dehn twist $\tau_P$ is supported inside $U$ and we have a commutative diagram:
\[
  \begin{tikzcd}
    \fU \arrow{r}{\tau_{\mathbf{P}}} \arrow{d}{ \pi} & \fU \arrow{d}{\pi} \\
    U \arrow{r}{\tau_P} & U
  \end{tikzcd}
\]
where $\pi:\fU \to U$ is the covering map.
As always, we assume that $L_0,L_1$ are equipped with $\mathbb{Z}-$gradings and spin structures.

\subsection{Correspondence of intersections}\label{s:Intersection}

We denote the set of generators in $C_0$ by $\cX(C_0)$, which is divided into two types $\cX_a(C_0)$ and $\cX_b(C_0)$:

\begin{itemize}
  \item $\cX_a(C_0)$: generators in $hom(\eP, L_1)\otimes_\Gamma hom(L_0,\eP)[1]$
  \item $\cX_b(C_0)$: generators in $hom(L_0,L_1)$
%  \item Type (II): differentials for the evaluation.
\end{itemize}

More precisely, $\cX_b(C_0)=L_0 \cap L_1$ and $\cX_a(C_0)$ is the set of elements
of the form $[\mathbf{q}^\vee\otimes g\mathbf{p} ]\sim [\mathbf{q}^\vee g\otimes \mathbf{p} ]\sim [(g^{-1}\mathbf{q})^\vee\otimes \mathbf{p}]$, where 
we are using the correspondence \eqref{e:LStoLift} and \eqref{e:LStoLiftvee}.
On the other hand, we denote $L_0\cap \tau_P L_1$ by $\cX(C_1)$ which is a set of generators for $C_1$.

Let $p \in L_0 \cap P$, $q \in L_1 \cap P$ and $\fp, \fq \in \fP$ be a lift of $p$ and $q$, respectively.  We also introduce the following notation
\begin{align}
\begin{split}
  &\fc_{\fp, \fq}:\text{ the unique intersection }T^*_\fp \fP \cap \tau_\fP (T^*_\fq \fP)\\
  &c_{\fp,\fq}:=\pi(\textbf{c}_{ \fp,\fq}) \text{, which is an intersection of $L_0\cap\tau_P L_1$}
\end{split}
\end{align}

\begin{lemma}\label{l:CorrIntersections}
 There is a grading-preserving bijection $\iota: \cX(C_0) \to \cX(C_1)$.
\end{lemma}

\begin{proof}
First, there is an obvious graded identification between $\cX_b(C_0)$ and the intersections of $L_0\cap \tau_P L_1$ outside $U$,
so we only need to explain how to define $\iota|_{\cX_a(C_0)}$.

% Let 
% We denote the point .

 %Consider a lift $F^1_{\mathbf{q} h^{-1}}\subset W$ of $F^1_{q}\subset U$.  Its Dehn twist intersects each $F^0_{\mathbf{p}g}$ at a single point, 
 %which we denote as $\textbf{c}_{h \mathbf{q}^\vee,\mathbf{p}g}$.  
 %Note that $c_{\fp,\fq} \in U$ is an intersection of $L_0\cap\tau_P L_1$. 
 %By equivariance, it is not hard to see that the map
We define $\iota|_{\cX_a(C_0)}$ by
\begin{equation}\label{e:genCorr}
     \iota|_{\cX_a(C_0)}: \mathbf{q}^\vee\otimes \mathbf{p} \mapsto  c_{\fp,\fq} 
\end{equation}
This map is well-defined because 
\begin{align}
 \iota(\mathbf{q}^\vee g^{-1} \otimes g\mathbf{p})=\iota((g\mathbf{q})^\vee\otimes g\mathbf{p})=\pi(\textbf{c}_{ g\fp,g\fq})=  c_{\fp,\fq}
\end{align}
The last equality comes from the equivariance of $\tau_\fP$.
It is clear that $\iota|_{\cX_a(C_0)}$ is a bijection from $\cX_a(C_0)$ to  the intersections of $L_0\cap \tau_P L_1$ inside $U$.

To see that $\iota|_{\cX_a(C_0)}$ preserves the grading, we only need to observe that $\pi$ interwines the canonical trivialization of 
$(\Lambda^{\otimes top}_{\mathbb{C}}(T^*\fU))^{\otimes 2}$ and $(\Lambda^{\otimes top}_{\mathbb{C}}(T^*U))^{\otimes 2}$ so the computation reduces to the case 
that $P=S^n$, which is well-known (see e.g. \cite{Se03}).
\end{proof}

%When $gh=e$, we will simply denote $c_{p,q,e}$ as $c_{p,q}$. 
Using Lemma \ref{l:CorrIntersections}, we define $\cX_a(C_1)=\iota(\cX_a(C_0))$ and $\cX_b(C_1)=\iota(\cX_b(C_0))$.
%Elements in $\cX_a(C_0), \cX_a(C_1)$ ($\cX_b(C_0), \cX_b(C_1)$, resp.) are called \textbf{Type (a) (Type (b), resp.)} generators.  
We summarize our notation in Figure \ref{fig:Notation}.

% Notation-wise, the intersection $(g_1 p,g_2 q)$ means the intersection between the lift $g_1F_{p}=F_{g_1p}$ and $g_2Z_{q}=Z_{g_2q}$, while $\pi(g_1p,g_2q)$ is an intersection in $L_0\cap \tau_P L_1$.  Clearly, $\pi(g_1p,g_2q)=\pi(gg_1p,gg_2q)$ for any $g\in G$.  

\begin{figure}
  \centering
  \includegraphics[]{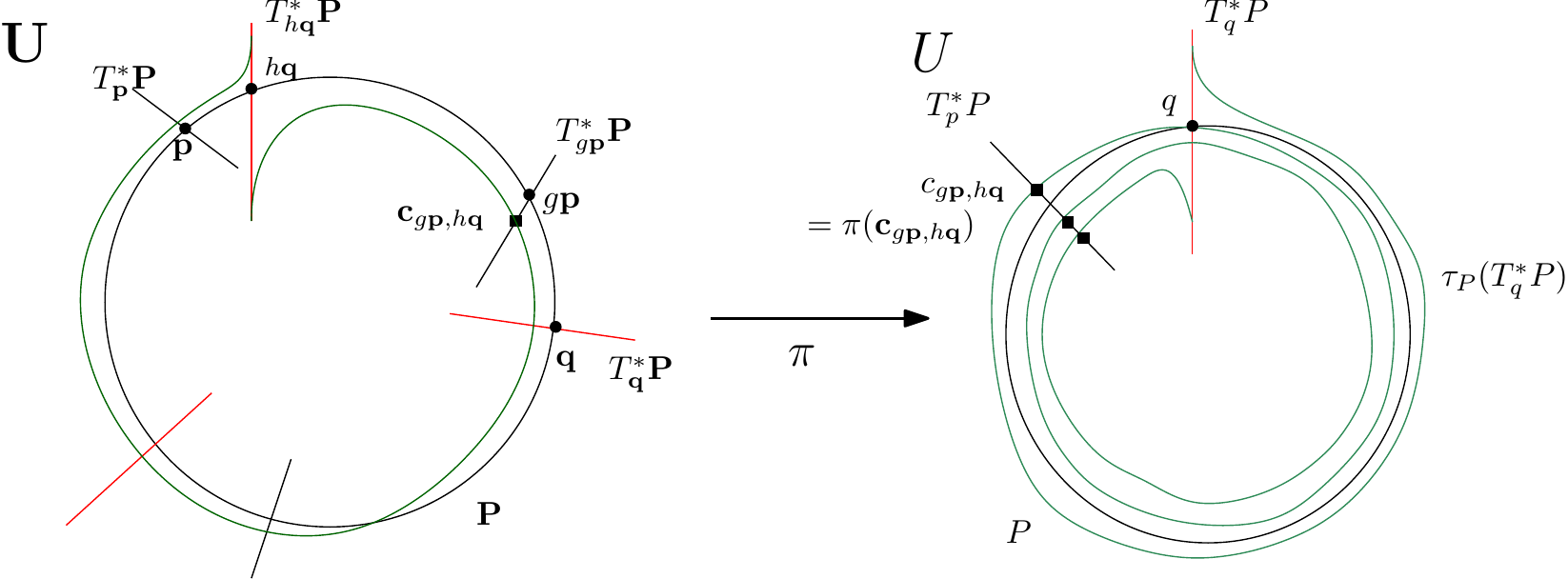}
  \caption{Generator correspondence between $C_0$ and $C_1$}
  \label{fig:Notation}
\end{figure}

\subsection{Overall strategy} % (fold)
\label{sub:correspondence_of_differentials}

The differentials in $C_0$ can be divided into four types.

\begin{itemize}
  \item Type (A1): differentials in $hom(L_0,\eP)$, i.e. pseudo-holomorphic strips in 
  $\eM( \mathbf{p'} ; \mathbf{p})$
  %with respect to some $J_{(L_0,\eP)}=(J_{(L_0,\eP),t})_{t \in [0,1]}$;
  \item Type (A2): differentials in $hom(\eP, L_1)$, i.e. pseudo-holomorphic strips in 
  $\eM((\mathbf{q}')^\vee ; \mathbf{q}^\vee)$ 
  %with respect to some $J_{(\eP,L_1)}=(J_{(\eP,L_1),t})_{t \in [0,1]}$;
  \item Type (B): differentials in $hom(L_0, L_1)$, i.e. pseudo-holomorphic strips in 
  $\eM(x_0;x_1)$ 
  %with respect to some $J_{(L_0,L_1)}=(J_{(L_0,L_1),t})_{t \in [0,1]}$;
  \item Type (C): differentials from the evaluation map, i.e. pseudo-holomorphic triangles in 
  $\eM(x; \mathbf{q}^\vee, \mathbf{p})$ 
  %with respect to some $J_{(L_0,\eP,L_1)}=(J_{(L_0,\eP,L_1),z})_{z \in S}$, 
  %where $S$ is the unit disc with $3$ boundary punctures and $J_{(L_0,\eP,L_1)}$ equals to $J_{(L_0,\eP)}$, $J_{(\eP,L_1)}$ and $J_{(L_0,L_1)}$ 
  %respectively over the corresponding strip-like ends.
\end{itemize}

\begin{figure}[tb]
  \centering
  \includegraphics[]{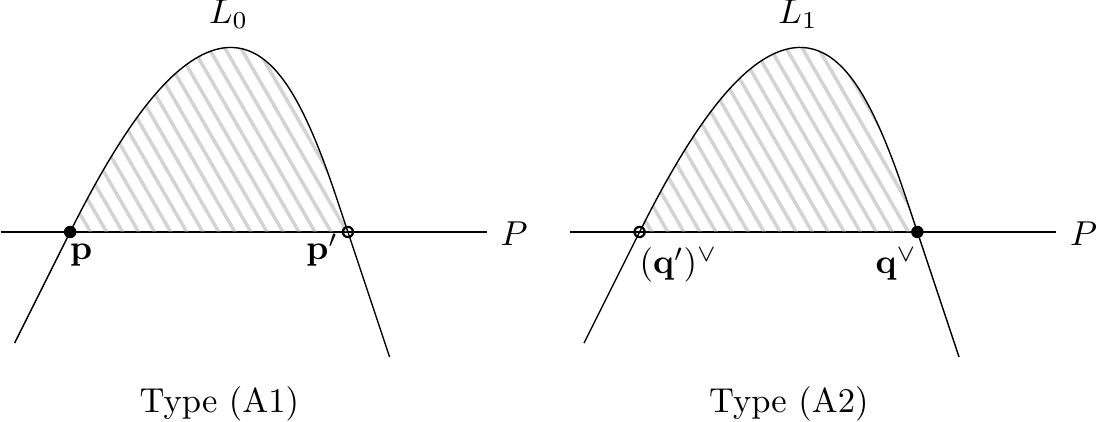}
  \includegraphics[]{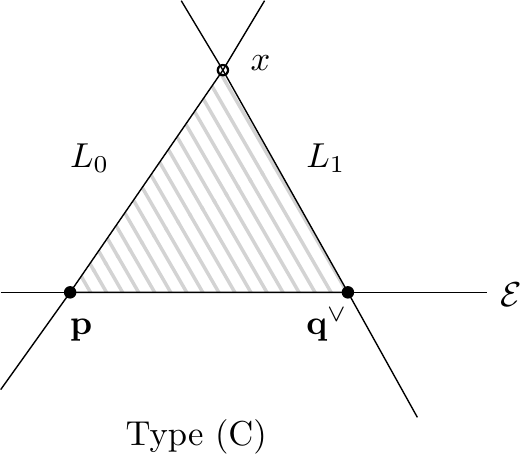}
  \caption{Types of holomorphic curves in $C_0$}
  \label{fig:C0}
\end{figure}

For $C_1$, we divide the differentials similarly, using correspondence of generators $\iota$.  Concretely, we have:

\begin{itemize}
  \item Type (A1'): pseudo-holomorphic strips in $\eM(c_{\fp',\fq}; c_{\fp,\fq})$;
  \item Type (A2'): pseudo-holomorphic strips in $\eM(c_{\fp,\fq'}; c_{\fp,\fq})$;
  \item Type (A3'): pseudo-holomorphic strips in $\eM(c_{\fp',\fq'}; c_{\fp,\fq})$ that are not in Type(A1') and (A2');
  \item Type (B'): pseudo-holomorphic strips in $\eM(x_0;x_1)$; 
  \item Type (C'): pseudo-holomorphic strips in $\eM(x; c_{\fp,\fq})$;
  \item Type (D'): pseudo-holomorphic strips in $\eM(c_{\fp,\fq};x)$;
\end{itemize}
%where all the moduli above are with respect to the same $J_{(L_0,\tau_P(L_1))}=(J_{(L_0,\tau_P(L_1)),t})_{t \in [0,1]}$.
where $x,x_0,x_1 \in \cX_b(C_1)$.

\begin{figure}[tb]
  \centering
  \includegraphics[]{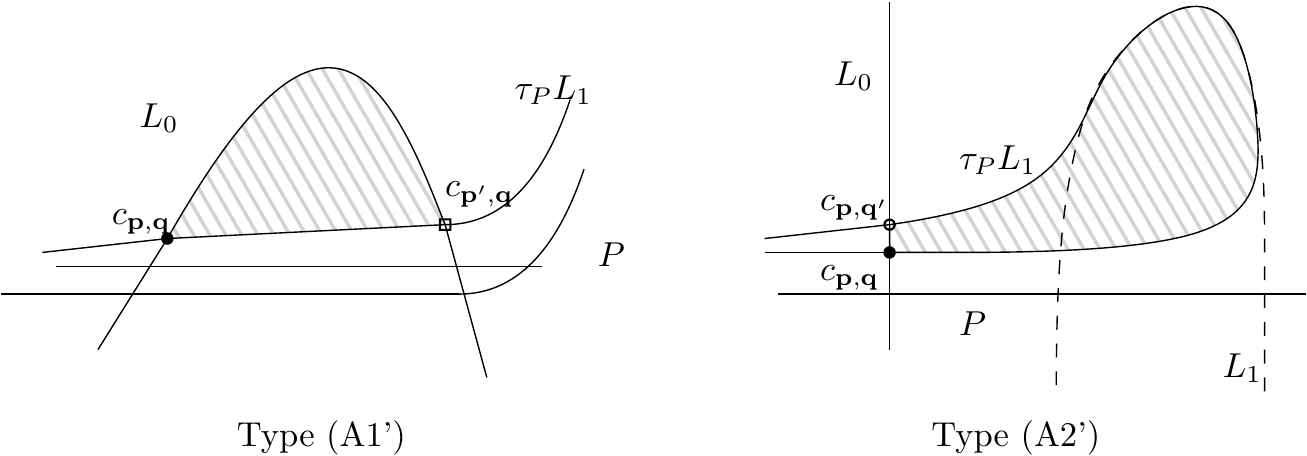}
  \includegraphics[]{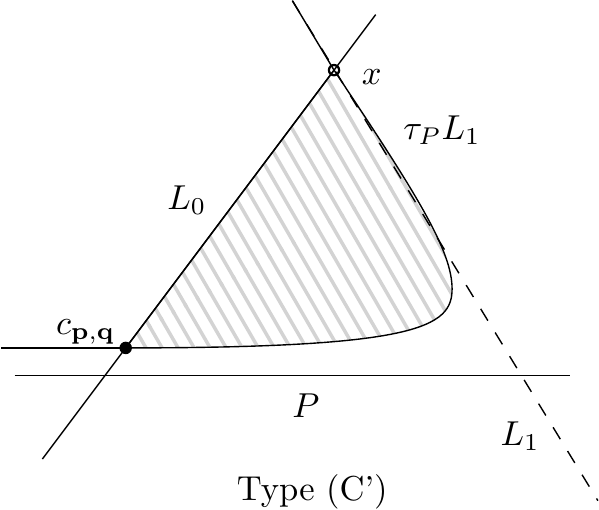}
  \caption{Types of holomorphic curves in $C_1$}
  \label{fig:C1}
\end{figure}

%The reader should be aware that we have use the abuse of notation pointed out in Remark \ref{rem:NonCanonical}: for example, the correspondence we need for Type (A1) should be between
%$\eM(h \mathbf{q}^\vee\otimes \mathbf{p'} g'; h \mathbf{q}^\vee\otimes \mathbf{p} g)$ and $\eM(c_{p',q}; c_{p,q,gh^{-1}})$; however, we may omit all of $g,h$ and $g'$ by re-choosing the lifts $\mathbf{p}$, $\mathbf{p}'$ and $\mathbf{q}$.  Therefore, the proof of the above correspondences indeed covers all needed situations.

By the discussion in Section \ref{sub:gluing_in_sft}, we know that for an appropriate choice of $\{J^\tau\}$
and $\tau \gg 1$, all the rigid $J^\tau$-holomorphic polygons in the moduli above are transversally cut out and 
they are in bijection to the corresponding holomorphic buildings.
By studying the holomorphic buildings, we will show that there are bijective correspondences
\begin{align}
 &\eM^{J^\tau}( \mathbf{p'} ; \mathbf{p}) \simeq \eM^{J^\tau}(c_{\fp',\fq}; c_{\fp,\fq}) \text{ for all }\fq; \\
 &\eM^{J^\tau}((\mathbf{q}')^\vee ; \mathbf{q}^\vee) \simeq \eM^{J^\tau}(c_{\fp,\fq'}; c_{\fp,\fq}) \text{ for all }\fp;\\
 &\eM^{J^\tau}(x_0;x_1) \simeq \eM^{J^\tau}(x_0;x_1); \label{eq:typeB}\\
 &\eM^{J^\tau}(x; \mathbf{q}^\vee, \mathbf{p}) \simeq \eM^{J^\tau}(x; c_{\fp,\fq}); \\
 & \text{Type (A3') and (D') are empty with respect to $J^\tau$}.
\end{align}
where the two sides of \eqref{eq:typeB} are with respect to boundary conditions $(L_0,L_1)$ and $(L_0, \tau_P(L_1))$, respectively.
In other words, for $\tau \gg 1$, $\iota:C_0 \to C_1$  is an isomorphism which clearly implies Proposition \ref{p:CohLevelIso}.

In the following subsections, we ignore the sign and only conisder the case that $\cchar(\K)=2$.
The complete proof of Proposition \ref{p:CohLevelIso}, where orientation of moduli is taken into account, will be given in Appendix \ref{sec:orientations}.

\subsection{Neck-stretching limits of holomorphic strips and triangles} % (fold)
\label{sub:neck_stretching_limits_of_holomoprhic_strips_and_triangles}

% subsection neck_stretching_limits_of_holomoprhic_strips_and_triangles (end)

In this section, we will list all possible holomorphic buildings $u_{\infty}=\{u_v\}_{v \in V(\eT)}$ that arises as the limit (when $\tau$ goes to infinity) 
of curves in the moduli discussed in Section \ref{sub:correspondence_of_differentials}.
By Proposition \ref{p:nosidebubble}, we know that $u_\infty$ satisfies the following conditions
\begin{align}
  \label{eq:SFTlimitCondition}
 \begin{array}{ll}
  (i)\text{ The total level }n_\eT=1, \\
  (ii) \text{ }\virdim(u_v)=0 \text{ for all }v. \\
  (iii)\text{ All compact edges in $\eT$ correspond to Reeb chords}.\\
  (iv)\text{ If }v \in V^{\partial}, \text{ then $l_\eT(v)=1$ and $u_v$ has exactly one boundary asymptote.}
 \end{array}
\end{align}
Therefore, we assume \eqref{eq:SFTlimitCondition} hold throughout this section.

\begin{lemma}\label{l:Maslov2}
 In the case (iv) of \eqref{eq:SFTlimitCondition}, let $v \in V^{\partial}$ and $x$ be the negative asymptote of $u_v$. Then $|x|=1$.
\end{lemma}

\begin{proof}
  By  Lemma \ref{l:VirdimFormula2} and $\virdim(u_v)=0$, we have
 \begin{align}
   0=\virdim(u_v)=|x|+\mbb(x)-2=|x|-1
 \end{align}
Therefore, $|x|=1$.

\end{proof}

\begin{comment}
\begin{lemma}\label{l:vanishing}\label{l:noUcurve}
   Let $F_0$, $F_1$ be two fibers in $T^*P$ with the canonical grading given by the vertical distribution, and the zero section $P$ also equipped with the canonical zero grading.  Then there is a unique chain $c_{01}\in CF^{n-1}(F_0,\tau_P F_1)$ represented by a geometric intersection.  

   Moreover, let $\{F_i\}_{1\le i\le r}$ be a set of cotangent fibers with arbitrary choice of gradings in $T^*P$.  Consider any holomorphic disks $u$ with the following properties:

       \begin{itemize}
       \item $\partial u\in \bigcup_i F_i\cup\bigcup_j\tau_P F_j$ or $\bigcup_i F_i\cup P$
       \item all positive ends of $u$ are asymptotic to Reeb chords, and there is either no negative ends, or its unique negative end is asymptotic to $c_{01}$.
     \end{itemize}

   Then $virdim(u)\ge n-1$ if it has a negative end, and $virdim(u)\ge n$ if it has none. In particular, if $u$ is regular, it is not rigid.  
\end{lemma}
\end{comment}

\begin{lemma}\label{l:noUcurve}
 If $l_\eT(v)=0$, then $u_v$ has at least one asymptote that is not a Reeb chord.
\end{lemma}

\begin{proof}
 Suppose not. Let $y_1,\dots,y_k$ be the asymptotes of $u_v$ which are all positive Reeb chord.
 Notice that the shift of gradings for any individual boundary condition does not affect the virtual dimension of $u_v$.  
Therefore we can use the canonical relative grading to compute the virtual dimension of $u_v$.
 By Lemma \ref{l:VirdimFormula2} and Corollary \ref{c:ReebDynamicMorsified}, we have
 \begin{align}
  \virdim(u_v)=n- \sum_{j=1}^k  |y_j| -(3-k) \ge n-3+k \ge n-2 >0
 \end{align}
which contradicts to the assumption \eqref{eq:SFTlimitCondition} that $\virdim(u_v)=0$.
\end{proof}

\begin{lemma}\label{l:A3Dtype}
 Every generator $c_{\fp,\fq} \in CF(T^*_p P, \tau_P(T^*_qP))$ satisfies $|c_{\fp,\fq}|=n-1$ with respect to the canonical relative grading.
 Moreover, if  $c_{\fp,\fq}$ is the only asymptote of a non-constant $J^-$-holomorphic map $u_v: \Sigma_v \to SM^-=T^*P$ that is not a Reeb chord, then 
 $c_{\fp,\fq}$ must be positive as an asymptote of $u_v$.
\end{lemma}

  \begin{proof}
       To see that $|c_{\fp,\fq}|=n-1$, it suffices to show that $|\fc_{\fp,\fq}|=n-1$.
       One can compute it directly by noting that $\tau_\fP (T^*_\fq \fP)=\fP[1] \# T^*_\fq \fP$, where $\fP[1]$ is the grading shift of $\fP$ by $1$
       and $\#$ denotes the graded Lagrangian surgery at the point $\fq$ (see \cite{SeGraded} or \cite{MW15}).
       Alternatively, one can see it using the Dehn twist exact sequence \cite{Se03}
       \begin{align}
        0 \to HF^k(T^*_\fp \fP, \tau_\fP(T^*_\fq\fP)) \to  \oplus_{a+b-1=k} HF^a(\fP, T^*_\fq\fP)\otimes HF^b(T^*_\fp \fP, \fP) \to 0
       \end{align}
and the fact that the second non-trivial term is non-zero only when $a=0$ and $b=n$.     

On the other hand, if $c_{\fp,\fq}$ is a negative asymptote and the remaining asymptotes are denoted by $y_1,\dots,y_k$, we would have (computed in canonical relative grading)
\begin{align}
 \virdim(u_v)=|c_{\fp,\fq}| - \sum_{i=1 }^k |y_i| - (2-k) \ge n-2>0
\end{align}
which contradicts to the assumption \eqref{eq:SFTlimitCondition} that $\virdim(u_v)=0$.
\end{proof}

Now, we can describe the SFT limits of various moduli.

\begin{lemma}[Type (A1)]\label{l:stretchingStripA}
 Let $u_{\infty}=(u_v)_{v \in V(\eT)}$ be a non-empty SFT limit of curves in $\eM^{J^\tau}( \mathbf{p'} ; \mathbf{p})$.
 Then $\eT$ consists of exactly two vertices $v_1,v_2$ and 
\begin{itemize}
  \item $u_{v_1}$ is a $J^-$-holomorphic triangle with negative asymptote $p':=\pi(\fp')$ and positive asymptotes $x,p$
  where $x$ is a Reeb chord with $|x|=0$ in the canonical relative grading; 
  \item $v_2 \in V^{\partial}$ so, by Lemma \ref{l:Maslov2}, $u_{v_2}$ is a $J^+$-holomorphic curve with one negative asymptote $x$ such that $|x|=1$ in the actual grading.
\end{itemize}
\end{lemma}

\begin{figure}[tb]
  \centering
  \includegraphics[scale=1.2]{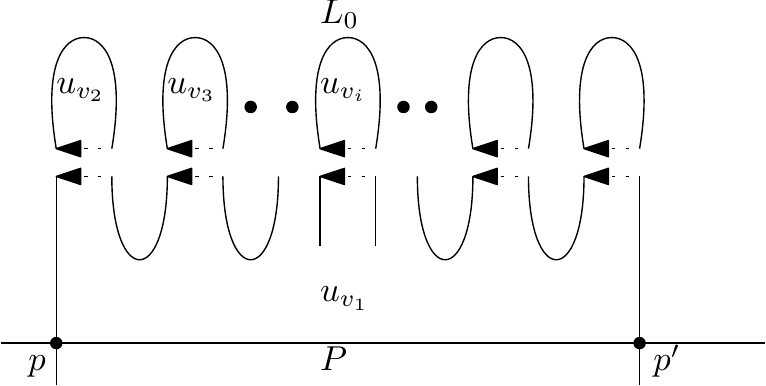}
  \caption{Multiple side bubbles}
  \label{fig:side}
\end{figure}

\begin{proof}
 Notice that, by the boundary condition $P$, $p$ and $p'$ must be asymptotes of the same $u_v$.
 We call it $u_{v_1}$.
 We label the other vertices of $\eT$ by $v_2, \dots, v_k$ for some $k \ge 0$.
 By boundary condition again, we know that $v_j \in V^{\partial}$ for $j > 1$.
 By \eqref{eq:SFTlimitCondition}, we have $l_\eT(v_j)=1$ for $j >1$.
 Moreover, all $v_j$ are adjacent to $v_1$ because $u_{v_j}$ has a negative asymptote (see Figure \ref{fig:side}).
%By Lemma \ref{l:Maslov2}, the negative asymptote of $u_{v_j}$ which is denoted by $y_j$ has $|y_j|=1$.
By Lemma \ref{l:VirdimFormula2} and Corollary \ref{c:ReebDynamicMorsified} again,
 \begin{align}
   0=\virdim(u_{v_1})=|p'|-|p|- \sum_{j=1}^k |y_j|-(1-k) \ge k-1
 \end{align}
so $k=0,1$. However, $k \neq 0$ by boundary condition.
As a result, $k=1$ and we denote $y_1$ by $x$.

Finally, to compute $|x|$ in the canonical relative grading, we just need to make a grading shift $|p'|-|p|=1$ to $T_{p'}^*P$.
It gives $|x|=0$ in the canonical relative grading.
 
\end{proof}

Similarly, we have

\begin{lemma}[Type (A1')]\label{l:stretchingStripA'}
 Let $u_{\infty}=(u_v)_{v \in V(\eT)}$ be a non-empty SFT limit of curves in $\eM^{J^\tau}( c_{\mathbf{p'},\fq} ; c_{\mathbf{p},\fq})$.
 Then $\eT$ consists of exactly two vertices $v_1,v_2$ and 
\begin{itemize}
  \item $u_{v_1}$ is a $J^-$-holomorphic triangle with negative asymptote $c_{\mathbf{p'},\fq}$ and positive asymptotes $x,c_{\mathbf{p},\fq}$
  where $x$ is a Reeb chord with $|x|=0$ in the canonical relative grading; 
  \item $v_2 \in V^{\partial}$ so $u_{v_2}$ is a $J^+$-holomorphic curve with one negative asymptote $x$ such that $|x|=1$ in the actual grading.
\end{itemize}
\end{lemma}

We omit the corresponding statements for type (A2) and (A2') because of the similarity.
Next we consider

\begin{lemma}[Type (B), (B')]\label{l:stretchingStripB}
 Let $u_{\infty}=(u_v)_{v \in V(\eT)}$ be a non-empty SFT limit of curves in $\eM^{J^\tau}(x_0;x_1)$.
 Then $\eT$ consists of exactly one vertex $v$ and $l_\eT(v)=1$.
\end{lemma}

\begin{proof}
 If $\eT$ has a vertex $v$ such that $l_\eT(v)=0$, then all the asymptotes of $v$ are Reeb chords which contradicts to Lemma \ref{l:noUcurve}.
 Therefore, $l_\eT(v)=1$ for all $v \in V(\eT)$ and it holds only when $\eT$ consists of exactly one vertex.
\end{proof}

\begin{lemma}[Type (C)]\label{l:triangle}
 Let $u_{\infty}=(u_v)_{v \in V(\eT)}$ be a non-empty SFT limit of curves in $\eM^{J^\tau}(x; \mathbf{q}^\vee, \mathbf{p})$.
 Then $\eT$ consists of exactly two vertices $v_1,v_2$ and 
\begin{itemize}
  \item $u_{v_1}$ is a $J^-$-holomorphic triangle with positive asymptotes $y,\mathbf{q}^\vee,\mathbf{p}$,
  where $y$ is a Reeb chord with $|y|=0$ in the canonical relative grading; 
  \item $u_{v_2}$ is a $J^+$-holomorphic curve with two negative asymptotes $x$ and $y$.
\end{itemize}
\end{lemma}

\begin{proof}
  Again, we use the same argument as in the proof of Lemma \ref{l:stretchingStripA}.
  There is $v_1 \in \eT$ such that
  $u_{v_1}$ is a holomorphic polygon and $\mathbf{q}^\vee$, $\mathbf{p}$
  are asymptotes of $u_{v_1}$.
  All other vertices are adjacent to $v_1$: otherwise, there will be components in $T^*\textbf{P}$ with only Reeb asymptotes, contradicting Lemma \ref{l:noUcurve}.  Denote these vertices by $v_2, \dots, v_k$.
  There is exactly one $j>1$ (say $j=2$) such that $v_j \notin V^{\partial}$ and $x$ is an asymptote of $u_{v_j}$.
  For $\eT$ to be a tree, $u_{v_2}$ has exactly one negative Reeb chord asymptote, which is denoted by $y_2$. 
  %See Figure \ref{fig:triangle}.
Let the negative asymptote for $u_{v_j}$ (for $j>2$) be $y_j$.
  
  %\begin{figure}[tb]
  %  \centering
  %  \includegraphics[scale=1.2]{Triangle.pdf}
  %  \caption{The limit of a holomorphic triangle}
  %  \label{fig:triangle}
  %\end{figure}

%For $u_{v_2}$  to be rigid, we have 
%\begin{align*}
% 0=-n+|x|+|y|+\mbb(y)-1=-n+|x|+|y|
%\end{align*}
For $u_{v_1}$  to be rigid, we have
\begin{align*}
 0=n-|\fp|-|\fq^\vee|-\sum_{j=2}^k |y_j|-(2-k) \ge n-n-0+k-2
\end{align*}
so $k \le 2$. However, we have $k \ge 2$ so we get $k=2$. 
Moreover, the canonical relative grading of $y_2$ is $0$.
%The computation of $|y|$ in canonical relative grading is similar to Lemma \ref{l:stretchingStripA}.
%Note that the same chord will have $\mu_{U}(\gamma)=n-|p|-|q^\vee|$ on the $U$ side, and any other asymptotic chords will have grading $1$ as reasoned in Lemma \ref{l:stretchingStrip}, therefore, the virtual index of the polygon is
%$n-|p|-|q^\vee|-(n-|p|-|q^\vee|)-k=-k$, where $k$ is the number of side bubbles.  Therefore, $k=0$.  Also, since $\mu_U(\gamma)=n-|p|-|q^\vee|=|q|-|p|$, it must be a short chord.

\end{proof}

\begin{rmk}\label{rem:dualizeC}
    Later on, we will also make use of the moduli space $\eM^{J^\tau}(\mathbf{p}^\vee;x^\vee,\mathbf{q}^\vee)$.  
    The shape of neck-stretching limit will remain the same as Type (C), because this is simply a modification of some of the strip-like ends 
    (from outgoing to incoming, and vice versa)
    and does not change the behavior of the underlying curve.
\end{rmk}

\begin{lemma}[Type (C')]\label{l:bigon}
 Let $u_{\infty}=(u_v)_{v \in V(\eT)}$ be a non-empty SFT limit of curves in $\eM^{J^\tau}(x; c_{\fp,\fq})$.
 Then $\eT$ consists of exactly two vertices $v_1,v_2$ and 
\begin{itemize}
  \item $u_{v_1}$ is a $J^-$-holomorphic bigon with positive asymptotes $y,c_{\mathbf{p},\fq}$,
  where $y$ is a Reeb chord with $|y|=0$ in the canonical relative grading; 
  \item $u_{v_2}$ is a $J^+$-holomorphic curve with two negative asymptotes $x$ and $y$.
\end{itemize}
\end{lemma}

\begin{proof}
  The argument is entirely parallel to Lemma \ref{l:triangle}.
  Let $u_{v_1}$ be the $J^-$-holomorphic curve such that $c_{\fp,\fq}$ is an asymptote of it.
  Let the other asymptotes of $u_{v_1}$ be $y_1, \dots, y_k$.
  For $u_{v_1}$ to be rigid, by Lemma \ref{l:A3Dtype},
  \begin{align}
   0=\virdim(u_{v_1})=n-|c_{\mathbf{p},\fq}|-\sum_{j=1}^k |y_j|-(2-k) \ge n-(n-1)-2+k=k-1
  \end{align}
 so $k=1$ because $u_{v_1}$ has at least one positive Reeb chord asymptote.
\end{proof}

Our final task is to show that type (A3') and (D') are empty for $\tau \gg 1$.

\begin{lemma}[Type (A3')]
 Let $u_{\infty}=(u_v)_{v \in V(\eT)}$ be a SFT limit of curves in $\eM^{J^\tau}(c_{\fp',\fq'}; c_{\fp,\fq})$ that are not in Type(A1') and (A2').
 Then $u_{\infty}$ is empty.
\end{lemma}

\begin{proof}
 There is $v \in V(\eT)$ such that $c_{\fp',\fq'}$ is a negative asymptote of $u_v$.
 By boundary condition, $c_{\fp,\fq}$ cannot be an asymptote of $u_v$.
 The existence of $u_v$ violates Lemma \ref{l:A3Dtype}.
\end{proof}

By Lemma \ref{l:A3Dtype} again, we have

\begin{lemma}[Type (D')]\label{l:stretchingStripD}
 Let $u_{\infty}=(u_v)_{v \in V(\eT)}$ be a SFT limit of curves in $\eM^{J^\tau}(c_{\fp,\fq};x)$.
 Then $u_{\infty}$ is empty.
\end{lemma}

\begin{comment}
\begin{lemma}[Type (A3), (D)]\label{l:twistedComplex}
      Consider the chain group $CF(L_0,\tau_P L_1)$.  Under the correspondence of generators $\iota$ \eqref{e:genCorr},

      \begin{enumerate}[(1)]
        \item modulo automorphisms, there is no rigid holomorphic strips with input in $\iota(CF(L_0,L_1))$ and output in $\iota( hom(\eP, L_1)\otimes_\Gamma hom(L_0,\eP))$,
        \item $\mu^1(c_{p,q,g})$ must be a linear combination of generators of the form $c_{p,q',g'}$ or $c_{p',q,g'}$ for some $p',q'\in P$ and $g\in\Gamma$.
      \end{enumerate}
      
\end{lemma}

\begin{proof}
   Assume the contrary, we take a rigid unparametrized holomorphic bigon $u$ representing a contribution of the differential in $CF(L_0,\tau_PL_1)$ with an output $c_{p,q,g}\in\iota(hom(\eP, L_1)\otimes_\Gamma hom(L_0,\eP))$.  Apply the neck-stretching along $ST^*P$.  Assuming the input (positive end) of $u$ is either outside $U$ or has the form $c_{p',q',g'}$ with $p\neq p'$ and $q\neq q'$,
   the irreducible component of the building containing $c_{p,q,g}$ must be a holomorphic polygon with a unique output $c_{p,q,g}$, and at least one input on the boundary, where all inputs are Reeb chords.  This means the virtual dimension of this component must be at least $n-1\ge2$ by Lemma \ref{l:vanishing}.  Hence, the rest of the holomorphic building must have an unparametrized component with virtual dimension less than zero.  This is a contradiction.

\end{proof}
\end{comment}

\subsection{Local contribution} % (fold)
\label{sub:twist_local_model}
In this section, we will determine the algebraic count of some moduli of rigid $J^-$-holomorphic curves in  $SM^-=T^*P$.

Let $q_1,q_2,q_3 \in P$ be three generic points such that $\cup_i \Lambda_{q_i}$ satisfies Corollary \ref{c:ReebDynamicMorsified}. 
Let $\fq_i \in \fP$ be a lift of $q_i$ for $i=1,2,3$.
Let $\fJ^-$ be the almost complex structure on $T^*\fP$ that is lifted from $J^-$. 
Since the contact form $\theta|_{\partial \fU}$ equals to the lift of $\alpha=\theta|_{\partial U}$,
by Lemma \ref{l:IndexChangeChord}, there is a unique Reed chord $x_{i,j}$ from $\Lambda_{\fq_i}$ to $\Lambda_{\fq_j}$
such that $|x_{i,j}|=0$ in the canonical relative grading.
Let $\fq_i \in CF(T^*_{\fq_i} \fP, \fP)$ and $\fc_{i,j} \in CF(T^*_{\fq_i} \fP, \tau_{\fP}(T^*_{\fq_j} \fP))$ be the chains represented by the unique geometric intersection in the respective chain complexes.

%Let $F_1, F'_{1}$ and $F_3$ be three generic distinct fibers $T^*S^n$ and $x_{1,2}$ denote the shortest chords connecting $F_1$ and $F_1'$ on the contact boundary.  Let $p_{1}\in CF(F_1,S^n)$, $p_{1'}\in CF(F_1', S^n)$, $c_{31'}\in CF(F_3,\tau_S F_1')$, $c_{31}\in CF(F_3,\tau_SF_1)$ etc. be the chain represented by the unique geometric intersections in the respective chain complexes, and $p_1^\vee$ will denote the Poincare dual chain.  

We are interested in the algebaic counts of following moduli spaces

% \textbf{some order is wrong here}
\begin{enumerate}[(1)]
   \item $\eM^{\fJ^-}(\fq_1;\fq_2,x_{1,2})$, $\eM^{\fJ^-}(q_{2}^\vee;x_{1,2},\fq_1^\vee)$ and $\eM^{\fJ^-}(\emptyset;\fq_2, x_{1,2}, \fq_1^\vee)$,
   \item $\eM^{\fJ^-}(\fc_{3,2};x_{1,2},\fc_{3,1})$,
   \item $\eM^{\fJ^-}(\fc_{1,3};\fc_{2,3},x_{1,2})$,
   \item $\eM^{\fJ^-}(\emptyset;\fc_{2,1},x_{1,2})$.
 \end{enumerate}

  \begin{figure}[]
   \centering
   \includegraphics[scale=0.9]{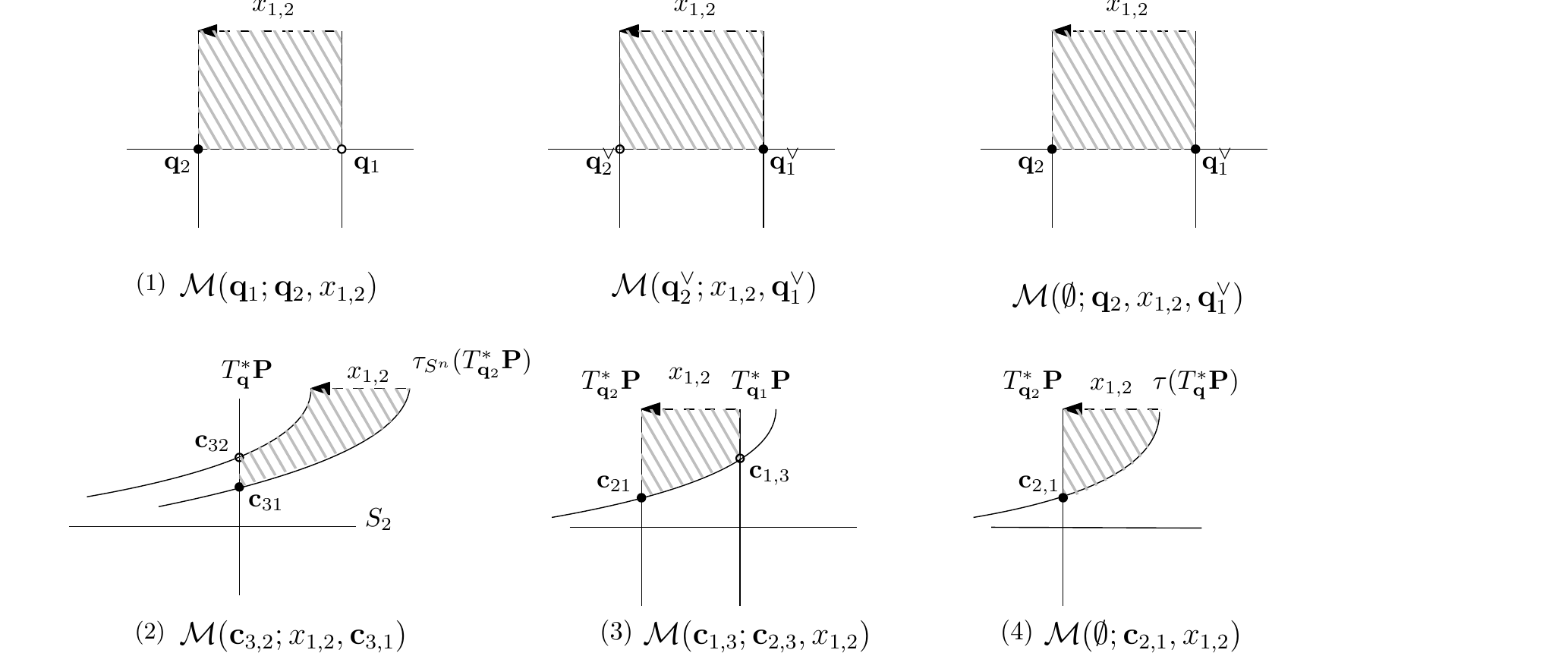}
   \caption{Six moduli spaces in Theorem \ref{t:localCount}}
 \end{figure}

\begin{thm}\label{t:localCount}
    The algebraic count of the above moduli spaces are all $\pm 1$. 
\end{thm}

% \subsubsection{Neck-stretching limits for holomorphic strips}

%To prove Theorem \ref{t:localCount}, we need to investigate the neck stretching limits of holomorphic curves in the big local model when we apply SFT stretching to $\partial T^*_\delta S^n$.  But we will work over the setting that accomodates all cases we need in this paper, that is, when $L, L'\subset M$ be exact symplectic manifolds and $U$ be a neighborhood of a spherical Lagrangian $P\subset M$.

% Take any holomorphic strips with dimension $1$.  We prove the stretching limit will have a $U$-component with a unique positive asymptote.

\begin{proof}[Proof of Theorem \ref{t:localCount}]

We will apply SFT stretching on the the following ``big local model''.

Consider an $A_3$ Milnor fiber consisting of the plumbing of three copies of $T^*S^n$.
We denote the Lagrangian spheres by $S_1$, $\fP$ and $S_3$, respectively, where $S_1 \cap S_3= \emptyset$.
%where the three Lagrangian spheres are $S_1$, $S_2$ and $S_3$, respectively.  
We can identify a neighborhood of $\fP$ with $\fU$.
By Hamiltonian isotopy if necessary, we assume that $\fU \cap S_j$ is a pair of disjoint cotangent fibers for $j=1,3$.  
We perturb $S_1$ to $S_2$ by a perfect Morse function, so that  $\fU \cap S_2$ is another cotangent fiber.

It will be clear that we should, for $j=1,2,3$, naturally abuse the notation to denote $\fq_j \in CF(S_j,\fP)$, which is the only generator in the corresponding cochain complex.
Let $e, pt\in CF(S_1,S_2)$ be the minimum and maximum of the Morse function, respectively, where $e$ represents the identity in cohomology.  
On the cohomological level, it is clear that $[\fq_2][e]=\pm [\fq_1]$ and $[e][\fq_1^\vee]=\pm [\fq_2^\vee]$.  This implies the algebraic count 

\begin{equation}\label{e:p12}
    \begin{aligned}
        &\#\eM(\fq_1; \fq_2, e)=\pm 1\\
        &\#\eM(\fq_2^\vee; e,\fq_1^\vee)=\pm 1.  
     \end{aligned}
\end{equation}

 \begin{figure}[]
   \centering
   \includegraphics[scale =0.5]{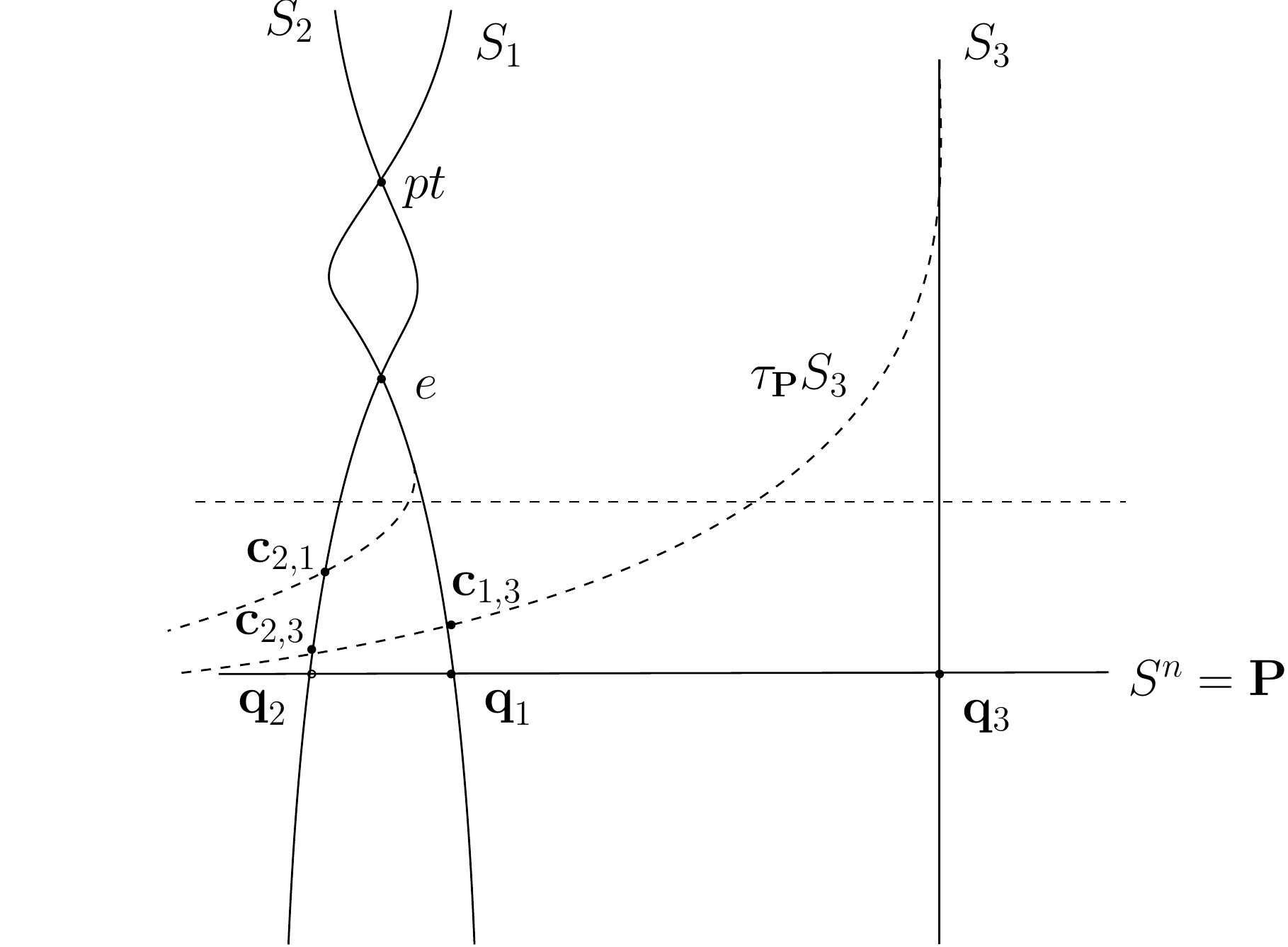}
   \caption{Big local model before stretch}
   \label{fig:local}
 \end{figure}

We now apply the same argument to other cochain complexes. For $i \neq j$, let $\fc_{i,j}\in CF^*(S_i,\tau_{\fP}(S_j))$.
%$c_{1'3}\in CF^*(S'_1,\tau_{S_2}S_3)$, $c_{31}\in CF^*(S_3, \tau_{S_2}S_1)$ and $c_{31'}\in CF^*(S_3, \tau_{S_2}S'_1)$ 
be the only generator in their corresponding complex.  Again, the multiplication by $[e]$ on $[\fc_{1,3}] $ and $ [\fc_{3,1}]$ yields

\begin{align}
   \label{e:p'12}\#\eM(\fc_{1,3}; \fc_{2,3}, e)=\pm 1, \\
   \label{e:ptilde12}\#\eM(\fc_{3,2}; e,\fc_{3,1})= \pm 1.
\end{align}

For the case of $\fc_{2,1}\in CF(S_2,\tau_\fP (S_1))$, it is immediate from Seidel's exact sequence that 
$\rank HF(S_2,\tau_\fP (S_1))=1$, concentrated on degree $0$.  
$CF(S_2,\tau_S S_1)$ has two additional generators $|\fc_{2,1}|=n-1$ and $|pt|=n$, which cancel each other.  Therefore, one has 

\begin{equation}\label{e:pbar12}
     \#\eM(pt; \fc_{2,1})=\pm 1
\end{equation}

To deduce Theorem \ref{t:localCount}, we perform a neck-stretching along $\partial \fU$.
It means that we choose a family of almost complex structure $\fJ^{\tau}$ adapted to $\partial \fU$ and see how the $\fJ^{\tau}$-holomorphic curves
converge as $\tau$ goes to infinity.
We require that the limiting almost complex structure on $S \fU$ coincides with $\fJ^-$ and we denote the limiting almost complex structure outside $\fU$ by $\fJ^+$.
$S_1$ and $S_2$ give two fibers in $\fU$, and every holomorphic curve in $\eM^{\fJ^\tau}(\fq_1; \fq_2,e)$ will converge, in the $\fU$ part, to a curve
in $\eM^{\fJ^-}(\fq_1;\fq_2, x_{1,2})$ (see Lemma \ref{l:stretchingStripA}).  This implies
\[
  (\#\eM^{\fJ^-}(\fq_1;\fq_2, x_{1,2}))\cdot (\#\eM^{\fJ^+}(x_{1,2};e))=\#\eM^{\fJ^\tau}(\fq_1; \fq_2,e)=\pm 1.
\]
Since all counts are integers, it follows that $\#\eM^{J^-}(\fq_1;\fq_2,x_{1,2})= \pm 1$
which implies the same is true for $\#\eM^{J^-}(\fq_2^\vee;x_{1,2},\fq_1^\vee)$ and $\#\eM^{J^-}(\emptyset;\fq_2,x_{1,2},\fq_1^\vee)$.

The same stretching argument, along with \eqref{e:p'12}\eqref{e:ptilde12}\eqref{e:pbar12} yields

\begin{align}
    &(\#\eM^{\fJ^-}(\fc_{1,3};\fc_{2,3}, x_{1,2}))\cdot (\#\eM^{\fJ^+}(x_{1,2};e))=\#\eM^{\fJ^\tau}(\fc_{1,3}; \fc_{2,3}, e)=\pm 1,\\
  &(\#\eM^{\fJ^-}(\fc_{3,2};x_{1,2},\fc_{31}))\cdot (\#\eM^{\fJ^+}(x_{1,2};e))=\#\eM^{\fJ^\tau}(\fc_{3,2};e, \fc_{3,1})=\pm 1,\\
    &(\#\eM^{\fJ^-}(\emptyset;\fc_{2,1},x_{1,2}))\cdot (\#\eM^{\fJ^+}(x_{1,2},pt;\emptyset))=\#\eM^{\fJ^\tau}(pt;\fc_{2,1})=\pm 1.
\end{align}
which give the remaining algebraic counts. 

Finally, notice that even though $S_2$ is obtained by a perturbation of $S_1$, we can actually Hamiltonian isotope
$S_2$ so that $S_2 \cap P$ is the preassigned $q_2$ and there is no new intersection 
between $S_2$ and $S_1,S_3$ being created during the isotopy. 
With this choice of $S_2$ and the stretching argument explained above, Theorem \ref{t:localCount} follows.

\end{proof}

One may define the analogous moduli spaces similarly on $T^*P$ for cotangent fibers $T^*_{q_i} P$.  
By equivariance, every rigid $J^-$-holomorphic curve lifts to $|\Gamma|$ many rigid $\fJ^-$-holomorphic curves and 
every rigid $\fJ^-$-holomorphic curve descends to a rigid $J^-$-holomorphic curve.
%For example, let $c_{\fp,\fq} \in CF(T^*_{} P,\tau_PF_q)$, 
%$g\in \Gamma$ denote the chains represented by the geometric intersections. $\gamma_{p,q,g}$ will denote a Reeb chord over a geodesic from $p$ to $q$, which, when lifted to $T^*\mathbf{P}$, is the short chord from $F_{\mathbf{p}}$ to $F_{\mathbf{q}g}$.

%\begin{rmk}\label{rem:choice}
%    Note, however, that since we do not make any canonical choices for the lifts $\mathbf{p}$ or $\mathbf{q}$ (this is equivalent to saying there is no canonical choice to label a geometric intersection in $CF(F_p,\tau_PF_q)$ as $c_{p,q,id}$), one may always assume that the lift under focus is precisely $\mathbf{p}$ and $\mathbf{q}$, (equivalently, the geometric intersection in $T^*P$ under focus is labelled by $id\in\Gamma$), unless otherwise stated.  We will suppress $id$ from the subscript for brevity.
%\end{rmk}

With this understood, we have

\begin{corr}\label{c:countU}
   The algebraic count of the following moduli spaces are $\pm1$.
\end{corr}

\begin{enumerate}[(1)]
  \item $\eM^{J^-}(p';p,x_{\fp',\fp})$, $\eM^{J^-}(q'^\vee;x_{\fq,\fq'},q^\vee)$ and $\eM^{J^-}(\emptyset;p,x_{\fq,\fp},q^\vee)$, 
  \item $\eM^{J^-}(c_{\fp,\fq'};x_{\fq,\fq'},c_{\fp,\fq})$
  \item $\eM^{J^-}(c_{\fp',\fq};c_{\fp,\fq},x_{\fp',\fp})$
  \item $\eM^{J^-}(\emptyset;c_{\fp,\fq},x_{\fq,\fp})$
\end{enumerate}
where $x_{\fp',\fp}$ is the unqiue Reeb chord of canonical relative grading $0$ from $\Lambda_{p'}$ to $\Lambda_{p}$ which can be lifted to a Reeb chord from
$\Lambda_{\fp'}$ to $\Lambda_{\fp}$. The definition of $x_{\fq,\fq'}$ and $x_{\fq,\fp}$ are similar.

\begin{comment}
\begin{proof}
  The moduli spaces above can be lifted to $\mathbf{U}$ and identified with 

  \begin{enumerate}[(i)]
  \item $\eM^{\fJ^-}(\textbf{p};\gamma_{\textbf{p,p'}},\textbf{p}')$ and $\eM^{\fJ^-}(\mathbf{p}'^\vee;\mathbf{p}^\vee,\gamma_{\textbf{p,p'}})$
  \item $\eM^{\fJ^-}(\emptyset;\textbf{p}',\textbf{p}^\vee,\gamma_{\textbf{p}\textbf{p}'})$
  \item $\eM^{\fJ^-}(c_{\textbf{p},\textbf{q}'};c_{\textbf{p},\textbf{q}},\gamma_{\textbf{q},\textbf{q}'})$
  \item $\eM^{\fJ^-}(c_{\textbf{p},\textbf{q}};\gamma_{\textbf{p},\textbf{p}'},c_{\textbf{p}',\textbf{q}})$
  \item $\eM^{\fJ^-}(\emptyset;c_{\textbf{p}',\textbf{p}},\gamma_{\textbf{p},\textbf{p}'})$
\end{enumerate}

  % \begin{enumerate}[(i)]
  %   \item $\eM^{\fJ^-}(\textbf{p}';\textbf{p},\gamma_{\textbf{p,p}'})$
  %   \item $\eM^{\fJ^-}(\textbf{c}^\vee_{\textbf{q},\textbf{p}};\gamma_{\textbf{p},\textbf{p}'},\textbf{c}^\vee_{\textbf{q},\textbf{p}'})$
  %   \item $\eM^{\fJ^-}(\mathbf{c}_{\textbf{p},\textbf{q}'};\mathbf{c}_{\textbf{p},\textbf{q}}, \gamma_{\textbf{q,q}'})$
  %   \item $\eM^{\fJ^-}(\emptyset;\textbf{c}_{\textbf{p}',\textbf{p}},\gamma_{\textbf{p},\textbf{p}'})$
  % \end{enumerate}
  
   Therefore, one may apply the count in Theorem \ref{t:localCount}.
\end{proof}
\end{comment}

% subsection surgery_local_model (end)

\subsection{Matching differentials} % (fold)
\label{sub:matching_differentials}

We now are ready to prove Proposition \ref{p:CohLevelIso}.  
The first lemma relates algebraic counts of differentials of Type (A1) and (A1').

\begin{lemma}\label{l:A1}
For $\tau \gg 1$, the algebraic count of following moduli spaces are equal

\begin{itemize}
  \item $\eM^{J^\tau}(c_{\fp',\fq}; c_{\fp,\fq})$, differentials in $hom(L_0,\tau_P (L_1))$ from $c_{\fp,\fq}$ to $c_{\fp',\fq}$,
  \item $\eM^{J^\tau}(\mathbf{p}';\mathbf{p})$, differentials in $hom(L_0, \eP)$ from $\mathbf{p}$ to $\mathbf{p}'$ 
  %(recall notations from the correspondence \eqref{e:LStoLift} and \eqref{e:LStoLiftvee}).
\end{itemize}
\end{lemma}

%Note that we actually need the general statement for $c_{p,q,g}$ and $c_{p,q',g'}$, but once we prove for the counts when $g$ and $g'$ are identity as above, it suffices to change the choice of the lift for $q$ and $q'$.  See Remark \ref{rem:NonCanonical}.

\begin{proof}
 % To prove the lemma, we choose a generic almost complex structure that is adjusted to $\partial U$, and take arbitrary holomorphic bigons $u_{c_{p,q}, c_{p,q'}}\in\eM^{\fJ^-}(c_{p,q'}; c_{p,q})$ and $u_{\textbf{q}, \textbf{q}'}\in \eM^{\fJ^-}(\textbf{q}';\textbf{q})$. They become holomorphic buildings under SFT limit after stretching $J$.  Without loss of generality, we may assume $|\textbf{q}|=0$ and $|\textbf{q}'|=1$ by a shift in the grading of $L_1$.

 To prove the lemma, we look at the SFT limit of these moduli when $\tau$ goes to infinity.
 Let $u_\infty^1$ and $u_{\infty}^2$ be a limiting holomorphic building from curves in $\eM^{J^\tau}(c_{\fp',\fq}; c_{\fp,\fq})$ and
 $\eM^{J^\tau}(\mathbf{p}';\mathbf{p})$, respectively.
 Lemma \ref{l:stretchingStripA} and \ref{l:stretchingStripA'}, 
 $u_{\infty}^i$ consist of a $J^-$-holomorphic curve $u_{v_1}^i$ and a $J^+$-holomorphic curve $u_{v_2}^i$.
 Moreover, $u_{v_2}^i$ lies in $\eM^{J^+}(x_{\fp,\fp'};\emptyset)$ for both $i$.
 On the othe hand, $u_{v_1}^1$ lies in $\eM^{J^-}(c_{\fp',\fq};c_{\fp,\fq}, x_{\fp',\fp})$ and
 $u_{v_1}^2$ lies in $\eM^{J^-}(\textbf{p}';\textbf{p},x_{\fp',\fp})$.
 
%Take the $U$-components of the limit curves for both moduli spaces.  Lemma \ref{l:stretchingStrip} implies the $U$-components $u_{c_{p,q}, c_{p,q'},\infty}^U\in\eM^{\fJ^-}(c_{p,q'};c_{p,q}, \gamma_{q,q'})$ and $u_{\textbf{q},\textbf{q}',\infty}^U\in\eM^{\fJ^-}(\textbf{q}';\textbf{q},\gamma_{q,q'})$ are both holomorphic triangles with a unique asymptotic chord to the short chord between $F_q$ and $F_q'$, and both moduli spaces have algebraic count equals $1$ by Corollary \ref{c:countU} (1) and (3).  

% Therefore, they are glued to the same moduli of rigid curves $\eM^{\fJ^-}(\gamma_{q,q'})$ in $M\backslash U$.

% By Corollary \ref{c:countU}, $\#\eM^{\fJ^-}(\textbf{q}';\textbf{q},\gamma_{q,q'})=\#\eM^{\fJ^-}(c_{p,q'};c_{p,q}, \gamma_{q,q'})=\pm1$.  

Therefore, for $\tau \gg 1$,

\begin{align*}
  &\#\eM^{J^\tau}(\textbf{p}',\textbf{p})\\
  =&\#\eM^{J^+}(x_{\fp,\fp'};\emptyset)\cdot\#\eM^{J^-}(p';p,x_{\fp',\fp})\\
  =&\#\eM^{J^+}(x_{\fp,\fp'};\emptyset)\cdot\#\eM^{J^-}(c_{\fp',\fq};c_{\fp,\fq}, x_{\fp',\fp})\\
  =&\#\eM^{J^\tau}(c_{\fp',\fq}; c_{\fp,\fq})
\end{align*}
where the second equality uses Corollary \ref{c:countU} (1) and (3).
\end{proof}

Similarly, we compare the differentials of Type (A2) and (A2').

\begin{lemma}\label{l:A2}
    For $\tau \gg 1$, the algebraic count of following moduli spaces are equal
    \begin{itemize}
  \item $\eM^{J^{\tau}}(c_{\fp,\fq'};c_{\fp,\fq})$,
  \item $\eM^{J^\tau}(\mathbf{q}'^\vee;\mathbf{q}^\vee)$.
\end{itemize}
\end{lemma}

\begin{proof}
The proof is almost word-by-word taken from Lemma \ref{l:A1}.  Lemma \ref{l:stretchingStripA}, \ref{l:stretchingStripA'} and Corollary \ref{c:countU} (1) and (2) implies

\begin{align*}
  &\#\eM^{J^\tau}(\textbf{q}'^\vee;\textbf{q}^\vee)\\
  =&\#\eM^{J^+}(x_{\fq,\fq'})\cdot\#\eM^{J^-}(q'^\vee;x_{\fq,\fq'},q^\vee)\\
  =&\#\eM^{J^+}(x_{\fq,\fq'})\cdot\#\eM^{J^-}(c_{\fp,\fq'};x_{\fq,\fq'}, c_{\fp,\fq})\\
  =&\#\eM^{J^\tau}(c_{\fp,\fq'}; c_{\fp,\fq})
\end{align*}

\end{proof}

% \subsubsection{The connecting map $hom(\eP, L_1) \otimes_\Gamma hom(L_0,\eP)\to hom(L_0,\tau_P L_1)$}

The last lemma addresses differentials of Type (C) and (C').

\begin{lemma}\label{l:C}
For $\tau \gg 1$, the algebraic count of following moduli spaces are equal
    \begin{itemize}
  \item $\eM^{J^\tau}(x; \mathbf{q}^\vee,\mathbf{p})$, for some $x\in CF^*(L_0,L_1)$ represented by an intersection outside $U$,
  \item $\eM^{J^\tau}(x; c_{\fp,\fq})$.
\end{itemize}
\end{lemma}

\begin{proof}
  The strategy is still similar.  Apply the same neck-stretching as in Lemma \ref{l:A1} and \ref{l:A2}, one obtains a building consisting of a 
  triangle and a bigon for $\eM^{J^\tau}(x; \mathbf{q}^\vee, \mathbf{p})$, thanks to Lemma \ref{l:triangle}; and a building consisting of 
  two bigons for $\eM^{J^\tau}(x; c_{\fp,\fq})$ from Lemma \ref{l:bigon}.  
Therefore
\begin{align*}
  &\#\eM^{J^\tau}(x; \mathbf{q}^\vee, \mathbf{p})\\
  =&\#\eM^{J^+}(x,x_{\fq,\fp};\emptyset) \cdot \#\eM^{J^-}(\emptyset;  p, x_{\fq,\fp},q^\vee)\\
  =&\#\eM^{J^+}(x,x_{\fq,\fp};\emptyset) \cdot \#\eM^{J^-}(\emptyset; c_{\fp,\fq}, x_{\fq,\fp})\\
  =&\#\eM^{J^\tau}(x; c_{\fp,\fq})
\end{align*}
where the second equality uses Corollary \ref{c:countU} (1) and (4).
\end{proof}

As the end product of this section, we have

\begin{proof}[Proof of Proposition \ref{p:CohLevelIso} when $\cchar(\K)=2$]
 For $\tau \gg 1$, the differential on $C_0$ and $C_1$ can be identified by Lemma \ref{l:A1}, \ref{l:A2}, \ref{l:stretchingStripB}, \ref{l:C} and \ref{l:A3Dtype}.
\end{proof}

The proof of Proposition \ref{p:CohLevelIso} when $\cchar(\K) \neq 2$ is given in Appendix \ref{sec:orientations}.

\section{Categorical level identification}\label{s:categorical}

In this section, we want to prove Theorem \ref{t:Twist formula} by showing the following:

\begin{thm}\label{t:reformulation}
 For any object $\eE^1 \in Ob(\cF)$, we can perform a Hamiltonian perturbation for $\eE^1$ to obtain another object $(\eE^1)'$ of $\cF$ such that
 there is a degree zero cochain $c_\eD \in hom^0_{\cF^{\perf}}(\tau_P((\eE^1)'),T_\eP(\eE^1))$ so that $c_\eD$ is a cocycle, and
 \begin{align}
  \mu^2(c_\eD, \cdot):hom^0_{\cF^{\perf}}(\eE^0,\tau_P((\eE^1)')) \to hom^0_{\cF^{\perf}}(\eE^0,T_\eP(\eE^1)) \label{eq:QIsoms}
 \end{align}
is a quasi-isomorphism for all $\eE^0 \in Ob(\cF)$

In particular, $\tau_P(\eE^1) \simeq \tau_P((\eE^1)') \simeq T_\eP(\eE^1)$ as perfect $A_{\infty}$ right $\cF$-modules.
\end{thm}

The overall strategy goes as follows.
By Proposition \ref{p:CohLevelIso}, Corollary \ref{c:ConeIndepAuxData} and the fact that Hamiltonian isotopic objects are quasi-isomorphic,
we know that
\begin{align}
 H(hom_{\cF^{\perf}}(\tau_P((\eE^1)'),T_\eP(\eE^1)))=HF(\tau_P(\eE^1),\tau_P(\eE^1)) \label{eq:CohIsom}
\end{align}
In particular, there are non-exact degree zero cocycles in $hom_{\cF^{\perf}}(\tau_P((\eE^1)'),T_\eP(\eE^1))$.
We will pick an appropriate one which is denoted by $c_\eD$.
By a Hamiltonian perturbation if necessary, to check that $c_\eD$ satisfies \eqref{eq:QIsoms},
it suffices to check it for those $\eE^0$ such that $L_0$ intersects $L_1$, $(L_1)'$ and $P$ transversally.
Therefore, we can apply neck-stretching along $\partial U$ to compute $\mu^2(c_\eD, \cdot)$ for $\tau \gg 1$ (see Section \ref{sub:gluing_in_sft}).

%Let $U$ be a Weinstein neighborhood of $P$.
%Every Lagrangian  $L$ is assumed to be cylindrical near $\partial U$.
%We use $L^{in}:=L \cap Int(U)$ and $L^{out}:=L \backslash U$.
%For any two Lagrangians $L_0$, $L_1$ below, we will apply Corollary \ref{c:SFT_compute} and compute $CF(L_0,L_1)$ at SFT limit.

\subsection{Hunting for degree zero cocycles}

To find a degree zero cocycle, we need to first analyze the differential of 
$hom_{\cF^{\perf}}(\tau_P((\eE^1)'),T_\eP(\eE^1))$ by neck-stretching.

Let $L_1'$ be a $C^2$-small Hamiltonian push-off of $L_1$ such that $L_1' \cap U$ is a union of cotangent fibers.
Let $q_1,\dots,q_{d_{L_1}} \in CF(L_1,P)$ and $q_1',\dots,q_{d_{L_1}}' \in CF(L_1',P)$
be the cochain representatives of the geometric intersection points, where $d_{L_1}=\# (P \cap L_1)=\#(P \cap L_1')$.  We also number the
intersection points so that $d_P(q_i,q_i')\ll\epsilon$ in the standard quotient round metric.
Let $\Lambda_{q_i},\Lambda_{q_j'} \subset \partial U$ be the cospheres at $q_i$ and $q_j'$, respectively.
We assume $q_i,q_j'$ satisfy Corollary \ref{c:ReebDynamicMorsified}.
Let $\fq_i,\fq_j'$ be a lift of $q_i,q_j'$, respectively, for all $i,j$.
Our focus will be the cochain complex
\begin{align}
\eD:=hom_{\cF^{\perf}}(\tau_P((\eE^1)'),T_\eP(\eE^1))=(CF(\eP,\eE^1)\otimes_{\Gamma} CF(\tau_P((\eE^1)'),\eP))[1] \oplus CF(\tau((\eE^1)'),\eE^1) \label{eq:eD}
\end{align}
which is generated by elements supported at the intersection points
\begin{align}
 \left\{
 \begin{array}{ll}
 q_i^\vee  \otimes \tau_P(q_j'), &\text{ for }   i,j=1,\dots,d_{L_1} \\
c_{i,g,j}^\vee:=c_{\fq_i,g\fq_j'}^\vee, &\text{ for }g \in \Gamma, i,j=1,\dots,d_{L_1}\\
w_k, &\text{ for }k=1,\dots,\# (L_1' \cap L_1)
 \end{array}
\right.
\end{align}
The first two kinds of intersection points are inside $U$ while $\{w_k\}$ are outside $U$.
Elements supported at $c_{i,g,j}^\vee$ and $w_k$ are given by 
\begin{align}
 Hom_\K(\tau_P((\eE^1)')_{c_{\fq_i,g\fq_j'}},\eE^1_{c_{\fq_i,g\fq_j'}}) \text{   and   } Hom_\K(\tau_P((\eE^1)')_{w_k},\eE^1_{w_k})
\end{align}
respectively.
On the other hand, the elements supported at $q_i^\vee  \otimes \tau_P(q_j')$ are generated by
\begin{align}\label{eq:EleSupportedAtTensor}
 (\psi^2 \otimes \fq_i^\vee)  \otimes (g\tau_\fP(\fq_j') \otimes \psi^1), &\text{ for } \psi^2 \in \eE^1_{q_i}, \psi^1 \in Hom_\K(\tau_P((\eE^1)')_{\tau_P(q_j')},\K),  g \in \Gamma
\end{align}
Here we use the commutativity $\pi(\tau_\fP(\fq_j'))=\tau_P(\pi(\fq_j'))=\tau_P(q_j')$.

\begin{lemma}
 With respect to canonical relative grading, we have
 \begin{align}
\left\{
\begin{array}{ll}
 |q_i^\vee|=0, &\text{for }q_i^\vee \in hom(P,T^*_{q_i} P) \\
 |\tau_P(q_j')|=1, &\text{for }\tau_P(q_j') \in hom(\tau_P(T^*_{q_j'} P),P) \\
 |c_{i,g,j}^\vee|=1, &\text{for }c_{i,g,j}^\vee= \pi(\tau_\fP(T^*_{g\fq_j'} \fP) \cap T^*_{\fq_i} \fP) \in hom(\tau_P(T^*_{q_j'} P),T^*_{q_i} P)
\end{array}
\right.
\end{align}
\end{lemma}

\begin{proof}
 The fact that $|q_i^\vee|=0$ follows from the definition of canonical relative grading \eqref{eq:CanonicalGrading}.
 $|c_{i,g,j}^\vee|=1$ follows from $|c_{i,g,j}|=n-1$ (see Lemma \ref{l:A3Dtype}).
 Finally, from the long exact sequence
\begin{align}
   HF^k(P, T^*_{q_j'} P) \to HF^k(P, \tau_P(T^*_{q_j'} P)) \to HF^{k+1}(P, P) \to HF^{k+1}(P, T^*_{q_j'} P) 
\end{align}
and the fact that $HF(P,\tau_P(T^*_{q_j'} P)) $ has rank $1$,
we know that $HF^{0}(P, P) \simeq HF^{0}(P, T^*_{q_j'} P)$ and 
$HF^k(P, \tau_P(T^*_{q_j'} P)) \to HF^{k+1}(P, P)$ is an isomorphism when $k=n-1$.
Therefore, $|\tau_P(q_j')^\vee|=n-1$ and $|\tau_P(q_j')|=n-|\tau_P(q_j')^\vee|=1$.
\end{proof}

Without loss of generality, we assume that there is a unique $w_k$ with degree $0$ and we denote it by $e_L$.
All other $w_k$ has $|w_k|>0$. With generators understood, we now recall that the differential for element $\psi_x$ supported at $x=c_{i,g,j}^\vee$ or $x=w_k$ is given by
$\mu^1(\psi_x)=\mu^1_\cF(\psi_x)$, and for element  supported at $q_i^\vee  \otimes \tau_P(q_j')$ is given by (see \eqref{eq:mappingCone})
\begin{align}
 \mu^1(\psi^2 \otimes \fq_i^\vee  \otimes g\tau_\fP(\fq_j') \otimes \psi^1)
 =& (-1)^{|g\tau_\fP(\fq_j')|}\mu^1_\cF(\psi^2 \otimes \fq_i^\vee) \otimes (g\tau_\fP(\fq_j') \otimes \psi^1)+(\psi^2 \otimes \fq_i^\vee) \otimes \mu^1_\cF(g\tau_\fP(\fq_j') \otimes \psi^1) \nonumber\\
 &+\mu^2_\cF(\psi^2 \otimes \fq_i^\vee,g\tau_\fP(\fq_j') \otimes \psi^1) \label{eq:differentials}
\end{align}
Our focus will be put on
 $\mu^1_\cF(\psi_{e_L})$ and  $\mu^2_\cF(\psi^2 \otimes \fq_i^\vee,g\tau_\fP(\fq_j') \otimes \psi^1)$.

\subsubsection{Computing $\mu^1_\cF(\psi_{e_L})$}
%\subsubsection{ $\mu^1_\cF(\psi_{e_L})$}

%We are not going to use Hamiltonian perturbation for any CR equation defining operations in this section.

Let $h:L_1 \to \R$ be a smooth function such that $dh=\theta|_{L_1}$.
We define $h_i:=h|_{\Lambda_{q_i}}$ which are constants because $L_1$ is cylindrical near $\Lambda_{q_i}$.
Hamiltonian push-off induces $h':L_1' \to \R$ such that $dh'=\theta|_{L_1'}$ and $h_i':=h'|_{\Lambda_{q_i'}}$ are constants.
%We reserve the right to assume $h_i'$ is as close to $h_i$ as we want.
By possibly reordering the index set of $i$, we assume that $h_1 \le  h_2 \le \dots \le h_d$.
For each $i$, by relabelling if necessary, we also assume that $q_i'$ is the closest to $q_i$ among points in $\{q_j'\}_{j=1}^{d_{L_1}}$, and 
$\fq_i'$ is the closest to $\fq_i$ among points in $\{g\fq_i'\}_{g \in \Gamma}$.
%The action of $p_ig \in hom(F_i,\eP)$ and
%$p_i'g \in hom(F_i',\eP)$ are $h_i$ and $h_i'$, respectively, for all $i$ and $g \in \Gamma$.

We recall from \eqref{eq:ActionChordDefn} that the action of a Reeb chord $x$ from $\Lambda_{q_j'}$ to $\Lambda_{q_i}$ is given by
\begin{align}
A(x):= L(x) +h_j' -h_i 
\end{align}

%Since we do not add Hamiltonian perturbation term, the energy of a pseudo-holomorphic curve $u$ is given by
%$$E(u):=\sum_{x \in \cX^{out}} \cA(x)- \sum_{x \in \cX^{in}} \cA(x)$$
%and differential decreases action.
%Here, $\cX^{out}$ and $\cX^{in}$ are the set of output and input of $u$, respectively.

\begin{lemma}\label{l:action_chord}
 There is a constant $\epsilon>0$ depending only on $\{q_i\}_{i=1}^{d_{L_1}}$ and $L_1$ such that when $L_1'$ is a sufficiently small Hamiltonian push-off of $L_1$,
 \begin{itemize}
 \item $A(x)> \epsilon$ if $x$  is a Reed chord from $\Lambda_{q_j'}$ to $\Lambda_{q_i}$ and $j>i$, and
 \item $A(x)> \epsilon$ if $x$  is a Reed chord from $\Lambda_{q_i'}$ to $\Lambda_{q_i}$ but not the shortest one.
 \end{itemize}
\end{lemma}

\begin{proof}
% Reeb chords from $F_j$ to $F_i$ in bijection to geodesic from $p_j$ to $p_i$ on $L_1$.
 There is a constant $\epsilon>0$ depending only on $\{q_i\}_{i=1}^{d_{L_1}}$ and $L_1$ such that $L(x) > 3\epsilon$
 if $x$ is either a Reeb chord from $\Lambda_{q_j}$ to $\Lambda_{q_i}$ and $i \neq j$, or it is a {\bf non-constant} Reeb chord from $\Lambda_{q_i}$ to itself.
 We can choose a small Hamiltonian perturbation such that $L(x) > 2\epsilon$
 if either $x$ is a Reeb chord from $\Lambda_{q_j'}$ to $\Lambda_{q_i}$, or a non-shortest  Reeb chord from $\Lambda_{q_i'}$ to $\Lambda_{q_i}$.
 If $j \ge i$, we have $h_j \ge h_i$ so we can assume the Hamiltonian chosen is small enough such that $h_j' -h_i > -\epsilon$ and therefore $A(x)=L(x) +h_j' -h_i > \epsilon$ in both cases listed in the lemma.
\end{proof}

For each $i$, we denote the shortest Reeb chord from $\Lambda_{q_i'}$ to $\Lambda_{q_i}$ by $x_{i',i}$.
With respect to canonical relative grading, we have $|x_{i',i}|=0$.
Since $\fq_i'$ is the closest to $\fq_i$ among points in $\{g\fq_i'\}_{g \in \Gamma}$, if we lift the Reeb chord $x_{i',i}$ to a Reeb chord starting from 
$\Lambda_{\fq_i'}$, then it ends on $\Lambda_{\fq_i}$.

The following lemmata (\ref{l:DiffIdentify}, \ref{l:actionFiltration} and \ref{l:IdentityTriangle}) concern some moduli of rigid bigon with input being $e_L$.
We start with the case when the output lies outside $U$.

\begin{lemma}\label{l:DiffIdentify}
For $\tau \gg 1$, rigid elements in $\eM^{J^{\tau}}(w_k;e_L)$ with respect to boundary conditions $(\tau_P(L_1'),L_1)$ and $(L_1',L_1)$ (i.e. they
contirbute to the differential in $CF(\tau_P(L_1'),L_1)$ and $CF(L_1',L_1)$), respectively, can be 
canonically identified.
\end{lemma}

\begin{proof}
By the same 
reasoning as in Lemma \ref{l:stretchingStripB}, as $\tau$ goes to infinity, 
the holomorphic building $u_{\infty}=(u_v)_{v \in V(\eT)}$ consists of exactly one vertex $v$ and $u_v$ maps to $SM^+$.
The result follows.
\end{proof}

In Lemma \ref{l:action_chord}, the $\epsilon$ is independent of perturbation.
 Therefore, we can choose a perturbation such that the action of $e_L$ in $hom(L_1',L_1)$ (and hence in $hom(\tau(L_1'),L_1)$) is less than $\epsilon$.
%We can choose the Hamiltonian perturbation to be so small so that $A(e_L) < \epsilon$. 
In this case, we have

\begin{lemma}\label{l:actionFiltration}
Let $\epsilon$ satisfy Lemma \ref{l:action_chord}.
If $A(e_L) < \epsilon$, then for all $j>i$ and $g \in \Gamma$ (or $j=i$ and $g \neq 1_\Gamma$), there is no rigid element in $\eM^{J^{\tau}}(c_{i,g,j}^\vee;e_L)$ for $\tau \gg 1$.
\end{lemma}

\begin{proof}
 Suppose not, then we will have a holomorphic building $u_{\infty}=(u_v)_{v \in V(\eT)}$ as $\tau$ goes to infinity.
 Let $u_{v_1}$ be the $J^-$-holomorphic curve such that $c_{i,g,j}^\vee$ is an asymptote of $u_{v_1}$.
 One can argue as in Lemma \ref{l:bigon} to show that $u_{v_1}$ has exactly one positive Reeb chord asymptote $x$.
 Moreover, $x$ can be lifted to a Reeb chord from $\Lambda_{g\fq_j'}$ to $\Lambda_{\fq_i}$ by boundary condition.
 When $j>i$ and $g \in \Gamma$ (or $j=i$ and $g \neq 1_\Gamma$), we have $A(x)> \epsilon$ by Lemma \ref{l:action_chord}.
 Since $A(e_L) < \epsilon$ by assumption, we get a contradiction by Lemma \ref{l:actionBoundGlobal}.
\end{proof}

\begin{lemma}\label{l:IdentityTriangle}
 For $L_1'$ sufficently close to $L_1$ and $\tau \gg 1$, the algebraic count of rigid elements in $\eM^{J^{\tau}}(c_{i,1_\Gamma,i}^\vee;e_L)$ is $\pm 1$.
\end{lemma}

\begin{proof}
Similar as before, every limiting holomorphic building $u_{\infty}=(u_v)_{v \in V(\eT)}$ consists of two vertices (see Lemma \ref{l:bigon}).
By boundary condition, the bottom level curve $u_{v_1}$ lies in  $\eM^{J^-}(c_{i,1_\Gamma,i}^\vee; x_{i',i})$, which has algebaic count $\pm 1$ by Corollary \ref{c:countU}(4).
Therefore, it suffices to determine the algebraic count of $\eM^{J^+}(x_{i',i};e_L)$.

We consider the rigid elements in the moduli $\eM(q_i^\vee;e_L,(q_i')^\vee)$ for a compactible almost complex structure $J$, which is reponsible to the $q_i^\vee$-coefficient of $\mu^2(e_L,(q_i')^\vee)$
for the operation $\mu^2(\cdot,\cdot):hom(L_1',L_1) \times hom(P,L_1') \to hom(P,L_1)$.
Therefore, it has algebraic count $\pm 1$
with respect to $J$ when $L_1'$ is $C^2$-close to $L_1$.
%\begin{align}
%\mu^2(e_L,\cdot):hom(L_1',L_1) \times hom(P,L_1') \to hom(P,L_1) 
%\end{align}
%which gives $\mu^2(e_L,(q_i')^\vee)=q_i^\vee$ with respect to some compactible almost complex structure $J$ when  $L_1'$ sufficently close to $L_1$.

Next, we will use a cascade (homotopy) type argument which goes back to Floer and argue that the algebraic count of $\eM^{J^\tau}(q_i^\vee;e_L,(q_i')^\vee)$ is
$\pm 1$ for all $\tau$.
A detailed account for a cascade (homotopy) type argument involving higher multiplications can be found in, for example, \cite{AbouzaidSeidel} (see also \cite[Section $10e$]{Seidelbook}, \cite{Auroux10}).  

Let us recall the overall strategy of the cascade argument tailored for our situation.
Pick a path of compactible almost complex structures $(J_t)_{t \in [0,1]}$ from $J$ to $J^\tau$ for some finite time $\tau$.
For a generic path of almost complex structure $(J_t)_{t \in [0,1]}$, there are finitely many $0<t_1<\dots<t_k <1$ such that
there exists $J_{t_l}$ stable maps with input $e_L,(q_i')^\vee$, output $q_i^\vee$ and consisting of more than one component.
When $(J_t)_{t \in [0,1]}$ is generic, these stable maps consist of exactly
%the moduli space $\eM^{J_{t_l}}(q_i^\vee;e_L,(q_i')^\vee)$ consists of both smooth holomorphic triangles and stable maps with 
two components, which are a $J_{t_l}$-holomorphic triangle and bigon, respectively.  
Moreover, one of the components must be of virtual dimension $0$, and the other one is of dimension $-1$.  
In this case, we say \textit{a bifurcation occurs at} $t_l$, and denote the component of virtual dimension $-1$ as $u$.

If a bifurcation occurs at $t_l$, for example, then $\eM^{J_{t}}(q_i^\vee;e_L,(q_i')^\vee)$
has the same diffeomorphism type when $t\in (t-\epsilon,t_l)$ for some small $\epsilon>0$.  
The \textit{birth-death} bifurcation gives a cobordism between $J_{t_l-\epsilon}$ and $J_{t_l}$-stable triangles; and the \textit{death-birth} bifurcation 
gives a cobordism between $J_{t_l-\epsilon}$ triangles and $J_{t_l}$ 
stable triangles with  exactly one component. 
When $t$ approaches $t_l$ from the right, we get the cooresponding cobordisms.
%Although the standard gluing theory does not apply to bifurcation, the degeneration for both $t\to t_l^+$ and $t\to t_l^-$ follows Gromov compactness. 
The change of algebraic count from 
$\eM^{J_{t_l-\epsilon}}(q_i^\vee;e_L,(q_i')^\vee)$ to 
$\eM^{J_{t_l+\epsilon}}(q_i^\vee;e_L,(q_i')^\vee)$ is called the contribution to $\eM^{J_{t}}(q_i^\vee;e_L,(q_i')^\vee)$
by the bifuration at time $t_l$.

 Therefore, to show that the algebraic count persists to be $\pm1$, we will analyze each bifurcation moment $t_l$ below and 
 prove the contribution to $\eM^{J_{t}}(q_i^\vee;e_L,(q_i')^\vee)$ is zero.  For simplicity we let $l=1$.   
 Since there are exactly two irreducible components at $t=t_1$, one of them has to has virtual dimension 0 and the other one
 has dimension $-1$.  Let $u$ denote the component of virtual dimension $-1$, and we divide the possible stable $J_{t_1}$-holomorphic triangles into three cases:

 \begin{enumerate}[(i)]
   \item both $q_i^\vee$ and $(q_i')^\vee$ are asymptotes of $u$;
   \item exactly one of $q_i^\vee$ and $(q_i')^\vee$ is an asymptote of $u$;
   \item none of $q_i^\vee$ and $(q_i')^\vee$ are asymptotes of $u$.
 \end{enumerate}

 % Part of the triangles degenerates to a stable map described above   And when $t'\in (t_1,t_1+\epsilon)$, $\eM^{J_{t'}}(q_i^\vee;e_L,(q_i')^\vee)$ is diffeomorphic to $\eM^{J_{t}}(q_i^\vee;e_L,(q_i')^\vee)$, with extra contributions coming from the gluing of the two components of the stable map at $t=t_1$.  

% bifurcations occur.
% In particular, when the first bifurcation occurs (at $t=t_1$), there is a $J^{t_1}$-holomorphic curve $u$ with virtual dimension $-1$
% and appropriate Lagrangian boundary condition
% such that there is exactly one of the asymptote of $u$ is not $q_i^\vee$, $e_L$ and $(q_i')^\vee$.

\noindent\textit{Case (i):} If both $q_i^\vee$ and $(q_i')^\vee$ are asymptotes of $u$, then the last asymptote $x$ of $u$ must be a generator of $CF(L_1', L_1)$ by boundary condition.
Moreover, $x$ is a degree $1$ element of $CF(L_1', L_1)$ because $\virdim(u)=-1$ and $|e_L|=0$.
This bifurcation contributes to a change in the algebraic count of $\eM^{J_t}(q_i^\vee;e_L,(q_i')^\vee)$
by the algebraic count of rigid elmements from $\eM^{J_{t_1}}(x;e_L)$ (when $t>t_1$, 
the moduli $\eM^{J_{t_1}}(x;e_L)$ and $\eM^{J_{t_1}}(q_i^\vee;x,(q_i')^\vee)$ glue together to give a change).
However, the algebraic count of rigid elmements from $\eM^{J_{t_1}}(x;e_L)$ is zero because $e_L$ is a cocycle.\\

\noindent\textit{Case (ii):} If exactly one of $q_i^\vee$ and $(q_i')^\vee$ is an asymptote of $u$, then $P$ is a Lagrangian boundary condition of one of 
the component of $\partial \Sigma_u$, where $\Sigma_u$ is the domain of $u$.
By this boundary component, there is another point $q_j$ or $q_j'$ for some $j \neq i$ which is an asymptote of $u$.
Since there is a lower bound between the distance from $q_i$ (or $q_i'$) to $q_j$ (or $q_j'$) for $j \neq i$,
we can apply monotonicity Lemma at an appropriate point in $Im(u) \cap P$ to get a constant $\delta >0$ depending only on $\{q_i\}_{i=1}^{d_{L_1}}$ but not $L_1'$ such that
the energy $E_\omega(u) > \delta$.
If we chose $L_1'$ to be sufficently close to $L_1$ such that $A(e_L)+A((q_i')^\vee)-A(q_i^\vee)< \delta$, then for $u$ to contribute
to a change of algebraic count of $\eM^{J_t}(q_i^\vee;e_L,(q_i')^\vee)$, $u$ has to be glued with a rigid $J_{t_1}$-holomorphic
curve of negative energy, which does not exist.\\

\noindent\textit{Case (iii):} If none of $q_i^\vee$ and $(q_i')^\vee$ are asymptotes of $u$, then $u$ is a bigon with one asymptote being $e_L$ and the other 
asymptote, denoted by $x$, being a generator of $CF(L_1', L_1)$.
Moreover, $|x|=0$ because $\virdim(u)=-1$.
It is a contradiction because $e_L$ is the only generator of $CF(L_1', L_1)$ with degree $0$ and constant maps have virtual dimension $0$.

As a result, no bifurcation can possibly contribute to a change to the algebraic count and $\#\eM^{J^{\tau}}(q_i^\vee;e_L,(q_i')^\vee)=\pm1$ for all $\tau$.
By letting $\tau$ go to infinity, the argument in Lemma \ref{l:triangle} implies that 
the limiting holomorphic building $u_{\infty}=(u_v)_{v \in V(\eT)}$ consist of two vertices.
Moreover, we have $u_{v_1} \in \eM^{J^-}(q_i^\vee;x_{i',i},(q_i')^\vee)$ and $u_{v_2} \in \eM^{J^+}(x_{i',i};e_L)$.
It implies that the algebraic count of rigid element in $\eM^{J^+}(x_{i',i};e_L)$ is $\pm 1$.
The proof finishes.

\end{proof}

\begin{rmk}\label{rem:possibleBifurcation}
    The proof above shows that the algebraic count of $\eM^{J^+}(x_{i',i};e_L)$ is $\pm 1$, which will be used in Proposition \ref{p:matrixA} again.
    %It is useful to observe the only possible bifurcation occurs in case (i) above.  When it happens, $u$ is a $J_{t_1}$-holomorphic triangle of virtual dimension $-1$, of which both $q_i^\vee$ and $(q_i')^\vee$ are asymptotes.  The other stable component is a differential from $e_L$ to some Lagrangian intersection $x$.
\end{rmk}

Using the Hamiltonian push-off, we have the identifications
\begin{align}
 \tau_P((\eE^1)')_{w_k} \simeq (\eE^1)'_{w_k} \simeq \eE^1_{w_k}, \text{ and } \tau_P((\eE^1)')_{c_{\fq_i,g\fq_j'}} \simeq \eE^1_{c_{\fq_i,g\fq_j'}}
\end{align}
for all $w_k$ and $c_{\fq_i,g\fq_j}$. In particular, we can define $t_\eD$ to be the identity morphism at $e_L$:
\begin{align}
 t_\eD:=id \in Hom_\K(\tau_P((\eE^1)')_{e_L}, \eE^1_{e_L}) \subset hom_{\cF^{\perf}}(\tau_P((\eE^1)'),T_\eP(\eE^1))=\eD \label{eq:teD}
\end{align}
Alternatively, we can tautologically view $t_\eD$ as an element in $Hom_\K((\eE^1)'_{e_L}, \eE^1_{e_L})$
and it is the cohomological unit of $CF((\eE^1)',\eE^1)$ (see Remark \ref{r:CohomologicalUnit}).
We denote $t_\eD$ by $e_\cE$ when we view it as an element in $CF((\eE^1)',\eE^1)$.

Let us take local systems on the Lagrangians into account.  Since $\pi_1(U \cap L_1)=1$, we can identify stalks of the local 
system $\eE^1_p$ over each $p\in U\cap L_1$ using the flat connection 
(equivalently, assume the connection is trivial in $U \cap L_1$).  
Similary, identify all $(\eE^1)'_{p'}$ for $p' \in U \cap L_1'$.  
This also induces an identification of stalks on $\tau_P(T^*_qP)$, since local systems therein are pushforwards of the ones over a fiber.

We can now summarize the previous lemmata.

\begin{prop}\label{p:matrixA}
 For $L_1'$ sufficienly close to $L_1$ and $\tau \gg 1$, we have
 \begin{align}
  \mu^1(t_\eD)=\sum_{i,j,g} \psi_{c_{i,g,j}^\vee}
 \end{align}
 where $\psi_{c_{i,g,j}^\vee} \in Hom_\K(\tau_P((\eE^1)')_{c_{\fq_i,g\fq_j'}},\eE^1_{c_{\fq_i,g\fq_j'}})$ and
 \begin{align}
  \left\{
  \begin{array}{ll}
   \psi_{c_{i,g,j}^\vee}=0 \text{ if }j>i \text{ and } g \in \Gamma \text{ (or }j=i \text{ and } g \neq 1_\Gamma\text{)} \\
   \psi_{c_{i,1_\Gamma,i}^\vee}=\pm id \in Hom_\K(\tau_P((\eE^1)')_{c_{\fq_i,\fq_i'}},\eE^1_{c_{\fq_i,\fq_i'}})
  \end{array}
\right.
 \end{align}

\end{prop}

\begin{proof}
 By Lemma \ref{l:DiffIdentify} and the fact that $e_\eE$ is a cocyle in $CF((\eE^1)', \eE^1)$, we know that $  \mu^1(t_\eD)=\sum_{i,j,g} \psi_{c_{i,g,j}^\vee}$.
 The fact that 
 $\psi_{c_{i,g,j}^\vee}=0$ if $j>i$ and $g \in \Gamma$ (or $j=i$ and $ g \neq 1_\Gamma$) follows from Lemma \ref{l:actionFiltration}.
 Finally, to see that $\psi_{c_{i,1_\Gamma,i}^\vee}=id$ we need to understand the moduli 
 $\eM^{J^{\tau}}(c_{i,1_\Gamma,i}^\vee;e_L)$ and the parallel transport maps given by the rigid elements in it.

 % As explained in Remark \ref{rem:possibleBifurcation}, when bifurcation occurs (say, a death-birth at $t=t_1$),
 % the moduli $\eM^{J_{t_1}}(x;e_L)$ and $\eM^{J_{t_1}}(q_i^\vee;x,(q_i')^\vee)$ glue together to give extra rigid elements in 
 % $\eM^{J_t}(q_i^\vee;e_L,(q_i')^\vee)$ for $t>t_1$.

 Consider the holomorphic building when $\tau=\infty$, we have two components
 $u_1 \in \eM^{J^-}(q_i^\vee;x_{i',i},(q_i')^\vee)$ and $u_2 \in \eM^{J^+}(x_{i',i};e_L)$ by Lemma \ref{l:triangle} and Remark \ref{rem:dualizeC}.  
 When $L_1'$ is sufficiently $C^2$-close to $L_1$, the action of $u_1,u_2$ can be as small as we want. 
 It implies that, by monotonicity lemma, $u_2$ lies in a Weinstein neigborhood of $L_1$.
 %the moduli $\eM^{J_{\tau}}(x;e_L)$ consists of differentials from $e_L$ to $x$ which are contained in a fixed Weinstein neigborhood of $L_1$ by monotonicity lemma for all finite $\tau$.  
 %Therefore, $u_1$ and $u_2$ have the same property (where the SFT limit is regarded as a degeneration of the almost complex structure).

 It in turn implies that, for each strip $u_2$ in the limit, the associated output is 
 $\psi_{i',i}=\pm id$ when the input at $e_L$ is $t_\eD$ (the sign of $\psi_{i',i}$ depends on the sign of $u_2$).  
 This is because we have identified the stalks of $\eE^1$ and $(\eE^1)'$ at the point $e_L$, and  
 the associated parallel transports $I_{\partial_0 u}$ and $I_{\partial_1 u}$ on their respective boundary conditions are inverse to 
 each other (in fact, the strip itself provides an isotopy after projecting to $L_1$ in the Weinstein neighborhood).   
 Since we have proved that the algebraic count of 
 $\eM^{J^+}(x_{i',i};e_L)$ is $\pm 1$ (see Remark \ref{rem:possibleBifurcation}), the associated
 output by all elements in $\eM^{J^+}(x_{i',i};e_L)$ is $\pm id$, when the input at $e_L$ is $t_\eD$.
 
 To get the proposition, we now replace $u_1$ by $u_1'\in\eM^{J^{\tau}}(c_{i,1_\Gamma,i}^\vee;x_{i',i})$.  
 As explained earlier, we have identified the fibers of the local systems of $\eE^1$ and $\tau_P(\eE^1)'$ at $c_{\fq_i,\fq_i'}$.
 Since the parallel transports of  $\eE^1$ and $\tau_P(\eE^1)'$ inside $U$ are trivial,
 if the input at $x_{i',i}$ is $\pm id$, so is the output.  
 By Lemma \ref{l:IdentityTriangle}, the algebraic count of $\eM^{J^{\tau}}(c_{i,1_\Gamma,i}^\vee;x_{i',i})$ is $\pm1$ 
 and each strip contributes $\pm id$ (and the sign of $\pm id$ only depends on the sign of the strip), 
 therefore,  the total countribution is $\pm id$, as desired.

 \end{proof}

\begin{rmk}
 In summary, when $L_1'$ is sufficently close to $L_1$,
 $e_L$ being a cohomological unit is responsible for the algebraic count of $\eM^{J^\tau}(q_i^\vee;e_L,(q_i')^\vee)$ being $\pm 1$ and hence the $q_i^\vee$-coefficient of
 $\mu^2(e_L,(q_i')^\vee)$ being $\pm 1$.
 On the other hand, $e_\eE$ being a cohomological unit is reponsible for the $\fq_i^\vee$-coefficient of $\mu^2(e_\eE,(\fq_i')^\vee)$ being $1$.
 Lemma \ref{l:IdentityTriangle} and Proposition \ref{p:matrixA} are obtained by replacing the bottom level curves at the SFT limit. 
\end{rmk}

\subsubsection{Computing $\mu^2_\cF(\psi^2 \otimes \fq_i^\vee,g\tau_\fP(\fq_j') \otimes \psi^1)$}

Next, we want to study $\mu^1((\psi^2 \otimes \fq_i^\vee)  \otimes (g\tau_\fP(\fq_j') \otimes \psi^1))$ (see \eqref{eq:EleSupportedAtTensor}, \eqref{eq:differentials}). 
In particular, we want to focus on the term $\mu^2_\cF(\psi^2 \otimes \fq_i^\vee,g\tau_\fP(\fq_j') \otimes \psi^1)$ so
we need to discuss the moduli $\eM(c_{i,g,j}^\vee;q_i^\vee,\tau_P(q_j'))$ and $\eM(w_k;q_i^\vee,\tau_P(q_j'))$.

\begin{lemma}\label{l:noOutside}
 For $\tau \gg1$, there is no rigid element in $\eM^{J^\tau}(w_k;q_i^\vee,\tau_P(q_j'))$.
\end{lemma}

\begin{proof}
 We argue by contradiction as before. Let $u_{\infty}=(u_v)_{v \in V(\eT)}$ be a limiting holomorphic building.
 By boundary condition, there is $v_1 \in V(\eT)$ such that $q_i^\vee, \tau_P(q_j')$ are asymptotes of $u_{v_1}$.
 The other asymptotes of $u_{v_1}$ are positive Reeb chords $y_1,\dots,y_k$.
 The virtual dimension of $u_{v_1}$  can be computed using canonical relative grading, and is given by
 \begin{align*}
 \virdim(u_{v_1})=n(1-0)-|q_i^\vee|-|\tau_P(q_j')|-\sum_{s=1}^k |y_s|-(1-k) \ge n-0-1-(1-k)>0  
 \end{align*}
 because $n \ge 3$. It contradicts to $\virdim(u_{v_1})=0$.
\end{proof}

\begin{figure}[tb]
  \centering
  \includegraphics[]{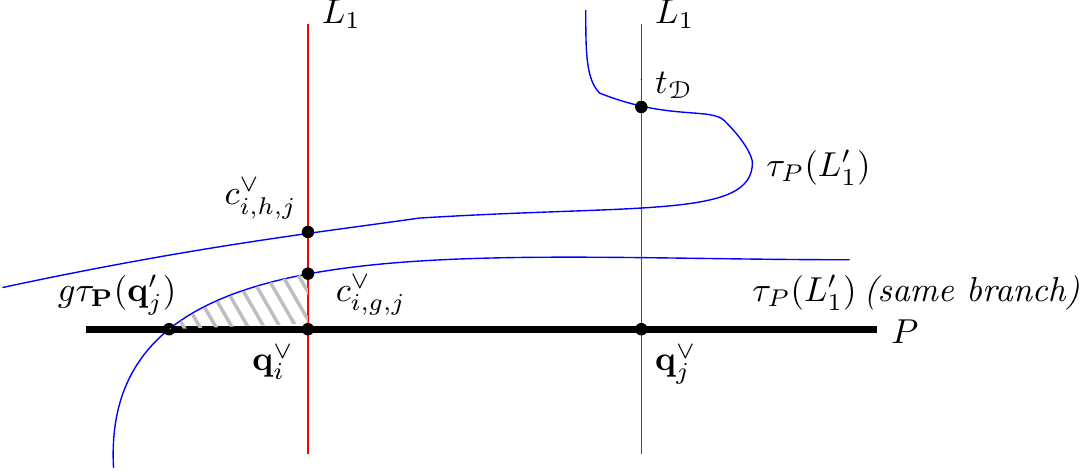}
  \caption{Holomorphic triangles in $\mathbf{U}$}
  \label{fig:smallMultiplication}
\end{figure}

\begin{lemma}\label{l:WrongBoundaryCondition}
 For $\tau \gg1$, there is no rigid element in $\eM^{J^\tau}(c_{\bar i,\bar g,\bar j}^\vee;q_i^\vee,\tau_P(q_{j}'))$ unless $c_{\bar i,\bar g,\bar j} = c_{i,g,j}$ for some $g \in \Gamma$.
\end{lemma}

\begin{proof}
 Assume $u_{\infty}=(u_v)_{v \in V(\eT)}$ be a limiting holomorphic building.
 If $c_{\bar i,\bar g,\bar j} \neq c_{i,g,j}$ for all $g \in \Gamma$, then $c_{\bar i,\bar g,\bar j} \notin T^*_{q_i}P \cap \tau_P(T^*_{q_j'} P)$.
 By boundary condition, there is $v_1 \in V(\eT)$ such that $q_i^\vee, \tau_P(q_j')$ are asymptotes of $u_{v_1}$ but $c_{\bar i,\bar g,\bar j}^\vee$ is {\bf not}
 an asymptote of $u_{v_1}$.
 Therefore, all other asymptotes of $u_{v_1}$ are positive Reeb chords and we get a contradiction as in Lemma \ref{l:noOutside}.
\end{proof}

The following lemma computes the $\mu^2$ map with trivial local systems on $L_1$ and $L_1'$.

\begin{lemma}\label{l:microlocaltriangle}
For $\tau \gg 1$, the  $c_{i,h,j}^\vee$-coefficient of $\mu^2(\fq_i^\vee,g\tau_P(\fq_j'))$ is $\pm 1$ when $h=g$ and is $0$ when $h \neq g$.
Here $\mu^2: hom(\eP,L_1) \times hom(\tau_P(L_1'),\eP) \to hom(\tau_P(L_1'),L_1)$
is the multiplication.
\end{lemma}

\begin{proof}
First, we want to argue that any $u \in \eM^{J^{\tau}}(c_{i,h,j}^\vee;q_i^\vee,\tau_P(q_j'))$ contributing to $\mu^2(\fq_i^\vee,g\tau_P(\fq_j'))$ has image completely lying
inside $U$ when $\tau \gg 1$.
We argue as before. Let $u_{\infty}=(u_v)_{v \in V(\eT)}$ be a limiting holomorphic building.
By boundary condition, there is $v_1 \in V(\eT)$ such that $q_i^\vee, \tau_P(q_j')$ are asymptotes of $u_{v_1}$.
If $c_{i,h,j}^\vee$ is not an asymptote of $u_{v_1}$, then we get a contradiction as in Lemma \ref{l:WrongBoundaryCondition}.
Therefore, $u_{v_1}$ has asymptotes $c_{i,h,j}^\vee$, $q_i^\vee$, $\tau_P(q_j')$ and positive Reeb chords $y_1,\dots,y_k$.
The virtual dimension of $u_{v_1}$ is given by
 \begin{align*}
 \virdim(u_{v_1})=|c_{i,h,j}^\vee|-|q_i^\vee|-|\tau_P(q_j')|-\sum_{s=1}^k |y_s|+k \ge 1-0-1+k=k  
 \end{align*}
 It means that $k=0$ so $u_{v_1}$ has no positive Reeb chord and the claim follows.
 
 In particular, we can lift $u \in \eM^{J^{\tau}}(c_{i,h,j}^\vee;q_i^\vee,\tau_P(q_j'))$ to the universal cover $\fU$.
 By considering the boundary condition, it is clear that we must have $h=g$ for $u$ to exist.
 Now, to compute the $c_{i,g,j}^\vee$-coefficient of $\mu^2(\fq_i^\vee,g\tau_P(\fq_j'))$, we use the following model. 
 %\footnote{Below, either stretch the local model or use $T^*S$ directly.} 
 
We consider an $A_3$-Milnor fiber as in the proof of Theorem \ref{t:localCount} but rename the objects to keep the notation aligned with the current situation.  For example, we denote the Lagrangian spheres by $S_1$ ,$S$ and $S_2$ such that $S_1 \cap S_2 =\emptyset$.
Let $\tau$ be the Dehn twist along $S$, $q^\vee \in CF(S,S_1)$, $q' \in CF(S_2,S)$, $\tau(q') \in CF(\tau(S_2),S)$, $c^\vee \in CF(\tau(S_2),S_1)$ and $e,f \in CF(S,S)$.
We have $|q^\vee|=0$, $|q'|=n$, $|\tau(q')|=1$, $|c^\vee|=1$, $|e|=0$ and $|f|=n$.
Consider the following commutative diagram (up to sign)
\begin{align*}
\xymatrixcolsep{4pc}\xymatrix{
HF(\tau(S_2),S)\times HF(S_1,\tau(S_2)) \times HF(S,S_1) \ar[d]^-{\mu^2(\tau(q'),\cdot) \times Id} \ar[r]^-{Id \times \mu^2} &
HF(\tau(S_2),S)\times HF(S,\tau(S_2)) \ar[d]^{\mu^2(\tau(p'),\cdot)}\\
HF(S_1,S)\times HF(S,S_1) \ar[r]^-{\mu^2}& HF(S,S)
}
\end{align*}
All the Floer cohomology has rank $1$ except that $HF(S,S)$ has rank $2$.
The bottom arrow gives $\mu^2(q,q^\vee)=f$.
By the long exact sequence
\begin{align}
 HF^{k}(S_1,S_2) \to HF^k(S_1,\tau(S_2)) \to HF^{k+1}(S_1,S) \to HF^{k+1}(S_1,S_2)
\end{align}
and the fact that $HF(S_1,S_2)=0$, we know that $HF^{n-1}(S_1,\tau(S_2)) \to HF^{n}(S_1,S)$ is an isomorphism.
Since $\tau(q')$ represents the unique (up to multiplications by a unit) non-zero class in $HF(\tau(S_2),S)$,
we know that $\mu^2(\tau(q'),\cdot)$ induces the isomorphism $HF^{n-1}(S_1,\tau(S_2)) \simeq HF^{n}(S_1,S)$.
Since this is true for any characteristic of $\mathbb{K}$, we must have $\mu^2(\tau(q'),c)= \pm q$.

%We claim that the left vertical arrow gives $\mu^2(\tau(q'),c)=q$.
%The reason is that $\tau(S_2)$ is quasi-isomorphic to $S[1] \oplus S_2$, which has a projection morphism (non-trivial) to $S$.
%Since $HF(\tau(S_2),S)$ is rank $1$, $\mu^2(\tau(p'),\cdot)$ is the  projection morphism.
%Moreover, $hom(S_1,S[1] \oplus S_2)=hom(S_1,S[1])$ because $hom(S_1,S_2)=0$.
%This implies that $\mu^2(\tau(p'),\cdot)$ is an isomorphism.
%Since this is true for any characteristic of $\mathbb{K}$, we must have $\mu^2(\tau(p'),c)=p$.

By the associativity of cohomological multiplication, we have $\mu^2(\tau(q'),\mu^2(c,q^\vee))=\pm f$.
It implies that $\mu^2(c,q^\vee)=\pm \tau(q')^\vee$.
Dually, we have $\mu^2(q^\vee,\tau(q'))=\pm c^\vee$ (it amounts to changing the asymptote $c$ from outgoing end to incoming end, and $\tau(q')$ from incoming end to outgoing end).

Since each  $u \in \eM^{J^{\tau}}(c_{i,h,j}^\vee;q_i^\vee,\tau_P(q_j'))$ can be lifted to $\fU$, there is a sign preserving bijective correspondence
$\eM^{J^{\tau}}(c_{i,h,j}^\vee;q_i^\vee,\tau_P(q_j')) \simeq \eM(c^\vee;q^\vee,\tau(q'))$ so we get the result.

\end{proof}

\begin{rmk}\label{r:GeomAlternative}
There is an alternative geometric argument as follows.
When the fibers corresponding $S_1$ and $S_2$ in the proof of Lemma \ref{l:microlocaltriangle} are fibers of antipodal points.
The moduli computing $c^\vee$-coefficient of $\mu^2(q^\vee,\tau(q'))$ is the constant map to the point $S_1 \cap S$.
One can check that this constant map is regular so the algebraic count is $\pm 1$.
In the more general case, where $S \cap S_2 $ is not the antipodal point of $S_1 \cap S$,
one can apply a homotopy type argument to conclude Lemma  \ref{l:microlocaltriangle}.
\end{rmk}

 Now we take the local system on $L_1$ and $L_1'$ into consideration.  Take the universal cover $\fU$ of the neighborhood of $P$, there is a unique path (up to homotopy) in $\tau_\fP(T^*_{g\fq_j'}\fP)$
from $\fc_{\fq_i,g\fq_j'}$ to $g\tau_{\fP}(\fq_j')$.
It descends to the unique path (up to homotopy) in $\tau_P(T^*_{q_j'}P)$ from 
$c_{\fq_i,g\fq_j'}$ to $\tau_{P}(q_j')$, which we denote by $[c_{\fq_i,g\fq_j'} \to \tau_{P}(q_j')]$.
Similarly, there is a unique path (up to homotopy) in $T^*_{q_i}P$ from 
 $q_i$ to $c_{\fq_i,g\fq_j'}$, which we denote by $[q_i \to c_{\fq_i,g\fq_j'}]$.  Then we have

\begin{prop}\label{p:xitauxj}
For $\tau \gg1$, we have (see \eqref{eq:EleSupportedAtTensor}), up to sign,
 \begin{multline}
\mu^2(\psi^2 \otimes \fq_i^\vee,g\tau_\fP(\fq_j') \otimes \psi^1)=I_{[q_i \to c_{\fq_i,g\fq_j'}]} (\psi^2) \otimes (\psi^1 \circ I_{[c_{\fq_i,g\fq_j'} \to \tau_{P}(q_j')]})\\
\in Hom_\K(\tau_P((\eE^1)')_{c_{\fq_i,g\fq_j'}}, \eE^1_{c_{\fq_i,g\fq_j'}}) \label{eq:LongFormula}
 \end{multline}
 for $\psi^2 \in \eE^1_{q_i}$ and $\psi^1 \in Hom_\K(\tau_P((\eE^1)')_{\tau_P(q_j')},\K)$.
In particular, the right hand side is supported at the intersection point $c_{i,g,j}^\vee$ only
and the morphism $\Phi^{\otimes}_{i,g,j}:=\mu^2(- \otimes \fq_i^\vee,g\tau_\fP(\fq_j') \otimes -)$
\begin{align}
 \Phi^{\otimes}_{i,g,j}:\eE^1_{q_i} \otimes Hom_\K(\tau_P((\eE^1)')_{\tau_P(q_j')},\K) \to Hom_\K(\tau_P((\eE^1)')_{c_{\fq_i,g\fq_j'}},\eE^1_{c_{\fq_i,g\fq_j'}}) \label{eq:mu2tensor}
\end{align}
is an isomorphism.

\end{prop}

\begin{proof}
By Lemma \ref{l:noOutside}, \ref{l:WrongBoundaryCondition} and \ref{l:microlocaltriangle}, we already know that
$\mu^2(\psi^2 \otimes \fq_i^\vee,g\tau_\fP(\fq_j') \otimes \psi^1)$ is supported at the intersection point $c_{i,g,j}^\vee$.
Moreover, as explained in the proof of Lemma \ref{l:microlocaltriangle}, the rigid elements contributing to 
$\mu^2(\fq_i^\vee,g\tau_P(\fq_j'))$ lie completely inside $U$.

To obtain the result, it suffices to understand the parallel transport maps.
Let $u \in \eM^{J^{\tau}}(c_{i,g,j}^\vee;q_i^\vee,\tau_P(q_j'))$.
The contribution to $\mu^2(\psi^2 \otimes \fq_i^\vee,g\tau_\fP(\fq_j') \otimes \psi^1)$ by $u$ is given by (up to sign)
\begin{align}
 (I_{\partial_2 u} \circ \psi^2) \otimes (\fq_i^\vee \circ I_{\partial_1 u} \circ g\tau_\fP(\fq_j')) \otimes (\psi^1 \circ I_{\partial_0 u}) \label{eq:CompoParallelTran}
\end{align}
Since the domain of $u$ is contractible, $u$ can be lifted to the universal cover and therefore the generator 
$c_{i,g,j}^\vee$ uniquely determine the homotopy class of the path $\partial_1 u$ on $P$ (and also $\partial_0 u$ on $\tau_P(L_1')$ and $\partial_2 u$ on $L_1$), which is
exactly the path such that $\fq_i^\vee \circ I_{\partial_1 u} \circ g\tau_\fP(\fq_j')=1$, 
where $g\tau_\fP(\fq_j')$ is regarded as an element of the universal local system at $q_j'$ and 
$\fq_i^\vee$ is regarded as an element of the dual of the universal local system at $q_i$.
In other words, we have $I_{\partial_1 u}(g\tau_\fP(\fq_j'))=\fq_i$.
On ther other hand, we have
$I_{\partial_0 u}=I_{[c_{\fq_i,g\fq_j'} \to \tau_{P}(q_j')]}$ and $I_{\partial_2 u}=I_{[q_i \to c_{\fq_i,g\fq_j'}]}$
so \eqref{eq:CompoParallelTran} reduces to
$I_{[q_i \to c_{\fq_i,g\fq_j'}]} (\psi^2) \otimes (\psi^1 \circ I_{[c_{\fq_i,g\fq_j'} \to \tau_{P}(q_j')]})$.
Now, \eqref{eq:LongFormula} follows immediately from Lemma \ref{l:microlocaltriangle}.
%Therefore, the parallel transport maps given by the boundary on $\tau_P(T^*_{q_j'}P)$ and $T^*_{q_i}P$ are uniquely determined (the parallel transport map 
%given by the boundary on $P$ is fixed by the inputs $\fq_i^\vee$ and $g\tau_\fP(\fq_j')$).
%As a result, the result follows immediately from Lemma \ref{l:microlocaltriangle}.

On the other hand, since $I_{[q_i \to c_{\fq_i,g\fq_j'}]}$ and $I_{[c_{\fq_i,g\fq_j'} \to \tau_{P}(q_j')]}$
are isomorphisms from $\eE^1_{q_i}$ to $\eE^1_{c_{\fq_i,g\fq_j'}}$
and from $\tau_P((\eE^1)')_{c_{\fq_i,g\fq_j'}}$ to $\tau_P((\eE^1)')_{\tau_P(q_j')}$, respectively, \eqref{eq:LongFormula} clear induces
the isomorphism
\begin{align}
 \eE^1_{q_i} \otimes Hom_\K(\tau_P((\eE^1)')_{\tau_P(q_j')},\K)
  \to \eE^1_{c_{\fq_i,g\fq_j'}} \otimes Hom_\K(\tau_P((\eE^1)')_{c_{\fq_i,g\fq_j'}},\K) 
\end{align}
which is exactly \eqref{eq:mu2tensor}.

 \end{proof}
 
 Recall that we have defined $\eD$ in \eqref{eq:eD}.

 \begin{corr}\label{c:SupportingPoints}
  For $L_1'$ sufficienly close to $L_1$ and $\tau \gg 1$,
 every degree $0$ class in $H^0(\eD)$ admits a cochain representative $\beta$
 which is a sum of elements supported at $e_L$ and $\{q_i^\vee \otimes \tau(q_j')\}_{i,j}$ only.
 Moreover, the term of $\beta$ supported at $e_L$ cannot be zero unless $\beta=0$.
 \end{corr}

 \begin{proof}
   Every degree $0$ cocycle in $\eD$
   is a sum of elements supported at $e_L$, $\{c_{i,g,j}^\vee\}_{i,j,g}$ and $\{q_i^\vee \otimes \tau(q_j')\}_{i,j}$ because $|w_k| \neq 0$ for $w_k \neq e_L$.
   Let $\beta$ be a degree $0$ cocycle which represents a class $[\beta]$.
   By Proposition \ref{p:xitauxj}, we can eliminate the terms of $\beta$ supported at $c_{i,g,j}^\vee$ by adding the differentials of certain 
   cochains supported at $q_i^\vee \otimes \tau(q_j')$. 
   Note that the term of $\beta$ supported at $c_{i,g,j}^\vee$ themselves might not be exact
   because  
   $\mu^1((\psi^2 \otimes \fq_i^\vee)  \otimes (g\tau_\fP(\fq_j') \otimes \psi^1))$
   involves more than just $\mu^2_\cF(\psi^2 \otimes \fq_i^\vee, g\tau_\fP(\fq_j') \otimes \psi^1)$ (see \eqref{eq:differentials}).
   However, it does not matter and there is a cochain $\beta'$ cohomologous to $\beta$ such that
   $\beta'$ is a sum of elements supported at $e_L$ and $\{q_i^\vee \otimes \tau(q_j')\}_{i,j}$ only.
   
   Now, suppose the term of $\beta'$ supported at $e_L$ is $0$.
   We write $\beta'=\sum_{(i,j)}\psi^{i,j}$, where, for all $i,j$, $\psi^{i,j}$ is an 
   element supported at $q_i^\vee \otimes \tau_P(q_j')$.
   If $\psi^{i_0,j_0} \neq 0$ for some $i_0,j_0$, then
   by the isomorphism statement in Proposition \ref{p:xitauxj}, the terms of $\mu^1(\beta')$ must contain a 
   non-trivial element supported at $c_{i_0,g,j_0}^\vee$ for some $g$, 
   because all other $\mu^1(\psi^{i,j})$ do not have non-zero element supported at $c_{i_0,g,j_0}^\vee$, which draws a contradiction.

   %  are given by applying 
   % $\Phi^\otimes_{i,g,j}$ to the terms of $\beta'$ supported at $q_i^\vee \otimes \tau(q_j')$.
   % Therefore, $\mu^1(\beta')=0$ implies that the terms of $\beta'$ supported at $q_i^\vee \otimes \tau(q_j')$ are $0$ for all $i,j$.
   % In other words, $\beta' \cong 0$ if $\beta'$ is a degree $0$ cocycle and it is a sum of elements supported at $\{q_i^\vee \otimes \tau(q_j')\}_{i,j}$ only.
 \end{proof}

By Corollary \ref{c:SupportingPoints}, we can write every degree $0$ cocyle $\beta$ of $\eD$ as 
\begin{align}
\beta=\psi_{e_L}+ \sum_{i,j} \psi_{q_i^\vee \otimes \tau_P(q_j')} \label{eq:firstForm}
\end{align}
where $\psi_{x}$ is an element supported at $x$.
Moreover, by \eqref{eq:EleSupportedAtTensor}, we can further decompose $\psi_{q_i^\vee \otimes \tau_P(q_j')}$
as
\begin{align}
 \psi_{q_i^\vee \otimes \tau_P(q_j')}= \sum_{g \in \Gamma} \sum_{k=1}^{n_{i,g,j}} \psi^2_{i,g,j,k} \otimes \fq_i^\vee \otimes g\tau_\fP(\fq_j') \otimes \psi^1_{i,g,j,k}
\end{align}
for some $\psi^2_{i,g,j,k} \in \eE^1_{q_i}$, $\psi^1_{i,g,j,k} \in Hom_\K(\tau_P((\eE^1)')_{\tau_P(q_j')},\K)$ and $n_{i,g,j} \in \N$.

 \begin{prop}[Cocycle elements]\label{p:cocycleElements}
 For $L_1'$ sufficienly close to $L_1$ and $\tau \gg 1$,
 there is a non-exact degree $0$ cocycle $c_\eD$ in $\eD$ of the form
 \begin{align}
  c_\eD= t_{\eD} +\sum_{g,k,i,j} \psi^2_{i,g,j,k} \otimes \fq_i^\vee \otimes g\tau_\fP(\fq_j') \otimes \psi^1_{i,g,j,k}
 \end{align}
 where
$ \psi^2_{i,g,j,k}=\psi^1_{i,g,j,k}= 0$ if either $j>i$ or ($j=i$ and $g \neq 1_\Gamma$), and (see \eqref{eq:mu2tensor})
 \begin{align}
 %\left\{
 %\begin{array}{ll}
 %  \psi^2_{i,g,j,k}=\psi^1_{i,g,j,k}= 0 \text{ for } j>i \text{ or } (j=i \text{ and } g \neq 1_\Gamma) \\
 %  \psi^2_{i,g,j,k}=\psi^1_{i,g,j,k}= 0 \text{ if } j=i \text{ and } g \neq 1_\Gamma \\
   \Phi^{\otimes}_{i,1_\Gamma,i}(\sum_k(\psi^2_{i,1_\Gamma,i,k} \otimes \fq_i^\vee \otimes \tau_\fP(\fq_i') \otimes \psi^1_{i,1_\Gamma,i,k}))=\pm id \label{eq:DecomIdentity}
 %\end{array}
%\right.
 \end{align}
 where $\pm id \in Hom_\K(\tau_P((\eE^1)')_{c_{\fq_i,\fq_i'}},\eE^1_{c_{\fq_i,\fq_i'}})$.
 %for some $\psi^1_k$, $\psi^2_k$ such that $\Phi^{\otimes}_{1_\Gamma}(\psi_{q_i^\vee \otimes \tau_P(q_i')})$ is the identity morphism in
 %$Hom_\K(\tau_P((\eE^1)')_{c_{\fq_i,\fq_i'}},\eE^1_{c_{\fq_i,\fq_i'}})$.
  \end{prop}
  
  \begin{proof}
   Let $\beta$ be a non-exact degree $0$ cocycle of $\eD$ (which exists from \eqref{eq:CohIsom}).
   We write $\beta$ in the form \eqref{eq:firstForm}.  Note that $\psi_{e_L}$ can be geometrically identified as an 
   element of $hom((\eE^1)',\eE^1)$.  Lemma \ref{l:DiffIdentify} implies that, for $\mu^1_{\eD}(\beta)=0$, we must have $\mu^1_{hom((\eE^1)',\eE^1)}(\psi_{e_L})=0$.

   Also, Corollary \ref{c:SupportingPoints} implies that the degree zero cocycle $\beta$ is uniquely determined by its $\psi_{e_L}$ component.   
   % If $\beta'$ is another degree $0$ cocycle of $\eD$
   % such that the term of $\beta'$ supported at $e_L$ is the same as that of $\beta$,
   % then $\beta-\beta'=0$ by Corollary \ref{c:SupportingPoints}.
   Therefore, 
   %the rank of 
   \begin{align}
    \rank(H^0(\eD)) \le \rank(H^0(hom((\eE^1)',\eE^1)))
   \end{align}
 %$H^0(\eD)$ is bounded above by $H^0(hom((\eE^1)',\eE^1))$.
   However, as explained in \eqref{eq:CohIsom}, we have
   \begin{align}
    \rank(H^0(\eD))=\rank(HF^0((\eE^1)',\eE^1))
   \end{align}
   %$H^0(\eD)$
   %equals to the rank of $HF^0(\tau_P(\eE^1),\tau_P(\eE^1))=HF^0((\eE^1)',\eE^1)$.
   It implies that for each degree $0$ cocycle $\psi_{e_L} \in hom((\eE^1)',\eE^1)$, there exists $\psi_{q_i^\vee \otimes \tau_P(q_j')}$
   such that $\psi_{e_L}+ \sum_{i,j} \psi_{q_i^\vee \otimes \tau_P(q_j')}$ is a degree $0$ cocycle in $\eD$.
   
   In particular, we can take $\psi_{e_L}=t_\eD$.
   For $\mu^1(t_\eD+\sum_{i,j} \psi_{q_i^\vee \otimes \tau_P(q_j'))})$ to be zero, 
   the terms of it supported at $c_{i,g,j}^\vee$ must be zero for all $i,j,g$.
   Therefore, we obtain the result by Proposition \ref{p:matrixA} and \ref{p:xitauxj} (see \eqref{eq:differentials}).
  \end{proof}

  \begin{comment}

 \begin{prop}
  When it is close to SFT limit, there exists $a_{ijg}$ as in Proposition \ref{p:matrixA}
  such that $e_L-\sum_{i,j,g} a_{ijg} p_i^\vee \otimes \tau(p_j')g$ is a degree $0$ cocycle in $hom(\tau(L_1'),T_{\eP}(L_1))$
  representing a non-trivial class in $H^0(\tau(L_1'),T_E(L_1))$.
 \end{prop}

 \begin{proof}
  Any degree $0$ element of $hom(\tau(L_1'),T_E(L_1))$ is of the form
  $$\alpha^e e_L +\sum_{i,j,g} \alpha^c_{i,g,j} c_{i,g,j}^\vee+ \sum_{i,j} \alpha^p_{ijg} p_i^\vee \otimes \tau(p_j')g$$
   By Lemma \ref{l:microlocaltriangle}, we have $\mu^1_{Tw}(p_i^\vee \otimes \tau(p_j')g)=c_{i,g,j}^\vee+\mu^1(p_i^\vee )\otimes  \tau(p_j')g+p_i^\vee  \otimes \mu^1( \tau(p_j')g)$.
  Therefore, we can pick a cochain representing a non-trivial degree $0$ cohomology class
  in the form
  $$c:=e_L+\sum_{i,j,g} \alpha^p_{ijg} p_i^\vee \otimes \tau(p_j')g$$
  For $c$ to be a cochain, we have $\mu^1(e_L)+\sum_{i,j,g} \alpha_{ijg} \mu^2(p_i^\vee, \tau(p_j')g)=0$.
  By Proposition \ref{p:matrixA} and \ref{p:xitauxj}, we have $\alpha^e_{ijg}=-a_{ijg}$.
 \end{proof}
\end{comment}
\subsection{Quasi-isomorphisms}
 
 Let $c_\eD$ be the degree $0$ cocycle obtained from Proposition \ref{p:cocycleElements}.
 In this section, we are going to study the map \eqref{eq:QIsoms} for $\eE^0 \in Ob(\cF)$.

 We assume that $L_0 \pitchfork L_1$, $L_0 \pitchfork \tau_P(L_1')$ and $L_0 \cap U$
 is a union of cotangent fibers $\cup_{i=1}^{d_{L_0}} T^*_{p_i} P$, where $d_{L_0}=\#(L_0 \cap P)$.
 Let $\fp_i$ be a choice of lift of $p_i$ in $\fP$.
 Let $C_0:=hom(\eE^0,\tau_P((\eE^1)'))$ and $C_1:=hom(\eE^0,T_\eP(\eE^1))$.
 We know from Lemma \ref{l:stretchingStripD} that, when $\tau$ is large enough, there is a subcomplex $C_0^{s} \subset C_0$
 generated by generators of $C_0$ outside $U$.
 Let $C_0^q:=C_0 /C_0^s$ be the quotient complex, which is generated by generators of $C_0$ inside $U$.
 Similarly, $C_1^s:=hom(\eE^0,\eE^1) \subset C_1$ is a subcomplex and $C_1^q:=C_1/C_1^s$ is the quotient complex.
 By definition (see \eqref{eq:mappingMu2}), for $\psi \in C_0$,
 \begin{align}
  \mu^2(c_\eD,\psi)=&\mu^2_{\cF}(e_\eD,\psi)+\sum_{i,j,g,k} (\psi^2_{i,g,j,k} \otimes \fq_i^\vee) \otimes \mu^2_{\cF}(g\tau_\fP(\fq_j') \otimes \psi^1_{i,g,j,k},\psi) \nonumber \\
                    &+\sum_{i,j,g,k} \mu^3_{\cF}(\psi^2_{i,g,j,k} \otimes \fq_i^\vee,g\tau_\fP(\fq_j') \otimes \psi^1_{i,g,j,k},\psi) \label{eq:mu2cE}
 \end{align}
 We define $\mu^2_s(c_\eD,-):=\mu^2(c_\eD,-)|_{C_0^s}:C_0^s \to C_1$. 
  %\begin{align}
  %\mu^2_s(c_\eD,\psi)=&\mu^2_{\cF}(e_\eD,\psi)+\sum_{i,j,g,k} \mu^3_{\cF}(\psi^2_{i,g,j,k} \otimes \fq_i^\vee,g\tau_\fP(\fq_j') \otimes \psi^1_{i,g,j,k},\psi)
 %\end{align}
 %for $\psi \in C_0^s$.
 
 \begin{lemma}\label{l:chainmapSub}
  For $\tau \gg 1$, the image of $\mu^2_s(c_\eD,-)$ is contained in $C_1^s$. Therefore, $\mu^2_s(c_\eD,-):C_0^s \to C_1^s$ is a chain map. 
 \end{lemma}

 \begin{proof}
  Note that the first and last term  on the right hand side of \eqref{eq:mu2cE} lie inside $C_1^s$.
  Therefore, it suffices to show that $\mu^2_{\cF}(g\tau_\fP(\fq_j') \otimes \psi^1_{i,g,j,k},\psi)=0$ for $\psi \in C_0^s$.
  We consider the moduli $\eM^{J^{\tau}}(p_s;\tau_P(q_j'),y)$ where $y \in (L_0 \cap \tau_P(L_1')) \backslash U$
  and $p_s \in L_0 \cap P$.
  Let $u_{\infty}=(u_v)_{v \in V(\eT)}$ be a holomorphic building converging from curves in $\eM^{J^{\tau}}(p_s;\tau_P(q_j'),y)$.
  By boundary condition, there exists $v_1 \in V(\eT)$ such that $p_s$ and $\tau_P(q_j')$ are asymptotes of $u_{v_1}$.
  The other asymptotes of $u_{v_1}$ are positive Reeb chords $y_1,\dots,y_m$. We have
  \begin{align}
   \virdim(u_{v_1})=|p_s|-|\tau_P(q_j')|- \sum_{l=1}^m |y_l|-(1-m) \ge n-1-(1-m) \ge n-2>0
  \end{align}
 Contradiction. Therefore, $\eM^{J^{\tau}}(p_s;\tau_P(q_j'),y)= \emptyset$ for $\tau \gg1$.
 \end{proof}

 \begin{lemma}\label{l:SubFirstTerm}
  For $\tau \gg 1$, $\mu^2_s(c_\eD,-)=\mu^2_{\cF}(t_\eD,-)$. 
 \end{lemma}
 
 \begin{proof}
  By Lemma \ref{l:chainmapSub}, it suffices to prove that $\eM^{J^{\tau}}(x;q_i^\vee, \tau_P(q_j'),y)= \emptyset$ for $\tau \gg1$,
  where $y \in (L_0 \cap \tau_P(L_1')) \backslash U$ and $x \in L_0 \cap L_1$.
  Let $u_{\infty}=(u_v)_{v \in V(\eT)}$ be a holomorphic building converging from curves in $\eM^{J^{\tau}}(x;q_i^\vee, \tau_P(q_j'),y)$.
  By boundary condition, there exists $v_1 \in V(\eT)$ such that $q_i^\vee$ and $\tau_P(q_j')$ are asymptotes of $u_{v_1}$.
  The other asymptotes of $u_{v_1}$ are positive Reeb chords $y_1,\dots,y_m$. We have
  \begin{align}
   \virdim(u_{v_1})=n-|q_i^\vee|-|\tau_P(q_j')|- \sum_{l=1}^m |y_l|-(1-m) \ge n-1-(1-m) \ge n-2>0
  \end{align}
 Therefore, $\eM^{J^{\tau}}(x;q_i^\vee, \tau_P(q_j'),y)= \emptyset$ for $\tau \gg1$,
 \end{proof}

 \begin{prop}\label{p:SubComQIso}
  For $\tau \gg 1$, $\mu^2_s(c_\eD,-)$ is a quasi-isomorphism.
 \end{prop}

 \begin{proof}
  For $y \in (L_0 \cap \tau_P(L_1')) \backslash U$ and $x \in L_0 \cap L_1$, the proof of Lemma \ref{l:stretchingStripB} implies that
  all rigid elements in $\eM^{J^{\tau}}(x;e_L,y)$ have their image completely outside $U$.
  
  As a result, the computation of $\mu^2_s(c_\eD,-)=\mu^2_{\cF}(t_\eD,-)$ picks up exactly the same holomorphic triangles
  that contributes to $\mu^2_\cF(e_\cE,-): C_0^s\cong hom(\eE^0,(\eE^1)') \to hom(\eE^0,\eE^1)\cong C_1^s$ via the tautological identification between $e_\cE$ and $t_\eD$ 
  (see \eqref{eq:teD} and the paragraph after it).
  Since $e_\cE$ is the cohomological unit, $\mu^2_s(c_\eD,-)$ is also a quasi-isomorphism.
 \end{proof}

 By Lemma \ref{l:chainmapSub}, we know that $\mu^2(c_\eD,-)$ induces a chain map on the quotient complexes $\mu_q^2(c_\eD,-):C_0^q \to C_1^q$.
 Since the first and last term  on the right hand side of \eqref{eq:mu2cE} are, by definition, lying inside $C_1^s$, the map $\mu_q^2(c_\eD,-)$ is given by
   \begin{align}
  \mu^2_q(c_\eD,\psi)=\sum_{i,j,g,k} (\psi^2_{i,g,j,k} \otimes \fq_i^\vee) \otimes \mu^2_{\cF}(g\tau_\fP(\fq_j') \otimes \psi^1_{i,g,j,k},\psi) \label{eq:Ending}
 \end{align}
 By Proposition \ref{p:SubComQIso} and the five lemma, to show that $\mu^2(c_\eD,-)$ is a quasi-isomorphism, it suffices to show that
 $\mu^2_q(c_\eD,-)$ is a quasi-isomorphism.
 
 We recall from Lemma \ref{l:CorrIntersections} that there is a bijective correspondence
 \begin{align}
  \iota:  hom(\eP,L_1') \otimes_{\Gamma} hom(L_0,\eP)   \to  (L_0 \cap \tau_P(L_1')) \cap U \label{eq:CorrIntersections2}
 \end{align}
 so we can write a point $y \in (L_0 \cap \tau_P(L_1')) \cap U$ as $c_{h\fp_s,\fq_l'}:=\iota(\fq_l'^\vee \otimes h\fp_s)$
 for some $h \in \Gamma$ and some $s,l$.
 We want to understand the moduli $\eM^{J^\tau}(p_m;\tau_P(q_j'),c_{h\fp_s,\fq_l'})$ for various $j,s,l,m$, which is responsible for (part of) the operation
 \begin{align}
  hom(\tau_P(L_1'),\eP) \times hom(L_0,\tau_P(L_1')) \to hom(L_0,\eP) \label{eq:microlocaltriangle}
 \end{align}
 Notice that, by switching the appropriate strip-like ends from incoming to outgoing (and backwards) for the same holomorphic triangles, \eqref{eq:microlocaltriangle} can be dualized to
  \begin{align}
 hom(\eP,L_0) \times hom(\tau_P(L_1'),\eP) \to hom(\tau_P(L_1'),L_0) \label{eq:microlocaltriangle2}
 \end{align}
 If we replace $L_0$ by $L_1$ (both of them are union of cotangent fibers in $U$), then we see that \eqref{eq:microlocaltriangle2} has already
 been studied in Lemma \ref{l:WrongBoundaryCondition} and \ref{l:microlocaltriangle}.
 The outcome is the following:
 
 \begin{lemma}\label{l:mtriangle}
  For $\tau \gg 1$, for $\psi_{c_{h\fp_s,\fq_l'}} \in C_0$ supported at $c_{h\fp_s,\fq_l'}$
  \begin{align}
   \mu^2_\cF(g\tau_\fP(\fq_j')\otimes \psi^1_{i,g,j,k},\psi_{c_{h\fp_s,\fq_l'}}) \label{eq:microlocaltriangle3}
  \end{align}
is $0$ if  $l \neq j$. When $l=j$, \eqref{eq:microlocaltriangle3} becomes
\begin{align}
gh\fp_s \otimes (I_{[\tau_P(q_j') \to p_s]} \circ \psi^1_{i,g,j,k} \circ I_{[c_{h\fp_s,\fq_j'} \to \tau_P(q_j')]} \circ \psi_{c_{h\fp_s,\fq_j'}} \circ I_{[p_s \to c_{h\fp_s,\fq_j'}]})
\end{align}
where all the parallel transport maps are the unique one determined by the boundary condition inside $U$ (cf. Proposition \ref{p:xitauxj}).
 \end{lemma}

\begin{proof} 
The argument largely resembles the proof of Lemma \ref{l:WrongBoundaryCondition}, \ref{l:microlocaltriangle}
and Proposition \ref{p:xitauxj}. 
A neck-stretching argument as in Lemma \ref{l:WrongBoundaryCondition} deduces that 
$\eM^{J^\tau}(p_m;\tau_P(q_j'),c_{h\fp_s,\fq_l'})$ is not empty only if $j=l$ and $m=s$.
The same dimension count implies that 
when $j=l$, $m=s$ and $\tau \gg 1$, every rigid element of $\eM^{J^\tau}(p_m;\tau_P(q_j'),c_{h\fp_s,\fq_l'})$
 has image inside $U$.
The local count and the chasing of local systems from Lemma \ref{l:microlocaltriangle} and Proposition \ref{p:xitauxj} 
applies directly to the current case because it is a computation in $\fU$ about cotangent fibers and their Dehn twists.
 %By the argument in Lemma \ref{l:WrongBoundaryCondition}, $\eM^{J^\tau}(p_m;\tau_P(q_j'),c_{h\fp_s,\fq_l'})$ is not empty only if $j=l$ and $m=s$.
 %By the argument in Lemma \ref{l:microlocaltriangle}, when $j=l$, $m=s$ and $\tau \gg 1$, every rigid element of $\eM^{J^\tau}(p_m;\tau_P(q_j'),c_{h\fp_s,\fq_l'})$
 %has image inside $U$.
 %Moreover
 In particular, if we remove the local systems on $L_0$ and $\tau_P(L_1')$, we get
 \begin{align}
  \mu^2_{\cF}(g\tau_\fP(\fq_j') ,c_{h\fp_s,\fq_j'})=\mu^2_{\cF}(g\tau_\fP(\fq_j') ,c_{gh\fp_s,g\fq_j'})=gh\fp_s \in hom(L_0, \eP)
 \end{align}
 The parallel transport maps are uniquely determined by boundary conditions, and after chasing all of them, we get the result.
 \end{proof}
 
 Let $V=hom(\eP,(\eE^1)') \otimes_{\Gamma} hom(\eE^0,\eP)$ which is generated by elements of the form
 \begin{align}
  (\Upsilon^2 \otimes (\fq_r')^\vee) \otimes (h\fp_t \otimes \Upsilon^1) \label{eq:EleSupportedAtTensors}
 \end{align}
 for $h \in \Gamma$, $r=1,\dots, d_{L_1}$, $t=1,\dots, d_{L_0}$, $\Upsilon^2 \in (\eE^1)'_{q_r'}$ and $\Upsilon^1 \in Hom_\K(\eE^0_{p_t},\K)$ (cf. \eqref{eq:EleSupportedAtTensor}).

 \begin{itemize}
   \item  For $s=1,\dots,d_{L_0}$, let $V_s$ be the subspace generated by elements in \eqref{eq:EleSupportedAtTensors} such that $t=s$.
   \item  For $s=1,\dots,d_{L_0}$ and
 $l=1,\dots,d_{L_1}$, let $V_{s,l}$ be the subspace of $V_s$ generated by elements in \eqref{eq:EleSupportedAtTensors} such that $r=l$.
 \item  For $s=1,\dots,d_{L_0}$,
 $l=1,\dots,d_{L_1}$ and $g \in \Gamma$, let 
 $V_{s,l,g}$ be the subspace of $V_{s,l}$ generated by elements in \eqref{eq:EleSupportedAtTensors} such that $h=g$.
 \end{itemize}

 Therefore, we have direct sum decompositions
 \begin{align}
  V=\oplus_s V_s, \text{   } V_s =\oplus_l V_{s,l}, \text{   } V_{s,l}= \oplus_{g} V_{s,l,g} \label{eq:DirectSums}
 \end{align}

 The bijective correspondence $\iota$  \eqref{eq:CorrIntersections2} extends to an isomorphism, also denoted by $\iota$,
 from $V$ to $C_0^q$ by keeping track of the (uniquely determined) parallel transport maps along Lagrangians inside $U$.
 On the other hand, there is an obvious isomorphism $F:hom(\eP,\eE^1) \otimes_{\Gamma} hom(\eE^0,\eP) \to V$ given by
 \begin{align}
  (\Upsilon^2 \otimes \fq_l^\vee) \otimes (h\fp_s \otimes \Upsilon^1) \mapsto (\Upsilon^2 \otimes (\fq_l')^\vee) \otimes (h\fp_s \otimes \Upsilon^1)
 \end{align}
where we used the identification $\eE^1_{q_l} \simeq (\eE^1)'_{q_l}$ by the Hamiltonian push-off.
As a result, we have a composition map
\begin{align}
 \Theta: V\xrightarrow{\iota} (L_0 \cap \tau_P(L_1')) \cap U\xrightarrow{\mu^2_q(c_\eD,-)} C_1^q\xrightarrow{F} V \label{eq:THETA}
\end{align}
which respects a filtration on $V$ in the following sense.

\begin{lemma}\label{l:FiltrationV}
We have
 \begin{align}
  \left\{
  \begin{array}{ll}
   \Theta(V_s) \subset V_s  &                          \text{ for all }s\\
   \Theta(V_{s,l}) \subset \oplus_{t \ge l} V_{s,t} & \text{ for all }s,l\\
   \Theta(V_{s,l,h}) \subset V_{s,l,h}+(\oplus_{t > l} V_{s,t}) & \text{ for all }s,l,h
  \end{array}
\right.
 \end{align}

\end{lemma}

\begin{proof}
 Explicitly, $\Theta$ is given by (see \eqref{eq:Ending} and Lemma \ref{l:mtriangle} )
 \begin{align}
  &(\Upsilon^2 \otimes (\fq_l')^\vee) \otimes (h\fp_s \otimes \Upsilon^1) \\
  \mapsto & \sum_{i,g,k} (\psi^2_{i,g,l,k} \otimes (\fq_i')^\vee) \otimes (gh\fp_s \otimes R(\psi^1_{i,g,l,k},\Upsilon^2,\Upsilon^1))
 \end{align}
where $R(\psi^1_{i,g,l,k},\Upsilon^2,\Upsilon^1)$ is a term depending on $\psi^1_{i,g,l,k},\Upsilon^2,\Upsilon^1$ given by composing parallel transport maps.
It is therefore clear that $\Theta(V_s) \subset V_s$.
By Proposition \ref{p:cocycleElements}, we know that $\psi^2_{i,g,l,k}=0$ unless $j \le i$ so $\Theta(V_{s,l}) \subset \oplus_{t \ge l} V_{s,t}$.

When $i=l$, $\psi^2_{i,g,l,k} \neq 0$ only if $g=1_\Gamma$ (by Proposition \ref{p:cocycleElements}).
Therefore, $\Theta(V_{s,l,h}) \subset V_{s,l,h}+(\oplus_{t > l} V_{s,t})$
\end{proof}

\begin{prop}\label{p:FinalIsom}
 $\mu^2_q$ is a quasi-isomorphism.
\end{prop}

\begin{proof}
 Since $\mu^2_q$ is a chain map, it suffices to show that $\mu^2_q$ is bijective.
 We know that $\iota$ and $F$ are isomorphisms so it suffices to show that $\Theta$ is surjective (see \eqref{eq:THETA}).
 By \eqref{eq:DirectSums} and Lemma \ref{l:FiltrationV}, it suffices to show that 
 \begin{align}
  \Theta|_{V_{s,l,h}}:V_{s,l,h} \to (V_{s,l,h}+(\oplus_{t > l} V_{s,t}))/(\oplus_{t > l} V_{s,t}) \label{eq:FilteredTheta}
 \end{align}
is bijective for all $s,l,h$.
For fixed $s,l,h$, the map \eqref{eq:FilteredTheta} can be identified with the map
\begin{align}
 (\eE^1)'_{q_l'} \otimes Hom_\K(\eE^0_{p_s},\K) \to (\eE^1)'_{q_l'} \otimes Hom_\K(\eE^0_{p_s},\K) \nonumber \\
 \Upsilon^2 \otimes \Upsilon^1 \mapsto \sum_k (\psi^2_{l,1_{\Gamma},l,k} \otimes  R(\psi^1_{l,1_{\Gamma},l,k},\Upsilon^2,\Upsilon^1)) \label{eq:LastIsom}
\end{align}
By \eqref{eq:DecomIdentity} and keeping track of the uniquely determined parallel transport maps, 
it is clear that \eqref{eq:LastIsom} is an isomorphism.
\end{proof}

\begin{proof}[Concluing the proof of Theorem \ref{t:Twist formula}, \ref{t:reformulation}]
 For each $\eE^1 \in Ob(\cF)$, we apply Proposition \ref{p:cocycleElements} to find a degree $0$ cocycle
 $c_\eD \in hom^0_{\cF^{\perf}}(\tau_P((\eE^1)'),T_\eP(\eE^1))$.
 Given any object $(\eE^0)' \in Ob(\cF)$, we consider a quasi-isomorphic $\eE^0$, which is a Hamiltonian isotopic copy and the underlying Lagrangian $L_0$ intersects transversally with $L_1,\tau_P(L_1')$ and $L_0\cap U$.
 
 Proposition \ref{p:SubComQIso} and \ref{p:FinalIsom}, together with the five lemma, then conclude that \eqref{eq:QIsoms} is a quasi-isomorphism.
 % Hamiltonian isotopic objects are quasi-isomorphic so the same is true for $(\eE^0)'$.
 % Since \eqref{eq:QIsoms} is a quasi-isomorphism for all objects in $\cF$, $\mu^2(c_\eD,-)$ is a quasi-isomorphism
 % between the perfect $A_{\infty}$ right $\cF$-modules $\tau_P((\eE^1)')$ and $T_\eP(\eE^1)$.
 
\end{proof}

\begin{comment}
\begin{proof}[Concluing the proof of Theorem \ref{t:Twist formula}, \ref{t:reformulation}]\footnote{original proof} 
 \textbf{For each $\eE^1 \in Ob(\cF)$, we apply Proposition \ref{p:cocycleElements} to find a degree $0$ cocycle
 element $c_\eD \in hom^0_{\cF^{\perf}}(\tau_P((\eE^1)'),T_\eP(\eE^1))$.
 Given any object $(\eE^0)' \in Ob(\cF)$, 
 the operation $\mu^2(c_\eD,-):hom((\eE^0)', \tau_P((\eE^1)')) \to hom((\eE^0)',T_\eP(\eE^1))$
 is defined by appropriate choice of auxiliary data $(K,J)$ (see Section \ref{ss:Fuk}) where $K$ is usually non-zero.
 However, we can Hamiltonian isotope $(\eE^0)'$ to get $\eE^0$
 such that $L_0$ intersect transversally with $L_1$, $\tau_P(L_1')$ and $L_0 \cap U$ is a union of cotangent fibers.
 For this $\eE^0$, we can choose $K=0$ and domain-independent $J$ so
 we can apply Proposition \ref{p:SubComQIso} and \ref{p:FinalIsom}, together with five lemma, to conclude that \eqref{eq:QIsoms} 
 is a quasi-isomorphism.
 Hamiltonian isotopic objects are quasi-isomorphic so the same is true for $(\eE^0)'$.
 Since \eqref{eq:QIsoms} is a quasi-isomorphism for all objects in $\cF$, $\mu^2(c_\eD,-)$ is a quasi-isomorphism
 between the perfect $A_{\infty}$ right $\cF$-modules $\tau_P((\eE^1)')$ and $T_\eP(\eE^1)$.}
 
\end{proof}
\end{comment}

\begin{proof}[Proof of Corollary \ref{c:PtwistSStwist}]
 When $P$ is diffeomorphic to $\mathbb{RP}^n$ and $n=4k-1$, $P$ is spin and can be equipped with the spin structure descended from $S^n$.
 When $\cchar(\K) \neq 2$, the universal local system $\eP$ is a direct sum of two rank $1$ local systems $\eE^1$ and $\eE^2$.
 This is because $\K[\Z_2]$ splits when $\cchar(\K) \neq 2$.
 Moreover, by Lemma \ref{l:semisimpleCal} and Corollary \ref{c:HSphere}, we have
 \begin{align}
  HF^*(\eE^i,\eE^j)=
  \left\{
  \begin{array}{ll}
  0 & \text{ if }i\neq j\\
  H^*(S^n)\text{ if }i= j
  \end{array}
\right.
 \end{align}
so $\eE^1$ and $\eE^2$ are orthogonal spherical objects.
In this case,
\begin{align}
 T_\eP(\eE)\simeq & Cone(\oplus_{i=1,2} (hom_\cF(\eE^i,\eE) \otimes \eE^i) \xrightarrow{ev} \eE) \label{eq:SStwist}
\end{align}
where $ev$ is the evaluation map.
The spherical twist to $\eE$ along $\eE^i$ is defined to be $Cone(hom_\cF(\eE^i,\eE) \otimes \eE^i \xrightarrow{ev} \eE)$.
A direct verification shows that \eqref{eq:SStwist} is the same as applying the spherical twist
to $\eE$ along $\eE^1$ and then $\eE^2$. It is the same as first applying spherical twist along $\eE^2$ and then $\eE^1$
because $\eE^1$ and $\eE^2$ are orthogonal objects.

When $P$ is diffeomorphic to $\mathbb{RP}^n$ and $\cchar(\K) = 2$, then $H^*(P)=H^*(\mathbb{RP}^n,\mathbb{Z}_2)$.
In this case, one can define a $\mathbb{P}$-twist along $P$ (see \cite{HT06}, \cite{Ha11}) which is an auto-equivalence on $\cF^{\perf}$.
More precisely, the algebra $H^*(\mathbb{RP}^n,\mathbb{Z}_2)$ is generated by a degree $1$ element instead of a degree $2$ element so the $\mathbb{P}$-twist along $P$ 
is not exactly, but a simple variant of, the $\mathbb{P}$-twist defined in \cite{HT06}.
To compare \eqref{eq:EqTwist} with the $\mathbb{P}$-twist, we note that $\K[\mathbb{Z}_2]$ fits into a non-split exact sequence
\begin{align}
 0 \to \K \to \K[\mathbb{Z}_2] \to \K \to 0
\end{align}
it implies that $\eP=Cone(P[-1] \to P)$ and the morphism in the cone is the unique non-trivial one.
In this case, the fact that $T_\eP(\eE)$ is the $\mathbb{P}$-twist of $\eE$ along $P$ is explained in \cite[Remark $4.4$]{Se17}.
\end{proof}

%\begin{rmk}
% Notice that, the quasi-isomorphism type of $\cF$ does not depend on the auxiliary data $(K,J)$ described in Section \ref{ss:Fuk}.
% Therefore, we are allowed to use $K=0$ and domain-independent $J$ for some of the $A_{\infty}$-operations, as long as the 
% moduli involved are regular (see Section \ref{ss:Regularity}).
%\end{rmk}

\section{Algebraic perspective}\label{s:Algebraic perspective}

Here, we present the algebraic mechanism underlying the auto-equivalence of Dehn twist along a spherical Lagrangian submanifold.
Every auto-equivalence is a twist along a spherical functor \cite{Se17}
and we will describe the spherical functor for our case.
%\cite{AL17}, \cite{Me16}.

We assume some basic terminology of $A_{\infty}$ categories (see \cite{Seidelbook} \cite{SeHom}, \cite{Gan13}, \cite[Section 7]{Seidelbook2}).
Every $A_{\infty}$ category is assumed to be small and c-unital (i.e. cohomological unital).
Every DG category/functor is regarded as an $A_{\infty}$ category/functor with no higher terms (see \cite[Section (1a)]{Seidelbook}).
Let $\cC$ be an $A_{\infty}$ category. Let $\text{\text{mod-}}\cC$ be the DG category of $A_{\infty}$ right $\cC$-module.
We have a cohomologically full and faithful Yoneda embedding $\cY_\cC: \cC \to \text{\text{mod-}}\cC$.

% \begin{comment}
% A {\it free} $\cC$-module is a direct sum of $\cC$-modules of the form $\cY_\cC(X)[k]$, where $X \in Ob(\cC)$. 
% A $\cC$-module $\cM$ is called { \it semi-free} if there exists a filtration $0=\cM_0 \subset \cM_1 \subset \dots =\cM$
% such that $\cM_{i+1}/\cM_i$ is free for all $i$.
% The full DG subcategory of semi-free $\cC$-modules is denoted by $\cS F-\cC$.
% A {\it finitely generated semi-free} $\cC$-module $\cM$ is a semi-free $\cC$-module
% such that $\cM_m=\cM$ for some $m \in \mathbb{N}$ and each  $\cM_{i+1}/\cM_i$ is a finite direct sum of shifted Yoneda modules.
% \end{comment}

\begin{defn}[\cite{Seidelbook2}, Section 7, or \cite{Or16} 2.2, 2.3]\label{d:PerfMod}
 The DG category of {\it  perfect} right $\cC$-modules $\cC^{perf}$ is the full DG subcategory of \text{mod-}$\cC$ consisting of all 
 $\cC$-modules that are quasi-isomorphic to  $\cC$-modules which can be constructed from image of $\cY_\cC$ and taking shifts, mapping cones and homotopy retracts.
 %are homotopy equivalent to a direct summand of a finitely generated semi-free $\cC$-module.
\end{defn}

%The homotopy category $H^0(\cC^{perf})$ is the smallest  full subcategory of $H^0(\text{mod-}\cC)$ that is triangulated, idempotent complete and containing 
%the image of Yoneda embbedings.

\begin{rmk}\label{r:K-projective}
 Every object in $\cC^{perf}$ is $K$-projective and $K$-injective in the sense of \cite[Lemma 1.16]{Seidelbook} so every $A_{\infty}$ functor (eg. DG functor) from $\cC^{perf}$
 preserves quasi-isomorphism.
 Together with the fact that $\cC^{perf}$ is pre-triangulated (ie. $H^0(\cC^{perf})$ is triangulated), 
 we know that its homotopy category coincides with its derived category (ie. $H^0(\cC^{perf})=D(\cC^{perf})$).
\end{rmk}

Let $\cC$ and $\cD$ be two $A_{\infty}$ categories.
%i.e. their cohomological categories are strictly unital.
Let $\cS:\cC^{perf} \to \cD^{perf}$ and $\cR: \cD^{perf} \to \cC^{perf}$  be two $A_{\infty}$ functors.
An {\it adjunction unit} $u_{rs}$ for the pair $(\cS,\cR)$ is an $A_{\infty}$ natural transformation
from the identity functor $Id_{\cC^{perf}}$ to $\cR \circ \cS$ such that
\begin{align*}
 hom_{\cD^{perf}}(\cS(c),d) \xrightarrow{\cR}  hom_{\cC^{perf}}(\cR(\cS(c)),\cR(d)) \xrightarrow{ -\circ u_{rs}(c)} hom_{\cC^{perf}}(c,\cR(d))
\end{align*}
is a quasi-isomorphism for all $c \in Ob(\cC^{perf})$ and $d \in Ob(\cD^{perf})$.
Dually, an {\it adjunction  counit} $c_{sr}$ for the pair $(\cS,\cR)$ is an $A_{\infty}$ natural transformation
from $\cS \circ \cR$ to $Id_{\cD^{perf}}$ such that
\begin{align*}
 hom_{\cC^{perf}}(c,\cR(d)) \xrightarrow{\cS}  hom_{\cD^{perf}}(\cS(c)),\cS(\cR(d))) \xrightarrow{  c_{sr}(d) \circ -} hom_{\cD^{perf}}(\cS(c),d)
\end{align*}
is a quasi-isomorphism for all $c \in Ob(\cC^{perf})$ and $d \in Ob(\cD^{perf})$.
If $(\cS,\cR)$ admits an adjunction unit or an adjunction counit, we call $\cR$ a {\it right adjoint} of $\cS$.
In this case, the induced natural transformation between the induced functors 
on the homotopy categories $H^0(\cC^{perf})$ and $H^0(\cD^{perf})$ is the adjunction unit or adjunction counit
in the ordinary sense so $[\cR]:H^0(\cD^{perf}) \to H^0(\cC^{perf})$ is a right adjoint of $[\cS]:H^0(\cC^{perf}) \to H^0(\cD^{perf})$
in the sense of triangulated categories.

A left adjoint of $\cS$ is an $A_{\infty}$ functor $\cL: \cD^{perf} \to \cC^{perf}$ such that 
$(\cL,\cS)$ admits an adjunction unit or an adjunction counit.

Now, suppose we have the following adjunction unit and counit $A_{\infty}$ natural transformations
\begin{align*}
 Id_{\cC^{perf}} \xrightarrow{u_{rs}} \cR \circ \cS \\
 \cS \circ \cR \xrightarrow{c_{sr}} Id_{\cD^{perf}}  \\
% Id_{\cD^{perf}} \xrightarrow{u_{sl}} \cS \circ \cL \\
 \cL \circ \cS \xrightarrow{c_{ls}} Id_{\cC^{perf}}  
\end{align*}

Let $\cF:\cC^{perf} \to \cC^{perf}$ and $\cT:\cD^{perf} \to \cD^{perf}$ be the mapping cone of the adjunction unit 
$u_{rs}$ and counit $c_{sr}$ in the DG category of $A_{\infty}$ functors $Fun(\cC^{perf},\cC^{perf})$
and $Fun(\cD^{perf},\cD^{perf})$, respectively.

\begin{thm}[\cite{AL17}, \cite{Me16}]\label{t:SphericalFun}
 If $\cF$ is an auto-equivalence and there is an natural equivalence of $A_{\infty}$ functors
 $\cR \simeq \cF \circ \cL$, then $\cT$ is an auto-equivalence.
 In this case, $\cS$ is called a spherical functor and $\cT$ is called the twist along
 the spherical functor $\cS$.
\end{thm}

\begin{proof}
 The Theorems in \cite{AL17} and \cite{Me16} are stated on the level of derived category and it implies Theorem \ref{t:SphericalFun} as follows:
 On the homotopy categories $H^0(\cC^{perf})$ and $H^0(\cD^{perf})$, $[\cR]$ is the right adjoint of $[\cS]$
 with adjunction units and counits $[u_{rs}]$, $[c_{sr}]$, respectively.
 Similarly, $[\cL]$ is the left adjoint of $[\cS]$ with adjunction counit $[c_{ls}]$.
 We also know that $[\cF]$ is an auto-equivalence and $[\cR] \simeq [\cF] \circ [\cL]$ so $[\cT]$ is an auto-equivalence by \cite{Me16}.
 Therefore, $\cT$ is an auto-equivalence.
\end{proof}

\begin{rmk}
 In \cite{AL17}, they prove Theorem \ref{t:SphericalFun} under a stronger assumption that 
  the composition
 \begin{align}
 \cR \xrightarrow{u_{sl}} \cR \circ \cS \circ \cL \to \cF \circ \cL \label{eq:compositionEguiv}
 \end{align}
 is an equivalence of DG functors for an adjunction unit $u_{sl} \in hom(Id_{\cD^{perf}}, \cS \circ \cL)$.
 Meachan provides a shorter proof of Theorem \ref{t:SphericalFun} and shows that any equivalence of DG functors
 $\cR \simeq \cF \circ \cL$ is sufficient to imply that the composition \eqref{eq:compositionEguiv} is an equivalence of DG functors.
 Therefore, we do not need to specify the equivalence $\cR \simeq \cF \circ \cL$ in Theorem \ref{t:SphericalFun}.
 %Moreover, 
\end{rmk}

%An $A_{\infty}$ functor $\cS:\cC \to \cD^{perf}$ between two $A_{\infty}$ categories $\cC$ and $\cD^{perf}$
%is called spherical if there are quasi-equivalences $q_C:\cC' \simeq \cC$ and $q_D:\cD^{perf} \simeq \cD^{perf}'$
%such that $q_D \circ \cS \circ q_C$ is a spherical DG functor between the DG categories $\cC'$ and $\cD^{perf}'$.
%We are going to prove the following in this section.

This section is devoted to the proof of Theorem \ref{t:Twist formula(alg)} and the structure goes as follows.
In Section \ref{ss:bimodules}, we recall 
the definitions and properties of various functors and adjunctons obtained from $A_{\infty}$ bimodules.
In Section \ref{ss:rSphericalTwist}, we introduce $R$-spherical objects for a $\K$-algebra $R$
and explain how to get an auto-equivalence from an $R$-spherical object using Theorem \ref{t:SphericalFun}.
Finally, in Section \ref{ss:R-spherical-Lag}, we explain the relation between the auto-equivalence and Dehn twist and prove Theorem \ref{t:Twist formula(alg)}.

\begin{comment}
\begin{prop}\label{p:Dehn=Spherical}
 If $P$ is a Lagrnagin $S^n/\Gamma$ and $\Gamma$ satisfies \eqref{eq:GammaCondition}, then $T_\eE$ (see \eqref{eq:TcE}) is an auto-equivalence which 
 is a twist along a spherical functor 
 \begin{align}
  \cS:\K[\Gamma]^{perf} \to \cF^{perf}
 \end{align}
 Here, $\K[\Gamma]$ is the group algebra regarded as an $A_{\infty}$-algebra concentrated at degree $0$.
 %Under the assumption of Theorem \ref{t:Twist formula}, $\cY_{\cF}(\tau_{P}(X))$ is quasi-isomorphic to $T_\eE(\cY_{\cF}(X))$ for any $X \in Ob(\cF)$.
\end{prop}
\end{comment}

\subsection{Bimodules}\label{ss:bimodules}

In this section, we gather some properties of $A_{\infty}$ bimodules which are well-known to experts.
We follow the sign conventions in \cite{LF2} (see \cite[Conventions $2.1$]{LF2}) and \cite{Gan13}.
We use $| \cdot|$ and $\| \cdot \|$ to denote the degree and the reduced degree (i.e. degree minus one) of an element, respectively.

\subsubsection{Tensor-Hom}

In this subsection, we want to explain the tensor-hom adjunction in the $A_{\infty}$ setting (Proposition \ref{p:TensorHom}).
%We use the sign convention in \cite[Section $2.4$]{Gan13} for the definition of DG category of $A_{\infty}$ bimodules.
We denote the DG category of $A_{\infty}$  $\cC\text{-}\cD$ bimodules by $\cC\text{-mod-}\cD$.
Let $\cC,\cC',\cD$ be $A_{\infty}$ categories and $\cB$ be an $A_{\infty}$ $\cC\text{-}\cD$ bimodule.
We have strictly unital $DG$ functors
\begin{align*}
-\otimes_\cC^{\mathbb{L}} \cB &:\cC'\text{-mod-}\cC \to \cC'\text{-mod-}\cD \\
 \cB \otimes_\cD^{\mathbb{L}} - &: \cD \text{-mod-}\cC' \to  \cC\text{-mod-}\cC' \\
 hom_{\text{mod-}\cD}(\cB,-)&: \cC'\text{-mod-}\cD  \to \cC'\text{-mod-}\cC \\
 hom_{\cC-mod}(\cB,-)&: \cC\text{-mod-}\cC'  \to \cC'\text{-mod-}\cD \\
 hom_{\text{mod-}\cD}(-,\cB)&: \cC'\text{-mod-}\cD  \to \cC\text{-mod-}\cC' \\
 hom_{\cC-mod}(-,\cB)&: \cC\text{-mod-}\cC'  \to \cC'\text{-mod-}\cD 
\end{align*}
Some formulae defining the functors above are explicitly given in \cite{Gan13}. 
We want to spell out more details about these functors for our applications.

\begin{defnlemma}\label{dl:tensorDGFun}
Let $\cB$ be an $A_{\infty}$ $\cC\text{-}\cD$ bimodule. 
Then we have a strictly unital $DG$ functor $F_{\otimes}(\cB):=-\otimes_\cC^{\mathbb{L}} \cB :\text{mod-}\cC \to \text{mod-}\cD$.
\end{defnlemma}

\begin{proof}
It is most familiar if $\cB$ is a DG bimodule.
In the $A_{\infty}$ context, the formulae are given explicitly as follows.

 For an $A_{\infty}$ right $\cC$ module $V$, the $A_{\infty}$ right $\cD$ module $V \otimes_\cC^{\mathbb{L}} \cB$ is defined by the following data
 $$V \otimes_\cC^{\mathbb{L}} \cB(X):=V \otimes_\K T\cC \otimes_\K \cB(X)$$
 is a cochain complex with differential
 \begin{align*}
  \mu^{1|0}_{V \otimes_\cC^{\mathbb{L}} \cB}(v \otimes g_s \dots g_1 \otimes x)&=\sum_k (-1)^{\maltese_0^k}\mu^{1|s-k}_V(v,g_s \dots ,g_{k+1})\otimes g_k\dots g_1 \otimes x \\
  &+\sum_{k,l} (-1)^{\maltese_0^k} v\otimes g_s \dots g_{k+l+1} \otimes \mu_{\cC}^{j}(g_{k+l} \dots g_{k+1})\otimes g_k \dots g_1 \otimes x \\
  &+\sum_k v\otimes g_s \dots g_{k+1} \otimes \mu^{k|1|0}_{\cB}(g_{k}, \dots g_1;x)
 \end{align*}
 where $\maltese_0^k= \sum_{j=1}^k \| g_j\|+|x|$.
 The higher right $\cD$ module structure is given by
 \begin{align*}
 %&\mu^{1|d}_{V \otimes_\cC^{\mathbb{L}} \cB}:V \otimes_\cC^{\mathbb{L}} \cB(X_d) \times hom_\cD(X_{d-1},X_d) \times \dots \times 
 %hom_\cD(X_0,X_1) \to V \otimes_R^{\mathbb{L}} \cB(X_0)\\
 & \mu^{1|d}_{V \otimes_\cC^{\mathbb{L}} \cB}(v \otimes g_s \dots g_1 \otimes x,a_d,\dots,a_1)
  =\sum_j v\otimes g_s \dots \otimes g_{j+1} \otimes \mu_{\cB}^{j|1|d}(g_j,\dots,x,\dots,a_1)
 \end{align*}
 The fact that $(\cB,\mu^{r|1|s}_{\cB})$ is an $A_{\infty}$ $\cC\text{-}\cD$ bimodule implies that $(V \otimes_R^{\mathbb{L}} \cB, \mu^{1|d}_{V \otimes_R^{\mathbb{L}} \cB})$
 is an $A_{\infty}$ right $\cD$ module.

 On the morphism level, $hom_{\text{mod-}\cC}(V_0,V_1) \to hom_{\text{mod-}\cD}(V_0  \otimes_\cC^{\mathbb{L}} \cB,V_1  \otimes_\cC^{\mathbb{L}} \cB)$ is given by
 \begin{align*}
  &\psi \mapsto (F_{\otimes}(\cB))^1(\psi):=(t_{\psi}^{1|d})_d \\
  &t_{\psi}^{1|0}(v \otimes g_s \dots g_1 \otimes x):= \sum_k (-1)^{|\psi| \cdot \maltese_0^k}\psi^{1|s-k}(v, g_s, \dots ,g_{k+1}) \otimes g_{k} \dots \otimes x \\
  &t_{\psi}^{1|d}(v \otimes g_s \dots g_1 \otimes x,a_d,\dots,a_1):=0
 \end{align*}
 where $\maltese_0^k= \sum_{j=1}^k \| g_j\|+|x|$.
 As an $A_{\infty}$ functor, we define $(F_{\otimes}(\cB))^k=0$ for $k \ge 2$.
 
 We want to show that $F_{\otimes}(\cB)$ defines an $A_{\infty}$ functor.
 Therefore, we need to show that
 \begin{align}
  \mu_{\text{mod-}\cD}^1((F_{\otimes}(\cB))^1(\psi))=&(F_{\otimes}(\cB))^1(\mu^1_{\text{mod-}\cC}(\psi)) \label{eq:mu1DG}\\
  \mu_{\text{mod-}\cD}^2((F_{\otimes}(\cB))^1(\psi_1),(F_{\otimes}(\cB))^1(\psi_0))=&(F_{\otimes}(\cB))^1(\mu^2_{\text{mod-}\cC}(\psi_1,\psi_0)) \label{eq:mu2DG}
 \end{align}

 We first consider \eqref{eq:mu1DG}.
 By definition (see \cite[Definition 2.18]{Gan13}, \cite[Equation (1.3)]{Seidelbook}), we have
 \begin{align}
  \mu_{\text{mod-}\cD}^1(t_{\psi})=(-1)^{|t_\psi|}\mu_{V_1 \otimes_\cC^{\mathbb{L}} \cB} \circ \widehat{t}_{\psi}-t_{\psi} \circ \widehat{\mu}_{V_0 \otimes_\cC^{\mathbb{L}} \cB}
 \end{align}
 %\begin{align*}
 % (\mu_{\text{mod-}\cD}^1(t_{\psi}))^1(v \otimes g_s \dots g_1 \otimes x) 
 % =&(-1)^{|t_\psi|}\mu_{V_1 \otimes_\cC^{\mathbb{L}} \cB} \circ \widehat{t}_{\psi}(v\otimes\dots \otimes x) 
 % -t_{\psi} \circ \widehat{\mu}_{V_0 \otimes_\cC^{\mathbb{L}} \cB}(v\otimes \dots \otimes x)  \\
 % =&(-1)^{|t_\psi|}\mu^{1|0}_{V_1 \otimes_\cC^{\mathbb{L}} \cB}(t^{1|0}_{\psi}(v\otimes\dots \otimes x)) 
 % -t^{1|0}_{\psi}(\mu^{1|0}_{V_0 \otimes_\cC^{\mathbb{L}} \cB}(v\otimes \dots \otimes x))
 % \end{align*}
  where $\widehat{t}_{\psi}$, $\widehat{\mu}_{V_0 \otimes_\cC^{\mathbb{L}} \cB}$ are the hat extension of 
  $t_\psi$ and $\mu_{V_0 \otimes_\cC^{\mathbb{L}} \cB}$, respectively (see \cite[Remark $2.5$]{Gan13}) and the $(-1)^{|t_\psi|}$ factor
  being in the front of $\mu_{V_1 \otimes_\cC^{\mathbb{L}} \cB} \circ \widehat{t}$ instead of $t_{\psi} \circ \widehat{\mu}_{V_0 \otimes_\cC^{\mathbb{L}} \cB}$
  comes from the convention of converting a DG category to an $A_{\infty}$ category (i.e. $\mu^1(a)=(-1)^{|a|}\partial (a)$).
  
  We first calculate $(\mu_{\text{mod-}\cD}^1(t_{\psi}))^{1|0}(v \otimes g_s  \dots g_1 \otimes x)$. One can easily check that
 \begin{align*}
  &\mu^{1|0}_{V_1 \otimes_\cC^{\mathbb{L}} \cB}(t^{1|0}_{\psi}(v \otimes \dots\otimes x)) \\
  =&\sum (-1)^{|\psi| \cdot \maltese_0^{k}} \mu_{V_1 \otimes_\cC^{\mathbb{L}} \cB}(\psi(v \dots g_{k+1}) \otimes g_k \dots \otimes x) \\
  =& \sum (-1)^{|\psi| \cdot \maltese_0^{k}+\maltese_0^{l}} \mu_{V_1}(\psi(v \dots g_{k+1}) \dots g_{l+1}) \dots  x 
  +  \sum (-1)^{|\psi| \cdot \maltese_0^{k}+\maltese_0^{l}} \psi(v \dots g_{k+1}) \dots \mu_\cC( \dots g_{l+1}) \dots x \\
  &+\sum (-1)^{|\psi| \cdot \maltese_0^{k}} \psi(v \dots g_{k+1})  \dots \mu_\cB( \dots x) \\
 \end{align*}
 \begin{align*}
  & t^{1|0}_{\psi}(\mu^{1|0}_{V_0 \otimes_\cC^{\mathbb{L}} \cB}(v \otimes\dots\otimes x))\\
  =&\sum (-1)^{\maltese_0^{k}} t_{\psi}(\mu_{V_0}(v \dots g_{k+1})\otimes g_k \dots \otimes x)
  + \sum (-1)^{\maltese_0^{k}} t_{\psi}(v \dots g_{l+1}\otimes\mu_\cC(g_l \dots,g_{k+1})\otimes g_k \dots x) \\
  &+ \sum  t_{\psi}(v \dots g_{k+1} \otimes \mu_{\cB}(g_k  \dots x)) \\  
  =& \sum (-1)^{\maltese_0^{k}+|\psi| \cdot \maltese_0^{l}} \psi(\mu_{V_0}(v \dots g_{k+1}) \dots g_{l+1}) \dots x  
  +  \sum (-1)^{\maltese_0^{k}+|\psi| \cdot \maltese_0^{l}} \psi(v \dots \mu_\cC( \dots g_{k+1}) \dots g_{l+1}) \dots x \\
  &+\sum (-1)^{\maltese_0^{k}+|\psi| \cdot (\maltese_0^{l}+1)} \psi(v \dots g_{l+1}) \dots \mu_\cC( \dots g_{k+1})\dots x  
  + \sum (-1)^{|\psi| \cdot (\maltese_0^{l}+1)}  \psi(v \dots g_{l+1})  \dots \mu_\cB(g_k \dots x) \\
  %=& \sum (-1)^{\maltese_{k+1}^{s+1}+\maltese_0^{l}} \psi(\mu_{V_0}(v \dots g_{k+1}) \dots g_{l+1}) \dots x  + \sum (-1)^{\maltese_{k+1}^{s+1}+\maltese_0^{l}} \psi(v \dots \mu_\cC( \dots g_{k+1}) \dots g_{l+1}) \dots x \\
  %&+\sum (-1)^{\maltese_{k+1}^{s+1}+\maltese_0^{l}} \mu_{V_1}(\psi(v \dots g_{k+1}) \dots g_{l+1}) \dots  x 
   \end{align*}
 where  $\maltese_0^{k}=|x|+\sum_{j=1}^k \|g_j\|$ and $\maltese_0^{l}=|x|+\sum_{j=1}^l \|g_j\|$.
 On the other hand, we calculate $(t_{\mu^1_{\text{mod-}\cC}(\psi)})^{1|0}(v \otimes g_s  \dots g_1 \otimes x)$:
 \begin{align*}
  (t_{\mu_{V_1} \circ \widehat{\psi}})^{1|0}(v \otimes g_s \dots g_1 \otimes x)
  %=&\sum (-1)^{\maltese_0^{l}} \mu^1_{\text{mod-}\cC}\psi(v \otimes g_s \dots g_{l+1}) \otimes g_l \dots x \\
  =&\sum (-1)^{\maltese_{l+1}^{k}+(|\psi|+1) \cdot \maltese_0^{l}} \mu_{V_1}(\psi(v \dots g_{k+1}) \dots g_{l+1}) \dots  x  \\
  (t_{\psi \circ \widehat{\mu}_{V_0}})^{1|0}(v \otimes g_s \dots g_1 \otimes x)
  %=&\sum (-1)^{\maltese_0^{l}} \mu^1_{\text{mod-}\cC}\psi(v \otimes g_s \dots g_{l+1}) \otimes g_l \dots x \\
  =& \sum (-1)^{\maltese_{l+1}^{k}+(|\psi|+1) \cdot \maltese_0^{l}} \psi(\mu_{V_0}(v \dots g_{k+1}) \dots g_{l+1}) \dots x  \\
  &+ \sum (-1)^{\maltese_{l+1}^{k}+(|\psi|+1) \cdot \maltese_0^{l}} \psi(v \dots \mu_\cC( \dots g_{k+1}) \dots g_{l+1}) \dots x 
 \end{align*}
 where $\maltese_{l+1}^{k}=\sum_{j=l+1}^k \|g_j\|$.
 Combining all of them gives  $(\mu_{\text{mod-}\cD}^1(t_{\psi}))^{1|0}=(t_{\mu^1_{\text{mod-}\cC}\psi})^{1|0}$, which is the 
 verification of \eqref{eq:mu1DG} up to first order.
 
 For the higher order terms, we have
 \begin{align*}
  &(\mu_{\text{mod-}\cD}^1(t_{\psi}))^{1|d}(v \otimes g_s \dots g_1 \otimes x,a_d,\dots,a_1) \\
  =& (-1)^{|t_{\psi}|} \mu_{V_1  \otimes_\cC^{\mathbb{L}} \cB} \circ \widehat{t}_{\psi} (v \otimes g_s \dots g_1 \otimes x,a_d,\dots,a_1))
  -t_{\psi} \circ \widehat{\mu}_{V_0  \otimes_\cC^{\mathbb{L}} \cB}(v \otimes g_s \dots g_1,a_d,\dots,a_1)\\
  =& \sum (-1)^{|\psi|+|\psi| \cdot (A_1^d+\maltese_{0}^{l})} \psi(v,\dots,g_{l+1}) \dots \mu_{\cB}(g_k,\dots,a_1) \\
   &-\sum (-1)^{|\psi|(A_1^d+\maltese_{0}^{l}+1)} \psi(v,\dots,g_{l+1}) \dots \mu_{\cB}(g_k,\dots,a_1) \\
  =&0  =(t_{\mu^1_{\text{mod-}\cC}\psi})^{1|d}(v \otimes g_s \dots g_1 \otimes x,a_d,\dots,a_1)
 \end{align*}
 where $A_1^d=\sum_{j=1}^d \|a_j\|$ and $\maltese_{0}^{l}=|x|+\sum_{j=1}^l \|g_j\|$.
 Therefore, \eqref{eq:mu1DG} holds so $(F_{\otimes}(\cB))^1$ is a chain map.

 It is straightforward to check that $\mu_{\text{mod-}\cD}^2(t_{\psi_1},t_{\psi_0})=(-1)^{|t_{\psi_0}|} t_{\psi_1} \circ \widehat{t}_{\psi_0}=t_{\mu_{\text{mod-}\cC}^2(\psi_1,\psi_0)}$ 
 (i.e. \eqref{eq:mu2DG} holds)
 and $Id_{V \otimes_\cC^{\mathbb{L}} \cB}=t_{Id_V}$.
 Therefore, $F_{\otimes}(\cB):\text{mod-}\cC \to \text{mod-}\cD$ defines a strictly unital DG functor.
 
\end{proof}

\begin{comment}
\begin{rmk}
 In the above proof, $\text{mod-}R$ refers to the DG category of $A_{\infty}$ right $R$ module so 
 it is possible that $\psi(v,g_s,\dots,g_1) \neq 0$ even if $s \ge 1$.

\end{rmk}
\end{comment}

In Definition/Lemma \ref{dl:tensorDGFun}, we associate a strictly unital DG functor $F_{\otimes}(\cB)$ to 
an $A_{\infty}$ $\cC\text{-}\cD$ bimodule $\cB$.
This can be upgraded to a strictly unital DG functor from 
the DG category of $A_{\infty}$ $\cC\text{-}\cD$ bimodules to the DG category of $A_{\infty}$-functors from $\text{mod-}\cC$ to $\text{mod-}\cD$, which is denoted by
$Fun(\text{mod-}\cC,\text{mod-}\cD)$, as follows:

\begin{defnlemma}\label{dl:BimodFunMor}(cf. \cite{SeHom})
There is a strictly unital DG functor $F_{\otimes}: \cC\text{-mod-}\cD \to Fun(\text{mod-}\cC,\text{mod-}\cD)$ such that $F_{\otimes}(\cB)=-\otimes_\cC^{\mathbb{L}} \cB$.
\end{defnlemma}

\begin{proof}
Let $\cB_0,\cB_1$ be $A_{\infty}$  $\cC\text{-}\cD$ bimodules and $t:\cB_0 \to \cB_1$ be an $A_{\infty}$ bimodule (pre-)homomorphism.
We want to define an $A_{\infty}$ (pre-)natural tranformation $F_{\otimes}^1(t)$ between the DG functors $ F_{\otimes}(\cB_0)$ and $F_{\otimes}(\cB_1)$.
Let $F_i:=F_{\otimes}(\cB_i)$.
 For each $V \in Ob(\text{mod-}\cC)$, we define
 \begin{align*}
 &(F_{\otimes}^1(t))^0_V \in hom_{\text{mod-}\cD}(F_0(V),F_1(V)) \\
%  ((F_{\otimes}^1(t))^0_V)^{1|0}(v \otimes g_s \dots g_1 \otimes x)&:= \sum v \otimes \dots t( \dots \otimes x) \\
  &((F_{\otimes}^1(t))^0_V)^{1|r}(v \otimes g_s \dots g_1 \otimes x,a_r,\dots,a_1):= \sum_j v \otimes \dots g_{j+1} \otimes t^{j|1|r}(g_j \dots \otimes x, a_r, \dots,a_1)
 \end{align*}
 The higher order terms $(F_{\otimes}^1(t))^r$ of the (pre-)natural tranformation $F_{\otimes}^1(t)$ are defined to be zero (for $r \ge 1$).
 Since we claim that $F_{\otimes}$ is a DG functor, we define $F_{\otimes}^k=0$ for $k \ge 2$ as an $A_{\infty}$ functor.
 We need to verify the analogue of \eqref{eq:mu1DG} and \eqref{eq:mu2DG}.
 
 To see that $F_{\otimes}^1$ is a chain map (i.e. the analogue of \eqref{eq:mu1DG} holds), one can check that for the first order:
 \begin{align*}
  &(\mu^1_{Fun(\text{mod-}\cC,\text{mod-}\cD)}(F_{\otimes}^1(t)))^0(v \otimes g_s \dots g_1 \otimes x,a_r,\dots,a_1) \\
  =&\mu^1_{\text{mod-}\cD}((F_{\otimes}^1(t))^0)(v \otimes g_s \dots g_1 \otimes x,a_r,\dots,a_1) \\
  =&((-1)^{|(F_{\otimes}^1(t))^0|}\mu_{F_1(V)} \circ \widehat{(F_{\otimes}^1(t))}^0- (F_{\otimes}^1(t))^0 \circ \widehat{\mu}_{F_0(V)})(v \otimes g_s \dots g_1 \otimes x,a_r,\dots,a_1) \\
  =& (F_{\otimes}^1((-1)^{|t|}\mu_{B_1} \circ \widehat{t} - t \circ \widehat{\mu}_{B_0}))^0(v \otimes g_s \dots g_1 \otimes x,a_r,\dots,a_1) \\
  =& (F_{\otimes}^1(\mu_{\cC\text{-mod-}\cD}^1(t)))^0(v \otimes g_s \dots g_1 \otimes x,a_r,\dots,a_1)
 \end{align*}
Some cancellation is made from the third line to the fourth line but it is straightforward so we do not spell it out.

We also need to check the higher order terms: for $\phi \in hom_{\text{mod-}\cC}(V_0,V_1)$, we have (see \cite[Equation $(1.8)$]{Seidelbook} for the sign convention of DG category of functors)
\begin{align*}
 (\mu^1_{Fun(\text{mod-}\cC,\text{mod-}\cD)}(F_{\otimes}^1(t)))^1(\phi):=\mu^2_{\text{mod-}\cD}(F_1(\phi),(F_{\otimes}^1(t))^0)+(-1)^{\|t\| \cdot \|\phi\|}\mu^2_{\text{mod-}\cD}((F_{\otimes}^1(t))^0,F_0(\phi))
\end{align*}
We can check that
\begin{align*}
 %&((\mu^1_{Fun(\text{mod-}\cC,\text{mod-}\cD)}(F_{\otimes}^1(t)))^1(\phi))(v \otimes g_s \dots g_1 \otimes x,a_r,\dots,a_1) \\
 &(\mu^2_{\text{mod-}\cD}(F_1(\phi),(F_{\otimes}^1(t))^0)+(-1)^{\|t\| \cdot \|\phi\|}\mu^2_{\text{mod-}\cD}((F_{\otimes}^1(t))^0,F_0(\phi)))(v \otimes g_s \dots g_1 \otimes x,a_r,\dots,a_1) \\
 =&((-1)^{|t|}F_1(\phi) \circ \widehat{(F_{\otimes}^1(t))^0)}+(-1)^{\|t\| \cdot \|\phi\|+|\phi|}(F_{\otimes}^1(t))^0 \circ \widehat{F_0(\phi)})(v \otimes g_s \dots g_1 \otimes x,a_r,\dots,a_1) \\
 =&\sum (-1)^{|t|+(A_1^r + \maltese_0^l+|t|)|\phi|} \phi(v,\dots,g_{l+1})\otimes g_{l} \dots g_{k+1} \otimes t(g_k,\dots, g_1, x,a_r,\dots,a_1) \\
 &+\sum (-1)^{\|t\| \cdot \|\phi\|+|\phi|+(A_1^r + \maltese_0^l)|\phi|} \phi(v,\dots,g_{l+1})\otimes g_{l} \dots g_{k+1} \otimes  t(g_k,\dots, g_1, x,a_r,\dots,a_1) \\
 =&0=(F_{\otimes}^1(\mu_{\cC\text{-mod-}\cD}^1(t)))^1
\end{align*}
where $A_{1}^{r}=\sum_{j=1}^r \|a_j\|$ and $\maltese_{0}^{l}=|x|+\sum_{j=1}^l \|g_j\|$.
Therefore, $F_{\otimes}^1$ is a chain map.
%For $k>1$, all the terms in $(\mu^1_{Fun(\text{mod-}\cC,\text{mod-}\cD)}(F_{\otimes}^1(t)))^k$ involves $\mu^l_{\text{mod-}\cD}$ for some $l>2$ so 
% $(\mu^1_{Fun(\text{mod-}\cC,\text{mod-}\cD)}(F_{\otimes}^1(t)))^k=0=(F_{\otimes}^1(\mu_{\cC\text{-mod-}\cD}^1(t)))^k$.
%It implies that $F_{\otimes}^1$ is a chain map.

It is straightforward to check that
\begin{align*}
 &\mu^2_{Fun(\text{mod-}\cC,\text{mod-}\cD)}(F_{\otimes}^1(t_1),F_{\otimes}^1(t_0)) =F_{\otimes}^1(\mu^2_{\cC\text{-mod-}\cD}(t_1, t_0)) \\
 &(F_{\otimes}^1(Id_{\cB}))=Id_{F_{\otimes}(\cB)}
\end{align*}
so $F_{\otimes}$ is a strictly unital DG functor.

\end{proof}

Using {\it hom} instead of {\it tensor}, we can construct a strictly unital DG functor from $\text{mod-}\cD$ to $\text{mod-}\cC$ by an $A_{\infty}$  $\cC\text{-}\cD$ bimodule.

\begin{defnlemma}\label{dl:homDGfun}
 $F_{hom}(\cB):=hom_{\text{mod-}\cD}(\cB,-): \text{mod-}\cD \to \text{mod-}\cC$ is a strictly unital DG functor.
\end{defnlemma}

\begin{proof}
 For $\cM \in Ob(\text{mod-}\cD)$, we have a cochain complex $(F_{hom}(\cB))(\cM):=hom_{\text{mod-}\cD}(\cB,\cM)$ 
 %with the differential given by 
 %\begin{align}
 % \mu^{1|0}_{(F_{hom}(\cB))(\cM)}(\psi):=\mu^1_{\text{mod-}\cD}
 %\end{align}
 %$$hom_{\text{mod-}\cA}(\cB_{\cF_\eE},X) \times R \dots \times R \to hom_{\text{mod-}\cA}(\cB_{\cF_\eE},X)$$
 The right $\cC$-module structure on $(F_{hom}(\cB))(\cM)$ is given by
 \begin{align}
 &\mu^{1|0}_{(F_{hom}(\cB))(\cM)}(\psi):=(-1)^{|\psi|}\mu^1_{\text{mod-}\cD}(\psi) \\
 &\mu^{1|s}_{(F_{hom}(\cB))(\cM)}(\psi,g_s,\dots,g_1)=(\psi^{1|r}_{g_s,\dots,g_1})_r  \text{ for }s>0\\ 
 &\psi^{1|r}_{g_s,\dots,g_1}(x,a_r,\dots,a_1)=\sum_k (-1)^{\dagger_k}\psi^{1|k}(\mu_\cB^{s|1|r-k}(g_s,\dots,g_1,x,a_r,\dots,a_{k+1}),a_k,\dots,a_1) \label{eq:RightGaction}
 \end{align}
 where $\dagger_k=(\sum_{j=1}^s \|g_j\|)(|\psi|+\sum_{j=1}^r \|a_r\|+|x|)+ |\psi|+\sum_{j=1}^k \|a_k\|+1$.
It is straightforward but involed to check that  the $\cC\text{-}\cD$ bimodule relations of $\cB$ implies that $(F_{hom}(\cB))(\cM)$ is a right $\cC$-module.
 
 For $\cM_0,\cM_1 \in Ob(\text{mod-}\cD)$, the map on morphism space is given by
 \begin{align}
& hom_{\text{mod-}\cD}(\cM_0,\cM_1) \to hom_{\text{mod-}\cC}((F_{hom}(\cB))(\cM_0),(F_{hom}(\cB))(\cM_1))  \nonumber\\
& t_1=(t_1^{1|d})_d \mapsto (F_{hom}(\cB))^1(t_1) \nonumber \\
& ((F_{hom}(\cB))^1(t_1))^{1|0}(t_0)=(-1)^{|t_0|}\mu^2_{\text{mod-}\cD}(t_1,t_0)=t_1 \circ \widehat{t}_0 \label{eq:hom-morphism}\\
& ((F_{hom}(\cB))^1(t_1))^{1|s}(t_0,g_s,\dots,g_1)=0 \text{ for } s>0 \nonumber
 \end{align}
It is routine to check that $\mu^1_{\text{mod-}\cC}((F_{hom}(\cB))^1(t_1))=(F_{hom}(\cB))^1(\mu^1_{\text{mod-}\cD}(t_1))$ so $(F_{hom}(\cB))^1$ is a chain map.
As an $A_{\infty}$ functor, we define $(F_{hom}(\cB))^k=0$ for $k \ge 2$.

Moreover, 
\begin{align*}
(F_{hom}(\cB))^1(\mu^2_{\text{mod-}\cD}(t_2, t_1))(t_0)&=(-1)^{|t_1|}t_2 \circ \widehat{t}_1 \circ \widehat{t}_0=\mu^2_{\text{mod-}\cC}((F_{hom}(\cB))^1(t_2), (F_{hom}(\cB))^1(t_1))(t_0) \\
 (F_{hom}(\cB))^1(Id_{\cM})(t_0)&=(-1)^{|t_0|}\mu^2_{\text{mod-}\cD}(Id_{\cM},t_0)=t_0=Id_{(F_{hom}(\cB))(\cM)}(t_0)
\end{align*} so
$F_{hom}(\cB)$ is a strictly unital DG functor.

\end{proof}

We use $(\cC\text{-mod-}\cD)^{opp}$ to denote the opposite category of $(\cC\text{-mod-}\cD)$.
By definition, the objects in $(\cC\text{-mod-}\cD)^{opp}$ are the same as the objects in $(\cC\text{-mod-}\cD)$.
Given two objects $\cB_0,\cB_1$, we have $hom_{(\cC\text{-mod-}\cD)^{opp}}(\cB_0,\cB_1)=hom_{\cC\text{-mod-}\cD}(\cB_1,\cB_0)$
as vector spaces.
The $A_{\infty}$-structure on $(\cC\text{-mod-}\cD)^{opp}$ is given by
\begin{align}
\mu^1_{(\cC\text{-mod-}\cD)^{opp}}(t)&=\mu^1_{\cC\text{-mod-}\cD}(t) \\
\mu^2_{(\cC\text{-mod-}\cD)^{opp}}(t_0,t_1)&=(-1)^{\|t_0\| \|t_1\|-1}\mu^2_{\cC\text{-mod-}\cD}(t_1,t_0)
\end{align}
and $\mu^k_{(\cC\text{-mod-}\cD)^{opp}}=0$ for $k \ge 3$.

\begin{rmk}\label{r:oppSignCon}
 This definition of opposite category is the same as the one in \cite[Definition 3.6]{NickSign} except an additional factor of $-1$ on $\mu^2$.
 We use this sign convention so that the following holds.
\end{rmk}

\begin{defnlemma}\label{dl:homDGfunMor}
There is a strictly unital DG functor $F_{hom}:(\cC\text{-mod-}\cD)^{opp} \to Fun(\text{mod-}\cD,\text{mod-}\cC)$ such that $F_{hom}(\cB)=hom_{\text{mod-}\cD}(\cB,-)$.
 
\end{defnlemma}

\begin{proof}
Let $\cB_0,\cB_1$ be $A_{\infty}$  $\cC\text{-}\cD$ bimodules and 
$t \in hom_{(\cC\text{-mod-}\cD)^{opp}}(\cB_0,\cB_1)=hom_{(\cC\text{-mod-}\cD)}(\cB_1,\cB_0)$.
%be an $A_{\infty}$ (pre-)bimodule homomorphism.
We want to define an $A_{\infty}$ (pre-)natural tranformation $F_{hom}^1(t)$ between the DG functors $ F_{hom}(\cB_0)$ and $F_{hom}(\cB_1)$.
Let $F_i:=F_{hom}(\cB_i)$.
 For $\cM \in Ob(\text{mod-}\cD)$ and $\psi \in F_0(\cM)$, we define
 \begin{align*}
 &(F_{hom}^1(t))^0_\cM \in hom_{\text{mod-}\cC}(F_0(\cM),F_1(\cM)) \\
%  ((F_{\otimes}^1(t))^0_V)^{1|0}(v \otimes g_s \dots g_1 \otimes x)&:= \sum v \otimes \dots t( \dots \otimes x) \\
& ((F_{hom}^1(t))^0_\cM)^{1|s}(\psi,g_s \dots g_1) \in F_1(\cM) \\
  &(((F_{hom}^1(t))^0_\cM)^{1|s}(\psi,g_s \dots g_1))^{1|r}(x, a_r, \dots,a_1)
  := \sum_k (-1)^{\dagger_k}\psi^{1|k}(t^{s|1|r-k}( g_s,\dots,x, \dots a_{k+1}),a_k,\dots,a_1)
 \end{align*}
 where $\dagger_k=\sum_{j=1}^s \|g_j\| (|\psi|+\sum_{j=1}^r \|a_j\|+|x|)+|t|(|\psi|+\sum_{j=1}^k \|a_j\|)$.
 The higher order terms $(F_{hom}^1(t))^r$ of the (pre-)natural tranformation $F_{hom}^1(t)$ are defined to be zero (for $r \ge 1$).
 We define $F_{hom}^k=0$ for $k \ge 2$ as an $A_{\infty}$ functor.
 
 To see that $F_{hom}^1$ is a chain map (i.e. the analogue of \eqref{eq:mu1DG} holds), we can check that for the first order:
 \begin{align*}
  &((\mu^1_{Fun(\text{mod-}\cD,\text{mod-}\cC)}(F_{hom}^1(t)))^0)^{1|s}(\psi,g_s \dots g_1))^{1|r}(x,a_r,\dots,a_1) \\
  =&((\mu^1_{\text{mod-}\cC}((F_{hom}^1(t))^0))^{1|s}(\psi,g_s \dots g_1))^{1|r}(x,a_r,\dots,a_1) \\
  =&(((-1)^{|t|}\mu_{F_1(\cM)} \circ \widehat{(F_{hom}^1(t))}^0-((F_{hom}^1(t))^0 \circ \widehat{\mu}_{F_0(\cM)}))^{1|s}(\psi,g_s \dots g_1))^{1|r}(x,a_r,\dots,a_1) \\
  %=&\sum_{k,h,l}\psi(\mu_{\cB_0}(g_s,\dots,g_{k+1},t(g_k,\dots,g_1,x,a_r,\dots,a_{r-l+1}),a_{r-l},\dots,a_{h+1}),a_h,\dots, a_1)) \\
  %&+\sum_{k,h,l} \psi(t(g_s,\dots,g_{k+1},\mu_{\cB_1}(g_k,\dots,g_1,x,a_r,\dots,a_{r-l+1}),a_{r-l},\dots,a_{h+1}),a_h,\dots, a_1)) \\
  =& ((F_{hom}^1((-1)^{|t|}\mu_{B_0} \circ \widehat{t} -t \circ \widehat{\mu}_{B_1}))^0)^{1|s}(\psi,g_s \dots g_1))^{1|r}(x,a_r,\dots,a_1) \\
  =& (((F_{hom}^1(\mu_{(\cC\text{-mod-}\cD)^{opp}}^1(t)))^0)^{1|s}(\psi,g_s \dots g_1))^{1|r}(x,a_r,\dots,a_1)
 \end{align*}
 For the higher order, we can check that
 \begin{align*}
  &(\mu^1_{Fun(\text{mod-}\cD,\text{mod-}\cC)}(F_{hom}^1(t)))^1(\phi) \\
  :=&\mu^2_{\text{mod-}\cC}(F_1(\phi),(F_{hom}^1(t))^0_{\cM_0})+(-1)^{\|t\| \|\phi\|}\mu^2_{\text{mod-}\cC}((F_{hom}^1(t))^0_{\cM_1},F_0(\phi))
 \end{align*}
 vanishes for every $\phi \in hom_{\text{mod-}\cD}(\cM_0,\cM_1)$ and every $\cM_0,\cM_1 \in Ob(\text{mod-}\cD)$.

% Also, it is routine to check that
% \begin{align*}
%(\mu^1_{Fun(\text{mod-}\cD,\text{mod-}\cC)}(F_{hom}^1(t)))^k=0=((F_{hom}^1(\mu_{(\cC\text{-mod-}\cD)^{opp}}^1(t)))^k  
% \end{align*}
%for all $k \ge 1$, so $F_{hom}^1$ is a chain map.

We can also check that (note the sign convention of the opposite category, see Remark \ref{r:oppSignCon})
\begin{align*}
 &\mu^2_{Fun(\text{mod-}\cD,\text{mod-}\cC)}(F_{hom}^1(t_1), F_{hom}^1(t_0)) = F_{hom}^1(\mu^2_{(\cC\text{-mod-}\cD)^{opp}}(t_1, t_0)) \\
 &(F_{hom}^1(Id_{\cB}))=Id_{F_{hom}(\cB)}
\end{align*}
so $F_{hom}$ is a strictly unital DG functor.
\end{proof}

\begin{rmk}
 There is also a strictly unital DG functor 
 $F_{hom}':\cC\text{-mod-}\cD \to Fun(\text{mod-}\cD,\cC-mod)$ such that
 $F_{hom}'(\cB):=hom_{\text{mod-}\cD}(-,\cB)$ for $\cB \in Ob(\cC\text{-mod-}\cD)$.
\end{rmk}

\begin{prop}[Tensor-Hom adjunction]\label{p:TensorHom}
 The DG functor $hom_{\text{mod-}\cD}(\cB,-)$ is right adjoint to 
 the DG functor $-\otimes_\cC^{\mathbb{L}} \cB$.
\end{prop}

\begin{proof}
 The adjunction unit $u$ is defined by
 \begin{align*}
  &u_V^0 \in hom_{\text{mod-}\cC}( V, hom_{\text{mod-}\cD}(\cB,V \otimes_\cC^{\mathbb{L}} \cB)) \text{ for }V \in Ob(\text{mod-}\cC)\\
  &u_V^0(v,g_s,\dots,g_1):=v \otimes g_s,\dots,g_1 \otimes Id_\cB  \in hom_{\text{mod-}\cD}(\cB,V \otimes_\cC^{\mathbb{L}} \cB)\\
  &(v \otimes g_s,\dots,g_1 \otimes Id_\cB  )^{1|0}(x):=(-1)^{|x|(|v|+\sum_{j=1}^s \|g_j\|)}v \otimes g_s,\dots,g_1 \otimes x \\
  &(v \otimes g_s,\dots,g_1 \otimes Id_\cB )^{1|r}(x,a_r,\dots,a_1):=0 \text{ if } r \ge 1\\
  %&u_V^0(v,g_s,\dots,g_1):=0 \text{ if } s \ge 1 \\
  &u^k:=0 \text{ if } k \ge 1 
 \end{align*}
and the adjunction counit $c$ is defined by
 \begin{align*}
  &c_\cM^0 \in hom_{\text{mod-}\cD}( hom_{\text{mod-}\cD}(\cB,\cM) \otimes_\cC^{\mathbb{L}} \cB, \cM) \text{ for }\cM \in Ob(\text{mod-}\cD) \\
  &(c_\cM^0 )^{1|r}(\phi  \otimes x,a_r,\dots,a_1):=(-1)^{|\phi|(|x|+\sum_{j=1}^r \|a_j\|)}\phi^{1|r}(x,a_r,\dots,a_1) \\
  &(c_\cM^0 )^{1|r}(\phi \otimes g_s \dots g_1 \otimes x,a_r,\dots,a_1):=0 \text{ if $s>0$} \\
  &c^k:=0 \text{ if } k \ge 1 
 \end{align*}
 It is routine to check that $u$ and $c$ are natural transformations.
 To see that $u$ and $c$ are indeed adjunction unit and counit, it suffices to observes that the natural transformations defined by $[u]$ and $[c]$ on the homotopy categories
 are the ordinary adjunction unit and counit natural transformations.
\end{proof}

\subsubsection{Graph}

In this subsection, we explain some properties of $F_{\otimes}(\cB)$
and $F_{hom}(\cB)$ when $\cB$ is a  graph bimodule.

Given an $A_{\infty}$ functor $\cF:\cC \to \cD$, the {\it graph bimodule} $\cB_\cF$ of $\cF$ 
is the $A_{\infty}$ $\cC\text{-}\cD$ bimodule defined as follows (see \cite[(7.24)]{Seidelbook2}):
For $V \in Ob(\cC)$ and $X \in Ob(\cD)$, we have a cochain complex $\cB_\cF(X,V):=hom_\cD(X,\cF(V))$.
The higher order bimodule structures are defined by
\begin{align}
&\mu_{\cB_{\cF}}^{r|1|s}(g_r,\dots,g_1,x,a_s,\dots,a_1)  \label{eq:GraphModmorphism}\\
:=&(-1)^{1+\sum_{j=1}^s \|a_j\|}\sum_{j;i_1+\dots+i_j=s} \mu_\cD(\cF^{i_j}(g_s,\dots,g_{s-i_j+1}),\dots,\cF^{i_1}(g_{i_1},\dots,g_1),x,a_s,\dots,a_1)
\end{align}

\begin{example}[Yoneda module]\label{eg:Yoneda}
 When $\cC=\K$, viewed as an $A_{\infty}$ category with a single object, then 
 a $\cC\text{-mod-}\cD$ bimodule is the same as a right $\cD$-module.
 If $X \in Ob(\cD)$ and $\cF: \K \to \cD$ is an $A_{\infty}$ functor sending the unique object in $\K$ to $X$, then 
 $\cB_\cF \simeq \cY_{\cD}(X)$ is the Yoneda module.
\end{example}

Note that, Example \ref{eg:Yoneda} implies that
the sign convention of the $A_{\infty}$ structural maps of Yoneda modules is different from the one in Section \ref{ss:EquivariantEv}
because we use the sign convention in \cite{LF2} (see \cite[Conventions $2.1$]{LF2}) in this section (Section \ref{s:Algebraic perspective})
to have better compatiblity with the literature of $A_{\infty}$ bimodules; while the rest of the paper uses the sign convention 
in \cite{Seidelbook} to have better compatiblity with the literature on Dehn twist
and twisted complexes (see Appendix \ref{sec:orientations}).
Since this section is mostly independent from the rest of the paper, it should not cause serious confusion.

In this sign convention, for $x \in hom_\cD(X_0,X_1)$, we have
\begin{align}
 \cY^1_\cD(x)(y,a_r,\dots,a_1)=(-1)^{\|x\|\cdot |y|+|x| (\sum_{j=1}^r \|a_j\|)}\mu_\cD(x,y,a_r,\dots,a_1) \label{eq:Yonedamorphism}
\end{align}
and the evaluation map (cf. \eqref{eq:evaluationMap}) is given by
\begin{align}
 (x \otimes y, a_r,\dots,a_1) \mapsto (-1)^{|y|}\mu_\cD (x,y, a_r,\dots,a_1) \label{eq:TwistedEvMap}
\end{align}

\begin{example}[Diagonal bimodule]
 The {\bf diagonal bimodule} $\Delta_\cC$ of an $A_{\infty}$ category $\cC$ is the graph bimodule of the 
 identity functor $Id_\cC$. 
 %and is specified by the following data \cite[Equation $(2.3)$]{LF2}
 %\begin{align}
 % \Delta_\cC(V_0,V_1):=&\hom_\cC(V_0,V_1) \\
 % \mu_{\Delta_\cC}^{r|1|s}(g_r,\dots,g_1,g,g_s',\dots,g_1'):=&(-1)^{\maltese_1^s+1}\mu_\cC^{r+1+s}(g_r,\dots,g_1,g,g_s',\dots,g_1')
 %\end{align}
%where $\maltese_1^s=\sum_{j=1}^s \|g_j'\|$. 
\end{example}

There are associated pull-back functors on modules for an $A_{\infty}$ functor (See \cite[Section 2.8]{Gan13}).
In this language, the graph bimodule is
$\cB_\cF=(\cF \times Id_\cD)^*\Delta_\cD$ where $(\cF \times Id_\cD)^*:\cD\text{-mod-}\cD \to \cC\text{-mod-}\cD$ is the pull-back functor.
If we have a natural equivalence $\cF \simeq \cG$, then $\cB_\cF=(\cF \times Id_\cD)^*\Delta_\cD \simeq (\cG \times Id_\cD)^*\Delta_\cD=\cB_\cG$.

%We use $\cY_\cA: \cA \to \text{mod-}\cA$ to denote the Yoneda embedding.
%Since we have assumed that $\cA$ is c-unital, $\cY_\cA$ is cohomologically full and faithful [Seidel].
The following is a basic property of tensor product with diagonal bimodule.

\begin{lemma}[\cite{SeHom}, \cite{Gan13}]\label{l:BarContraction}
$\cB \otimes_\cD^{\mathbb{L}} \Delta_\cD$ is canonically quasi-isomorphic to $\cB$ as an object in $\cC\text{-mod-}\cD$.
 Similarly,  $\Delta_\cC \otimes_\cC^{\mathbb{L}} \cB$ is canonically quasi-isomorphic to $\cB$ in $\cC\text{-mod-}\cD$.
\end{lemma}

\begin{proof}
This is proved in  \cite[Section 2]{SeHom} and \cite[Proposition 2.2]{Gan13}.
An explicit quasi-isomorphism from $\cB \otimes_\cD^{\mathbb{L}} \Delta_\cD$  to $\cB$ is given by
% We define $t \in hom_{\text{mod-}\cA}(R\otimes_R^{\mathbb{L}} \cB_{\cF_\eE},\cY_\cA(\eE))$ by 
 \begin{align*}
&t^{r|1|s}(g_r,\dots,g_1,x \otimes a_l \dots a_1 \otimes a,a_s',\dots,a_1')\\
=&(-1)^{\sum_{j=1}^s \|a_j'\|+|a|+\sum_{j=1}^l\|a_j\|}\mu_{\cB}^{r|1|s+l+1}(g_r,\dots,g_1,x,a_l \dots a_1,a,a_s',\dots,a_1')
 \end{align*}
An explicit quasi-isomorphism from $\Delta_\cC \otimes_\cC^{\mathbb{L}} \cB$  to $\cB$ can be defined similarly.

\end{proof}

\begin{rmk}\label{r:InverseBarContraction}
 There is an quasi-inverse of the quasi-isomorphism in Lemma \ref{l:BarContraction} when $\cD$ is strictly unital (see \cite[Equation $(2.9)$]{SeHom}).
 We only need the case that $\cC=\K$. Explicitly, we have a natural equivalence as follows.
 \begin{align*}
 & \eta \in hom_{Fun(\text{mod-}\cD,\text{mod-}\cD)}(Id_{\text{mod-}\cD}, F_{\otimes}(\Delta_\cD))\\
  &\eta^0_\cM \in hom_{\text{mod-}\cD}(\cM , \cM \otimes_\cD^{\mathbb{L}} \Delta_\cD) \\
  &(\eta^0_\cM)^{1|r}(x,a_r,\dots,a_1):=x \otimes a_r,\dots,a_1 \otimes e_\cD \\
  &\eta^k=0 \text{ if } k \ge 1 
 \end{align*}
Here, $e_\cD$ is the strict unit for the appropriate object in $\cD$.
One can check that $\eta$ is a natural transformation: for the first order
\begin{align*}
 &(\mu^1_{Fun(\text{mod-}\cD,\text{mod-}\cD)}(\eta))^0(x,a_r,\dots,a_1)\\
 =&\mu^1_{\text{mod-}\cD}(\eta^0_\cM)(x,a_r,\dots,a_1) \\
 =& (\mu_{\cM \otimes_\cD^{\mathbb{L}} \Delta_\cD} \circ \widehat{\eta}^0_\cM - \eta^0_\cM \circ \widehat{\mu}_\cM)(x,a_r,\dots,a_1) \\
 =& \mu_{\cM \otimes_\cD^{\mathbb{L}}\Delta_\cD}^{1|0}(x\otimes a_r \dots a_1\otimes e_\cD) +\mu_{\cM \otimes_\cD^{\mathbb{L}}\Delta_\cD}^{1|1}((x\otimes a_r \dots a_2 \otimes e_\cD),a_1) \\
 & -\sum  (-1)^{A_1^k}\mu_\cM(x,a_r,\dots,a_{k+1})\otimes a_k,\dots,a_1 \otimes e_\cD  \\
 &-\sum (-1)^{A_1^k} x\otimes a_r \dots a_{l+1} \otimes \mu_{\cD}(a_{l},\dots,a_{k+1})\otimes a_k \dots \otimes a_1 \otimes e_\cD \\ 
 =&0
 \end{align*}
 where $A_1^k=\sum_{j=1}^k \|a_j\|$, and we used $\mu_{\Delta_\cD}^{0|1|1}(e_\cD,a_1)=a_1$, $\mu_{\Delta_\cD}^{1|1|0}(a_1,e_\cD)=-(-1)^{|e_\cD|}a_1=-a_1$ (see \cite[(2.7)]{LF2}). 
 One can also check that
 \begin{align*}
  (\mu^1_{Fun(\text{mod-}\cD,\text{mod-}\cD)}(\eta))^1(\phi):=\mu^2_{\text{mod-}\cD}(F_{\otimes}(\Delta_\cD)(\phi),\eta^0_{\cM_0})+(-1)^{\|\phi\|}\mu^2_{\text{mod-}\cD}(\eta^0_{\cM_1},\phi)
 \end{align*}
 vanishes for $\phi \in hom_{\text{mod-}\cD}(\cM_0,\cM_1)$ so $\eta$ is a natural transformation.
 For the quasi-isomoprhism $t \in hom_{\text{mod-}\cD}(\cM \otimes_\cD^{\mathbb{L}} \Delta_\cD, \cM)$ in Lemma \ref{l:BarContraction},
 it is clear that $t \circ \eta^0_\cM (x)=\mu_{\Delta_\cD}^{0|1|1}(x,e_\cD)=x$ and $(t \circ \widehat{\eta}^0_\cM)^{1|r}=0$ for $r>0$.
 Therefore, $t \circ \eta^0_\cM=Id_\cM$ and $\eta$ is a quasi-inverse of $t$.
\end{rmk}

One of the most important property of graph bimodules
is that tensor product is compatible with composition in the following sense.

\begin{lemma}\label{l:CompoGraphBimod}
 Let $\cF_0:\cC_0 \to \cC_1$ and $\cF_1: \cC_1 \to \cC_2$ be $A_{\infty}$ functors.
 Then $\cB_{\cF_1 \circ \cF_0}$ is quasi-isomorphic to $\cB_{\cF_0} \otimes_{\cC_1}^{\mathbb{L}} \cB_{\cF_1}$ as $\cC_0-\cC_2$ bimodules.
 
 As a result, we have a natural equivalence $F_{\otimes}(\cB_{\cF_1 \circ \cF_0}) \simeq F_{\otimes}(\cB_{\cF_1}) \circ F_{\otimes}(\cB_{\cF_0})$ in $Fun(\text{mod-}\cC_0,\text{mod-}\cC_2)$.
\end{lemma}

\begin{proof}
 By Lemma \ref{l:BarContraction}, we have a quasi-isomorphism $\cB_{\cF_1} \simeq \Delta_{\cC_1} \otimes_{\cC_1}^{\mathbb{L}} \cB_{\cF_1}$
 as $\cC_1-\cC_2$ bimodules.
 The pull-back $(\cF_0 \times Id_{\cC_2})^*$ preserves quasi-isomorphism.
On the left, we have
 \begin{align*}
(\cF_0 \times Id_{\cC_2})^*\cB_{\cF_1}&=   (\cF_0 \times Id_{\cC_2})^*((\cF_1 \times Id_{\cC_2})^*\Delta_{\cC_2}) \\
& \simeq (\cF_1 \circ \cF_0 \times Id_{\cC_2})^*\Delta_{\cC_2} \\
& =\cB_{\cF_1 \circ \cF_0}
 \end{align*}
On the right, since $\mu^{r|1|s}_{\Delta_{\cC_1} \otimes_{\cC_1}^{\mathbb{L}} \cB_{\cF_1}}=0$ if $r,s \ge 1$, one can easily check that
\begin{align*}
 (\cF_0 \times Id_{\cC_2})^*(\Delta_{\cC_1} \otimes_{\cC_1}^{\mathbb{L}} \cB_{\cF_1})&= (\cF_0 \times Id_{\cC_1})^*(\Delta_{\cC_1}) \otimes_{\cC_1}^{\mathbb{L}} \cB_{\cF_1} \\
 &=\cB_{\cF_0} \otimes_{\cC_1}^{\mathbb{L}} \cB_{\cF_1}
\end{align*}

It is clear from the definition that $F_{\otimes}(\cB_{\cF_0} \otimes_{\cC_1}^{\mathbb{L}} \cB_{\cF_1})=F_{\otimes}(\cB_{\cF_1}) \circ F_{\otimes}(\cB_{\cF_0})$.
Since $F_{\otimes}$ preserves quasi-isomorphism, we have $F_{\otimes}(\cB_{\cF_1 \circ \cF_0})=F_{\otimes}(\cB_{\cF_1}) \circ F_{\otimes}(\cB_{\cF_0})$.
\end{proof}

As a consequence of Example \ref{eg:Yoneda} and Lemma \ref{l:CompoGraphBimod}:

\begin{corr}\label{c:perfect}(\cite[Section 7, (7.25)]{Seidelbook2})
The following diagram is commutative up to quasi-isomorphisms of $A_{\infty}$ functors
\begin{align*}
\begin{CD}
\cC    @>\cF>> \cD\\
@VV\cY_\cC V                     @VV \cY_\cD V\\
\text{mod-}\cC     @>F_{\otimes}(\cB_\cF)>> \text{mod-}\cD
\end{CD}
\end{align*}
where $\cY_\cC$ and $\cY_\cD$ are the Yoneda embeddings of $\cC$ and $\cD$, respectively.
Therefore, $F_{\otimes}(\cB_\cF)$ sends perfect $\cC$-modules to perfect $\cD$-modules.
\end{corr}

We have the dual statement of Lemma \ref{l:CompoGraphBimod}.
\begin{lemma}\label{l:GeneralTenHom}
  Let $\cF_0:\cC_0 \to \cC_1$ and $\cF_1: \cC_1 \to \cC_2$ be $A_{\infty}$ functors.
 Then  we have a natural equivalence 
 $F_{hom}(\cB_{\cF_1 \circ \cF_0}) \simeq F_{hom}(\cB_{\cF_0}) \circ F_{hom}(\cB_{\cF_1})$ in $Fun(\text{mod-}\cC_2,\text{mod-}\cC_0)$.
\end{lemma}

\begin{proof}
 One can regard it as a generalization of Proposition \ref{p:TensorHom}.
 Let $G_0:=F_{hom}(\cB_{\cF_0} \otimes_{\cC_1}^{\mathbb{L}} \cB_{\cF_1}) $ and $G_1:=F_{hom}(\cB_{\cF_0}) \circ F_{hom}(\cB_{\cF_1})$.
 An explicit natural equivalence is given as follows
 \begin{align*}
  &\eta \in hom_{Fun(\text{mod-}\cC_2,\text{mod-}\cC_0)}(G_0,G_1) \\
  & \eta^0_\cM \in hom_{\text{mod-}\cC_0}(G_0(\cM),G_1(\cM)) \\
  & (\eta^0_\cM)^{1|0}(\phi):= \widetilde{\phi}  \in G_1(\cM)\\
  &(\widetilde{\phi})^{1|s}(x,b_s,\dots,b_1):=\widetilde{\phi}_{[x\otimes b_s \dots \otimes b_1]} \\
  &(\widetilde{\phi}_{[x\otimes b_s \dots \otimes b_1]})^{1|t} (y,c_t,\dots,c_1):=(-1)^\dagger \phi^{1|t}(x\otimes b_s \dots \otimes b_1 \otimes y,c_t,\dots,c_1) \\
  & (\eta^0_\cM)^{1|r}(\phi,a_r,\dots,a_1):=0  \text{ if } r >0\\
  & \eta^k:=0 \text{ if } k >0
 \end{align*}
 where $\dagger=(|x|+\sum_{j=1}^s \|b_j\|)(|y|+\sum_{j=1}^t \|c_j\|)$, 
 $\cM \in Ob(\text{mod-}\cC_2)$, $\phi \in G_0(\cM)$, $a_i \in Mor(\cC_0)$, $b_i \in Mor(\cC_1)$, $c_i \in Mor(\cC_2)$, $x \in \cB_0$ and $y \in \cB_1$.
It is routine to check that $\eta$ is a natural transformation (i.e. $(\mu^1_{Fun(\text{mod-}\cC_2,\text{mod-}\cC_0)}(\eta))^0=0$
and $(\mu^1_{Fun(\text{mod-}\cC_2,\text{mod-}\cC_0)}(\eta))^1(\psi)=0$ for $\psi \in hom_{\text{mod-}\cC_2}(\cM_0,\cM_1)$).

To see that $\eta^0_\cM$ is a quasi-isomorphism for all $\cM$, it suffices to observe that $\eta^0_\cM$ is bijective.
\end{proof}

\begin{lemma}\label{l:uniquenessAdjoint}
 If $\cF:\cC \to \cD$ is a quasi-equivalence, then $F_{\otimes}(\cB_\cF) \circ F_{hom}(\cB_\cF) \simeq Id_{\text{mod-}\cD}$
 and $F_{hom}(\cB_\cF) \circ F_{\otimes}(\cB_\cF) \simeq Id_{\text{mod-}\cC}$.
\end{lemma}

\begin{proof}
 Let $\cG$ be an quasi-inverse of $\cF$.
 By Lemma \ref{l:CompoGraphBimod}, we have  $F_{\otimes}(\cB_\cF) \circ F_{\otimes}(\cB_\cG) \simeq Id$ and 
 $F_{\otimes}(\cB_\cG) \circ F_{\otimes}(\cB_\cF) \simeq Id$.
 From Proposition \ref{p:TensorHom}, we have adjunction counit $c_{\cB_{\cF}}$ for the pair $(F_{\otimes}(\cB_\cF),F_{hom}(\cB_\cF))$.
 It means that the composition
 \begin{align*}
  hom(V,F_{hom}(\cB_\cF)(X)) \xrightarrow{F_{\otimes}(\cB_\cF)}& hom(F_{\otimes}(\cB_\cF)(V),F_{\otimes}(\cB_\cF)(F_{hom}(\cB_\cF)(X))) \\
  \xrightarrow{c_{\cB_{\cF}}(X) \circ -}& hom(F_{\otimes}(\cB_\cF)(V),X)
 \end{align*}
 is a quasi-isomorphism for all $V \in Ob(\text{mod-}\cC)$ and $X \in Ob(\text{mod-}\cD)$.
 Since $F_{\otimes}(\cB_\cF)$ is a quasi-equivalence, the first map in the composition is a quasi-isomorphism.
 Therefore, composition with $c_{\cB_{\cF}}(X)$ is a quasi-isomorphism for all $V$ and $X$.
 It means that $c_{\cB_{\cF}}$ is a natural equivalence and hence $F_{\otimes}(\cB_\cF) \circ F_{hom}(\cB_\cF) \simeq Id_{\text{mod-}\cD}$.
 In other words, $F_{hom}(\cB_\cF) \simeq F_{\otimes}(\cB_\cG)$.
\end{proof}

\begin{lemma}\label{l:conjugationInvariant}
 Let $\cG: \cD \to \cD'$ be a quasi-equivalence.
 Let $\cF:\cC \to \cD$ and $\cF':=\cG \circ \cF: \cC \to \cD'$ be two $A_{\infty}$ functors.
 Let $c_{sr}$ and $c_{sr}'$ be the adjunction counits for the pairs $(F_{\otimes}(\cB_{\cF}),F_{hom}(\cB_{\cF}))$
 and $(F_{\otimes}(\cB_{\cF'}),F_{hom}(\cB_{\cF'}))$ in Proposition \ref{p:TensorHom}, respectively.
 Then the adjunction counits are compatible with the quasi-equivalence $F_{\otimes}(\cB_{\cG})$ in the sense that
 $F_{\otimes}(\cB_{\cG}) \circ Cone(c_{sr}) \simeq Cone(c_{sr}') \circ F_{\otimes}(\cB_{\cG}) $.
\end{lemma}

\begin{proof}
 
By Lemma \ref{l:CompoGraphBimod}, we have $\cB_{\cF'} \simeq \cB_{\cF} \otimes_{\cD}^{\mathbb{L}} \cB_{\cG}$ 
and hence $F_{\otimes}(\cB_{\cF'}) \simeq  F_{\otimes}(\cB_\cG) \circ F_{\otimes}(\cB_{\cF})$.
By Lemma \ref{l:GeneralTenHom}, we also have
%\begin{align*}
 $F_{hom}(\cB_{\cF'}) \simeq F_{hom}(\cB_{\cF}) \circ F_{hom}(\cB_{\cG})$
%\end{align*}
We consider
\begin{align}
 &F_{\otimes}(\cB_\cG) \circ  Cone(c_{sr}) \circ F_{hom}(\cB_{\cG})\\
 \simeq & Cone( F_{\otimes}(\cB_\cG) \circ F_{\otimes}(\cB_{\cF}) \circ F_{hom}(\cB_{\cF}) \circ F_{hom}(\cB_{\cG}) \to F_{\otimes}(\cB_\cG) \circ F_{hom}(\cB_{\cG})) \label{eq:conjugation1}\\
 \simeq& Cone(F_{\otimes}(\cB_{\cF'}) \circ F_{hom}(\cB_{\cF'}) \xrightarrow{c_{sr}'} Id_{(\cD')^{perf}} ) \label{eq:conjugation2}\\
 \simeq& Cone(c_{sr}') 
\end{align}
From \eqref{eq:conjugation1} to \eqref{eq:conjugation2}, we use that the adjunction counit of the pair $(F_{\otimes}(\cB_{\cG}),F_{hom}(\cB_{\cG}))$ 
is a natural equivalence, which is proved in
Lemma \ref{l:uniquenessAdjoint}, so that we can identify $F_{\otimes}(\cB_\cG) \circ F_{hom}(\cB_{\cG})$ with $Id_{(\cD')^{perf}}$.
One can directly compute and verify that the morphism in \eqref{eq:conjugation2} is indeed $c_{sr}'$, which boils down to checking evaluation of the pair 
$(F_{\otimes}(\cB_{\cF}),F_{hom}(\cB_{\cF}))$ followed by evaluation of the pair $(F_{\otimes}(\cB_{\cG}),F_{hom}(\cB_{\cG}))$ coincides with the evaluation
of the pair $(F_{\otimes}(\cB_{\cF'}),F_{hom}(\cB_{\cF'}))$. 
We leave this computation to the readers.

Finally, the result follows from Lemma \ref{l:uniquenessAdjoint} where we showed that $F_{hom}(\cB_{\cG})$ is an inverse of $F_{\otimes}(\cB_{\cG})$.
\end{proof}

\subsection{$R$-spherical twist}\label{ss:rSphericalTwist}

We have finished the preparation on $A_{\infty}$ bimodules.
In this section, we want to introduce $R$-spherical objects and its associated auto-equivalences.
Let $R$ be a $\K$-algebra viewed as an $A_{\infty}$ algebra concentrated at degree $0$ with $\mu^k_R=0$ for $k \neq 2$
and $\mu^2_R(g_2,g_1):=(-1)^{|g_1|}g_2g_1=g_2g_1$.
%Assume that $R$ is smooth so that every right $A_{\infty}$ module over $R$ is perfect.
We are particularly interested in the case where $R$ is a group algebra $\K[\Gamma]$ but we do not impose this assumption in this section.

\begin{defn}\label{d:R-twist}
An $R$-spherical object $\eE$ is an object of a c-unital $A_{\infty}$ category $\mathcal{A}$ such that
\begin{itemize}
\item there is a graded $\K$-algebra isomorphism $\overline{\Theta}:H^*(S^n) \otimes_{\K} R \simeq Hom_{\cA}^*(\eE,\eE)$\footnote{we do not assume that $\eE$ is formal, cf. Remark \ref{r:Nonformal}}
\item the composition map $f:Hom_{\cA}^k(X,\eE) \to Hom_{\text{mod-}R}(Hom^{n-k}_{\cA}(\eE,X),Hom^n(\eE,\eE))$ is an isomorphism for all $X \in Ob(\cA)$ and for all $k$
\end{itemize} 
\end{defn}
Let $\Theta:R \to Hom_{\cA}^0(\eE,\eE)$ be the restriction of $\overline{\Theta}$ to the degree $0$ part.
Then, for $x \in Hom_{\cA}^k(X,\eE)$,
$y \in Hom^{n-k}_{\cA}(\eE,X)$ and $g \in R$, 
we define $f(x)(y):=\mu^2(x,y)$ and $(f(x)g)y:=\mu^2(x,\mu^2(y,\Theta(g)))=\mu^2(\mu^2(x,y),\Theta(g))=(f(x)(y))g$.
In particular, $f$ is well-defined.
%From now on, we avoid abusing notation and use $\Theta:R \to Hom^0(\eE,\eE)$ to denote the canonical identification

\begin{lemma}\label{l:initialFunctor}
 There is a cohomologically faithful $A_{\infty}$ functor $\cF_\eE:R \to \cA$ with image 
 $\eE$, where $R$ is viewed as an $A_{\infty}$ category with one object.
 Moreover, this functor is unique up to homotopy and $\K$-algebra automorphism of $R$. 
 %such that $H(\cF_\eE):R \to Hom_{\cA}(\eE,\eE)$ is injective
\end{lemma}

\begin{proof}
 Without loss of generality, we can assume that $\cA$ is minimal, then $hom_{\cA}(\eE,\eE)$ has non-negative degree and $\cF_\eE^k=0$ for $k \ge 2$ by degree reason.
 As a result, $A_{\infty}$ functors $\cF_\eE:R \to \cA$ with image $\eE$ one-to-one correspond to $\K$-algebra homomorphisms $\cF_{\eE}^1:R \to hom^0_{\cA}(\eE,\eE)$.
 Since $hom^0_{\cA}(\eE,\eE)=R$ as a $\K$-algebra, the result follows.
 
\end{proof}

Lemma \ref{l:initialFunctor} induces a graph bimodule $\cB_{\cF_\eE}$.
As a graph bimodule, the DG functor $F_{\otimes}(\cB_{\cF_\eE}): \text{mod-}R \to \text{mod-}\cA$ sends perfect modules to perfect modules (see Corollary \ref{c:perfect}).
Therefore, we can restrict the functor to $F_{\otimes}(\cB_{\cF_\eE}):R^{perf} \to \cA^{perf}$.
However, it is not necessarily true that $F_{hom}(\cB_{\cF_\eE})$ sends perfect modules to perfect modules so we want to make the following assumption.
\begin{assumption}\label{a:perfect}
 For any $X \in Ob(\cA^{perf})$, we assume $hom_{\text{mod-}\cA}(\cB_{\cF_\eE},X)$
 is an object in $R^{perf}$.
\end{assumption}
Assumption \ref{a:perfect} is automatically satisfied if $R$ is smooth (ie. the diagonal bimodule $R$ is a perfect $R-R$ bimodule), 
for example, when $R=\K[\Gamma]$ and $\cchar(\K)$ does not divide $|\Gamma|$.
In our geometric context (eg. Dehn twist), Assumption \ref{a:perfect} is always satisfied (see Lemma \ref{l:homIdnetify}).

We assume Assumption \ref{a:perfect} holds throughout this section.
In this case, $F_{hom}(\cB_{\cF_\eE})$ is a right adjoint of $F_{\otimes}(\cB_{\cF_\eE}):R^{perf} \to \cA^{perf}$ by Proposition \ref{p:TensorHom}.

\begin{comment}
Let $mod_c-\cA$ be the full subcategory of compact objects in $\text{mod-}\cA$.
By [Proposition 2.4 of Orlov smooth and proper noncommutative schemes and gluing of DG categories] or nlab page, 
$mod_c-\cA$ is the smallest split closed triangulated full subcategory of $\text{mod-}\cA$ split-generated by the Yondea image of $\cA$.
It is clear that the image of $-\otimes_R^{\mathbb{L}} \cB_{\cF_\eE}$ lies in $mod_c-\cA$
so $-\otimes_R^{\mathbb{L}} \cB_{\cF_\eE}:\text{mod-}R \to mod_c-\cA$ is also a DG functor with right adjoint $hom_{\text{mod-}\cA}(\cB_{\cF_\eE},-)$.
We can now state the proposition we are going to establish in this section.
\end{comment}

\begin{prop}\label{p:R-twist}
 Let $\eE$ be an $R$-spherical object and $\cB_{\cF_\eE}$ be the $R-\cA$ bimodule as above.
 Suppose Assumption \ref{a:perfect} is satisfied.
 Then $\cS:=F_{\otimes}(\cB_{\cF_\eE}):R^{perf} \to \cA^{perf}$ is a spherical functor with right adjoint $\cR:=F_{hom}(\cB_{\cF_\eE})$,
 left adjoint $\cL:=F_{hom}(\cB_{\cF_\eE})[n]$ and the twist auto-equivalence is given by
 $\cT_\cS:=Cone(\cS \circ \cR \xrightarrow{c_{sr}} Id_{\cA^{perf}} )$.
\end{prop}

Note that every c-unital $A_{\infty}$ category is a quasi-equivalent to a minimal and strictly unital $A_{\infty}$ category (see e.g. \cite[Lemma 2.1]{Seidelbook}).
By Lemma \ref{l:CompoGraphBimod}, \ref{l:GeneralTenHom} and \ref{l:conjugationInvariant}, we can assume that $\cA$ is minimal and strictly unital when we prove Proposition \ref{p:R-twist}.
The proof of Proposition \ref{p:R-twist} is a direct application of Theorem \ref{t:SphericalFun}
and we are going to verify the assumptions of Theorem \ref{t:SphericalFun} in the following subsections.

\subsubsection{Mapping cone of the adjunction unit}

%Since $\cA$ is c-unital, it is quasi-equivalent to a minimal and strictly unital $A_{\infty}$ category ([Seidel]).
%As a result, without loss of generality, we assume that $\cA$ is minimal and strictly unital (see Section \ref{ss:Concluding}).
Let $\cA$ be minimal and strictly unital.
Then $(hom_\cA(\eE,\eE),\mu^2_\cA)$ is isomorphic to $H^*(S^n) \otimes_\K R$ as a graded $R$-algebra (but $\mu^k_\cA$ on $hom_\cA(\eE,\eE)$ could be non-zero for some $k>2$)
so we have an isomorphism $\Theta:R \simeq hom^0_\cA(\eE,\eE)$.
We also have a strict unit $e_\eE:=\Theta(1) \in hom^0_\cA(\eE,\eE)$.

Let $\cS:=-\otimes_R^{\mathbb{L}} \cB_{\cF_\eE}$ and $\cR:=hom_{\cA^{perf}}(\cB_{\cF_\eE},-)$.
In this subsection, we want to show that $Cone(Id_R \xrightarrow{u_{rs}} \cR \circ \cS) \simeq Id_R[-n]$.
The idea goes as follows:
First, we show that the $R\text{-}R$ bimodule structure on 
$Hom_{\cA^{perf}}(\cB_{\cF_\eE},\cB_{\cF_\eE}) \simeq R \oplus R[-n]$ is the canonical one
and show that  $\cR \circ \cS$ 
is natural equivalent to $F_{\otimes}(hom_{\cA^{perf}}(\cB_{\cF_\eE},\cB_{\cF_\eE}))$.
Then we show that
$Cone(Id_R \xrightarrow{u_{rs}} \cR \circ \cS)$ is natural equivalent to 
$F_{\otimes}(Cone(R \to hom_{\cA^{perf}}(\cB_{\cF_\eE},\cB_{\cF_\eE})) \simeq F_{\otimes}(R[-n])$
and conclude the proof.

\begin{lemma}\label{l:RRstrEE}
 $Hom_{\cA^{perf}}(\cB_{\cF_\eE},\cB_{\cF_\eE})=R \oplus R[-n]$ as an ordinary $R-R$ bimodule.
\end{lemma}

%Let $\cY_R:R \to \text{mod-}R$ be the Yoneda embedding.
%\begin{lemma}\label{l:GGstructure}
% $H^*((\cY_R \otimes \cY_R)^*\cB_{\cR\cS})=R \oplus R[-n]$ as an $R-R$ bimodule
%\end{lemma}

\begin{proof}
If we forget the left $R$-module structure on $\cB_{\cF_\eE}$, 
then the right $\cA$-module structure on $\cB_{\cF_\eE}$ gives $Hom_{\cA^{perf}}(\cB_{\cF_\eE},\cB_{\cF_\eE})=Hom_{\cA^{perf}}(\cY_\cA(\eE),\cY_\cA(\eE))=R \oplus R[-n]$.
%where $\cY$ is the Yoneda embedding of $\cA$.
Moreover, as an endomorphism algebra, we have 
$Hom_{\cA^{perf}}(\cY_\cA(\eE),\cY_\cA(\eE))=H^*(S^n) \otimes R$ by Defintion \ref{d:R-twist}.

On the other hand, $hom_{\cA^{perf}}(\cB_{\cF_\eE},\cB_{\cF_\eE})$ is equipped with an  $A_{\infty}$ $R-R$ bimodule structure,
by the left $R$-module structure on $\cB_{\cF_\eE}$, as follows:
\begin{align*}
&\mu^{0|1|0}(\psi)=(-1)^{|\psi|} \mu^{1}_{\cA^{perf}}(\psi)\\
&\mu^{0|1|s}(\psi,g_s,\dots,g_1)=(\psi^{1|r}_{g_s,\dots,g_1})_r \\
&\psi^{1|r}_{g_s,\dots,g_1}(x,a_r,\dots,a_1)=\sum (-1)^{\dagger_{1,k}}\psi^{1|r}(\mu_{\cB_{\cF_\eE}}(g_s,\dots,g_1,x,a_r,\dots,a_{k+1}),a_k, \dots,a_1) \\
&\mu^{s|1|0}(h_s,\dots,h_1,\psi)=(_{h_s,\dots,h_1}\psi^{1|r})_r \\
&_{h_s,\dots,h_1}\psi^{1|r}(x,a_r,\dots,a_1)=\sum (-1)^{\dagger_{2,k}} \mu_{\cB_{\cF_\eE}}(h_s,\dots,h_1,\psi^{1|r}(x,a_r,\dots, a_{k+1}),a_k,\dots,a_1) \\
&\mu^{r|1|s}=0 \text{ if $r,s \ge 1$}
 \end{align*}
where 
\begin{align*}
 \dagger_{1,k}&=(\sum_{j=1}^s \|g_j\|)(|\psi|+\sum_{j=1}^r \|a_r\|+|x|)+ |\psi|+\sum_{j=1}^k \|a_k\|+1 \\
 \dagger_{2,k}&=(|x|+ \sum_{j=1}^r \|a_j\|)\sum_{j=1}^s \|h_j\|+|\psi| \sum_{j=1}^k \|a_j\|
\end{align*}
%$\dagger_{1,k}=(\sum_{j=1}^s \|g_j\|)(|\psi|+\sum_{j=1}^r \|a_r\|+|x|)+ |\psi|+\sum_{j=1}^k \|a_k\|+1$ and
%$\dagger_{2,k}=(|x|+ \sum_{j=1}^r \|a_j\|)\sum_{j=1}^s \|h_j\|+|\psi| \sum_{j=1}^k \|a_j\|$.

We are interested in the first order $R-R$ bimodule structure.
For any $g \in R$, $x \in hom_\cA(\eE,\eE)$, $y \in \cB_{\cF_\eE}(X_d)=hom_\cA(X_d,\eE)$ and $a_i \in hom_\cA(X_{i-1},X_i)$, we have (see \eqref{eq:Yonedamorphism})
%\footnote{See \cite[after (4.11)]{AbGen} for the convention of Yoneda embedding}
 \begin{align*}
  &(\mu^{1|1|0}_{hom_{\cA^{perf}}(\cB_{\cF_\eE},\cB_{\cF_\eE})}(g,\cY^1(x)))^{1|r}(y,a_r,\dots,a_1) \\
  =&\sum_j (-1)^{\dagger_1}\mu_{\cB_{\cF_\eE}}(g,\cY^1(x)(y,a_r,\dots,a_{p+1}),\dots,a_1) \\
  =&\sum_j (-1)^{\dagger_2}\mu_{\cA}(\Theta(g),\mu_\cA(x,y,a_r,\dots,a_{p+1}),\dots,a_1) \\
  =&(-1)^{|x|+1}(\cY^1(\Theta(g)) \circ \cY^1(x))^{1|r}(y,a_d,\dots,a_1) \\
  =&-\mu^2_{\cA^{perf}}(\cY^1(\Theta(g)), \cY^1(x))^{1|r}(y,a_d,\dots,a_1)
 \end{align*}
% where $\cY$ is the Yoneda embedding of $\eE$
where $\dagger_1=|y|+\sum_{j=1}^r \|a_j\|+|x|(\sum_{j=1}^p \|a_j\|)$ and $\dagger_2=\dagger_1+(\sum_{j=1}^p \|a_j\|)+1+\|x\|\cdot |y|+|x|(\sum_{j=p+1}^r \|a_j\|)$.
 Similarly
  \begin{align*}
  &(\mu^{0|1|1}_{hom_{\cA^{perf}}(\cB_{\cF_\eE},\cB_{\cF_\eE})}(\cY^1(x),g))^{1|r}(y,a_r,\dots,a_1) \\
  =&\sum_j (-1)^{\dagger_1} \cY^1(x)(\mu_{\cB_{\cF_\eE}}(g,y,a_r,\dots,a_{p+1}),\dots) \\
  =&\sum_j (-1)^{\dagger_2} \mu_{\cA}(x,\mu_\cA(\Theta(g),y,a_r,\dots,a_{p+1}),\dots) \\
  =&(\cY^1(x) \circ \cY^1(\Theta(g)))^{1|r}(y,a_r,\dots,a_1)\\
  =&\mu^2_{\cA^{perf}}(\cY^1(x), \cY^1(\Theta(g)))^{1|r}(y,a_d,\dots,a_1)
 \end{align*}
 where $\dagger_1=(|x|+|y|+\sum_{j=1}^r \|a_j\|)+|x|+(\sum_{j=1}^p \|a_j\|)+1$ and 
 $\dagger_2=\dagger_1+(\sum_{j=p+1}^r \|a_j\|)+1+\|x\|(|y|+\sum_{j=p+1}^r \|a_j\|)+|x|(\sum_{j=1}^p \|a_j\|)$.
 
 Note that, as an $A_{\infty}$ $R-R$ bimodule, the structural map on $H^*(S^n) \otimes R$ is given by (see \eqref{eq:GraphModmorphism})
 $\mu^{1|1|0}(g,x)=-\mu^2(g,x)$ and $\mu^{0|1|1}(x,g)=(-1)^{\|g\|+1}\mu^2(x,g)=\mu^2(x,g)$
 for $x \in R \oplus R[-n]$ and $g \in R$.
 
 Therefore, the $R-R$ bimodule structure on $Hom_{\cA^{perf}}(\cB_{\cF_\eE},\cB_{\cF_\eE})$, which has been shown above to be exactly
 given by the left and right multiplication by elements in $Hom^0_{\cA^{perf}}(\cB_{\cF_\eE},\cB_{\cF_\eE})=R$, is the canonical 
 $R-R$ bimodule structure on $H^*(S^n) \otimes R \simeq R \oplus R[-n]$.
\end{proof}

\begin{lemma}\label{l:tensorOUT}
Under Assumption \ref{a:perfect}, the functor
 $-\otimes_R^{\mathbb{L}} hom_{\cA^{perf}}(\cB_{\cF_\eE},\cB_{\cF_\eE})$ is natural equivalent to $\cR \circ \cS$ in $Fun(R^{perf},R^{perf})$.
\end{lemma}

\begin{proof}
Let  $F_0:=-\otimes_R^{\mathbb{L}} hom_{\cA^{perf}}(\cB_{\cF_\eE},\cB_{\cF_\eE})$ and 
  $F_1:=hom_{\cA^{perf}}(\cB_{\cF_\eE},- \otimes_R^{\mathbb{L}} \cB_{\cF_\eE})$.
We want to construct a natural equivalence $\eta$ from $F_0$ to $F_1$.
 There is an inclusion $(\eta^0_V)^{1|0} :F_0(V) \to F_1(V)$
 given by
 \begin{align*}
 & v \otimes g_s \dots g_1 \otimes t \mapsto _{v \otimes g_s \dots g_1}t \\
 & _{v \otimes g_s \dots g_1}t(x,a_r,\dots,a_1)=v \otimes g_s \dots g_1 \otimes t(x,a_r,\dots,a_1)
 \end{align*}
 We define $(\eta^0_V)^{1|r}=0$ if $r>0$. It is straightforward to check that $\eta$ is a natural transformation.
 
 \begin{comment}
 We observe that
 \begin{align*}
  &(\eta^0_V)^{1|0}(\mu^{1|r}_{F_0(V)}(v \otimes g_s \dots g_1 \otimes t,h_r,\dots,h_1)) \\
  =&(\eta^0_V)(v \otimes g_s \dots g_1 \otimes t_{h_r,\dots,h_1})\\
  =&\mu^{1|r}_{F_1(V)}((\eta^0_V)^{1|0}(v \otimes g_s \dots g_1 \otimes t),h_r,\dots,h_1)
 \end{align*}
so $(\eta^0_V) \in hom_{R^{perf}}(F_0(V),F_1(V))$.
Here, $t_{h_r,\dots,h_1}$ is defined by Equation \eqref{eq:RightGaction} which comes from the left $R$-module structure on $\cB_{\cF_\eE}$.
 
 To see  that $(\eta^0_V)$ defines a natural transformation from $F_0$ to $F_1$, we also need to check, for $\psi \in hom_{\text{mod-}R}(V,W)$, 
  \begin{align*}
   &(\eta^0_V)\circ F_0(\psi)(v \otimes g_s \dots g_1 \otimes t) \\
   =&(\eta^0_V)^{1|0}(\sum_k \psi(v \dots g_{k+1}) \otimes g_k \dots g_1 \otimes t) \\
   =&_{\sum_k \psi(v \dots g_{k+1}) \otimes g_k \dots g_1}t \\
   =&F_1(\psi)(_{v \otimes g_s \dots g_1}t) \\
   =&F_1(\psi) \circ (\eta^0_V)^{1|0}(v \otimes g_s \dots g_1 \otimes t)
  \end{align*}
\end{comment}
  Moreover, when $V=R$, we have
  \begin{align*}
   (\eta^0_R)^{1|0}:R \otimes_R^{\mathbb{L}} hom_{\cA^{perf}}(\cB_{\cF_\eE},\cB_{\cF_\eE}) \to hom_{\cA^{perf}}(\cB_{\cF_\eE},R \otimes_R^{\mathbb{L}} \cB_{\cF_\eE})
  \end{align*}
 By Lemma \ref{l:BarContraction} and Remark \ref{r:InverseBarContraction},
 we know that $(\eta^0_R)^{1|0}$ a quasi-isomorphism.
 Since the Yoneda image of $R$ split-generates $R^{perf}$,
 we conclude that $(\eta^0_V)^{1|0}$ is a quasi-isomorphism for all $V \in Ob(R^{perf})$ and hence $\eta$ is a natural equivalence.
 \end{proof}

By Lemma \ref{l:RRstrEE}, $Hom_{\cA^{perf}}(\cB_{\cF_\eE},\cB_{\cF_\eE})=R \oplus R[-n]$ as an ordinary $R-R$ bimodule
so $hom_{\cA^{perf}}(\cB_{\cF_\eE},\cB_{\cF_\eE})$ is quasi-isomorphic to $(R \oplus R[-n],\mu^*_{RR})$
for some minimal $A_{\infty}$ $R-R$ bimodule structure $\mu^*_{RR}$.
Since $\cA$ is minimal, we can also identify $R \oplus R[-n]$ with $hom_\cA(\eE,\eE)$.

Now, we can construct an $A_{\infty}$ $R-R$ bimodule homomorphism $\bar{u}$ as follows.
\begin{align*}
 &\bar{u}:R=hom^0(\eE,\eE) \to (hom_\cA(\eE,\eE),\mu^*_{RR}) \\
 &\bar{u}^{0|1|0}(x):=x \\
 &\bar{u}^{r|1|s}=0 \text{ if $(r,s) \neq (0,0)$ }
\end{align*}
Since $hom_{\cA^{perf}}(\cB_{\cF_\eE},\cB_{\cF_\eE})$ is quasi-isomorphic to $(R \oplus R[-n],\mu^*_{RR})$,
we can push $\bar{u}$ to an $A_{\infty}$ $R-R$ bimodule homomorphism
from $R$ to $hom_{\cA^{perf}}(\cB_{\cF_\eE},\cB_{\cF_\eE})$, which is denoted by $u$.
In particular, we have
\begin{align}
 %&u:R=hom^0(\eE,\eE) \to hom_{\cA^{perf}}(\cB_{\cF_\eE},\cB_{\cF_\eE}) \\
 &u^{0|1|0}(x):=\cY^1(x) 
\end{align}
The higher order terms $u^{r|1|s}$ are not explicit.

%Since 
%\begin{align*}
%[u^{0|1|0}(\mu_R^{0|1|1}(x,g))]&=[\cY^1(\mu_\cA^2(x,\Theta(g)))] \\
%&= [\mu_{\cA^{perf}}^2(\cY^1(x),\cY^1(\Theta(g)))]\\
%&= [\mu_hom_{\cA^{perf}}(\cB_{\cF_\eE},\cB_{\cF_\eE})^{0|1|1}(u^{0|1|0}(x),g)]
%\end{align*}
%We can pick $u^{0|1|1}(x,g)$ such that 

\begin{lemma}\label{l:simplifyCotwist}
There is a natural equivalence
 \begin{align*}
 Cone(Id_{R^{perf}}\xrightarrow{u_{rs}} \cR \circ \cS) \simeq -\otimes_R^{\mathbb{L}} Cone(R  \xrightarrow{u} hom_{\cA^{perf}}(\cB_{\cF_\eE},\cB_{\cF_\eE}))  
 \end{align*}

\end{lemma}

\begin{proof}
 The $A_{\infty}$ bimodule homomorphism $u$ induces a natural transformation $F_\otimes^1(u)$
 from $F_{\otimes}(R)$ to $F_{\otimes}(hom_{\cA^{perf}}(\cB_{\cF_\eE},\cB_{\cF_\eE}))$.
 For any $V \in Ob(\text{mod-}R)$ and $v \in V$, the composition
 \begin{align*}
  V &\simeq V\otimes_R^{\mathbb{L}} R=V\otimes_R^{\mathbb{L}}  hom^0(\eE,\eE) \\
    &\xrightarrow{F_\otimes^1(u)} V\otimes_R^{\mathbb{L}} hom_{\cA^{perf}}(\cB_{\cF_\eE},\cB_{\cF_\eE}) \\
    &\simeq hom_{\cA^{perf}}(\cB_{\cF_\eE},V\otimes_R^{\mathbb{L}} \cB_{\cF_\eE})=\cR \circ \cS(V)
 \end{align*}
 is given by
 \begin{align*}
  v &\mapsto v \otimes 1=v \otimes e_\eE \\
    &\mapsto v \otimes \cY^1(e_\eE)=v \otimes Id_{\cY(\eE)} \\
    &\mapsto (\eta^0_V)(v \otimes Id_{\cY(\eE)})
 \end{align*}
 where the first map is given in Remark \ref{r:InverseBarContraction} and the last map $(\eta^0_V)$ is the natural equivalence in Lemma \ref{l:tensorOUT}.
 
 By Proposition \ref{p:TensorHom}, we also have $u_{rs}(v)=(\eta^0_V)(v \otimes Id_{\cY(\eE)})$.
 In other words, $F_\otimes^1(u)$ coincides with $u_{rs}$ up to natural equivalences $Id \simeq -\otimes_R^{\mathbb{L}} R$ and $(\eta^0_V)$.
 Since the functor from  $R\text{-mod-}R$ to $Fun(\text{mod-}R,\text{mod-}R)$ is a DG functor, we have
 \begin{align*}
  F_\otimes^1(Cone(u)) \simeq Cone(F_\otimes^1(u)) \simeq Cone(Id_{R^{perf}}\xrightarrow{u_{rs}} \cR \circ \cS)
 \end{align*}

\end{proof}

\begin{lemma}\label{l:Cotwist}
 $ Cone(R \xrightarrow{u} hom_{\cA^{perf}}(\cB_{\cF_\eE},\cB_{\cF_\eE}))$ is quasi-isomorphic to the formal $R-R$ bimodule $R[-n]$.
 As a result, there is a natural equivalence of DG functors
 $Cone(Id_{R^{perf}}\xrightarrow{u_{rs}} \cR \circ \cS) \simeq Id[-n]$.
\end{lemma}

\begin{proof}
 By the cohomology long exact sequence, $H(Cone(R \xrightarrow{u} hom_{\cA^{perf}}(\cB_{\cF_\eE},\cB_{\cF_\eE})))=R[-n]$.
 Moreover, Lemma \ref{l:RRstrEE} implies that
 the ordinary $R-R$ bimodule structure on $H(Cone(R \to hom_{\cA^{perf}}(\cB_{\cF_\eE},\cB_{\cF_\eE})))$ coincide with the standard bimodule structure on $R[-n]$ 
 (ie. given by left and right multiplication).
 By degree reason, $R[-n]$ only supports one minimal $A_{\infty}$ $R-R$ bimdoule structure which is the formal one.
 As a result,
 \begin{align*}
  Cone(Id_{R^{perf}}\xrightarrow{u_{rs}} \cR \circ \cS) & \simeq -\otimes_R^{\mathbb{L}} Cone(R  \xrightarrow{u} hom_{\cA^{perf}}(\cB_{\cF_\eE},\cB_{\cF_\eE}))  \\
  & \simeq -\otimes_R^{\mathbb{L}} R[-n] \simeq [-n]
 \end{align*}

\end{proof}

\subsubsection{An auto-equivalence}\label{ss:Concluding}

%Let $c:hom_{\cA^{perf}}(\cB_{\cF_\eE},\cB_{\cF_\eE})) \to Cone(R \to hom_{\cA^{perf}}(\cB_{\cF_\eE},\cB_{\cF_\eE})) \simeq R[-n]$ be the inclusion.
%In this step, we want to show that the induced natural transformation 
%$n_c[n]$ from $\cR \circ \cS [n]$ to $Id$ is an adjunction counit.
%It is sufficent to show that $\cR[n]$ is a left adjoint of $\cS$ and we will conclude the proof of Proposition \ref{p:R-twist}.

By Lemma \ref{l:RRstrEE}, $Hom_{\cA^{perf}}(\cB_{\cF_\eE},\cB_{\cF_\eE})[n]=R[n] \oplus R$ as an ordinary $R-R$ bimodule
so $hom_{\cA^{perf}}(\cB_{\cF_\eE},\cB_{\cF_\eE})[n]$ is quasi-isomorphic to $(R[n] \oplus R,\mu^*_{RR'})$
for some minimal $A_{\infty}$ $R-R$ bimodule structure $\mu^*_{RR'}$.
We construct an $A_{\infty}$ $R-R$ bimodule homomorphism $\bar{c}$ as follows.
\begin{align*}
&\bar{c}:(R[n] \oplus R,\mu^*_{RR'}) \to R=Hom_{\cA^{perf}}^n(\cB_{\cF_\eE},\cB_{\cF_\eE})[n] \\
 &\bar{c}^{0|1|0}(r)=r \text{ if $r \in R$ } \\
 &\bar{c}^{0|1|0}(r)=0 \text{ if $r \in R[n]$ } \\
 &\bar{c}^{r|1|s}=0 \text{ if $(r,s) \neq (0,0)$ }
 \end{align*}
 Here, the $R-R$ bimodule structure on $Hom_{\cA^{perf}}^n(\cB_{\cF_\eE},\cB_{\cF_\eE})[n]$ is induced from the quotient 
 $(R[n] \oplus R,\mu^*_{RR'})/R[n]$, which is the same as the standard $R-R$ bimodule structure on $R$, and $\bar{c}$ is the quotient bimodule homomorphism.
 We can pull back $\bar{c}$ to an $A_{\infty}$  $R-R$ bimodule homomorphism from $hom_{\cA^{perf}}(\cB_{\cF_\eE},\cB_{\cF_\eE})[n]$ to $R$, which is denoted by $c$.
%By Lemma \ref{l:tensorOUT} and a variant of \ref{dl:BimodFunMor},
It induces a natural transformation $F^1_{\otimes}(c)$ which fits into
\begin{align*}
 \cR \circ \cS[n] \simeq -\otimes_R^{\mathbb{L}} hom_{\cA^{perf}}(\cB_{\cF_\eE},\cB_{\cF_\eE})[n] \xrightarrow{F^1_{\otimes}(c)} -\otimes_R^{\mathbb{L}} R \simeq Id_{\text{mod-}R}
\end{align*}

\begin{lemma}\label{l:LeftAdjoint}
For any $X \in Ob(\cA)$, the composition
 \begin{align*}
&hom_{\cA^{perf}}(\cY_{\cA}(X),\cS(R)) \\
&\xrightarrow{\cR[n]}  hom_{\text{mod-}R}(\cR(\cY_{\cA}(X))[n],\cR(\cS(R))[n]) \\
&\xrightarrow{F^1_{\otimes}(c)} hom_{\text{mod-}R}(\cR(\cY_{\cA}(X))[n],R)
 \end{align*}
% \begin{align*}
%&hom_{\cA^{perf}}(\cY_{\cA}(X),R \otimes_R^{\mathbb{L}} \cB_{\cF_\eE}) \\
%&\xrightarrow{\cR[n]}  hom_{\text{mod-}R}(hom_{\cA^{perf}}(\cB_{\cF_\eE},\cY_{\cA}(X))[n],hom_{\cA^{perf}}(\cB_{\cF_\eE},R \otimes_R^{\mathbb{L}} \cB_{\cF_\eE})[n]) \\
%&\xrightarrow{n_c[n]} hom_{\text{mod-}R}(hom_{\cA^{perf}}(\cB_{\cF_\eE},\cY_{\cA}(X))[n],R)
% \end{align*}
is a quasi-isomorphism.

\end{lemma}

\begin{proof}
We have Yoneda embeddings 
\begin{align*}
Hom_\cA(X,\eE) &\simeq Hom_{\cA^{perf}}(\cY_{\cA}(X),R \otimes_R^{\mathbb{L}} \cB_{\cF_\eE}) \\
&= Hom_{\cA^{perf}}(\cY_{\cA}(X),\cS(R))
\end{align*}
and 
\begin{align*}
Hom_{\text{mod-}R}(Hom_\cA(\eE,X)[n],R) &\simeq H(hom_{\text{mod-}R}(hom_{\cA^{perf}}(\cB_{\cF_\eE},\cY_{\cA}(X))[n],R)) \\
& =Hom_{\text{mod-}R}(\cR(\cY_{\cA}(X))[n],R)
\end{align*}
The second equality also uses that $R$ is free.

For any cocycle $t_1 \in hom_{\cA^{perf}}(\cY_{\cA}(X),\cS(R))$ and $t_0 \in \cR(\cY_{\cA}(X))[n]$, we have (see \eqref{eq:hom-morphism})
 \begin{align*}
  &(\cR[n](t_1))(t_0)=t_1 \circ \widehat{t}_0 \in \cR(\cS(R))[n]
 \end{align*}
On the other hand, $F^1_{\otimes}(c)$ is given by composition with $c \in hom_{\text{mod-}R}(\cR(\cS(R))[n],R)$.
On the cohomological level, $[c]: H^{*}(\cR(\cS(R)))[n] \simeq Hom^{*}(\eE,\eE)[n] \to R=Hom^n(\eE,\eE)[n]$ is given by
\begin{align*}
 &[c](x)=x \text{ if $x \in Hom^{n}(\eE,\eE)$}\\
 &[c](x)=0 \text{ if $x \in Hom^{0}(\eE,\eE)$}
\end{align*} 
Therefore, the cohomological level map $[F^1_{\otimes}(c) \circ \cR[n]]$ coincides with the map $f$ in Definition \ref{d:R-twist} 
which, by assumption, is an isomorphism.
 \end{proof}

 \begin{corr}\label{c:LeftAdjoint}
  $\cR[n]$ is a left adjoint of $\cS$.
 \end{corr}

 \begin{proof}
  $\cA^{perf}$ is split-generated by the Yoneda image of $\cA$ and $R^{perf}$ is split-generated by $R$.
  Therefore, Lemma \ref{l:LeftAdjoint} proves that $F^1_{\otimes}(c)[n]$ is an adjunction counit form $\cR[n] \circ \cS$ to $Id_{\text{mod-}R}$
  and hence $\cR[n]$ is a left adjoint of $\cS$.
 \end{proof}

\begin{proof}[Proof of Proposition \ref{p:R-twist}]\label{pf:R-twist}
 By Lemma \ref{l:Cotwist}, we have an equivalence
 $\cF:=Cone(Id_{R^{perf}}\xrightarrow{u_{rs}} \cR \circ \cS) \simeq Id[-n]$, which is clearly an auto-equivalence.
 By Corollary \ref{c:LeftAdjoint}, we have 
 \begin{align*}
\cF \circ \cL   & \simeq Id[-n] \circ \cR[n] \\
& \simeq \cR 
 \end{align*}
 so Proposition \ref{p:R-twist} follows from Theorem \ref{t:SphericalFun}.
\end{proof}

\begin{rmk}\label{r:Nonformal}
 The proof of Proposition \ref{p:R-twist} works exactly the same when $R=\K[u]/u^2$ and $|u|=-1$.
 The corresponding auto-equivalence is a $\mathbb{P}$-twist which has been observed by \cite{Se17}.
 Our proof implies that the formality assumption he put is not needed, as he expected.
\end{rmk}

\subsection{$R$-spherical objects from spherical Lagrangians}\label{ss:R-spherical-Lag}

In this section, we consider $R$-spherical objects arising from spherical Lagrangians.
Let $\eP$ be the universal local system (see Section \ref{ss:UniSystem}) on $P$ and $P$ satisfies \eqref{eq:GammaCondition}.
%Let $\Gamma$ be a finite subgroup of $O(n+1)$ and $R=\K[\Gamma]$.
%Let $L$ be an exact Lagrangian $S^n/\Gamma$ and $\eP$ be $L$ equipped with the universal local system.

We construct an $R-\cA$ bimodule $\widetilde{\cB}_\eP$ as follows.
For an object $X \in \cA$ 
(and the only object of $R$), we define $\widetilde{\cB}_\eP(X,\cdot)=hom_{\cA}(X,\eP)$ as a chain complex.
 Then, we define (cf. \eqref{eq:GraphModmorphism})
\begin{align}
 &\mu^{0|1|s}_{\widetilde{\cB}_\eP}(\psi,x_s,\dots,x_1)=(-1)^{\sum_{j=1}^s \|x_j\|+1}\mu^{s+1}_{\cA}(\psi,x_s,\dots,x_1) \label{eq:sphericalBimod1}\\
 &\mu^{1|1|0}_{\widetilde{\cB}_\eP}(g,\psi)=(-1)^{|\psi|+1}g \psi \label{eq:sphericalBimod2}
\end{align}
and $\mu^{r|1|s}_{\widetilde{\cB}_\eP}=0$ otherwise.

\begin{lemma}
 $(\widetilde{\cB}_\eP,\mu^{r|1|s}_{\widetilde{\cB}_\eP})$ is a $R-\cA$ bimodule.
\end{lemma}

\begin{proof}
First notice that $\mu^{r|1|s}$ is of degree $1-r-s$.
For bimodule relations, the only non-void relation involving both the left and right module structure is the following:
$$\mu_{\widetilde{\cB}_\eP}(g,\mu_{\widetilde{\cB}_\eP}(\psi,x_s,\dots,x_1))+(-1)^{\sum_{j=1}^s\|x_j\|}\mu_{\widetilde{\cB}_\eP}(\mu_{\widetilde{\cB}_\eP}(g,\psi),x_s,\dots,x_1)=0$$
which, by \eqref{eq:sphericalBimod1} and \eqref{eq:sphericalBimod2}, reduces to
$$g\mu_\cA(\psi,x_s,\dots,x_1)-\mu_\cA(g\psi,x_s,\dots,x_1)=0$$
and it is the content of Corollary \ref{c:ActionCommute}.
\end{proof}

We can assume that elements of $hom_\cA(\eP,\eP)$ have non-negative degree and 
\begin{align*}
 hom^0_\cA(\eP,\eP)=Hom_\K(E^0_{e_L},E^1_{e_L})=Hom_\K(R,R)
\end{align*}
for a geometric transversal intersection point $e_L$.
The identity map $Id_R \in Hom_\K(R,R)$ (with respect to the canonical identification between $E^0_{e_L}$ and $E^1_{e_L}$) is a representative of the cohomological unit
so we also denote it by $e_\eP$.
We have a cohomological faithful functor $\cF_\eP:R \to \cA$ with image $\eP$ such that $g \in R$ is sent to $g e_\eP$ (ie. using the left action of $hom^0_\cA(-,\eP)$).
Therefore, the graph of $\cF_\eP$ gives another $R-\cA$ bimodule $\cB_{\cF_\eP}$
and we have
\begin{align}
 [\mu_{\cB_{\cF_\eP}}^{1|1|0}(g,\psi)]=[-\mu^2_{\cA}(g e_\eP,\psi)]=(-1)^{|\psi|+1}[g\psi] \label{eq:sphericalBimod3}
\end{align}
where the last equality comes from Corollary \ref{c:ActionCoincide}.
One should note the similarity between \eqref{eq:sphericalBimod2} and  \eqref{eq:sphericalBimod3}. 

\begin{lemma}
 $\cB_{\cF_\eP}$ is quasi-isomorphic to $\widetilde{\cB}_\eP$ as $R-\cA$ bimodule.
\end{lemma}

\begin{proof}
First recall the bimodule structure on $\cB_{\cF_\eP}$.
Since we have arranged such that $hom_\cA(\eP,\eP)$ has non-negative degree, the $A_{\infty}$ functor $\cF_{\eP}:R \to \cA$
satisfies $\cF_{\eP}^k=0$ for $k>1$, by degree reason.
Therefore, we have
\begin{align*}
  &\mu_{\cB_{\cF_\eP}}^{r|1|s}(g_r,\dots,g_1,x,a_s,\dots,a_1) \\
   =&(-1)^{1+\sum_{j=1}^s \|a_j\|}\mu_\cA^{r+s+1}(\cF_{\eP}^1(g_r),\dots,\cF_{\eP}^1(g_1),x,a_s,\dots,a_1) \\
   =&(-1)^{1+\sum_{j=1}^s \|a_j\|} \mu_\cA^{r+s+1}(g_re_\eP,\dots,g_1e_\eP,x,a_s,\dots,a_1) 
\end{align*}
%where the last equality follows form Lemma \ref{}.

 The Yondea embedding $hom_\cA(\eP,\eP) \to hom_{\text{mod-}\cA}(\cY(\eP),\cY(\eP))$ is cohomologically full and faithful.
 In particular, we have $[Id_{\cY(\eP)}]=[\cY^1(e_\eP)]$.
 Let $\cH \in hom_{\text{mod-}\cA}^{-1}(\cY(\eP),\cY(\eP))$ be a homotopy between them. That is (up to sign)
 \begin{align}\label{eq:YonedaHomotopy}
  \mu_{\cY(\eP)} \circ \widehat{\cH} + \cH \circ \widehat{\mu}_{\cY(\eP)}=Id_{\cY(\eP)}+\cY^1(e_\eP)
 \end{align}
 We are going to construct an explicit quasi-isomorphism from  $\cB_{\cF_\eP}$ to $\widetilde{\cB}_\eP$  as follows 
 (we do it without signs and leave the checking for signs as an exercise).
 \begin{align*}
   &t^{0|1|0}(x)=x \\
   &t^{0|1|s}(x,a_s,\dots,a_1)=0 \\
   &t^{r|1|s}(g_r,\dots,g_1,x,a_s,\dots,a_1)=g_r \cH^{1|r+s-1}(g_{r-1}e_\eP,\dots,g_1e_\eP,x,a_s,\dots,a_1) \text{ if $r>0$}
 \end{align*}
 %Notice that, there are $r-1$ many $e_\eP$ in the last term $\cH^{1|r+s-1}(e_\eP,\dots,e_\eP,x,a_s,\dots,a_1)$ so it will also be denoted by
 %$\cH^{1|r+s-1}(e_\eP^{\otimes r-1},x,a_s,\dots,a_1)$.
 To check that $t \in hom_{R\text{-mod-}\cA}^0(\cB_{\cF_\eP},\widetilde{\cB}_\eP)$ is a bimodule homomorphism (ie. $t$ is closed), we need to compute
 \begin{align*}
  &\mu_{\widetilde{\cB}_\eP} \circ \widehat{t} (g_r,\dots,g_1,x,a_s,\dots,a_1) \\
 =& \mu_{\widetilde{\cB}_\eP}^{r|1|s}(g_r,\dots,g_1,x,a_s,\dots,a_1) \\
 &+\sum_{k,l} \mu_{\widetilde{\cB}_\eP}^{r-k|1|l}(g_r,\dots,g_{k+1},g_k\cH^{1|k+s-l-1}(g_{k-1}e_\eP,\dots ,g_{1}e_\eP,x,a_s,\dots,a_{l+1}),a_l,\dots,a_1) \\
 =&
 \begin{cases}
  g_rg_{r-1}\cH^{1|r+s-2}(g_{r-2}e_\eP,\dots ,g_{1}e_\eP,x,a_s,\dots,a_{1}) \\
  + \sum_l \mu_\cA^{l+1}(g_r\cH^{1|r+s-l-1}(g_{r-1}e_\eP,\dots ,g_{1}e_\eP,x,a_s,\dots,a_{l+1}),a_l,\dots,a_1 ) &\text{ if $r\ge 2$}\\
  \sum_l \mu_\cA^{l+1}(g_1\cH^{1|s-l}(x,a_s,\dots,a_{l+1}),a_l,\dots,a_1 ) &\text{ if $r=1$}\\
  \mu_\cA^{s+1}(x,a_s,\dots,a_{1}) &\text{ if $r=0$} \\
 \end{cases} \\
 \end{align*}
 We also have
\begin{align*}
 &t \circ \widehat{\mu}_{\cB_{\cF_\eP}}(g_r,\dots,g_1,x,a_s,\dots,a_1) \\
 =&\sum_{j=1}^{r-1} t^{r-1|1|s}(g_r,\dots,g_{j+2},g_{j+1}g_j,g_{j-1},\dots,g_1,x,a_s,\dots,a_1) \\
 &+\sum_{k,l} t^{r-k|1|l}(g_r,\dots,g_{k+1},\mu_{\cA}^{k+s-l+1}(g_ke_\eP,\dots,g_1e_\eP,x,a_s,\dots,a_{l+1}),a_l,\dots,a_1) \\
 &+\sum_{k,l} t^{r|1|k+l+1}(g_r,\dots,g_1,x,a_s,\dots,a_{s-k+1},\mu_\cA^{s-k-l}(a_{s-k},\dots,a_{l+1}),a_l,\dots,a_1)\\
 =& g_rg_{r-1}\cH^{1|r+s-2}(g_{r-2}e_\eP,\dots,g_1e_\eP,x,a_s,\dots,a_1) \\
 &+\sum_{j=1}^{r-2}g_r\cH^{1|r+s-2}(g_{r-1}e_\eP, \dots,g_{j+2}e_\eP,g_{j+1}g_je_\eP,g_{j-1}e_\eP,\dots,g_1e_\eP,x,a_s,\dots,a_1) \\
 &+\mu_{\cA}^{r+s+1}(g_re_\eP,\dots,g_1e_\eP,x,a_s,\dots,a_1) \\
 &+\sum_{k,l} g_r \cH^{1|r+l-k-1}(g_{r-1}e_\eP,\dots, g_{k+1}e_\eP, \mu_{\cA}^{k+s-l+1}(g_ke_\eP,\dots,g_1e_\eP,x,a_s,\dots,a_{l+1}),a_l,\dots,a_1 ) \\
 &+\sum_{k,l} g_r \cH^{1|r+l+k}(g_{r-1}e_\eP,\dots,g_1e_\eP,x,a_s,\dots,a_{s-k+1},\mu_\cA^{s-k-l}(a_{s-k},\dots,a_{l+1}),a_l,\dots,a_1)
\end{align*}
On the other hand,
\begin{align*}
 &(Id_{\cY(\eP)}+\cY^1(e_\eP))(g_{r-1}e_\eP,\dots,g_1e_\eP,x,a_s,\dots,a_1) \\
 =&\begin{cases}
    \mu_\cA^{r+s+1}(e_\eP,g_{r-1}e_\eP,\dots,g_1e_\eP,x,a_s,\dots,a_1) &\text{ if $(r,s)\neq (1,0)$} \\
    x+\mu^2_\cA(e_\eP,x) &\text{ if $(r,s)= (1,0)$}
   \end{cases}\\
\end{align*}
Notice that $\cH^{1|r-1}(h_re_\eP,\dots,h_1e_\eP)=0$ for degree reason so we have
\begin{align*}
 &(\mu_{\cY(\eP)} \circ \widehat{\cH} + \cH \circ \widehat{\mu}_{\cY(\eP)})(g_{r-1}e_\eP,\dots,g_1e_\eP,x,a_s,\dots,a_1) \\
 =&\sum_l \mu_\cA^{l+1}(\cH^{1|r+s-l-1}(g_{r-1}e_\eP,\dots,g_1e_\eP,x,a_s,\dots,a_{l+1}),a_l,\dots,1) \\
 &+\sum_{j=1}^{r-2}\cH^{1|r+s-2}(g_{r-1}e_\eP, \dots,g_{j+2}e_\eP,g_{j+1}g_je_\eP,g_{j-1}e_\eP,\dots,g_1e_\eP,x,a_s,\dots,a_1) \\
 &+\sum_{k,l}  \cH^{1|r+l-k-1}(g_{r-1}e_\eP,\dots, g_{k+1}e_\eP, \mu_{\cA}^{k+s-l+1}(g_ke_\eP,\dots,g_1e_\eP,x,a_s,\dots,a_{l+1}),a_l,\dots,a_1 ) \\
 &+\sum_{k,l}  \cH^{1|r+l+k}(g_{r-1}e_\eP,\dots,g_1e_\eP,x,a_s,\dots,a_{s-k+1},\mu_\cA^{s-k-l}(a_{s-k},\dots,a_{l+1}),a_l,\dots,a_1)
\end{align*}
By comparing the terms, one can see the multiplying the Equation \eqref{eq:YonedaHomotopy} by $g_r$ on the left implies that 
$\mu_{\widetilde{\cB}_\eP} \circ \widehat{t}+t \circ \widehat{\mu}_{\cB_{\cF_\eP}}=0$
and hence $t$ is a bimodule homomorphism.
The fact that $t^{0|1|0}(x)=x$ implies that $t$ is a quasi-isomorphism.
\end{proof}

\begin{corr}
 There are natural equivalences of DG functors
 \begin{align*}
 -\otimes_R^{\mathbb{L}} \widetilde{\cB}_\eP &\simeq -\otimes_R^{\mathbb{L}} \cB_{\cF_\eP} \\
hom_{\text{mod-}\cA}(\widetilde{\cB}_\eP,-) &\simeq  hom_{\text{mod-}\cA}(\cB_{\cF_\eP},-) 
 \end{align*}

\end{corr}

\begin{proof}
 Since every object in $R\text{-mod-}\cA$ is both $K$-projective and $K$-injective, the DG functors $F_{\otimes}$ and $F_{hom}$ preserves quasi-isomorphisms.
 Therefore, we have $F_{\otimes}(\widetilde{\cB}_\eP) \simeq F_{\otimes}(\cB_{\cF_\eP}) $ and 
 $F_{hom}(\widetilde{\cB}_\eP) \simeq F_{hom}(\cB_{\cF_\eP}) $.
\end{proof}

%\begin{lemma}
% The image of $hom_{\text{mod-}\cA}(\widetilde{\cB}_\eP,-)$ lies in the full category of DG right $R$-modules.
% In particular, it is sufficent to consider adoint pairs between $dg\text{mod-}R$ and $mod_c-\cA$.
%\end{lemma}

%\begin{proof}
% Clear from formula.
%\end{proof}

\begin{lemma}\label{l:homIdnetify}
For any $X \in Ob(\cA)$,
 $hom_{\text{mod-}\cA}(\widetilde{\cB}_\eP,\cY_\cA(X))$ is quasi-isomorphic to $hom_{\cA}(\eP,X)$ as object in $\text{mod-}R$.
 As a result, $hom_{\text{mod-}\cA}(\widetilde{\cB}_\eP,\cM)$ is quasi-isomorphic to a perfect DG $R$-module for all $\cM \in \cA^{perf}$ so Assumption \ref{a:perfect} is satisfied.
\end{lemma}

\begin{proof}
%Looking at the definition of the right $R$-module structure on $hom_{\text{mod-}\cA}(\widetilde{\cB}_\eP,\cY_\cA(X))$ from Definition/Lemma \ref{dl:homDGfun},
To show that 
$hom_{\text{mod-}\cA}(\widetilde{\cB}_\eP,\cY_\cA(X))$ is quasi-isomorphic to $hom_{\cA}(\eP,X)$, it suffices to show that the Yoneda embedding 
\begin{align*}
 &hom_{\cA}(\eP,X) \to hom_{\text{mod-}\cA}(\widetilde{\cB}_\eP,\cY_\cA(X)) \\
 &x \mapsto \cY^1_\cA(x)
\end{align*}
is a right $R$-module homomorphism in the ordinary sense so $\cY^1_\cA$ is a quasi-isomorphism in $\text{mod-}R$.
One can compute, for $\phi \in hom_{\text{mod-}\cA}(\widetilde{\cB}_\eP,\cY_\cA(X))$, we have
\begin{align*}
 \mu^{0|1|1}_{hom_{\text{mod-}\cA}(\widetilde{\cB}_\eP,\cY_\cA(X))}(\phi,g)(y,a_r,\dots,a_1)
 =&(-1)^{|y|+1}\phi(\mu^{1|1|0}_{\widetilde{\cB}_\eP})(g,\phi),a_r,\dots,a_1) \\
 =&\phi(gy,a_r,\dots,a_1)
\end{align*}
When $\phi=\cY^1_\cA(x)$, it becomes
\begin{align*}
 (-1)^{\|x\| \cdot |y|+|x| (\sum_j \|a_j\|)}\mu_{\cA}(x,gy,a_r,\dots,a_1)
 =&(-1)^{\|x\| \cdot |y|+|x| (\sum_j \|a_j\|)}\mu_{\cA}(xg,y,a_r,\dots,a_1)\\
 =&\cY^1_\cA(xg)(y,a_r,\dots,a_1) \\
\end{align*}
so $\cY^1_\cA$ is a quasi-isomorphism in $\text{mod-}R$.

Moreover, $hom_{\cA}(\eP,X)$ is by construction a finitely generated semi-free $R$-module (ie. it can be constructed from free $R$ modules, taking shifts and taking mapping cones)
so $hom_{\cA}(\eP,X)$ is perfect.
\end{proof}

Given a DG right $R$-module $V$, we have an $A_{\infty}$ right $\cA$ module $V \otimes_R \eP$ (see Section \ref{ss:EquivariantEv})
which plays a key role (when $V=hom(\eP,L)$) in the previous sections. 
Now, we want to identify it with $V \otimes_R^{\mathbb{L}} \widetilde{\cB}_\eP$.

\begin{lemma}\label{l:tensorIdentify}
 For a DG right $R$-module $V$,
 $V \otimes_R^{\mathbb{L}} \widetilde{\cB}_\eP$ is quasi-isomorphic to $V \otimes_R \eP$ as object in $\text{mod-}\cA$.
\end{lemma}

\begin{proof}

 For $X \in Ob(\cA)$, we have the projection
 \begin{align*}
  t^{1|0}:&V \otimes_\K TR \otimes_\K hom_\cA(X,\eP) \to V \otimes_R hom_\cA(X,\eP) \\
          & v \otimes g_s \dots g_1 \otimes x \mapsto 
          \begin{cases}
           v \otimes_R x &\text{ if $s =0$} \\
           0 &\text{ if $s \ge 1$} \\
          \end{cases}
 \end{align*}
Let $t^{1|r}=0$ for $r>0$.
One can check that $t=(t^{1|r})_r$ defines an $A_{\infty}$ right $\cA$ module homomorphism.
%because $\mu \circ t((v \otimes g \otimes x)=0=t(vg \otimes x+ v \otimes gx)=t \circ \mu(v \otimes g \otimes x)$
%and $\mu \circ t(v \otimes x,a_r,\dots,a_1)=\mu(v \otimes_R x,a_r,\dots,a_1)=t \circ \mu(v \otimes x,a_r,\dots,a_1)$
Moreover, the differential $\mu^1_{Cone(t)(X)}$ on $Cone(t)(X)$ does not depend on the $A_{\infty}$ right $\cA$ module on $\eP$.

It is clear that $t^{1|0}$ can be obtained by applying the (underived) functor $V \otimes_R -$ to the morphism $\phi$
\begin{align*}
 R \otimes_\K TR \otimes_\K hom_\cA(X,\eP) \to hom_\cA(X,\eP) \\
 r \otimes g_s \dots g_1 \otimes x \mapsto 
          \begin{cases}
           rx &\text{ if $s =0$} \\
           0 &\text{ if $s \ge 1$} \\
          \end{cases}
\end{align*}
under the identification $V \otimes_R R \otimes TR \otimes_\K hom_\cA(X,\eP) \simeq V \otimes_\K TR \otimes_\K hom_\cA(X,\eP)$.
Since $\phi$ is the standard bar resolution of $hom_\cA(X,\eP)$ and $hom_\cA(X,\eP)$ is already an {\it K-projective} object (see Remark \ref{r:K-projective}) as a left DG $R$ 
module before resolution, $Cone(t)(X) \simeq V \otimes_R Cone(\phi) $ is acyclic.
%\footnote{More details about these facts can be found in \cite[Proposition $10.3.3$]{Yek}}

%Since $hom_\cA(X,\eP)$ is an {\it h-projective} object as a left DG $R$ module, $Cone(t)(X)$ is acyclic. 
%quasi-isomorphism from $V \otimes_R \eP(X)$ to $V \otimes_R \widetilde{\cB}_\eP(X)$ for all $X$.

\end{proof}

\begin{prop}\label{p:algebraicLagTwist}
 $T_\eP(X)=\cT_{F_{\otimes}(\widetilde{\cB}_\eP)}(X)$ for all $X \in Ob(\cA)$, where $T_{F_{\otimes}(\widetilde{\cB}_\eP)}$ is the twist auto-equivalence.
% , which is the twist along the spherical functor $\cF_\eP$.
\end{prop}

\begin{proof}

By Lemma \ref{l:homIdnetify} and \ref{l:tensorIdentify}, we have quasi-isomorphisms
\begin{align*}
 &hom_{\text{mod-}\cA}(\widetilde{\cB}_\eP,\cY_\cA(X)) \otimes_R \eP \to hom_{\text{mod-}\cA}(\widetilde{\cB}_\eP,\cY_\cA(X))\otimes_R^{\mathbb{L}} \widetilde{\cB}_\eP \\
 &x \otimes y \mapsto \cY^1_{\cA}(x) \otimes y
\end{align*}
and 
\begin{align*}
 &hom_{\text{mod-}\cA}(\widetilde{\cB}_\eP,\cY_\cA(X)) \otimes_R \eP \to hom_{\cA}(\eP,X) \otimes_R \eP  \\
 &x \otimes y \mapsto x \otimes y
\end{align*}
By Proposition \ref{p:TensorHom}, the adjunction counit $(c_{sr})_{\cY_\cA(X)}^0$ gives
 %\begin{align*}
 % (c_{sr})_{\cY_\cA(X)}^0 \in hom_{\text{mod-}\cA} ( hom_{\text{mod-}\cA}(\widetilde{\cB}_\eP,\cY_\cA(X))\otimes_R^{\mathbb{L}} \widetilde{\cB}_\eP, \cY_\cA(X))
 %\end{align*}
 \begin{align}
  (c_{sr})_{\cY_\cA(X)}^0(\cY^1_{\cA}(x) \otimes y,a_r,\dots,a_1)&=(-1)^{|x|(|y|+\sum_{j=1}^r \|a_j\|)}\cY_\cA^1(x)(y,a_r,\dots,a_1) \\
  &=(-1)^{|y|} \mu_\cA(x,y,a_r,\dots,a_1) \label{eq:twistedEv}
 \end{align}
 which coincides with the $A_{\infty}$ evaluation morphism \eqref{eq:TwistedEvMap}.
 Therefore, we have $T_\eP(X)=\cT_{F_{\otimes}(\widetilde{\cB}_\eP)}(X)$ for all $X \in Ob(\cA)$.
\end{proof}

%\begin{rmk}
% Let $R^{perf}_{dg}$ be the full subcategory of $R^{perf}$ consisting of DG $R$-modules. 
% In fact, we can further restrict our adjoint functors to $F_{\otimes}(\widetilde{\cB}_\eP):R^{perf}_{dg} \to \cA^{perf}$ 
% and $F_{hom}(\widetilde{\cB}_\eP):\cA^{perf} \to R^{perf}_{dg}$.
%\end{rmk}

\begin{corr}\label{c:functoralityAuto}
 $T_{\cG(\eP)}(\cG(X)) \simeq \cG(T_\eP(X))$ for any quasi-equivalence $\cG:Fuk \to Fuk'$.
\end{corr}

\begin{proof}
 It follows from Lemma \ref{l:conjugationInvariant} and Proposition \ref{p:algebraicLagTwist}.
\end{proof}

\begin{proof}[Proof of Theorem \ref{t:Twist formula(alg)}]
Theorem \ref{t:Twist formula(alg)} is a consequence of Proposition \ref{p:algebraicLagTwist} and Theorem \ref{t:SphericalFun}.
\end{proof}

\appendix

\section{Orientations} % (fold)
\label{sec:orientations}

In this appendix, we will discuss the orientation of various moduli spaces appeared in this paper.  
Our goal is to prove Proposition \ref{p:CohLevelIso} when $\cchar(\K) \neq 2$.
We follow the sign convention in \cite{Seidelbook}.
Readers are referred to \cite[Section 11,12]{Seidelbook} for basic definitions, from which we follow largely in the expositions.

\subsection{Orientation operator}

A linear Lagrangian brane $\Lambda^{\#}=(\Lambda,\alpha^\#,P^\#)$
consists of 

\begin{itemize}
  \item a Lagrangian subspace $\Lambda \subset \C^n$
  \item a phase $\alpha^\# \in \R$ such that $e^{2\pi \sqrt{-1} \alpha^\#}=Det^2_{\Omega}(\Lambda)$
  \item a $Pin_n$-space $P^\#$ together with an isomorphism $P^\# \times_{Pin_n} \R^n \cong \Lambda$.
\end{itemize}

Here, $Det^2_{\Omega}$ is the square of the standard complex volume form on $\C^n$.
The $k$-fold shift $\Lambda^{\#}[k]$ of $\Lambda^{\#}$ is given by $(\Lambda,\alpha^\#-k,P^\# \otimes \lambda^{top}(\Lambda)^{\otimes k})$, where $\lambda^{top}$ is
the top exterior power.
For every pair of linear Lagrangian branes $(\Lambda_0^\#,\Lambda_1^\#)$,
one can define the index $\iota(\Lambda_0^\#,\Lambda_1^\#)$ and an orientation line (i.e. a rank one $\R$-vector space)
$o(\Lambda_0^\#,\Lambda_1^\#)$.

Now, we explain how the indices and orientation lines are related to Fredholm operators.
Let $S \in \cR^{d+1}$, and $E=S \times \C^n$ be regarded as a trivial symplectic vector bundle over $S$.
Let $F \subset E$ be a Lagrangian subbundle over $\partial S$.
For each strip-like end $\epsilon^i$, we assume $F|_{\epsilon^{i}(s,j)}$ is independent of $s$ for $j=0,1,$.
On top of that, we pick a continuous function $\alpha^\#:\partial S \to \R$ and a $Pin$-structure $P^\#$ on $F$ such that 
$e^{2\pi \sqrt{-1} \alpha^\#(x)}=Det^2_{\Omega}(F_x)$
for all $x \in \partial S$.
In this case, we get a pair of linear Lagrangian branes $(\Lambda^\#_{\xi^i,0},\Lambda^\#_{\xi^i,1})$ for each puncture $\xi^i$, where
$\Lambda^\#_{\xi^i,j}=(F|_{\epsilon^{i}(s,j)},\alpha^\#(\epsilon^{i}(s,j)), P^\#_{\epsilon^{i}(s,j)})$ for $j=0,1$.
We can associate a Fredholm operator $D_{S,F}$ to these data and we have \cite[Proposition $11.13$]{Seidelbook}
\begin{align}
 \ind(D_{S,F})=&\iota(\Lambda^\#_{\xi^0,0},\Lambda^\#_{\xi^0,1})- \sum_{i=1}^d \iota(\Lambda^\#_{\xi^i,0},\Lambda^\#_{\xi^i,1})\\
  o(\Lambda^\#_{\xi^0,0},\Lambda^\#_{\xi^0,1}) \cong&  \det(D_{S,F}) \otimes o(\Lambda^\#_{\xi^d,0},\Lambda^\#_{\xi^d,1}) \otimes \dots \otimes o(\Lambda^\#_{\xi^1,0},\Lambda^\#_{\xi^1,1}) \label{eq:operator}
\end{align}
where $\ind(D_{S,F})$ and $\det(D_{S,F})$ are the index and determinant line of the operator, respectively.

In the reverse direction, given $(\Lambda_0^\#,\Lambda_1^\#)$, one can pick $S$ to be the upper half plane $H$
and $(F,\alpha,P)$ such that the pair of linear Lagrangian branes at the puncture of $S$ is $(\Lambda_0^\#,\Lambda_1^\#)$.
In this special case, the operator $D_{H,F}$ has the property that $\ind(D_{H,F})=\iota(\Lambda_0^\#,\Lambda_1^\#)$
and $\det(D_{H,F}) \cong o(\Lambda_0^\#,\Lambda_1^\#)$.
We call $D_{H,F}$ an \textit{orientation operator} of $(\Lambda_0^\#,\Lambda_1^\#)$.

Let $\rho$ be a path of Lagrangian branes from $\Lambda_1^\#$ to $\Lambda_1^\#[1]$.
Let $S$ be the closed unit disk $D$ and $(F,\alpha^\#,P^\#)$ be given by $\rho(\theta)$ at the point $e^{2\pi \sqrt{-1} \theta} \in \partial S$.
We denote the corresponding operator by $D_{D,\rho}$ and call it a shift operator.
There are gluing theorems concerning how indices and determinant lines are related before and after gluing two operators at a puncture or a boundary point
\cite[$(11.9)$, $(11.11)$]{Seidelbook}.
In particular, we can glue an orientation operator of $(\Lambda_0^\#,\Lambda_1^\#)$ with $D_{D,\rho}$ at boundary points that both fibers are $\Lambda_1^\#$ and obtain
\begin{align}
o(\Lambda_0^\#,\Lambda_1^\#) \otimes \det(D_{D,\rho}) \cong o(\Lambda_0^\#,\Lambda_1^\#[1]) \otimes \lambda^{top}(\Lambda_1)
\end{align}
By \cite[Lemma $11.17$]{Seidelbook}, there is a canonical isomorphism $\det(D_{D,\rho}) \cong \lambda^{top}(\Lambda_1)$
so we have a canonical isomorphism
\begin{align}
 \sigma: o(\Lambda_0^\#,\Lambda_1^\#) \cong o(\Lambda_0^\#,\Lambda_1^\#[1]) \label{eq:shiftIsom}
\end{align}
Therefore, there is a canonical isomorphism between $o(\Lambda_0^\#,\Lambda_1^\#)$
and $o(\Lambda_0^\#,\Lambda_1^\#[k])$ for all $k \in \Z$.

Similarly, we can consider a path of Lagrangian branes $\tau$ from $\Lambda_0^\#[1]$ to $\Lambda_0^\#$.
We can use $S=D$ and $\tau$ to define an operator $D_{D,\tau}$ which we call a front-shift operator.
In this case, we can glue an  orientation operator of $(\Lambda_0^\#,\Lambda_1^\#)$ with $D_{D,\tau}$
at boundary points that both fibers are $\Lambda_0^\#$ and obtain
\begin{align}
o(\Lambda_0^\#,\Lambda_1^\#) \otimes \det(D_{D,\tau}) \cong o(\Lambda_0^\#[1],\Lambda_1^\#) \otimes \lambda^{top}(\Lambda_0)
\end{align}
By \cite[Lemma $11.17$]{Seidelbook}, there is a canonical isomorphism $\det(D_{D,\tau}) \cong \lambda^{top}(\Lambda_0)$
so we have a canonical isomorphism
\begin{align}
 \eta: o(\Lambda_0^\#,\Lambda_1^\#) \cong o(\Lambda_0^\#[1],\Lambda_1^\#) \label{eq:DshiftIsom}
\end{align}

\subsection{Floer differential and product}\label{ss:Ashift}

Let $L_i$, $i=0,1$, be closed Lagrangian submanifolds equipped with a grading function
$\theta_{L_i}:L_i \to \R$ (see Section \ref{ss:Grading})
and a spin structure.
We assume that $L_0 \pitchfork L_1$.
At each point $x \in L_i$, we have a Lagrangian brane $T_xL_i^\#=(T_xL_i, \theta_{L_i}(x),Pin_x)$ inside $T_xM$ where
$Pin_x$ is the $Pin_n$-space determined by the spin structure on $L_i$.
The $k$-fold shift $L_i[k]$ of $L_i$ is given by applying $k$-fold shift to $T_xL_i^\#$ for all $x \in L_i$.
For each $x \in L_0 \cap L_1$, we have a pair of Lagrangian branes $(T_xL_0^\#,T_xL_1^\#)$ inside $T_xM$.
Therefore, we have the grading $|x|:=\iota(T_xL_0^\#,T_xL_1^\#)$ and the orientation line $o(x):=o(T_xL_0^\#,T_xL_1^\#)$.
We define $|o(x)|_\K$ to be the one dimensional $\K$-vector space generated by the two orientations of $o(x)$
modulo the relation that their sum is zero.
An isomorphism $c:o(x) \to o(x')$ between two orientation lines can induces an isomorphism $|c|_\K: |o(x)|_\K \to |o(x')|_\K$.

Let $x_0,x_1 \in L_0 \cap L_1$ and $u:S=\R\times [0,1] \to M$
be a rigid element in $\eM(x_0;x_1)$.
Using the trivialization of $\Lambda^{top}_\C(M,\omega)$ together with the grading functions and spin structures on $L_i$,
we get a trivial bundle $E=u^*TM=S \times \C$ and a Lagrangian subbundle $F$ together with $(\alpha^\#,P^\#)$ over $\partial S$.
By \eqref{eq:operator}, we get a canonical isomorphism
\begin{align}
 \det(D_u) \cong o(x_0) \otimes o(x_1)^\vee
\end{align}
On the other hand, the $s$-translation $\R$-action on $u$ induces a short exact sequence
\begin{align}
 \R \to T_u\widetilde{\eM}(x_0;x_1) \to T_u\eM(x_0;x_1)
\end{align}
where $\widetilde{\eM}(x_0;x_1)$ is the moduli space of strips before modulo the $\R$-action.
Therefore, we have an identifcation of the top exterior power of $T_u\widetilde{\eM}(x_0;x_1)$
and $ T_u\eM(x_0;x_1)$, respectively.
As a result, an orientation of $\eM(x_0;x_1)$ gives an isomorphism (see \eqref{eq:operator})
\begin{align}
 c_u:o(x_1) \to o(x_0) \label{eq:diffOri}
\end{align}
Therefore, we can define the Floer cochain complex by
\begin{align}
 CF(L_0,L_1)=\oplus_{x \in L_0 \cap L_1} |o(x)|_\K
\end{align}
and the differential $\partial$ on $|o(x)|_\K$ is given by summing
\begin{align}
 \partial^{x',x}= \sum_{u \in \eM(x';x)} |c_u|_\K: |o(x)| \to |o(x')|
\end{align}
over all $x'$ such that $|x'|=|x|+1$. We have $\partial^2=0$ \cite[Section (12f)]{Seidelbook}.
Similarly, given a collection of pairwisely transversally intersecting Lagrangian branes $\{L_j\}_{j=0}^d$,
$x_j \in L_{j-1} \cap L_j$, $j=1,\dots,d$, and $x_0 \in L_0 \cap L_d$, we get an isomorphism (after an orientation of $\cR^{d+1}$ is chosen)
\begin{align}
 c_u:o(x_d) \otimes \dots \otimes o(x_1) \to o(x_0)
\end{align}
for each rigid element $u \in \eM(x_0;x_d,\dots,x_1)$, and hence a multilinear map between the relevant Floer cochain complexes.
Assuming the convention of orientations in \cite{Seidelbook}.
The actual $A_{\infty}$ structural map $\mu^d(x_d,\dots,x_1)$ is 
given by summing over all $|c_u|_\K$ with a sign twist given by $(-1)^{\dagger}$ (see \cite[Section (12g)]{Seidelbook}), where
\begin{align}
 \dagger=\sum_{k=1}^d k|x_k| \label{eq:signTwist}
\end{align}
In particular, $\mu^1(x)=(-1)^{|x|} \partial(x)$.

We are interested in how Floer differentials and $\mu^2$-products (i.e. $d=1,2$) behave under shifts \eqref{eq:shiftIsom}, \eqref{eq:DshiftIsom}.
Let $x \in L_0 \cap L_1$ be equipped with a pair of Lagrangian branes $(T_xL_0^\#,T_xL_1^\#)$.  We use $\tilde{x}$ (resp. $\bar{x}$) to denote the same intersection $x$ being
equipped with the pair of Lagrangian branes $(T_{x}L_0^\#,T_xL_1^\#[1])$ (resp. $(T_{x}L_0^\#[1],T_xL_1^\#)$).
We denote the canonical isomorphism \eqref{eq:shiftIsom} (resp. \eqref{eq:DshiftIsom}) at $x$ by $\sigma_x:o(x) \to o(\tilde{x})$
(resp. $\eta_x:o(x) \to o(\bar{x})$).
For $x_0,x_1 \in L_0 \cap L_1$ and a rigid element $u \in \eM(x_0;x_1)$, we denote $u$ by $\tilde{u}$ (resp. $\bar{u}$) when we regard it as an element in $\eM(\tilde{x}_0;\tilde{x}_1)$
(resp. $\eM(\bar{x}_0;\bar{x}_1)$).
It is explained in \cite[Section $12h$]{Seidelbook} that
\begin{align}
 \sigma_{x_0} \circ c_u = c_{\tilde{u}} \circ  \sigma_{x_1} \label{eq:DiffComShift}
\end{align}
It is instructive to recall the reasoning behind \eqref{eq:DiffComShift}.
Consider orientation operators $D_{H,x_i}$, $D_{H,\tilde{x}_i}$, the shift operators $D_{D,\rho,x_i}$ at $x_i$
and the linearized operator $D_{u}$ defining the Floer differential.  The left hand side of \eqref{eq:DiffComShift} $\sigma_{x_0} \circ c_u $ is obtained by first gluing $D_{u}$ with $D_{H,x_1}$,  then $D_{u}\# D_{H,x_1}$ with $D_{D,\rho,x_0}$; the right hand side $c_{\tilde{u}} \circ  \sigma_{x_1}$ is obtained from gluing $D_{H,x_1}$ with $D_{D,\rho,x_1}$ first, and then $D_{u}$ with $D_{H,x_1} \# D_{D,\rho,x_1}$.

Since the operators $(D_{u}\# D_{H,x_1}) \# D_{D,\rho,x_0}$ and $D_{u} \# (D_{H,x_1} \# D_{D,\rho,x_1})$
are homotopic (meaning the underlying path of Lagrangian subspace on the boundary are homotopic), the associativity of determinant line under gluing implies \eqref{eq:DiffComShift}.

Similarly, we have
\begin{align}
 \eta_{x_0} \circ c_u = c_{\bar{u}} \circ  \eta_{x_1} \label{eq:DiffComShift2}
\end{align}
so \eqref{eq:DiffComShift} and \eqref{eq:DiffComShift2} implies that
\begin{align}
   \sigma\circ\partial=\partial\circ\sigma, \hskip 1cm 
   \eta\circ\partial=\partial\circ\eta \label{e:commute2}
 \end{align} 
Now, we consider the Floer product.
Let $u \in \eM(x_0;x_2,x_1)$ where $x_0 \in L_0 \cap L_2$ and $x_j \in L_{j-1} \cap L_j$ for $j=1,2$.
We use $u'$ to denote $u$ when we regard it as an element in $\eM(x_0;\bar{x}_2,\tilde{x}_1)$.
We continue to use $D_{H,*}$ to denote an orientation operator of a Lagrangian intersection point $*$ (equipped with pair of Lagrangian branes).
The gluings of $D_{D,\rho,x_1}$ and $D_{D,\tau,x_2}$ induce the $\sigma$-operator at $x_1$ and $\eta$-operator at $x_2$, respectively.
The operator $(D_u \# D_{H,x_2}) \# D_{H,x_1}$
is homotopic to $(D_{u'} \# D_{H,\bar{x}_2}) \# D_{H,\tilde{x}_1}$,
and $D_{H,\bar{x}_2}\sim D_{H,x_2} \# D_{D,\tau,x_2}$, $D_{H,\tilde{x}_1}\sim D_{H,x_1} \# D_{D,\rho,x_1}$
are homotopies of operators.
It implies that there is an equality
\begin{align}
 c_u=(-1)^{|x_1|}c_{u'} \circ (\eta_{x_2} \otimes \sigma_{x_1}) \label{eq:ProdComShift}
\end{align}
%\begin{align}
% c_u: o(x_2) \otimes o(x_1) \to o(x_0)
%\end{align}
%and
%\begin{align}
% c_u' \circ (\eta_{x_2} \otimes \sigma_{x_1}): o(x_2) \otimes o(x_1) \to o(x_0)
%\end{align}
where the sign $(-1)^{|x_1|}$ comes from \eqref{eq:DiffComShift} when moving $D_{D,\tau,x_2}$ pass $D_{H,x_1}$. 

 % when applying $\eta_{x_2} \otimes \sigma_{x_1}$.

We abuse the notation and denote the canonical isomorphism from $CF(L_0,L_1)$ to $CF(L_0,L_1[1])$ (resp. $CF(L_0[1],L_1)$) by $\sigma$ (resp. $\eta$).
Denote the operator 
\begin{align}
 (-1)^{\deg}:a\mapsto(-1)^{|a|}(a)
\end{align}
for elements of pure degree $|a|$ (and extend linearly), then $\mu^1=\partial\circ(-1)^{deg}$.
%be the isomorphism such that $(-1)^{\deg-1} \sigma (a_x)=(-1)^{|x|-1} \sigma(a_x)$ 
%and $(-1)^{\deg-1} \eta (a_x)=(-1)^{|x|-1} \eta(a_x)$
%for $a_x \in |o(x)|_\K$.
Combining \eqref{eq:signTwist}, \eqref{e:commute2}, \eqref{eq:ProdComShift} we have 
\begin{align}
 \mu^1 \circ ((-1)^{\deg} \circ \sigma)&= ((-1)^{\deg} \circ \sigma) \circ \mu^1     \label{eq:Fcomm1}\\
 \mu^1 \circ \eta&= - \eta \circ \mu^1 \label{eq:Fcomm3}\\
 \mu^2&=\mu^2 \circ  ( \eta \otimes  ((-1)^{\deg} \circ \sigma)) \label{eq:Fcomm2}
\end{align}
Note that \eqref{eq:Fcomm1} is equivalent to $\mu^1 \circ \sigma=-\sigma \circ \mu^1$
but $(-1)^{\deg} \circ \sigma$ will be used later so we prefer to write in this form.

\subsection{Matching orientations}\label{ss:FirstKind}

We use the notations in Section \ref{sec:quasi_isomorphism}.
In Section \ref{sub:matching_differentials}, we proved that there are bijective identifications between the moduli
\begin{align}
  \eM^{J^{\tau}}(\fp';\fp) &\simeq \eM^{J^{\tau}}(c_{\fp',\fq};c_{\fp,\fq}) \label{eq:biM1} \\
  \eM^{J^\tau}(x; \mathbf{q}^\vee,\mathbf{p}) &\simeq \eM^{J^\tau}(x; c_{\fp,\fq})\label{eq:biM2}\\
  \eM^{J^\tau}(\mathbf{q}'^\vee;\mathbf{q}^\vee) &\simeq \eM^{J^{\tau}}(c_{\fp,\fq'};c_{\fp,\fq}) \label{eq:biM3}
\end{align}
Let $\mu^{1,1}$, $\mu^{1,2}$ and $\mu^{1,3}$ be the terms of the differential of $CF(L_0,\tau_P(L_1))$ contributed by the moduli on the right hand side of \eqref{eq:biM1}, 
\eqref{eq:biM2} and \eqref{eq:biM3}, respectively.
In particular, we have
\begin{align}
 \mu^1=\mu^{1,1}+\mu^{1,2}+\mu^{1,3}
\end{align}
and (after modulo signs)
\begin{align}
 \iota \circ (id \otimes \mu^1) &=\mu^{1,1} \circ \iota \\
 \iota \circ \mu^2 &=\mu^{1,2} \circ \iota \\
 \iota \circ (\mu^1 \otimes id) &=\mu^{1,3} \circ \iota
\end{align}
To finish the proof of Proposition \ref{p:CohLevelIso}, it suffices to find a collection of isomorphisms
\begin{align}
 I_{\fp,\fq}: o(\fq^\vee) \otimes o(\fp) \to o(c_{\fp,\fq})
% I_x: o(x) \to o(\iota(x))
\end{align}
for all $\fq^\vee \otimes \fp \in \cX_a(C_0)$
%and $x \in \cX_b(C_0)$ 
such that 
\begin{align}
 |I|_{\K} \circ (id \otimes \mu^1) &= \mu^{1,1} \circ |I|_\K \label{eq:biMM1} \\
  |I|_{\K} \circ \mu^2 &= \mu^{1,2} \circ |I|_\K \label{eq:biMM2} \\
   |I|_{\K} \circ (\mu^1 \otimes (-1)^{\deg-1}) &= \mu^{1,3} \circ |I|_\K \label{eq:biMM3}
\end{align}
where $I=(\oplus_{\fq^\vee \otimes \fp \in \cX_a(C_0)} I_{\fp,\fq}) \oplus ( \oplus_{x \in \cX_b(C_0)} id_{o(x)})$,
and $id_{o(x)}$ is the identity morphism from $o(x)$ to $o(\iota(x))=o(x)$ for $x \in \cX_b(C_0)$.
Notice that, the sign in \eqref{eq:biMM3} (and no sign in \eqref{eq:biMM1}, \eqref{eq:biMM2}) comes from the fact that (see Section \ref{ss:EquivariantEv})
\begin{align}
 \mu^1(\fq^\vee \otimes \fp)=(-1)^{|\fp|-1} \mu^1(\fq^\vee) \otimes \fp+ \fq^\vee \otimes \mu^1(\fp)+\mu^2(\fq^\vee, \fp)
\end{align}

In this section, we give the definition of $I_{\fp,\fq}$ and check that \eqref{eq:biMM1}, \eqref{eq:biMM2}, \eqref{eq:biMM3} hold.
Since the sign computation is local in nature and it is preserved under the covering map $T^*\fU \to T^*U$,
we assume that $\eE=\fP=S^n$.

First, we consider the case that $\fq^\vee \otimes \fp \in CF(\fP,L_1) \otimes CF(L_0,\fP)$ satisfies $|\fq^\vee|=1$.
In this case, we can perform a graded Lagrangian surgery (see \cite{SeGraded} or \cite{MW15})
$\fP \#_\fq T^*_\fq \fP$,
which means that $\fP \#_\fq T^*_\fq \fP$ can be equipped with a grading function so that its restriction to $\fP \backslash \{\fq\}$
 and $T^*_\fq \fP \backslash \{\fq\}$ are the same as the grading functions on $\fP \backslash \{\fq\}$
 and on $T^*_\fq \fP \backslash \{\fq\}$, respectively.
 Moreover, all $\fP$, $T^*_\fq \fP$ and $\fP \#_\fq T^*_\fq \fP$ are spin and the (unique) spin structure on 
 $\fP \#_\fq T^*_\fq \fP$ restricts to the (unique) spin structure on $\fP \backslash \{\fq\}$ and on $T^*_\fq \fP \backslash \{\fq\}$, respectively.
 %\footnote{Maybe add a reference on this uniqueness of glued spin structure, and mention the condition of relative cohomology vanishing condition?} 
 
 In this case, we have a canonical identification of $o(\fp)$, viewed as a subspace of 
 $CF(T^*_\fp \fP, \fP)$ and of $CF(T^*_\fp \fP, \fP \#_\fq T^*_\fq \fP)$, respectively.
 Moreover, $\fP \#_\fq T^*_\fq \fP$ is Hamiltonian isotopic to $\tau_\fP (T^*_\fq \fP)$, which sends $\fp$ to $c_{\fp,\fq}$,
 and the Hamiltonian interwines the brane structures (i.e. grading functions and spin structures on the Lagrangians).
 Therefore, we have an isomorphism
 \begin{align}
  \Phi_{Ham}: o(\fp)  \cong o(c_{\fp,\fq})
 \end{align}
from $o(\fp) \subset CF(T^*_\fp \fP, \fP)$ to $o(c_{\fp,\fq}) \subset CF(T^*_\fp \fP, \tau_\fP (T^*_\fq \fP))$.
Any choice of an isomorphism 
\begin{align}
 \Phi_{sur}: o(\fq^\vee) \to \R
\end{align}
will give us an isomorphism
\begin{align}
 \Phi:=\Phi_{sur} \otimes \Phi_{Ham}: o(\fq^\vee) \otimes o(\fp) \to \R \otimes o(c_{\fp,\fq}) =o(c_{\fp,\fq})
\end{align}
for every $\fq^\vee \otimes \fp$ such that $|\fq^\vee|=1$.
We assume that a choice of $\Phi_{sur}$ is made for the moment 
(the actually choice will be uniquely determined by a property used in Lemma \ref{l:biMM2}).

Now, for general $\fq^\vee \otimes \fp$, we consider the isomorphism (see Section \ref{ss:Ashift})
\begin{align}
 \phi:= \eta \otimes ((-1)^{\deg} \circ \sigma): CF(\fP,L_1) \otimes CF(L_0,\fP) \to CF(\fP[1],L_1) \otimes CF(L_0,\fP[1])
\end{align}
and we define $I_{\fp,\fq}$ by
\begin{align}
 I_{\fp,\fq}:= \Phi \circ \phi^{1-|\fq^\vee|}: o(\fq^\vee) \otimes o(\fp) \to o(c_{\fp,\fq})
\end{align}
Notice that $|\sigma^{1-|\fq^\vee|}(\fp)|=|\fp|+|\fq^\vee|-1=|c_{\fp,\fq}|$, and one should view this isomorphism as 
identifying $o(\fp)$ with $o((\sigma)^{1-|\fq^\vee|}(\fp))$ by a {\bf sign-twisted} shift followed by identifying  $o((\sigma)^{1-|\fq^\vee|}(\fp))$  and $o(c_{\fp,\fq})$
by a Hamiltonian isotopy.
Readers should be convinced from \eqref{eq:Fcomm1} that it is sensible to use the sign-twisted shift $(-1)^{\deg} \circ \sigma$.

\begin{lemma}\label{l:biMM2}
 There is a choice of $\Phi_{sur}$ such that \eqref{eq:biMM2} holds.
\end{lemma}

\begin{proof}
To prove \eqref{eq:biMM2}, we start with the case that $|\fq^\vee|=1$.
The bijection \eqref{eq:biM2} is obtained by the bijection  
$  \eM^{J^-}(\emptyset; \mathbf{q}^\vee,\mathbf{p}, x_{\fq,\fp}) \simeq \eM^{J^-}(\emptyset; c_{\fp,\fq},x_{\fq,\fp})$.
As before, we identify $o(c_{\fp,\fq})$ with $o(\fp)$
by the Hamiltonian isotopy defining $\Phi_{Ham}$.
In this case, the linearized operator $D_{c_{\fp,\fq},x_{\fq,\fp}}$ corresponding to the latter moduli
is homotopic to $D_{\mathbf{q}^\vee,\mathbf{p}, x_{\fq,\fp}} \# D_{H,\fq^\vee}$, where 
$D_{\mathbf{q}^\vee,\mathbf{p}, x_{\fq,\fp}}$ is the linearized operator corresponding to the former moduli
and $D_{H,\fq^\vee}$ is an orientation operator of $\fq^\vee$.
The fact that these two operators are homotopic is a reflection of the fact that we can perform a graded Lagrangian surgery 
$\fP \#_\fq T^*_\fq \fP$ compactible with the spin structures when $|\fq^\vee|=1$.
As a result, there is a choice of $\Phi_{sur}$ such that 
\begin{align}
 c_u = c_u' \circ (\Phi_{sur} \otimes \Phi_{Ham}): o(\fq^\vee) \otimes o(\fp) \to o(x)
\end{align}
where $u \in \eM^{J^\tau}(x; \mathbf{q}^\vee,\mathbf{p})$
and $u' \in \eM^{J^\tau}(x; c_{\fp,\fq})$ is the element corresponding to $u$ under the bijection \eqref{eq:biM2}.
We use such a choice of $\Phi_{sur}$ from now on.
In particular, it means that
\begin{align}
 \mu^2 = \mu^{1,2} \circ |\Phi|_\K \label{eq:FFcomm2}
\end{align}
for $\fq^\vee \otimes \fp$ such that $|\fq^\vee|=1$.
For general $\fq^\vee \otimes \fp$, we use \eqref{eq:Fcomm2} and \eqref{eq:FFcomm2} to deduce that
\begin{align}
 |I|_\K \circ \mu^2= |\Phi|_\K \circ \mu^2 \circ |\phi^{1-|\fq^\vee|}|_\K= \mu^{1,2} \circ |I|_\K
\end{align}
which is exactly the desired \eqref{eq:biMM2}.
\end{proof}

With the choice of $\Phi_{sur}$ chosen in Lemma \ref{l:biMM2}, we can now proceed and prove \eqref{eq:biMM1}, \eqref{eq:biMM3}.

\begin{lemma}\label{l:biMM1}
 The equation \eqref{eq:biMM1} holds.
\end{lemma}

\begin{proof}
To show \eqref{eq:biMM1}, we again first consider $\fq^\vee \otimes \fp$ such that $|\fq^\vee|=1$.
Let $\fp' \in L_0 \cap \fP$ such that $|\fp'|=|\fp|+1$.
The bijection \eqref{eq:biM1} is obtained from the bijection
$\eM^{J^{-}}(\fp';\fp,x_{\fp',\fp}) \simeq \eM^{J^{-}}(c_{\fp',\fq};c_{\fp,\fq},x_{\fp',\fp})$.
By the Hamiltonian isotopy defining $\Phi_{Ham}$, we see that the linearized operator corresponding to the former moduli is homotopic to 
the linearized operator corresponding to the latter moduli.
It implies that
\begin{align}
 \Phi_{Ham} \circ c_u =c_{u'} \circ \Phi_{Ham}: o(\fp) \to o(c_{\fp',\fq})
\end{align}
where $u \in \eM^{J^{\tau}}(\fp';\fp)$ and $u' \in \eM^{J^{\tau}}(c_{\fp',\fq};c_{\fp,\fq})$ is the element corresponding to $u$ under the bijection \eqref{eq:biM1}.
It implies that (note that $|\fp|=|c_{\fp,\fq}|$ and $\mu^1(a)=(-1)^{|a|} \partial(a)$, see \eqref{eq:signTwist})
\begin{align}
 |\Phi_{Ham}|_\K \circ \mu^1 = \mu^{1,1} \circ |\Phi_{Ham}|_\K 
\end{align}
for $\fq^\vee \otimes \fp$ such that $|\fq^\vee|=1$. It also means that, 
no matter what isomoprhism we choose for $\Phi_{sur}$, we have
\begin{align}
 |\Phi|_\K \circ (id \otimes \mu^1) = \mu^{1,1} \circ |\Phi|_\K \label{eq:FFcomm1}
\end{align}
For general $\fq^\vee \otimes \fp$, we use \eqref{eq:Fcomm1} and \eqref{eq:FFcomm1} to deduce that
\begin{align}
 |I|_\K \circ (id \otimes \mu^1)&=|\Phi_{sur} \circ \eta^{1-|\fq^\vee|}|_\K \otimes |\Phi_{Ham} \circ ((-1)^{\deg} \circ \sigma)^{1-|\fq^\vee|}|_\K \circ \mu^1   \\
 &=|\Phi_{sur} \circ \eta^{1-|\fq^\vee|}|_\K \otimes  (\mu^{1,1} \circ |\Phi_{Ham} \circ ((-1)^{\deg} \circ \sigma)^{1-|\fq^\vee|}|_\K) \\
 &= \mu^{1,1} \otimes |I|_\K
\end{align}
which is exactly the desired \eqref{eq:biMM1}.
\end{proof}

%\begin{rmk}
% The choice of $\Phi_{sur}$ is not important for the proof of Lemma \ref{l:biMM1} (in particular \eqref{eq:FFcomm1}) boils down to the trivial fact that
% $\mu^1(a)=-\mu^1(-a)$.
%\end{rmk}

%\subsection{The remaining ones}\label{ss:Remaining}

%We are going to establish \eqref{eq:biMM2} and \eqref{eq:biMM3} in this section.

\begin{lemma}\label{l:biMM3}
 The equation \eqref{eq:biMM3} holds.
\end{lemma}

\begin{proof}
To prove \eqref{eq:biMM3}, we appeal to an algebraic argument instead of identifying the moduli directly.
Let $V_{m,n}$ be the subspace of $CF(\fP,L_1) \otimes CF(L_0,\fP)$
generated by $o(\fq^\vee) \otimes o(\fp)$ such that $|\fq^\vee|=m$ and $|\fp|=n$.
The bijection \eqref{eq:biM3} comes from the bijection
$ \eM^{J^-}(\mathbf{q}'^\vee;x_{\fq,\fq'},\mathbf{q}^\vee) \simeq \eM^{J^{-}}(c_{\fp,\fq'};x_{\fq,\fq'},c_{\fp,\fq})$.
Therefore, for each $a \in \Z$, there is $f(a) \in \{0,1\}$
such that
\begin{align}
 |\Phi|_\K \circ (\mu^1 \otimes id)|_{V_{1,a}} =(-1)^{f(a)} \mu^{1,3} \circ |\Phi|_\K |_{V_{1,a}} \label{eq:FFcomm3}
\end{align}
We remark that the existence of $f$ follows from the fact that the sign only depends on $|\fp|$ and $|\fq^\vee|$ 
(because once $|\fp|$ and $|\fq^\vee|$ are determined, the local model computing the sign is determined).

By \eqref{eq:Fcomm3}, we have $\phi \circ (\mu^1 \otimes id)=-(\mu^1 \otimes id) \circ \phi$ so we get
\begin{align}
 (-1)^{1-k} |\Phi \circ \phi^{1-k}|_\K \circ (\mu^1 \otimes id)|_{V_{k,a+1-k}}=(-1)^{f(a)}\mu^{1,3} \circ |\Phi \circ \phi^{1-k}|_\K |_{V_{k,a+1-k}}
\end{align}
by precomposing \eqref{eq:FFcomm3} by $|\phi^{1-k}|_\K$.
By relabelling the subscripts, we have
\begin{align}
 |I|_\K \circ (\mu^1 \otimes id)|_{V_{m,n}}= (-1)^{f(m+n-1)+1-m} \mu^{1,3} \circ |I|_\K \label{eq:FFFcomm3}
\end{align}

The $A_{\infty}$-relations on $CF(\fP,L_1) \otimes CF(L_0,\fP)$ give
\begin{align}
 \mu^1 \circ \mu^2 + \mu^2 \circ (id \otimes \mu^1)+\mu^2 \circ (\mu^1 \otimes (-1)^{\deg-1})=0 \label{eq:AlgTrick1}
\end{align}
On the other hand, $CF(L_0, \tau_P(L_1))$ is a cochain complex so by considering the square of differential with input in $\cX_a(C_1)$
and output in $\cX_b(C_1)$, we get
\begin{align}
 \mu^1 \circ \mu^{1,2} + \mu^{1,2} \circ \mu^{1,1} +\mu^{1,2} \circ \mu^{1,3}=0 \label{eq:AlgTrick2}
\end{align}
Since we have already proved \eqref{eq:biMM1} and \eqref{eq:biMM2}, when we apply $|I|_\K$ to the left of \eqref{eq:AlgTrick1}
and on the right of \eqref{eq:AlgTrick2}, we get (after cancellation)
\begin{align}
 \mu^{1,2} \circ |I|_\K \circ (\mu^1 \otimes  (-1)^{\deg-1})=\mu^{1,2} \circ \mu^{1,3} \circ |I|_\K 
\end{align}
Applying it to $V_{m,n}$ and plugging in \eqref{eq:FFFcomm3}, we have
\begin{align}
 (-1)^{(f(m+n-1)+1-m)+(n-1)}\mu^{1,2} \circ \mu^{1,3} \circ |I|_\K =\mu^{1,2} \circ \mu^{1,3} \circ |I|_\K 
\end{align}
When $\mu^{1,2} \circ \mu^{1,3} \circ |I|_\K \neq 0$, it is possible only when $(f(m+n-1)+1-m)+(n-1)$ is even.
In particular, we have $f(a)=a-1$ modulo $2$.
Put it back to \eqref{eq:FFFcomm3}, we get \eqref{eq:biMM3}.
\end{proof}

\begin{proof}[Proof of Proposition \ref{p:CohLevelIso}]
 It follows from Lemma \ref{l:biMM1}, \ref{l:biMM2} and \ref{l:biMM3}.
\end{proof}


\begin{thebibliography}{CRGG15}

\bibitem[Abb04]{Abbas}
Casim Abbas.
\newblock Pseudoholomorphic strips in symplectizations. {II}. {F}redholm theory
  and transversality.
\newblock {\em Comm. Pure Appl. Math.}, 57(1):1--58, 2004.

\bibitem[Abo12a]{Abouzaid12}
Mohammed Abouzaid.
\newblock Nearby {L}agrangians with vanishing {M}aslov class are homotopy
  equivalent.
\newblock {\em Invent. Math.}, 189(2):251--313, 2012.

\bibitem[Abo12b]{Ab12}
Mohammed Abouzaid.
\newblock On the wrapped {F}ukaya category and based loops.
\newblock {\em J. Symplectic Geom.}, 10(1):27--79, 2012.

\bibitem[AHK05]{AHK05}
Paul~S. Aspinwall, R.~Paul Horja, and Robert~L. Karp.
\newblock Massless {D}-branes on {C}alabi-{Y}au threefolds and monodromy.
\newblock {\em Comm. Math. Phys.}, 259(1):45--69, 2005.

\bibitem[AL17]{AL17}
Rina Anno and Timothy Logvinenko.
\newblock Spherical {DG}-functors.
\newblock {\em J. Eur. Math. Soc. (JEMS)}, 19(9):2577--2656, 2017.

\bibitem[AS10a]{AbbS12}
Alberto Abbondandolo and Matthias Schwarz.
\newblock Floer homology of cotangent bundles and the loop product.
\newblock {\em Geom. Topol.}, 14(3):1569--1722, 2010.

\bibitem[AS10b]{AbouzaidSeidel}
Mohammed Abouzaid and Paul Seidel.
\newblock An open string analogue of {V}iterbo functoriality.
\newblock {\em Geom. Topol.}, 14(2):627--718, 2010.

\bibitem[Aur10]{Auroux10}
Denis Auroux.
\newblock Fukaya categories of symmetric products and bordered
  {H}eegaard-{F}loer homology.
\newblock {\em J. G\"okova Geom. Topol. GGT}, 4:1--54, 2010.

\bibitem[BC13]{BC13}
Paul Biran and Octav Cornea.
\newblock Lagrangian cobordism. {I}.
\newblock {\em J. Amer. Math. Soc.}, 26(2):295--340, 2013.

\bibitem[BC14]{BC2}
Paul Biran and Octav Cornea.
\newblock Lagrangian cobordism and {F}ukaya categories.
\newblock {\em Geom. Funct. Anal.}, 24(6):1731--1830, 2014.

\bibitem[BC17]{BCIII}
Paul Biran and Octav Cornea.
\newblock Cone-decompositions of {L}agrangian cobordisms in {L}efschetz
  fibrations.
\newblock {\em Selecta Math. (N.S.)}, 23(4):2635--2704, 2017.

\bibitem[BEH{\etalchar{+}}03]{SFTcompact}
F.~Bourgeois, Y.~Eliashberg, H.~Hofer, K.~Wysocki, and E.~Zehnder.
\newblock Compactness results in symplectic field theory.
\newblock {\em Geom. Topol.}, 7:799--888, 2003.

\bibitem[{Bou}02]{Bo02}
Frederic {Bourgeois}.
\newblock {A Morse-Bott approach to contact homology}.
\newblock {\em Ph.D. Thesis, Stanford University}, 2002.

\bibitem[CEL10]{CEL10}
K.~Cieliebak, T.~Ekholm, and J.~Latschev.
\newblock Compactness for holomorphic curves with switching {L}agrangian
  boundary conditions.
\newblock {\em J. Symplectic Geom.}, 8(3):267--298, 2010.

\bibitem[CDGG15]{CDGG}
Baptiste Chantraine, Georgios~Dimitroglou Rizell, Paolo Ghiggini, and Roman
  Golovko.
\newblock Floer theory for lagrangian cobordisms.
\newblock {\em arXiv:1511.09471}, 2015.

\bibitem[Dam12]{Da12}
Mihai Damian.
\newblock Floer homology on the universal cover, {A}udin's conjecture and other
  constraints on {L}agrangian submanifolds.
\newblock {\em Comment. Math. Helv.}, 87(2):433--462, 2012.

\bibitem[Dra04]{Dragnev}
Dragomir~L. Dragnev.
\newblock Fredholm theory and transversality for noncompact pseudoholomorphic
  maps in symplectizations.
\newblock {\em Comm. Pure Appl. Math.}, 57(6):726--763, 2004.

\bibitem[DS14]{DonovanSegal}
Will Donovan and Ed~Segal.
\newblock Window shifts, flop equivalences and {G}rassmannian twists.
\newblock {\em Compos. Math.}, 150(6):942--978, 2014.

\bibitem[DW16]{DonovanWemyss}
Will Donovan and Michael Wemyss.
\newblock Noncommutative deformations and flops.
\newblock {\em Duke Math. J.}, 165(8):1397--1474, 2016.

\bibitem[EES05]{EES05}
Tobias Ekholm, John Etnyre, and Michael Sullivan.
\newblock Orientations in {L}egendrian contact homology and exact {L}agrangian
  immersions.
\newblock {\em Internat. J. Math.}, 16(5):453--532, 2005.

\bibitem[EES07]{EES07}
Tobias Ekholm, John Etnyre, and Michael Sullivan.
\newblock Legendrian contact homology in {$P\times\Bbb R$}.
\newblock {\em Trans. Amer. Math. Soc.}, 359(7):3301--3335, 2007.

\bibitem[ESW]{ESW}
Jonathan~David Evans, Ivan Smith, and Michael Wemyss.
\newblock in preparation.

\bibitem[Gan]{Gan13}
Sheel Ganatra.
\newblock {Symplectic cohomology and duality for the wrapped Fukaya category}.
\newblock {\em arXiv:1304.7312}.

\bibitem[Gei08]{GeigesBook}
Hansj\"org Geiges.
\newblock {\em An introduction to contact topology}, volume 109 of {\em
  Cambridge Studies in Advanced Mathematics}.
\newblock Cambridge University Press, Cambridge, 2008.

\bibitem[Har11]{Ha11}
Richard Harris.
\newblock Projective twists in {$A_\infty$} categories.
\newblock {\em arXiv:1111.0538}, 2011.

\bibitem[Hin04]{Hind04}
R.~Hind.
\newblock Lagrangian spheres in {$S^2\times S^2$}.
\newblock {\em Geom. Funct. Anal.}, 14(2):303--318, 2004.

\bibitem[HK09]{HaghighatKlemm}
Babak Haghighat and Albrecht Klemm.
\newblock Topological strings on {G}rassmannian {C}alabi-{Y}au manifolds.
\newblock {\em J. High Energy Phys.}, (1):029, 31, 2009.

\bibitem[HLM]{HLM}
Hansol Hong, Siu-Cheong Lau, and Cheuk-Yu Mak.
\newblock in preparation.

\bibitem[HLS16]{HS16}
Daniel Halpern-Leistner and Ian Shipman.
\newblock Autoequivalences of derived categories via geometric invariant
  theory.
\newblock {\em Adv. Math.}, 303:1264--1299, 2016.

\bibitem[Hor05]{Horja05}
R.~Paul Horja.
\newblock Derived category automorphisms from mirror symmetry.
\newblock {\em Duke Math. J.}, 127(1):1--34, 2005.

\bibitem[HT06]{HT06}
Daniel Huybrechts and Richard Thomas.
\newblock {$\Bbb P$}-objects and autoequivalences of derived categories.
\newblock {\em Math. Res. Lett.}, 13(1):87--98, 2006.

\bibitem[Laz00]{La00}
L.~Lazzarini.
\newblock Existence of a somewhere injective pseudo-holomorphic disc.
\newblock {\em Geom. Funct. Anal.}, 10(4):829--862, 2000.

\bibitem[Laz11]{La11}
Laurent Lazzarini.
\newblock Relative frames on {$J$}-holomorphic curves.
\newblock {\em J. Fixed Point Theory Appl.}, 9(2):213--256, 2011.

\bibitem[Lip06]{Lip06}
Robert Lipshitz.
\newblock A cylindrical reformulation of {H}eegaard {F}loer homology.
\newblock {\em Geom. Topol.}, 10:955--1097, 2006.

\bibitem[Mat]{Mat}
Diego Matessi.
\newblock Lagrangian submanifolds from tropical hypersurfaces.
\newblock arXiv:1804.01469.

\bibitem[Mea]{Me16}
Ciaran Meachan.
\newblock A note on spherical functors.
\newblock arXiv:1606.09377.

\bibitem[Mik]{Mik}
Grigory Mikhalkin.
\newblock Examples of tropical-to-{L}agrangian correspondence.
\newblock arXiv:1802.06473.

\bibitem[MR]{MR18}
Cheuk-Yu Mak and Helge Ruddat.
\newblock Tropically constructed {L}agrangians in mirror quintic 3-folds.
\newblock in preparation.

\bibitem[MW]{MW15}
Cheuk-Yu Mak and Weiwei Wu.
\newblock Dehn twists exact sequences through {L}agrangian cobordism.
\newblock arXiv:1509.08028.

\bibitem[MWW18]{MWWfunctor}
S.~Ma'u, K.~Wehrheim, and C.~Woodward.
\newblock {$A_\infty$} functors for {L}agrangian correspondences.
\newblock {\em Selecta Math. (N.S.)}, 24(3):1913--2002, 2018.

\bibitem[Orl16]{Or16}
Dmitri Orlov.
\newblock Smooth and proper noncommutative schemes and gluing of {DG}
  categories.
\newblock {\em Adv. Math.}, 302:59--105, 2016.

\bibitem[Per]{Pe18}
Alexandre Perrier.
\newblock {S}tructure of $j$-holomorphic disks with immersed {L}agrangian
  boundary conditions.
\newblock arXiv:1808.01849.

\bibitem[RS93]{RobbinSalamon}
Joel Robbin and Dietmar Salamon.
\newblock The {M}aslov index for paths.
\newblock {\em Topology}, 32(4):827--844, 1993.

\bibitem[Seg18]{Se17}
Ed~Segal.
\newblock All autoequivalences are spherical twists.
\newblock {\em Int. Math. Res. Not. IMRN}, (10):3137--3154, 2018.

\bibitem[Sei00]{SeGraded}
Paul Seidel.
\newblock Graded {L}agrangian submanifolds.
\newblock {\em Bull. Soc. Math. France}, 128(1):103--149, 2000.

\bibitem[Sei03]{Se03}
Paul Seidel.
\newblock A long exact sequence for symplectic {F}loer cohomology.
\newblock {\em Topology}, 42(5):1003--1063, 2003.

\bibitem[Sei08a]{SeHom}
Paul Seidel.
\newblock {$A_\infty$}-subalgebras and natural transformations.
\newblock {\em Homology, Homotopy Appl.}, 10(2):83--114, 2008.

\bibitem[Sei08b]{Seidelbook}
Paul Seidel.
\newblock {\em Fukaya categories and {P}icard-{L}efschetz theory}.
\newblock Zurich Lectures in Advanced Mathematics. European Mathematical
  Society (EMS), Zurich, 2008.

\bibitem[Sei13]{Seidelbook2}
Paul Seidel.
\newblock {\em Categorical Dynamics and Symplectic Topology}.
\newblock http://www-math.mit.edu/\~{}seidel/. 2013.

\bibitem[Sei17]{LF2}
Paul Seidel.
\newblock Fukaya {$A_\infty$}-structures associated to {L}efschetz fibrations.
  {II}.
\newblock In {\em Algebra, geometry, and physics in the 21st century}, volume
  324 of {\em Progr. Math.}, pages 295--364. Birkh\"{a}user/Springer, Cham,
  2017.

\bibitem[She15]{NickSign}
Nick Sheridan.
\newblock Formulae in noncommutative hodge theory.
\newblock {\em arXiv:1510.03795}, 2015.

\bibitem[ST01]{ST01}
Paul Seidel and Richard Thomas.
\newblock Braid group actions on derived categories of coherent sheaves.
\newblock {\em Duke Math. J.}, 108(1):37--108, 2001.

\bibitem[Tod07]{Toda07}
Yukinobu Toda.
\newblock On a certain generalization of spherical twists.
\newblock {\em Bull. Soc. Math. France}, 135(1):119--134, 2007.

\bibitem[WW16]{WeWo16}
Katrin Wehrheim and Chris~T. Woodward.
\newblock Exact triangle for fibered {D}ehn twists.
\newblock {\em Res. Math. Sci.}, 3:Paper No. 17, 75, 2016.

\end{thebibliography}
\end{document}